\documentclass[10pt, a4paper]{article} 
\usepackage{stmaryrd}
\usepackage[a4paper, left=2cm, right=2cm, top=3.5cm, bottom=3.5cm]{geometry}
\usepackage{amsmath}
\usepackage{amssymb}
\usepackage{mathrsfs}
\usepackage{amsthm}
\usepackage[english]{babel}
\usepackage{xcolor}
\usepackage{verbatim}
\usepackage{hyperref}
\usepackage[nameinlink,capitalize]{cleveref}
\usepackage{cancel}
\usepackage{enumitem}
\usepackage{csquotes}
\usepackage{dsfont}
\usepackage{comment}
\usepackage[textsize=tiny]{todonotes}

\usepackage[
backend=biber,
style=alphabetic,
maxnames=99,
sorting=nyt,
doi=true,
isbn=false,
url=false,
]{biblatex}
\def\ud{\mathrm d}

\newcommand{\tnorm}[1]{\left\vert\kern-0.25ex\left\vert\kern-0.25ex\left\vert #1 
    \right\vert\kern-0.25ex\right\vert\kern-0.25ex\right\vert}
\newcommand{\R}{\mathbb{R}}
\newcommand{\E}{\mathbb{E}}

\newcommand{\MF}{{\mathcal{F}}}
\newcommand{\BBeta}{\boldsymbol{\eta}}
\newcommand{\BD}{\mathbf{D}}
\newcommand{\LL}{\mathbb{L}}
\newcommand{\HH}{\mathbb{H}}
\newcommand{\by}{{\bar{y}}}
\newcommand{\bz}{{\bar{z}}}
\newcommand{\be}{\begin{equation}}
\newcommand{\ee}{\end{equation}}
\newcommand{\esssup}{\mathop{\mathrm{ess\,sup}}}
\newcommand{\bk}[1]{\llbracket#1\rrbracket}

\theoremstyle{definition}
\newtheorem{defn}{Definition}[section]

\newtheorem{rmk}[defn]{Remark}

\newtheorem{assumptions}[defn]{Assumptions}
\theoremstyle{plain}
\newtheorem{thm}[defn]{Theorem}
\newtheorem*{thm*}{Theorem}
\newtheorem{lemma}[defn]{Lemma}
\newtheorem{coroll}[defn]{Corollary}
\newtheorem{prop}[defn]{Proposition}

\newtheorem*{assum*}{Assumptions}

\title{On the Pontryagin Maximum Principle for \\
Rough Stochastic Differential Equations}
\author{Peter K. Friz\thanks{Technische Universit\"at Berlin and Weierstraß-Institut (friz@math.tu-berlin.de)}, Ulrich Horst\thanks{Humboldt Universit\"at zu Berlin (horst@math.hu-berlin.de)}, Stefanos Theodorakopoulos\thanks{Technische Universit\"at Berlin (stefanos.theodorakopoulos@tu-berlin.de)},  Huilin Zhang\thanks{{Shandong University (huilinzhang@sdu.edu.cn)}}    
}
\allowdisplaybreaks

\begin{document}

\maketitle

\begin{abstract}
We establish first- and second-order Pontryagin (a.k.a. stochastic) maximum principles for rough stochastic control problems with general nonlinear $\mathrm{Lip}^\gamma$-rough vector fields. This complements the dynamic programming approach of Friz–Lê–Zhang \cite{10.1214/26-AOP1852},
and extends the maximum principle of Horst–Zhang \cite{horst2025pontryagin} 
beyond affine rough vector fields. The paper concludes with applications to pathwise stochastic optimal control.
\end{abstract}

\smallskip
\noindent \textbf{Keywords.} Stochastic optimal control, Rough paths

\smallskip
\noindent \textbf{MSC (2020).} 60H20, 93E20.

\tableofcontents

\section{Introduction}

Classical stochastic control concerns the choice of an adapted control that minimizes an expected cost associated with a stochastic differential equation. Its two principal approaches are dynamic programming and Pontryagin's maximum principle; see \cite{yong1999stochastic,pham2009continuous,carmona2018probabilistic}. Dynamic programming describes the value function through a Hamilton--Jacobi--Bellman equation, whereas the maximum principle expresses optimality through the state and backward adjoint processes. The theory of backward stochastic differential equations (BSDEs), initiated by Pardoux and Peng \cite{pardoux1990adapted}, provides the natural language for these adjoints; see also \cite{zhang2017backward}. When the diffusion depends on the control and the control space is nonconvex, the second-order maximum principle of Peng \cite{peng1990general} shows that a second adjoint is needed to account for the case of spike kindly perturbed diffusion.
An extension of this setting arises when the state is also affected by an exogenous signal whose realization is fixed throughout the optimization. The remaining randomness is still described by Brownian motions, but the fixed signal may be too irregular for ordinary differential calculus. 
Rough stochastic differential equations (RSDEs) provide a framework in which the two sources of noise can be treated simultaneously: one retains It\^o integration for the Brownian motion and uses rough integration for the fixed signal. This separates the probabilistic information available to the controller from the temporal irregularity of the exogenous environment. 
Additionally, separation between fixed rough noise and stochastic dynamics is useful beyond control. It underlies robust filtering with correlated noise and multidimensional observations \cite{CDFO13}, rough filtering equations and rough Kalman--Bucy theory \cite{bugini2025rough}, and conditional pricing in local stochastic volatility models \cite{bank2025rough}. It also appears in McKean--Vlasov equations with rough common noise \cite{friz2025mckeanvlasovequationsroughcommon}, linear-quadratic mean-field games with rough common noise \cite{friz2026meanfieldgamesroughcommon}, and the compactification approach to rough common-noise games \cite{bayraktar2026meanfieldgamesroughcommon}. These developments motivate optimality conditions for the underlying controlled RSDEs beyond the affine setting.

For mixed Brownian and rough dynamics, Diehl, Oberhauser and Riedel \cite{DOR15} constructed a joint lift of Brownian motion and a deterministic rough path, with applications to filtering and rough PDEs. In recent years, L\^e's stochastic sewing lemma \cite{le2020stochastic} made it possible to exploit adaptedness and conditional cancellations, leading Friz, Hocquet and L\^e \cite{FHL21} to rough stochastic integration and further to an intrinsically formulated RSDE theory with stochastic controlled vector fields. The linear RSDE theory was developed by Bugini, Coghi and Nilssen \cite{bugini2024malliavin}, Horst and Zhang \cite{horst2025pontryagin}. Related $p$-variation estimates and integration theories were established by Friz and Zorin-Kranich \cite{friz2023rough} and, in a setting allowing jumps, by Allan and Pieper \cite{allan2026rough}. 
Within the RSDE framework, Friz, L\^e and Zhang \cite{10.1214/26-AOP1852} established a rough dynamic programming principle and investigated the associated rough HJB equations. Their work on randomization \cite{FLZ25} clarifies how a fixed rough signal can be replaced by a random one and how the resulting dynamics relate to conditional laws. Subsequently, Horst and Zhang \cite{horst2025pontryagin} initiated the maximum-principle approach for systems with an affine rough coefficient. By employing a flow transformation they derived necessary and sufficient conditions which accommodate unbounded coefficients, including linear-quadratic models. 
In the present paper, we extend the maximum-principle approach by establishing necessary and sufficient conditions for such systems with a nonlinear rough coefficient. The drift and Brownian diffusion may both depend on the control, and the coefficients may be random. We develop a direct variational approach that combines estimates for forward rough SDEs, well-posedness of backward adjoint equations, and rough stochastic integration by parts. Our results cover both linear perturbations of Euclidean-valued controls and spike perturbations in a general separable metric control space. The essential point is to carry out the first- and second-order duality arguments when the rough coefficients of the variational and adjoint equations depend on the stochastic state.

\subsection{Controlled stochastic systems from a pathwise viewpoint}
The pathwise interpretation on stochastic control has a long history. Davis and Burstein \cite{DB92} studied anticipative control through deterministic control problems parametrized by Wiener paths and suitable Lagrange multipliers. Lions and Souganidis \cite{LS98,LS98b} developed a theory of fully nonlinear stochastic PDEs with applications to pathwise stochastic control. Buckdahn and Ma \cite{BM07} established a dynamic programming principle and a stochastic viscosity characterization of the value field, using a Doss--Sussmann transformation to relate the pathwise problem to an auxiliary stochastic control problem. Further connections between pathwise control and possibly non-adapted stochastic PDEs were investigated in \cite{BCO23}. These works explain why conditioning on an exogenous noise leads to questions that are not covered directly by the usual Markovian control theory.
Rough path theory, introduced by Lyons \cite{lyons1998differential}, gives a deterministic meaning to differential equations driven by irregular signals with extra algebraic (Lie) group structure. Gubinelli's controlled-path formulation \cite{gubinelli2004controlling} is particularly convenient for the expansions and integration arguments used here; we refer to \cite{FV10,friz2020course} for systematic description on both formulations of rough paths. A related application to fully nonlinear stochastic PDEs was developed by Caruana, Friz and Oberhauser \cite{caruana2011}. 
For the application of rough paths in optimal control, Diehl, Friz and Gassiat \cite{DiehlFrizGassiat2017} established a maximum principle, a rough HJB characterization, and a stochastic control duality formula for controlled rough differential equations. Allan and Cohen \cite{AC20} studied pathwise control in connection with robust filtering, including the degeneracy that can arise when the rough coefficient is directly controlled. Relaxed controls and rough HJB equations are considered by Chakraborty, Honnappa and Tindel \cite{CHT24}; see also \cite{ashkarian2026pontryaginmaximumprincipleqfunctions}.
Zheng and Hu \cite{zhenghu2024global} studied partially observed systems driven by Brownian and fractional Brownian motions, using transformations for systems whose fractional-noise coefficients are linear in the state. Recently, Lew \cite{lew2026rough} derived a first-order maximum principle for deterministic controls and Gaussian rough drivers, with a pathwise adjoint construction and an application to indirect shooting. 

We here derive maximum principles for a pathwise stochastic control problem where the system is driven by an internal Brownian noise and another external Brownian noise. In this setting, controls remain adapted to the internal Brownian motion driving the It\^o integral and may or may not depend (see Remark \ref{rmk 5.3}) on the entire trajectory of the external noise, thereby allowing anticipation of its future evolution. The objective is to minimize the conditional expected cost given this trajectory. By lifting the external Brownian motion to its geometric rough path and freezing its realization, we reduce the problem to the rough stochastic control framework developed in earlier sections. The main results then 
yield Hamiltonian necessary conditions for conditional optimality, expressed through rough adjoint BSDEs, together with sufficient conditions under the stated convexity and rough-remainder assumptions.
These expand the maximum-principle approach to pathwise control of \cite{horst2025pontryagin} to nonlinear rough vector fields, and complement the dynamic programming approach of \cite{10.1214/26-AOP1852} through a forward–backward characterization of optimal controls.

\subsection{Problem formulation and main results}\label{subsec:theProblem}
Let $W$ be a $w$-dimensional Brownian motion and let
$\boldsymbol{\eta} \in \mathscr{C}^{0,\alpha}_g([0,T];\mathbb{R}^\ell)$ be a (deterministic) geometric rough path with
$\alpha \in (\frac{1}{3},\frac{1}{2})$. For an admissible control $u \in \mathcal{A}$, we consider the $\mathbb{R}^d$-valued state process governed by the rough stochastic differential equation (RSDE)
\begin{equation}\label{eq:rough_stochastic_dynamics}
\begin{cases}
\ud X^{\boldsymbol{\eta},u}_s
=
b(s,X^{\boldsymbol{\eta},u}_s,u_s)\,\ud s
+
\sigma(s,X^{\boldsymbol{\eta},u}_s,u_s)\,\ud W_s
+
(f,f')(s,X^{\boldsymbol{\eta},u}_s)\,\ud\boldsymbol{\eta}_s,
\qquad s \in [t,T], \\
X^{\boldsymbol{\eta},u}_t
= x.
\end{cases}
\end{equation}
Here the Brownian integral is understood in the It\^o sense, and $(f,f')$ is a stochastic controlled vector field relative to $\eta$, with the prime representing its Gubinelli derivative (see e.g. \cite{friz2020course}). This formulation allows us to admit progressively measurable controls with respect to the Brownian filtration, while retaining control dependence in the Brownian diffusion. Throughout, we suppress the dependence of the coefficients on $\omega$ when no confusion can arise.

The objective is to minimize the cost functional
\begin{equation}\label{eq: cost_function_intro}
\mathcal{J}^{\boldsymbol{\eta}}(u;t,x)
=
\mathbb{E}\bigg[
\int_t^T h\bigl(s,X^{\boldsymbol{\eta},u}_s,u_s\bigr)\,\ud s
+
g\bigl(X^{\boldsymbol{\eta},u}_T\bigr)
\bigg]
\end{equation}
over a class $\mathcal{A}$ of progressively measurable controls on $[t,T]$. The corresponding value function is
\begin{equation}\label{eq:rough-vf}
\mathbb{V}^{\boldsymbol{\eta}}(t,x)
:=
\inf_{u \in \mathcal{A}} \mathcal{J}^{\boldsymbol{\eta}}(u;t,x).
\end{equation}
The analytic results we obtained lead to two forms of Pontryagin's maximum principle. The first gives a derivative formula and a Hamiltonian minimization condition for Euclidean-valued controls and coefficients differentiable in the state and control variables. The second is a rough stochastic counterpart of the second-order principle in \cite{peng1990general}, allowing a nonconvex control space and a controlled Brownian diffusion, differentiability is required only in the state variable, with stronger spatial regularity to justify the second-order expansion.

For the first maximum principle, we consider linear perturbations $u^\epsilon=u+\epsilon v$. Under Assumptions~\ref{A1}--\ref{A3} and \ref{assu: lin PMP}, we prove Gateux differentiability of the cost functional and represent its derivative through a first adjoint equation. More precisely, let
\[
H(\omega,s,x,y,z,u)
:=y^\top b(\omega,s,x,u)+\operatorname{Tr}\bigl(z^\top\sigma(\omega,s,x,u)\bigr)+h(\omega,s,x,u).
\]
Writing $X=X^{\boldsymbol{\eta},u}$, the adjoint pair $((Y,Y'),Z)$, solves
\[
\begin{split}
Y_s={}&\partial_x g(X_T)^\top
+\int_s^T \partial_x H(r,X_r,Y_r,Z_r,u_r)^\top\,\ud r\\
&+\int_s^T \partial_x f(r,X_r)^\top Y_r\,\ud\boldsymbol{\eta}_r
-\int_s^T Z_r\,\ud W_r.
\end{split}
\]
The rough integral is understood with the associated controlled derivative (computable from $Y'_r= \partial_x f(r,X_r)^\top Y_r$). 
The following is a loose formulation of our first findings, with full details left to  Proposition~\ref{Der_formula_lin} and Theorem~\ref{thm:roughpmp_lin}.

\begin{thm}[Linear perturbation]
The functional $\mathcal{J}^{\boldsymbol{\eta}}$ is Gateux differentiable, and for every $v\in\mathcal{A}$ its derivative at the point $u\in\mathcal{A}$ is given by 
\[
\frac{\ud}{\ud\epsilon}\mathcal{J}^{\boldsymbol{\eta}}(u+\epsilon v;t,x)\Big|_{\epsilon=0}
=\mathbb{E}\bigg[\int_t^T
\partial_u H(s,X_s,Y_s,Z_s,u_s)\cdot v_s\,\ud s\bigg].
\]
Thus, an optimal control $(u_t)_{t\in[0,T]}\in \mathcal{A}$ of $\mathcal{J}^{\boldsymbol{\eta}}$ satisfies the usual stationarity condition
$$
 \partial_u H(s, X_s, Y_s, Z_s, u_s) = 0,\hspace{.2cm} \mathbb{P}\otimes\text{Leb}(\ud s)-a.s.
 $$
 If additionally $H$ is convex in the control variable, this yields the standard pointwise minimization of the Hamiltonian, where for every $u\in\mathbb{R}^\kappa$ we have
 $$
H(s,X_s,Y_s,Z_s,u_s) \leq H(s,X_s,Y_s,Z_s,u),\hspace{.2cm} \mathbb{P}\otimes\text{Leb}(\ud s)-a.s.
 $$
 Notice that although the rough coefficient does not enter $H$ explicitly, it affects both the state and the adjoint processes.
\end{thm}

Our second maximum principle applies to controls with values in a general separable metric space $U$.
We replace an optimal control $\bar u$ by another admissible control on an interval of length $\epsilon$. The resulting spike perturbation of the Brownian integral has size $\epsilon^{1/2}$, so its quadratic contribution must be retained when expanding the cost to order $\epsilon$.  Under Assumptions~\ref{P-control}--\ref{D-cost} and \ref{assu: per PMP} we construct first and second variations $V^\epsilon$ and $\zeta^\epsilon$ and prove
\[
V^\epsilon=O(\epsilon^{1/2}),\quad
\zeta^\epsilon=O(\epsilon), \qquad X^{\boldsymbol{\eta},u^\epsilon}
=\bar X+V^\epsilon+\zeta^\epsilon+o(\epsilon),
\]
in the moment and stochastic controlled $p$-variation norms specified in Lemma~\ref{prop:diffusion-spike-estimates}. Using again the same Hamiltonian, but now denoted with $\bar{H}$ due to the domain difference at the control argument, we consequently denote the corresponding first adjoint by $((\bar Y,\bar{Y}'),\bar Z)$. Furthermore, the second-order duality argument introduces a matrix-valued second adjoint pair $((\bar{P},\bar{P}'),\bar{Q})$, which in addition to terms linear in $\bar{P}$, its rough driver contains the contraction of the first adjoint with $\partial_{xx}^2 f(\cdot,\bar X)$; see \eqref{eq 88}. This is precisely the additional term caused by the nonlinearity of the rough coefficient. 
Once again, we summarize our main findings as stated in Theorem~\ref{thm:roughpmp} in the following informal Theorem.
\begin{thm}[Spike perturbation]
 Let $((\bar{X},\bar{X}'),(\bar u_t)_{t\in[0,T]})$ be an optimal pair of $\mathcal{J}^{\boldsymbol{\eta}}$. Then, for every $u\in U$ we have
\[
\begin{split}
\bar H(s,\bar X_s,\bar Y_s,\bar Z_s,\bar u_s)
-\bar H(s,\bar X_s,\bar Y_s,\bar Z_s,u) -\frac12\operatorname{Tr}
\bigl(\delta\sigma_s(u)^\top \bar{P}_s\delta\sigma_s(u)\bigr)\leq0, \hspace{.2cm} \mathbb{P}\otimes\text{Leb}(\ud s)-a.s.,
\end{split}
\]
with
$\delta\sigma_s(u):=\sigma(s,\bar X_s,u)-\sigma(s,\bar X_s,\bar u_s)$. The quadratic correction accounts for the controlled Brownian diffusion. When $\sigma$ is independent of the control, it vanishes and the condition reduces to pointwise minimization of $\bar H$.
\end{thm}

For an application to pathwise stochastic control problems, let a product space $\Omega=\Omega'\times\Omega''$, with $W=W(\omega')$ an idiosyncratic Brownian motion and $W^\perp=W^\perp(\omega'')$ an independent Brownian motion, lifted to a Brownian rough path $\textbf{W}^\perp$. Evaluating the Euclidean-valued pathwise controls, with a suitably modified admissibility conditions, at
\[
    \boldsymbol{\eta}=\textbf{W}^\perp(\omega'')
\]
produces a conditional expected cost functional as a {\sl random} objective, in the {\sl essinf}-sense, of the form
\begin{equation*}
    \overline{\mathcal{J}}(\overline u;t,x;\omega):= \mathbb E \bigg[  \int_t^T h(s, \overline X^{\overline u}_s, \overline u_s) \ud s + g(\overline X^{\overline u}_T) \bigg| \mathcal{F}^{\perp}_T \bigg],
\end{equation*}
where $\overline{X}^{\overline{u}}$ satisfies the equation
\begin{equation*}
\begin{cases}
\ud \overline{X}^{\overline{u}}_s
=
b(s,\overline{X}^{\overline{u}}_s,\overline{u}_s)\,\ud s
+
\sigma(s,\overline{X}^{\overline{u}}_s,\overline{u}_s)\,\ud W_s(\omega')
+
f(s,\overline{X}^{\overline{u}}_s)\ud W^{\perp}_s (\omega''),
\quad s \in [t,T], \\
\overline{X}^{\overline{u}}_t
= x.
\end{cases}
\end{equation*}
Furthermore, under Assumptions \ref{(A1-R)}--\ref{(A3-R)}, \ref{assu: lin PMP}, for every $\omega''\in\Omega''$ let $((Y,Y'),Z)(\omega'')$ be the per $\boldsymbol{\eta} = \textbf{W}^{\perp} (\omega'')$ adjoint pair. Once again, the following informal Theorem summarizes our main findings in the case of linear perturbations as stated in Proposition~\ref{thm-PMP-for-psc}.

\begin{thm}[Pathwise SMP]
An optimal control $(\overline{u}_t)_{t\in[0,T]}\in \overline{\mathcal{A}}$ of $\overline{\mathcal{J}}$ satisfies $\mathbb{P}''-\text{a.s.}$ the usual stationarity condition
 $$
        \partial_u H(s,\overline{X}_s^{\overline{u}}, Y_s, Z_s, \overline{u}_s) = 0,\hspace{.2cm} \mathbb{P}'\otimes\text{Leb}(\ud s)-a.s.
 $$
If additionally $H$ is convex in the control variable, this yields the standard pointwise minimization of the Hamiltonian, where for every $u\in\mathbb{R}^\kappa$ we have
\begin{equation*}
        H(s,\overline{X}_s^{\overline{u}},Y_s,Z_s,\overline{u}_s) \leq H(s,\overline{X}_s^{\overline{u}},Y_s,Z_s,u),\hspace{.2cm} \mathbb{P}'\otimes\text{Leb}(\ud s)-a.s.
        \end{equation*}
\end{thm}
Analog results can be obtained in the case of spike perturbations.

\subsection{Main ideas}
A formal differentiation of the state equation can suggest the appropriate variational equations, but of course it does not justify the maximum principles by itself. One must show that the variational and adjoint equations are well-defined, that the produced remainders are small in seminorms compatible with rough integration, and that the forward--backward products admit a valid integration-by-parts formula. Overall, four main features of the nonlinear problem determine the analysis.

$\bullet$ \emph{Global estimates for affine rough SDEs.}
As in the classical setting, our analysis repeatedly requires us to compare states corresponding to different controls. This involves estimating differences between functions evaluated along distinct state trajectories. The natural approach is to linearize these differences, 
with exact Taylor remainders expressed in terms of averages of derivatives along the line segments joining the two trajectories. This leads to affine RSDEs with additional parameters accounting for these averaged coefficients, see Subsection \ref{sec 3.1} for details.
Unlike in the classical setting, however, no suitable form of Grönwall's inequality is readily available in our framework due to the presence of rough integrals. To overcome this difficulty, in Theorems~\ref{thm 3.9} and \ref{thm 3.11}, we establish global estimates for affine rough SDEs in both Hölder and $p$-variation settings. These estimates rely on an alternative iteration procedure based on rough path estimates. They provide uniform control of the linearized equations relative to the nonlinear remainders as the perturbation parameter tends to zero, and constitute our main tool for deriving the SMPs.

$\bullet $ \emph{Preserving the scale of spike perturbations in rough stochastic systems.}
One can immediately observe that the second-order cost expansion must have an error $o(\epsilon)$. This, among other things, dictates the rates of our approximations. For example, a spike in the Brownian diffusion should produce a state variation of order $\epsilon^{1/2}$. However, a direct use of H\"older estimates can lose the required powers of $\epsilon$. To see the issue already for just a Brownian forcing term, let $I$ be an interval of length $\epsilon$ and set
$M^\epsilon_t :=\int_0^t\mathds{1}_I(r)\ud W_r$.
For $k\geq2$, $\alpha<1/2$ and $p>2$, its increments satisfy (see Section \ref{sec:prelim} for definition of related seminorms)
\[
\sup_{s<t}\frac{\|\delta M^\epsilon_{s,t}\|_k}{|t-s|^\alpha}
\asymp\epsilon^{1/2-\alpha}
\quad\text{and}\quad
\|\delta M^\epsilon\|_{p\text{-var},k}
\asymp\epsilon^{1/2}.
\]
Hence, smallness in a global H\"older seminorm does not retain the natural spike scale. Our estimates under $p$-variation seminorms for affine RSDEs provide the sharper bounds designed to solve such loss of regularity. Combined with second-order expansions of the stochastic controlled vector fields, it yields the $O(\epsilon^{1/2})$ for the first variation,  $O(\epsilon)$ for the second variation, and  $o(\epsilon)$ remainder error under the related stochastic rough seminorms, see Lemma \ref{prop:diffusion-spike-estimates} for details.

$\bullet$ \emph{Adjoints multidimensional rough BSDEs with stochastic coefficients.}
The first adjoint (similar for the second adjoint) equation contains the rough integral of $\partial_x f(\cdot,X)^\top Y$ (and its Gubinelli derivative) while its martingale term involves an unknown integrand $Z$ controlled in an $\mathbb H^m$ norm, which does not by itself give a fixed H\"older modulus in time. 
The backward equations required by this approach are related to the rough BSDEs originated by Diehl and Friz \cite{Diehl2012}, who used a flow transformation and quadratic BSDE theory in the scalar case. Liang and Tang \cite{LiangTang2025} developed a multidimensional theory, including a general square-integrable solution theory for affine rough drifts with random, time-dependent coefficients. More recently, Friz et al~\cite{FrizSongZhangZhang2026} developed an intrinsic theory of nonlinear rough BSDEs using rough semimartingales and conditional $p$-variation norms with BMO-type properties. 
However, \cite{LiangTang2025} only applies to the case where coefficients of the rough integrands have uniformly bounded rough remainders (see the beginning of Subsection \ref{sec:a-rbsde} for details). In contrast to our adjoint equations, the rough coefficients of our rough BSDEs have unbounded remainders since they nonlinearly depend on the forward rough SDE and the martingale part of rough SDE is unbounded. On the other hand, rough BSDEs considered in \cite{FrizSongZhangZhang2026} only work for the one dimensional case (for the multidimensional case they need an extra a priori condition). Thus, we develop an affine rough BSDE theory in a general affine and multidimensional case in Subsection \ref{sec:a-rbsde} for our dual argument (see Remark \ref{rmk:multidimensional-rough-bsde} for more details).  

$\bullet$ \emph{Duality with compatible regularity and moments.}
The integrability of the aforementioned solutions to rough BSDEs is restricted by the variational parameter, see Theorem~\ref{ex_uni_lin_bsde}. The main reason for this is that we would like the drift and diffusion parts of the solution to not contribute to its Gubinelli derivative, see Lemmas \ref{lem:maring} and \ref{lem:smoothint}. Moreover, the previously established product rules (or Ito's formulas) for rough stochastic processes require strong integrability conditions on the drift and diffusion coefficients, see \cite{FHL21}, \cite{bugini2024malliavin}. To conduct the dual argument, we introduce  Theorem~\ref{rough_product_formula}, with substantially weaker integrability conditions on the coefficients, to give the required product rule in a mixed H\"older--p-variation setting, see also Remark \ref{rmk 3.18}. The moment restrictions in Assumptions~\ref{assu: lin PMP} and \ref{assu: per PMP} ensure that the products and stochastic integrals appearing in the duality argument are integrable in a compatible way. Geometric rough integration supplies the chain-rule cancellations, while the Brownian quadratic covariation produces the diffusion correction in the maximum principle. Together, these tools justify the variational argument directly for the rough stochastic system.

\subsection{Structure of the paper}
The remainder of the paper is organized as follows. Section~\ref{sec:prelim} recalls the required notation and results on stochastic controlled rough paths in H\"older and $p$-variation seminorms. Section~\ref{sec:tool_box} develops the estimates for affine rough SDEs, the well-posedness theory for rough BSDEs, and the rough stochastic product rule. Section~\ref{sec:max_principles} applies these tools to stochastic optimal control: Subsection~\ref{subsec: 4.1} treats linear perturbations and the first-order maximum principle, while Subsection~\ref{subsec: 4.2} establishes the second-order expansion for spike perturbations and the maximum principle for a general control space. In Section \ref{sec:application} we apply our maximum principles to the pathwise stochastic control problem.

\section{Preliminaries}\label{sec:prelim}
\subsection{Frequently used notation}\label{subsec:FUN} 
For an $n-$tensor $A,$ we use $A^\top$ to denote matrix transposes when $n=2$ and (one of the) well-defined pairwise tensor transposes for $n>2.$  By slight abuse of notation, different pairwise tensor transposes will not be distinguished when they can be inferred from the context.

For every modulus of continuity $\overline{\omega}:[0,\infty) \rightarrow [0,\infty)$ that appears below, besides the standard properties of being nondecreasing, continuous at $0$ with $\overline{\omega}(0)=0$, we will always also assume that is subadditive.

For every pair $n,n'\in\mathbb{N}$ we shall denote by $\mathbb{R}^{n\times n'}$ the space of $n\times n'$-matrices with real entries endowed with its Euclidean norm, \emph{i.e.}, for $A\in\mathbb{R}^{n\times n'}$ $|A|^2:=\textrm{Tr}[AA^\top]$, where $A^{\top}$ is the transpose of $A$. We shall frequently identify $\mathbb{R}^n$ with $\mathbb{R}^{n \times 1}$.

Throughout this work we fix a complete filtered probability space $(\Omega,\mathcal F,(\mathcal F_t)_{t\in[0,T]},\mathbb{P})$ which satisfies the usual conditions, $k\in(0,\infty),T\in\mathbb{R}_+, d,w \in\mathbb{N}$, and we shall use $(V,|\cdot|_{V})$, $(Q,|\cdot|_{Q})$, $(E,|\cdot|_{E})$ to denote finite dimensional Banach spaces. Moreover, we will use ${L}(V,Q)$ to denote the space of linear maps between $V,Q$, and $B(V\times Q, E)$ to denote the space of bilinear maps between $V\times Q,E$. The Borel $\sigma-$algebra on a metric space $\mathcal{X}$ is denoted by $\mathcal{B}(\mathcal{X}).$ For our analysis we are going to need the following spaces:

    \begin{align*}
\mathbb{L}^k(\mathcal{F}_t;V)&:=\Big\{Z:\Omega\rightarrow V:\mathcal{F}_t-\text{measurable and}\hspace{.2cm}\|Z\|_k:=\mathbb{E}\big[|Z|^k\big]^{\frac{1}{k}}<\infty\Big\},\\
\mathbb{L}^\infty(\mathcal{F}_t;V)&:=\Big\{Z:\Omega\rightarrow V:\mathcal{F}_t-\text{measurable and}\hspace{.2cm}\|Z\|_\infty:=\esssup_\omega |Z(\omega) |<\infty\Big\},\\
    \mathbb{S}^k([0,T];V)&:=\Big\{Z: [0,T]\times\Omega\to V: \text{ prog. measurable and}\hspace{.2cm} \|Z\|_{\mathbb{S}^k}:=\mathbb{E}\Big[\sup_{t\in[0,T]}|Z_t|^k\Big]^{\frac{1}{k}}<\infty\Big\},
\end{align*}
and
\begin{align*}\label{eq:bbHspace} \mathbb{H}^k([0,T];\mathbb{R}^{d\times w}):=\bigg\{Z&: [0,T]\times\Omega\to\R^{d\times w}:\text{ prog. measurable and} \hspace{.2cm}
\|Z\|_{\mathbb{H}^k}:=\mathbb{E}\bigg[\Big(\int_0^T |Z_t|^2\ud t\Big)^{\frac{k}{2}}\bigg]^{\frac{1}{k}}<\infty\bigg\}.
  \end{align*}

Moreover, we denote by $(\mathcal{C}_b(V;Q),|\cdot|_{\infty}) $ the Banach space of continuous and bounded maps, namely
	\begin{equation*}
    \mathcal{C}_b(V;Q)=\left \{f: V \rightarrow Q :\text{ continuous and  }
	 |f|_\infty<\infty
	 \right \},\quad |f|_\infty:=\sup_{x\in V}|f(x)|\,.
	\end{equation*}
	For every $\alpha\in(0,1]$ and every function $g:V\to Q$, we denote by $[g]_\alpha$ its H\"older seminorm, \emph{i.e.}
	\begin{equation*}
		[g]_\alpha=\sup_{x,y\in V:x\neq y}\frac{|g(x)-g(y)|_Q}{|x-y|_V^\alpha}.
	\end{equation*}
	For $\gamma=N+\alpha$ where $N$ is a non-negative integer and $0<\alpha\le 1$, $\mathcal{C}^\gamma_b(V;Q)$ denotes the Lipschitz space of bounded functions $f: V\rightarrow Q$ such that $f$ has Fr\'echet derivatives up to order $N$, $D^i f$, $i = 1,\dots,N$ are bounded functions  and $D^Nf$ is globally H\"older continuous with exponent $\alpha$.
	Recall that for each $v\in V$, $Df(v)\in L(V,Q)$, $D^2f(v)\in L(V\otimes V,Q)$\footnote{Note that $
{B}(V \times V,Q)
\cong
{L}(V \otimes V,Q).
$} and so on.
Whenever clear from the context, we simply write  $\mathcal{C}^\gamma_b$ (or $\text{Lip}^\gamma$) for $\mathcal{C}_b^\gamma(V;Q)$.
For each $f$ in $\mathcal{C}_b^\gamma(V;Q)$, we have seminorm and norm, respectively, given by
	\begin{equation*}
			[f]_\gamma :=\sum\nolimits_{i=1}^N|D^i f|_\infty+[D^Nf]_\alpha
			\hspace{.2cm}\text{and}\hspace{.2cm} |f|_\gamma :=|f|_\infty+[f]_\gamma.
	\end{equation*}

Finally, we define the mixed moment norms, see \cite[Subsection 2.2]{FHL21}. 
Let $1< m\leq n \le \infty$, and  
fix a sub-sigma-field $\mathcal{G} \subset \mathcal{F}$.
For a $V$-valued r.v. $\xi(\omega)$, we have $|\xi |_{V}  \in \mathbb{L}^{0,+}(\mathcal{F};\R)$ so that  
\begin{equation*}\label{Lmn_norm}
		 \|\xi |\mathcal{G} \|_m := \E[| \xi |_{V}^m | \mathcal{G}]^{\frac{1}{m}} \in \mathbb{L}^{0,+}(\mathcal{G};\R), \quad \| \xi \|_{m, n} :=
   \left\{\begin{array}{l}
     \|\| \xi |\mathcal{G} \|_m \|_n \qquad \text{if}\hspace{.1cm} \| \xi | \mathcal{G} \|_m \in \mathbb{L}^n,\\
     + \infty \text{\qquad \qquad  \quad  otherwise}
   \end{array}\right. 
\end{equation*}
is well-defined (the cases $m$ or $n=\infty$ are understood in $\esssup$-sense).  This yields a scale of  mixed moment spaces, 
\begin{equation*}\label{Lmn_space}
\mathbb{L}^{m,n} (V) = \mathbb{L}^{m,n}_{\mathcal{G}, \mathcal{F}} (V) := \{ \xi \in \mathbb{L}^0(V) :  \|\xi\|_{m,n} < \infty \}.
\end{equation*}
Note that $(\mathbb{L}^{m,n}(V),\|\cdot\|_{m,n})$ is a Banach space which coincides with $(\mathbb{L}^{m}(V),\|\cdot\|_{m})$ when $m=n$.

\begin{rmk}
We remind that throughout the present work, to ease the heavy notation, we shall omit the $\omega$-dependence from the notation whenever its possible.
\end{rmk}
    
\subsection{Rough path theory} \label{subseec:RPpreliminaries}
\subsubsection{H\"older seminorms}
Let us recall some basic results on rough paths and rough stochastic integrals. We refer to the textbook \cite{friz2020course} for more details on rough analysis. 

In what follows, $\alpha \in (\frac13, \frac12)$ is a fixed H\"older exponent corresponding to the underlying rough path and $(V,|\cdot|_V)$ a finite dimensional Banach space.  
For a one-parameter function/path $\eta:[0,T] \rightarrow V, s,t\in[0,T]$ we set $\delta \eta_{s,t}:= \eta_t - \eta_s.$ The space of $V$-valued $\alpha$-H\"older continuous functions is denoted $C^\alpha([0,T];V)$. For any two-parameter function $A: \Delta_{[s,t]} \rightarrow V,$ where
\[
    \Delta_{[s,t]} := \{(u,v) \in [s, t]^2 ~ | ~ s \leq u < v \leq t\},
\]    
the difference operator $\delta A_{s,u,t}$ is defined by
$$
    \delta A_{s,u,t}:= A_{s,t}-A_{s,u}- A_{u,t},\ \ \ s \leq u \leq t,
$$
and the $\alpha$-H\"older seminorm is denoted
$$
\lvert A \rvert_{\alpha,[s,t]} := \sup_{(u, v) \in \Delta_{[s, t]}} \frac{\lvert A_{u,v} \rvert_V}{|v - u|^\alpha}.
$$
The set of $\alpha$-H\"older continuous two-parameter functions is denoted $C^\alpha_2([0,T];V)$. More generally, for a random two-parameter $V$-valued function $A: \Omega \times \Delta_{[s,t]} \rightarrow V$ we set 
$$
\lVert A \rVert_{\alpha, k,[s,t]} := \sup_{(u,v) \in \Delta_{[s, t]}} \frac{\lVert A_{u, v} \rVert_k}{\lvert v - u \rvert^\alpha}.
$$ 

The space of all functions for which the above seminorm is finite is denoted $C^\alpha_2([0,T];V;\mathbb{L}^k)$ and $C^\alpha_2 \mathbb{L}^k$ denotes the subspace of all $\mathcal{F} \otimes \mathcal{B}(\Delta_{[0,T]})/\mathcal{B}(V)$-measurable elements of $C^\alpha_2([0,T];V;\mathbb{L}^k )$. 

In the rest of the work, we usually omit the time interval in our notations and norms if $[s,t]=[0,T]$ and write `$\lesssim_{\theta}$' to indicate that an inequality holds up to a generic constant that depends on a parameter $\theta$, or if $\theta$ is independent to the argument we will simply write `$\lesssim$'.

From now on, for simplicity reasons, and due to the natural isomorphisms between finite dimensional Banach spaces, $\eta$ will take values on $\mathbb{R}^\ell$ for some $\ell >0$.

\begin{defn}\label{defn 2.1}

\begin{itemize}
\item[(i)] A pair of (deterministic) continuous functions $\boldsymbol{\eta}:=(\eta, \eta^{(2)}):[0,T] \times\Delta([0,T]) \rightarrow \R^\ell \times \R^{\ell \times \ell}$ is called a $\alpha$-H\"older continuous {\it rough path}, if the following two conditions are satisfied.
\[
  |\delta \eta |_{\alpha;[0,T]} + |\eta^{(2)} |_{2\alpha;[0,T]} < \infty \ \ \text{and}\ \ \eta^{(2)}_{s,t}- \eta^{(2)}_{s,u}- \eta^{(2)}_{u,t}= \delta \eta_{s,u} \otimes \delta \eta_{u,t},\ \text{ for any $s\leq u \leq t$.}
\]  
The space of $\alpha$-H\"older continuous rough paths is denoted  $\mathscr{C}^{\alpha}([0,T];\R^\ell)$ and is equipped with the following rough path norm and rough path metric: for any $\boldsymbol{\eta}=(\eta,\eta^{(2)}), \bar{\boldsymbol{\eta}}=(\bar \eta, {\bar \eta}^{(2)} ) \in \mathscr{C}^\alpha([0,T];\R^\ell),$ 
$$
\tnorm{\boldsymbol{\eta}}_\alpha :=|\eta_0|+ \lvert \delta \eta \rvert_{\alpha;[0,T]} +\lvert \eta^{(2)}\rvert_{2 \alpha;[0,T]}^{\frac{1}{2}} ,\hspace{.1cm}  
\rho_{\alpha}(\boldsymbol{\eta}, \bar{\boldsymbol{ \eta}}) := |\eta_0-\bar{\eta}_0|+\lvert \delta \eta - \delta \bar \eta \rvert_{\alpha;[0,T]} + \lvert \eta^{(2)} - \bar \eta^{(2)} \rvert_{2\alpha;[0,T]}.
$$

\item[(ii)] With $\mathscr{C}^{0,\alpha}_g([0,T];\R^\ell)$ we shall denote the space of $\alpha$-H\"older geometric rough paths, defined as the completion of canonically lifted smooth paths (\emph{i.e.} if $\eta$ is smooth, $\boldsymbol{\eta}:=(\eta, \int \eta  \ud\eta)$ is the canonical lift of $\eta$) under the rough path metric $\rho_\alpha$. Note that $\mathscr{C}^{0,\alpha}_g([0,T];\R^\ell)$ is Polish and $\mathscr{C}^{0,\alpha}_g([0,T];\R^\ell) \subset \mathscr{C}^{\alpha}([0,T];\R^\ell)$ (see e.g. \cite{friz2020course} for further details).

\end{itemize}
\end{defn}

\begin{rmk}
We shall always endow $\mathscr{C}^{0,\alpha}_g([0,T];\R^\ell)$ with the rough path metric $\rho_\alpha$. Moreover, we 
may write $\rho_\alpha(\boldsymbol{\eta})$ instead of $\rho_\alpha(\boldsymbol{\eta},0)$.
\end{rmk}

To introduce stochastic controlled rough paths and to define rough integrals, we introduce for any integrable random two-parameter process $A$ the two-parameter conditional expectations process $\E_\bullet A$ by
$$
(s,t;\omega)\mapsto \E[A_{s,t}|\mathcal{F}_s](\omega) =: \E_s[A_{s,t}](\omega).
$$
At this point we set $\bar{V} := L(\mathbb{R}^\ell,V)$ and $\bar{Q} := L(\mathbb{R}^\ell,Q)$, another (finite-dimensional) Banach spaces.

\begin{defn}[{\cite[Definition 3.1]{FHL21}}](Stochastic controlled rough paths)\label{def:contro-rp} 
Let $\boldsymbol{\eta} \in \mathscr{C}^{\alpha}([0,T];\R^\ell)$ be a rough path and $\beta, \beta' \in (0,1]$, $m\in[2,\infty)$ and $n\in[m,\infty]$.
A pair $(X,X')$ is called a \textit{stochastic controlled rough path} with values in $V$ of $(m,n)$-integrability and $(\beta,\beta')$-H\"older regularity with respect to $(\mathcal{F}_t)_{t\in[0,T]}$ if the following hold: 

\begin{enumerate}  
\item[(i)]  The pair
	   \[
			(X,X') : \Omega \times [0,T] \rightarrow  V \times 
            \bar{V}
	   \]
    is progressively measurable and 
     \[
        \sup_{t\in[0,T]}\|X'_t\|_{n}< \infty\,.
	   \]

\item[(ii)] 
Letting
		\begin{align*}\label{def.bk}
				\|Z\|_{\mu;m,n}:=\sup_{(s,t)\in \Delta([0,T]):s\neq t}\frac{\|\|Z_{s,t}|\mathcal{F}_s\|_m\|_n}{(t-s)^\mu}\,,
		\end{align*}
			the quantities
			$\|\delta X\|_{\beta;m,n}$ and $\|\delta X'\|_{\beta';m,n}$ are finite.%

\item[(iii)] For the map $(s,t)\mapsto R^X_{s,t} := X_t-X_s-X'_s \delta \eta_{s,t}$ we have
			\[
			\|\mathbb{E}_\bullet R^X\|_{\beta+\beta';n}
			:=\sup_{(s,t)\in \Delta([0,T]):s\neq t}\frac{\|\mathbb{E}_sR^{X}_{s,t}\|_n}{(t-s)^{\beta+\beta'}}
			<\infty.
			\]
\end{enumerate}
The class of such pairs is denoted by $(X,X') \in \mathbf{D}_{\eta}^{\beta,\beta'} \mathbb{L}^{m,n}([0,T];V)$. Moreover, we define the quantity (seminorm)
\begin{align*}
    \|(X,X')\|_{\mathbf{D}_{\eta}^{\beta,\beta'} \mathbb{L}^{m,n}} : =  \sup_{t\in[0,T]}\|X'_t\|_n + \|\delta X\|_{\beta;m,n} + \|\delta X'\|_{\beta';m,n} + \|R^X\|_{\beta+\beta';n}.
\end{align*} 
\end{defn} 

\begin{rmk}
In case $m =n$ we might drop the two index notation from the above quantities and work with the following equivalent definition
\begin{align*}
    \|(X,X')\|_{\mathbf{D}_{\eta}^{\beta,\beta'} \mathbb{L}^{n}} : =  \|X'_0\|_n + \|\delta X\|_{\beta;n} + \|\delta X'\|_{\beta';n} + \|R^X\|_{\beta+\beta';n}.
\end{align*} 
\end{rmk}

Having introduced rough paths, we proceed to introduce rough stochastic integrals. 

\begin{prop}[{\cite[Theorem 3.4]{FHL21}}] \label{prop:roughinteg}
Let $p\in\mathbb{N},$ $\boldsymbol{\eta} \in \mathscr{C}^{\alpha}([0,T];\R^\ell)$ a rough path, with $\alpha\in(\frac{1}{3},\frac{1}{2})$ and
 $0<\beta'\leq \beta\leq\alpha$ such that $\alpha+\beta+\beta'>1$. For any $k \in [2,\infty )$ and any $\boldsymbol{\eta}$-stochastic controlled rough path $(Z,Z') \in \mathbf{D}_{\eta}^{\beta,\beta'}\mathbb{L}^k$, the following rough stochastic integral is well-defined:
\begin{equation*}
  \int_0^t Z_s \ud\boldsymbol{\eta}_s :=  \int_0^t (Z_s,Z'_s) \ud\boldsymbol{\eta}_s :=  \lim_{|\Pi| \rightarrow 0} \sum_{[u,v] \in \Pi} ( Z_u \delta \eta_{u,v} + Z'_u \eta^{(2)}_{u,v}), \quad t\in[0,T].
\end{equation*}
Here, $\Pi$ is any partition of $[0,t]$ with mesh $|\Pi|$ and the limit is taken in the sense of convergence in probability. Moreover, the process $\int_0^\cdot Z_s\, \ud \boldsymbol{\eta}_s$ is continuous adapted $\mathbb{L}^k$-integrable with null initial value such that for any $(s,t) \in \Delta,$ we have
\begin{align*} 
    &\bigg\| \int_s^t Z_u \ud\boldsymbol{\eta}_u - Z_s \delta \eta_{s,t} \bigg\|_k \lesssim    (\tnorm{\boldsymbol{\eta}}_\alpha \| \delta Z \|_{\beta,k} + \tnorm{\boldsymbol{\eta}}_\alpha^2 \sup_{r\in [s,t]} \|Z'\|_k )|t-s|^{\alpha+\beta} \\  
    & \qquad \qquad \qquad \qquad \qquad \ \ \  +   ( \tnorm{\boldsymbol{\eta}}_\alpha \|\E_{\bullet} R^Z\|_{\beta+\beta',k} + \tnorm{\boldsymbol{\eta}}_\alpha^2 \| \E_{\bullet} \delta Z'\|_{\beta',k}) |t-s|^{\alpha+\beta+\beta'},\\  
   & \bigg\| \mathbb{E}_{s} \bigg[ \int_{s}^{t} Z_{u} \ud\boldsymbol{\eta}_u - Z_{s} \delta \eta_{s,t} - Z'_s \eta^{(2)}_{s,t} \bigg] \bigg\|_{k} 
    \lesssim   \left(  \tnorm{\boldsymbol{\eta}}_\alpha \|\E_{\bullet} R^{Z} \|_{ {\beta}+ \beta', k} + \tnorm{\boldsymbol{\eta}}_\alpha^2 \| \E_{\bullet} \delta Z' \|_{\beta',k} \right) |t-s|^{\alpha +   \beta + \beta'}.
\end{align*}
In particular, 
$$
    (X,X'):=\Big(\int_0^. Z_s \ud\boldsymbol{\eta}_s, Z\Big) \in \mathbf{D}^{\alpha, \beta}_{\eta} \mathbb{L}^k.
$$
\end{prop}

\begin{defn}[{\cite[Definition 3.7]{FHL21}}](Stochastic controlled vector fields)\label[definition]{def.scvec}
Let $\boldsymbol{\eta} \in \mathscr{C}^{\alpha}([0,T];\R^\ell)$ be a rough path, $\beta, \beta' \in (0,1]$ and $\gamma\in(1,\infty)$. We call $(f,f')$ {\em stochastic controlled vector field on $V$}
of $(m,n)$-integrability and 
$(\gamma,\beta,\beta')$-H\"older regularity with respect to $(\mathcal{F}_t)_{t\in[0,T]}$
if the following hold:
		\begin{enumerate}
			\item[$(\text{i})$] 
			The pair
			\[
			(f,f') : \Omega \times  [0,T] \times V  \rightarrow  \bar{V} \times L(\mathbb{R}^\ell,\bar{V})
			\]
			is $ \text{Prog}(\Omega \times [0,T]) \otimes \mathcal{B}(V)/ \mathcal{B}(\bar{V}) \otimes \mathcal{B}( L(\mathbb{R}^\ell,\bar{V}))$-measurable,  and for every $t \in [0,T]$ we have $(f(t,\cdot),f'(t,\cdot)) \in \mathcal{C}^\gamma_{b}(V;\bar{V}) \times \mathcal{C}^{\gamma-1}_{b}(V;L(\mathbb{R}^\ell,\bar{V}))$. Moreover, $(f,f')$ is bounded in the sense
			\[
			\sup_{t\in [0,T]}\||f_t|_{\gamma}\|_n +  \sup_{t\in [0,T]}\||f'_t|_{\gamma-1}\|_n< \infty.
			\]

			\item[$(\text{ii})$] 
			Letting
			\begin{align*}\label{def.bk}
				\bk{Z}_{\mu;m,n}:=\sup_{(s,t)\in \Delta([0,T]):s\neq t}\frac{\|\|\sup_{v\in V}|Z_{s,t}(v)|\big|\mathcal{F}_s\|_m\|_n}{(t-s)^\mu},
			\end{align*}
			the quantities
			$\bk{\delta f}_{\beta;m,n}$, $\bk{\delta f'}_{\beta';m,n}$ and $ \bk{\delta \partial_{v} f}_{\beta';m,n}$ are finite.%

			  \item[$(\text{iii})$] 
            For the map $(s,t)\mapsto R^f_{s,t} := f_t-f_s-f'_s \delta \eta_{s,t}$ we have
			\[
			\bk{ R^f}_{\beta+\beta';n}
			:=\sup_{(s,t)\in \Delta([0,T]):s\neq t}\frac{\|\sup_{v\in V}\mathbb{E}_s|R^{f}_{s,t}(v)|\|_n}{(t-s)^{\beta+\beta'}}
			<\infty.
			\]
		\end{enumerate}
		The class of such vector fields is denoted by $\mathbf{D}_\eta^{\beta,\beta'}\mathbb{L}^{m,n}\mathcal{C}^\gamma_b( [0,T];V)$, or simply
	 $\mathbf{D}_\eta^{\beta,\beta'}\mathbb{L}^{m,n}\mathcal{C}^\gamma_b$ whenever the tuple $([0,T];V)$ is clear from the context. 
	\end{defn}
\begin{rmk}
  In case $f,f'$ are deterministic we will use the notation $\mathbf{D}_\eta^{\beta,\beta'}\mathcal{C}^\gamma_b( [0,T];V)$, or simply
	 $\mathbf{D}_\eta^{\beta,\beta'}\mathcal{C}^\gamma_b$.
\end{rmk}

Finally, we also need the following definition.
\begin{defn}[{\cite[Definition 3.2]{bugini2024malliavin}}](Stochastic controlled linear vector fields)\label{lin-vec} 
Let $\boldsymbol{\eta} \in \mathscr{C}^{\alpha}([0,T];\R^\ell)$ be a rough path, $\beta, \beta' \in (0,1]$, $m\in[2,\infty)$ and $n\in[m,\infty]$. 
We call $(f,f')$ \textit{stochastic controlled linear vector field} from $V$ to $\bar{Q}$ of $(m,n)$-integrability and $(\beta,\beta')$-H\"older regularity with respect to $(\mathcal{F}_t)_{t\in[0,T]}$ if the following hold: 

\begin{enumerate}  
\item[$(\text{i})$]  The pair
	   \[
			(f,f') : \Omega \times [0,T] \rightarrow  \bar{Q} \times L(\mathbb{R}^\ell,\bar{Q})
	   \]
    is progressively measurable and bounded in the sense
     \[
			\sup_{t\in [0,T]} \|f_t\|_{\infty} +  \sup_{t\in [0,T]}\|f'_t\|_{\infty}< \infty\,.
	   \]
    
\item[$(\text{ii})$] 
Letting
		\begin{align*}\label{def.bk}
				\|Z\|_{\mu;m,n}:=\sup_{(s,t)\in \Delta([0,T]):s\neq t}\frac{\|\|Z_{s,t}|\mathcal{F}_s\|_m\|_n}{(t-s)^\mu}\,,
		\end{align*}
			the quantities
			$\|\delta f\|_{\beta;m,n}$ and $\|\delta f'\|_{\beta';m,n}$ are finite.
            
\item[$(\text{iii})$] For the map $(s,t)\mapsto R^f_{s,t} := f_t-f_s-f'_s \delta \eta_{s,t}$ we have
			\[
			\|\mathbb{E}_\bullet R^f\|_{\beta+\beta';n}
			:=\sup_{(s,t)\in \Delta([0,T]):s\neq t}\frac{\|\mathbb{E}_sR^{f}_{s,t}\|_n}{(t-s)^{\beta+\beta'}}
			<\infty.
			\]
\end{enumerate}
The class of such linear vector fields is denoted by $\mathbf{D}_{\eta}^{\beta,\beta'} \mathbb{L}^{m,n}_{lin}\big([0,T];L(V,\bar{Q})\big)$, or simply $\mathbf{D}_{\eta}^{\beta,\beta'} \mathbb{L}^{m,n}_{lin}$ whenever the tuple $([0,T];V,\bar{Q})$ is clear from the context. Moreover, we define the quantity  
\begin{align*}
    \|(f,f')\|_{\mathbf{D}_{\eta}^{\beta,\beta'} \mathbb{L}^{m,n}_{lin}} : = \bk{\delta f}_{\beta;m,n} + \bk{\delta f'}_{\beta';m,n} + \bk{ R^f}_{\beta+\beta';n}.
\end{align*}
\end{defn} 

\subsubsection{\texorpdfstring{$p$-Variation seminorms}{p-Variation seminorms}}

\begin{defn}\label{defn: V^pL^q and V^pL^q,r spaces}
Let $p,m \in [1,\infty)$ and $n \in [m,\infty]$. Given a two-parameter stochastic process $F = (F_{s,t})_{(s,t) \in\Delta([0,T])}$ with values in $V$, we define
\begin{equation*}
\|F\|_{p\text{-var},m;[s,t]} := \bigg(\sup_{\mathcal{P} \subset [s,t]} \sum_{[u,v] \in \mathcal{P}} \|F_{u,v}\|_{m}^p\bigg)^{\frac{1}{p}} = \bigg(\sup_{\mathcal{P} \subset [s,t]} \sum_{[u,v] \in \mathcal{P}} \mathbb{E}[|F_{u,v}|^m]^\frac{p}{m}\bigg)^{\frac{1}{p}}
\end{equation*}
and
\begin{equation*}
\|F\|_{p\text{-var},m,n;[s,t]} := \bigg(\sup_{\mathcal{P} \subset [s,t]} \sum_{[u,v] \in \mathcal{P}} \|\|F_{u,v}|\mathcal{F}_u\|_m\|_n^p\bigg)^{\hspace{-2pt}\frac{1}{p}} = \bigg(\sup_{\mathcal{P} \subset [s,t]} \sum_{[u,v] \in \mathcal{P}} \|\E_u[|F_{u,v}|^m]^{\frac{1}{m}}\|_{n}^{p}\bigg)^{\frac{1}{p}}
\end{equation*}
for $(s,t) \in\Delta([0,T])$. For a stochastic process $X = (X_t)_{t \in [0,T]}$, we then let
\begin{equation*}
\|X\|_{p\text{-var},m;[s,t]}:= \|\delta X\|_{p\text{-var},m;[s,t]} \quad \text{and} \quad \|X\|_{p\text{-var},m,n;[s,t]} := \|\delta X\|_{p\text{-var},m,n;[s,t]}.
\end{equation*}
\end{defn}

The following are the variation analogs of the H\"older setting.

\begin{defn} \label{defn 2.1-r}

\begin{itemize}
\item[$(1)$] A pair of (deterministic) continuous functions $\boldsymbol{\eta}:=(\eta, \eta^{(2)}):[0,T] \times\Delta([0,T]) \rightarrow \R^\ell \times \R^{\ell \times \ell}$ is called a $p$-variation continuous {\it rough path}, if the following two conditions are satisfied.
\[
(\text{i})\  |\delta \eta |_{p\text{-var};[0,T]},\ |\eta^{(2)} |_{\frac{p}{2}\text{-var};[0,T]} < \infty; \ \ (\text{ii})\  \eta^{(2)}_{s,t}- \eta^{(2)}_{s,u}- \eta^{(2)}_{u,t}= \delta \eta_{s,u} \otimes \delta \eta_{u,t},\ \text{ for any $s\leq u \leq t$.}
\]  
The space of $p$-variation continuous rough paths is denoted  $\mathscr{C}^{p\text{-var}}([0,T];\R^\ell)$ and is equipped with the following rough path norm and rough path metric: for any $\boldsymbol{\eta}=(\eta,\eta^{(2)}), \bar{\boldsymbol{\eta}}=(\bar \eta, {\bar \eta}^{(2)} ) \in \mathscr{C}^{p\text{-var}}([0,T];\R^\ell),$ 
\begin{align*}
\tnorm{\boldsymbol{\eta}}_{p\text{-var}} &:=|\eta_0|+ \lvert \delta \eta \rvert_{p\text{-var};[0,T]} +\lvert \eta^{(2)}\rvert_{\frac{p}{2}\text{-var};[0,T]}^{\frac{1}{2}},\\ 
\rho_{p\text{-var}}(\boldsymbol{\eta}, \bar{\boldsymbol{ \eta}}) &:= |\eta_0-\bar{\eta}_0|+\lvert \delta \eta - \delta \bar \eta \rvert_{p\text{-var};[0,T]} + \lvert \eta^{(2)} - \bar \eta^{(2)} \rvert_{\frac{p}{2}\text{-var};[0,T]}.
\end{align*}

\item[$(2)$] With $\mathscr{C}^{0,p\text{-var}}_g([0,T];\R^\ell)$ we shall denote the space of $p$-variation geometric rough paths, defined as the completion of canonically lifted smooth paths (\emph{i.e.} if $\eta$ is smooth, $\boldsymbol{\eta}:=(\eta, \int \eta  \ud\eta)$ is the canonical lift of $\eta$) under the rough path metric $\rho_{p\text{-var}}$. Note that $\mathscr{C}^{0,p\text{-var}}_g([0,T];\R^\ell)$ is Polish and $\mathscr{C}^{0,p\text{-var}}_g([0,T];\R^\ell) \subset \mathscr{C}^{p\text{-var}}([0,T];\R^\ell)$ (see e.g. \cite{allan2026rough} for further details).
\end{itemize}
\end{defn}

\begin{rmk}
Trivially $\mathscr{C}^{\alpha}([0,T];\R^\ell) \subseteq \mathscr{C}^{\frac{1}{\alpha}\text{-var}}([0,T];\R^\ell)$ and $\mathscr{C}^{0,\alpha}_g([0,T];\R^\ell)\subseteq \mathscr{C}^{0,\frac{1}{\alpha}\text{-var}}_g([0,T];\R^\ell) $.
\end{rmk}

\begin{defn}[{\cite[Definition 4.1]{LiangTang2025}}](Stochastic controlled rough paths)\label{def:contro-rp-r} 
Let $\boldsymbol{\eta} \in \mathscr{C}^{p\text{-var}}([0,T];\R^\ell)$ be a rough path and $\beta, \beta' \in [1,\infty)$, $p,m\in[2,\infty)$ and $n\in[m,\infty]$.
A pair $(X,X')$ is called a \textit{stochastic controlled rough path} with values in $V$ of $(m,n)$-integrability and $(\beta,\beta')$-variation with respect to $(\mathcal{F}_t)_{t\in[0,T]}$ if the following hold: 

\begin{enumerate}  
\item[$(\text{i})$]  The pair
	   \[
			(X,X') : \Omega \times [0,T] \rightarrow  V \times 
            \bar{V}
	   \]
    is progressively measurable and 
     \[
             \sup_{t\in[0,T]}\|X'_t\|_{n}< \infty\,.
	   \]

\item[$(\text{ii})$] 
Letting $(s,t)\mapsto R^X_{s,t} := X_t-X_s-X'_s \delta \eta_{s,t}$ we have
		\begin{align*}\label{def.bk}
				\|\delta X\|_{\beta\text{-var},m,n;[0,T]} + \|\delta X'\|_{\beta'\text{-var},m,n;[0,T]} + \|\mathbb{E}_\bullet R^X\|_{\frac{\beta \beta'}{\beta + \beta'}\text{-var},n;[0,T]} < \infty.
		\end{align*}
\end{enumerate}
The class of such pairs is denoted by $(X,X') \in \mathbf{D}_{\eta}^{(\beta,\beta')\text{-var}} \mathbb{L}^{m,n}([0,T];V)$. Moreover, we define the quantity (seminorm)
\begin{multline*}
    \|(X,X')\|_{\mathbf{D}_{\eta}^{(\beta,\beta')\text{-var}} \mathbb{L}^{m,n}} : =  \sup_{t\in[0,T]}\|X'_t\|_n + \|\delta X\|_{\beta\text{-var},m,n;[0,T]} + \|\delta X'\|_{\beta'\text{-var},m,n;[0,T]} + \|\mathbb{E}_\bullet R^X\|_{\frac{\beta \beta'}{\beta + \beta'}\text{-var},n;[0,T]}.
\end{multline*} 
\end{defn}

\begin{rmk}
In case $m =n$ we might drop the two index notation from the above quantities and work with the following equivalent definition
\begin{align*}
    \|(X,X')\|_{\mathbf{D}_{\eta}^{(\beta,\beta')\text{-var}} \mathbb{L}^{n}} : =  \|X'_0\|_n + \|\delta X\|_{\beta\text{-var},n;[0,T]} + \|\delta X'\|_{\beta'\text{-var},n;[0,T]} + \|\mathbb{E}_\bullet R^X\|_{\frac{\beta \beta'}{\beta + \beta'}\text{-var},n;[0,T]}.
\end{align*} 
\end{rmk}

\begin{lemma}[{\cite[Theorem 4.10]{LiangTang2025}}]\label{lem:roughstochint}
Let $\boldsymbol{\eta} \in \mathscr{C}^{p\text{-var}}([0,T];\R^\ell)$ a rough path, with $p\in(2,3)$ and $p\leq\beta\leq \beta'$
such that
$\frac{1}{p}+\frac{1}{\beta}+\frac{1}{\beta'}>1$.
For any $k \in [2,\infty )$ and any $\boldsymbol{\eta}$-stochastic controlled rough path $(Z,Z') \in \mathbf{D}_{\eta}^{(\beta,\beta')\text{-var}}\mathbb{L}^k$, the following rough stochastic integral is well-defined:
\begin{equation*}
  \int_0^t Z_s \ud\boldsymbol{\eta}_s :=  \int_0^t (Z_s,Z'_s) \ud\boldsymbol{\eta}_s :=  \lim_{|\Pi| \rightarrow 0} \sum_{[u,v] \in \Pi} ( Z_u \delta \eta_{u,v} + Z'_u \eta^{(2)}_{u,v}), \quad t\in[0,T].
\end{equation*}
Here, $\Pi$ is any partition of $[0,t]$ with mesh $|\Pi|$ and the limit is taken in the sense of convergence in probability. Moreover, the process $\int_0^\cdot Z_s\, \ud \boldsymbol{\eta}_s$ is continuous adapted $\mathbb{L}^k$-integrable with null initial value such that for any $(s,t) \in \Delta,$ we have
\begin{equation}\label{eq:rough-int-estimate}
\begin{aligned}
&\Big\|
\int_s^t Z_u\,\ud\boldsymbol{\eta}_u
-
Z_s\,\delta \eta_{s,t}
-
Z'_s\,\eta^{(2)}_{s,t}
\Big\|_{k}\\
&\ \ \ \lesssim
\Bigl(
\|\delta Z\|_{\beta\text{-var},k;[s,t]}
+  \|Z'\|_{C^{\beta'\text{-var} }\LL^k} |\delta \eta|_{p\text{-var};[s,t]} +
\|\E_\bullet R^Z\|_{ \frac{\beta \beta'}{\beta+\beta'}\text{-var}, k;[s,t] }
\Bigr)
|\delta \eta|_{p\text{-var};[s,t]}
\\
& \ \  \ \ \ \ \ +
\|\delta Z'\|_{\beta'\text{-var},k;[s,t]}
|\eta^{(2)}|_{\frac{p}{2}\text{-var};[s,t]},
\end{aligned}
\end{equation}
where $\| Z' \|_{C^{\beta'\text{-var} }\LL^k}:=\sup_{u\in [s,t]}\|Z'_u\|_{k} + \|\delta Z' \|_{\beta'\text{-var}, k;[s,t] }$.

In particular, 
$$
    (X,X'):=\Big(\int_0^. Z_s \ud\boldsymbol{\eta}_s, Z\Big) \in \mathbf{D}^{(p,\beta)\text{-var}}_{\eta} \mathbb{L}^k.
$$
\end{lemma}

\begin{proof}
The proof of the last estimate follows by \cite[Theorem 4.10]{LiangTang2025}, and the fact that 
\[
\begin{split}
\| \delta Z^M \|_{\beta\text{-var},k} &\le \|\delta Z\|_{\beta\text{-var},k} +\|\delta Z^J\|_{\beta\text{-var},k} \\
&\le  \|\delta Z\|_{\beta\text{-var},k} + \sup_t\|Z'_t\| |\delta \eta|_{\beta\text{-var}} + \|R^{Z^J}\|_{\frac{\beta \beta'}{\beta+\beta'}\text{-var},k}.
\end{split}
\]

\end{proof}

\begin{defn}(Stochastic controlled linear vector fields)\label{lin-vec-r} 
Let $\boldsymbol{\eta} \in \mathscr{C}^{p\text{-var}}([0,T];\R^\ell)$ be a rough path, $\beta, \beta' \in [1,\infty)$, $p,m\in[2,\infty)$ and $n\in[m,\infty]$. 
We call $(f,f')$ \textit{stochastic controlled linear vector field} from $V$ to $\bar{Q}$ of $(m,n)$-integrability and $(\beta,\beta')$-variation with respect to $(\mathcal{F}_t)_{t\in[0,T]}$ if the following hold: 

\begin{enumerate}  
\item[$(\text{i})$]  The pair
	   \[
			(f,f') : \Omega \times [0,T] \rightarrow  \bar{Q} \times L(\mathbb{R}^\ell,\bar{Q})
	   \]
    is progressively measurable and bounded in the sense
     \[
			\sup_{t\in [0,T]} \|f_t\|_{\infty} +  \sup_{t\in [0,T]}\|f'_t\|_{\infty}< \infty\,.
	   \]
    
\item[$(\text{ii})$] 
Letting $(s,t)\mapsto R^f_{s,t} := f_t-f_s-f'_s \delta \eta_{s,t}$ we have
		\begin{align*}\label{def.bk}
				\|\delta f\|_{\beta\text{-var},m,n;[0,T]} + \|\delta f'\|_{\beta'\text{-var},m,n;[0,T]} + \|\mathbb{E}_\bullet R^f\|_{\frac{\beta \beta'}{\beta + \beta'}\text{-var},n;[0,T]} < \infty.
		\end{align*}
\end{enumerate}
The class of such linear vector fields is denoted by $\mathbf{D}_{\eta}^{(\beta,\beta')\text{-var}} \mathbb{L}^{m,n}_{lin}\big([0,T];L(V,\bar{Q})\big)$, or simply $\mathbf{D}_{\eta}^{(\beta,\beta')\text{-var}} \mathbb{L}^{m,n}_{lin}$ whenever the tuple $([0,T];V,\bar{Q})$ is clear from the context. Moreover, we define the quantity  
\begin{align*}
    \|(f,f')\|_{\mathbf{D}_{\eta}^{(\beta,\beta')\text{-var}} \mathbb{L}^{m,n}_{lin}} : = \|\delta f\|_{\beta\text{-var},m,n;[0,T]} + \|\delta f'\|_{\beta'\text{-var},m,n;[0,T]} + \|\mathbb{E}_\bullet R^f\|_{\frac{\beta \beta'}{\beta + \beta'}\text{-var},n;[0,T]}.
\end{align*}
\end{defn} 

\begin{rmk}\label{rmk 2.16}
Let $\boldsymbol{\eta} \in \mathscr{C}^{p\text{-var}}([0,T];\R^\ell)$ be a rough path, $\beta, \beta' \in [1,\infty)$, $p,m\in[2,\infty)$, $n\in[m,\infty]$, $(X,X')$ stochastic controlled rough path and $(f,f')$ stochastic controlled linear vector field. Then, from \cite[Lemma 2.11 and Remark 4.2]{allan2026rough} (adapted to our continuous setting) we see that
\begin{align*}
    &\|\delta X\|_{\beta\text{-var},m,n;[0,T]}^\beta,  \|\delta X'\|_{\beta'\text{-var},m,n;[0,T]}^{\beta'}, \|\delta f\|_{\beta\text{-var},m,n;[0,T]}^\beta,  \|\delta f'\|_{\beta'\text{-var},m,n;[0,T]}^{\beta'},\\
    & \|\mathbb{E}_\bullet R^X\|_{\frac{\beta \beta'}{\beta + \beta'}\text{-var},n;[0,T]}^{\frac{\beta \beta'}{\beta + \beta'}}, \|\mathbb{E}_\bullet R^f\|_{\frac{\beta \beta'}{\beta + \beta'}\text{-var},n;[0,T]}^{\frac{\beta \beta'}{\beta + \beta'}}, \| R^X\|_{\beta\text{-var},m,n;[0,T]}^{\beta}, \| R^f\|_{\beta\text{-var},m,n;[0,T]}^{\beta}
\end{align*}
are deterministic (continuous) controls.
\end{rmk}

\section{The main tool box}\label{sec:tool_box}

\subsection{Parameterized affine rough stochastic differential equations}\label{sec 3.1}
From this point on let $W$ be a $w-$dimensional Brownian motion, we shall need to study equations of the form 
\begin{equation}\label{eq ex_lin}
\begin{cases}
    X_t = \xi + F_t + \int_0^t b_s X_s \ud s + \int_0^t \sigma_s X_s\ud W_s +\int_0^t \big(\int_0^1\Phi_s(\lambda) \ud\lambda, \int_0^1\Phi'_s(\lambda)\ud\lambda \Big) X_s\ud \boldsymbol{\eta}_s,\\
     F_t = \big(F,\int_0^1 F_\cdot'(\lambda)\ud\lambda\big)\in \mathbf{D}_{\eta}^{\alpha,\beta}\mathbb{L}^k([0,T]; \mathbb{R}^{d}).
    \end{cases}
\end{equation}
Our approach and definitions naturally extend those founded in \cite[Section 3]{bugini2024malliavin}.

\begin{defn}[Extended stochastic controlled linear vector fields]
Let $\boldsymbol{\eta} \in \mathscr{C}^{\alpha}([0,T];\R^\ell)$ be a rough path, $\beta, \beta' \in (0,1]$, $m\in[2,\infty)$ and $n\in[m,\infty]$. We call $\big(\Phi(\lambda),\Phi'(\lambda)\big)$ \textit{extended stochastic controlled linear vector field} from $V$ to $\bar{Q}$ of $(m,n)$-integrability and $(\beta,\beta')$-H\"older regularity with respect to $(\mathcal{F}_t)_{t\in[0,T]}$ if the following hold:
\begin{enumerate}
    \item[(i)] The pair
	   \[
			\big(\Phi(\lambda),\Phi'(\lambda)\big) : \Omega \times [0,T] \times [0,1] \rightarrow  L(V,\bar{Q}) \times L\big(\mathbb{R}^\ell,L(V,\bar{Q})\big)
	   \]
    is $ \text{Prog}(\Omega\times [0,T]) \otimes \mathcal{B}([0,1])/ \mathcal{B}(L(V,\bar{Q})) \otimes \mathcal{B}( L(\mathbb{R}^\ell,L(V,\bar{Q})))$-measurable.

    \item[(ii)] For every $\lambda\in[0,1]$ we have $\big(\Phi_\cdot(\lambda),\Phi_\cdot'(\lambda)\big)\in \mathbf{D}_{\eta}^{\beta,\beta'} \mathbb{L}^{k,\infty}_{lin}\big([0,T];L(V,\bar{Q})\big)$ and 
    $$
     \sup_{\lambda\in[0,1]} \sup_{t\in[0,T]}\| \Phi_t(\lambda)\|_{\infty} + \sup_{\lambda\in[0,1]} \sup_{t\in[0,T]}\| \Phi'_t(\lambda)\|_{\infty} + \sup_{\lambda\in[0,1]}\big\|\big(\Phi_\cdot(\lambda),\Phi'_\cdot(\lambda)\big)\big\|_{\mathbf{D}_{\eta}^{\beta,\beta'} \mathbb{L}^{k,\infty}_{lin}} <\infty.
     $$
\end{enumerate}
The class of such linear vector fields is denoted by $\mathbf{D}_{\eta}^{\beta,\beta'} \mathbb{L}^{m,n}_{\lambda-ex,lin}\big([0,T];L(V,\bar{Q})\big)$, or simply $\mathbf{D}_{\eta}^{\beta,\beta'} \mathbb{L}^{m,n}_{\lambda-ex,lin}$ whenever the tuple $([0,T];V,\bar{Q})$ is clear from the context.
\end{defn}

\begin{prop}\label{Prop 3.3}
Let $\boldsymbol{\eta} \in \mathscr{C}^{\alpha}([0,T];\R^\ell)$ be a rough path, $\beta, \beta' \in (0,1]$, $m\in[2,\infty)$, $n\in[m,\infty]$ and $\big(\Phi_\cdot(\lambda),\Phi_\cdot'(\lambda)\big)\in \mathbf{D}_{\eta}^{\beta,\beta'} \mathbb{L}^{m,n}_{\lambda-ex,lin}\big([0,T];L(V,\bar{Q})\big)$. Then, $(\int_0^1\Phi_\cdot(\lambda) \ud\lambda, \int_0^1\Phi'_\cdot(\lambda) \ud\lambda) \in \mathbf{D}_{\eta}^{\beta,\beta'} \mathbb{L}^{m,n}_{lin}\big([0,T];L(V,\bar{Q})\big)$. Furthermore,
 \begin{align*}
    \sup_{t\in[0,T]} \bigg\|\int_0^1\Phi_t(\lambda) \ud\lambda\bigg\|_\infty +  \sup_{t\in[0,T]} \bigg\|\int_0^1\Phi'_t(\lambda) \ud\lambda \bigg\|_\infty &\leq  \sup_{\lambda\in[0,1]} \sup_{t\in[0,T]} \|\Phi_t(\lambda)\|_{\infty} + \sup_{\lambda\in[0,1]}\sup_{t\in[0,T]} \|\Phi'_t(\lambda)\|_{\infty},\\
    \bigg\|\bigg(\int_0^1\Phi_\cdot(\lambda) \ud\lambda, \int_0^1\Phi'_\cdot(\lambda) \ud\lambda \bigg)\bigg\|_{\mathbf{D}_{\eta}^{\beta,\beta'} \mathbb{L}^{m,n}_{lin}} &\leq \sup_{\lambda\in[0,1]}\big\|\big(\Phi_\cdot(\lambda),\Phi'_\cdot(\lambda)\big)\big\|_{\mathbf{D}_{\eta}^{\beta,\beta'} \mathbb{L}^{m,n}_{lin}}.
 \end{align*}
\end{prop}

\begin{proof}
Immediate from the definitions, Jensen's inequality and conditional Fubini's Theorem.
\end{proof}

\begin{rmk}\label{rmk 3.3}
For the remainder of Section \ref{sec:tool_box}, for notational simplicity, we will refrain from using general Banach-space notation. However, it should be kept in mind that all the results remain valid, with the appropriate modifications, when considering arbitrary finite-dimensional Banach spaces.
\end{rmk}

\begin{defn}[Solution in the sense of Davie]\label{dfn ex_lin_Davie}
Let $\boldsymbol{\eta} \in \mathscr{C}^{\alpha}([0,T];\R^\ell)$, with $\alpha\in(\frac{1}{3},\frac{1}{2})$ and 
 $0<\beta'\leq \beta\leq\alpha$ such that $\alpha+\beta+\beta'>1$. Let $b:\Omega \times [0,T]\rightarrow \mathbb{R}^{d\times d}$, $\sigma :\Omega\times [0,T] \rightarrow L(\mathbb{R}^d,\mathbb{R}^{d\times w})$ progressive measurable and such that $\sup_{t\in[0,T]}\{\|\| b_t |\mathcal{F}_t \|_k \|_\infty + \|\| \sigma_t |\mathcal{F}_t \|_k \|_\infty \} <\infty$. Moreover, let $k\in[2,\infty)$, a pair $\big(\Phi(\lambda),\Phi'(\lambda)\big)\in \mathbf{D}_{\eta}^{\beta,\beta'} \mathbb{L}^{k,\infty}_{\lambda-ex,lin}\big([0,T];L(\mathbb{R}^d,\mathbb{R}^{d\times \ell})\big)$, and $\big(F_\cdot,\int_0^1 F_\cdot'(\lambda)\ud\lambda\big)\in \mathbf{D}_{\eta}^{\alpha,\beta}\mathbb{L}^k([0,T]; \mathbb{R}^{d})$. Then, an $\mathbb{L}^k$-integrable solution of \eqref{eq ex_lin} in the sense of Davie is a continuous adapted process $X=(X_t)_{t\in[0,T]}$ with $\sup_{t\in[0,T]}\|X_t\|_k <\infty$ such that
 \begin{enumerate}
     \item[(i)] $X_0 = \xi$.

     \item[(ii)] There exist $C_1, C_2 \geq 0$
     satisfying
\begin{equation*}
    \|X^\natural_{s,t}\|_k \leq C_1 |t-s|^{\alpha +\beta}\hspace{.2cm}\text{and}\hspace{.2cm}  \|\mathbb{E}_s[X^\natural_{s,t}]\|_k \leq C_2 |t-s|^{\alpha + \beta + \beta'},
\end{equation*}
for every $(s,t) \in\Delta([0,T])$, where 
\begin{gather*}
    X^\natural_{s,t} := \delta X_{s,t} - \delta F_{s,t} - \int_s^t b_r X_r \ud r - \int_s^t \sigma_r X_r \ud W_r \\
    - \Big(\int_0^1\Phi_s(\lambda)X_s \ud\lambda\Big)\delta \eta_{s,t} - \Big(\int_0^1\Big(\Phi'_s(\lambda)X_s +\Phi_s(\lambda)\Big[\int_0^1(\Phi_s(\widetilde{\lambda})X_s + F'_s(\widetilde{\lambda})\big)\ud \widetilde{\lambda}\Big]\Big)\ud\lambda\Big) \eta^{(2)}_{s,t}.
\end{gather*}
 \end{enumerate}
\end{defn}

From Proposition \ref{Prop 3.3}, following \cite[Section 3]{bugini2024malliavin} we easily get the following string of results.

\begin{prop}\label{Prop 3.2}
Let $\boldsymbol{\eta} \in \mathscr{C}^{\alpha}([0,T];\R^\ell)$, with $\alpha\in(\frac{1}{3},\frac{1}{2})$ and 
 $0<\beta'\leq \beta\leq\alpha$ such that $\alpha+\beta+\beta'>1$. Let $k\in[2,\infty)$, $(X,X')\in \mathbf{D}_{\eta}^{\beta,\beta'}\mathbb{L}^k([0,T]; \mathbb{R}^{d})$ with $\sup_{t\in[0,T]}\|X_t\|_k <\infty$ and $\big(\Phi_\cdot(\lambda),\Phi_\cdot'(\lambda)\big)\in \mathbf{D}_{\eta}^{\beta,\beta'} \mathbb{L}^{k,\infty}_{\lambda-ex,lin}\big([0,T];L(\mathbb{R}^d,\mathbb{R}^{d\times \ell})\big)$. 
 Then,
 \begin{equation*}
     \bigg(\int_0^1 \Phi_\cdot(\lambda)X_\cdot\ud\lambda, \int_0^1 \Big(\Phi'_\cdot(\lambda)X_\cdot+\Phi_\cdot(\lambda)X'_\cdot\Big)\ud\lambda\bigg) \in \mathbf{D}_{\eta}^{\beta,\beta'}\mathbb{L}^k([0,T]; \mathbb{R}^{d\times \ell}).
 \end{equation*}
\end{prop}

\begin{lemma}\label{lemma 3.4}
  Let $X = (X_t)_{t\in[0,T]}$ be a continuous and adapted $\mathbb{R}^d$-valued process with $\sup_{t\in[0,T]}\|X_t\|_k$ $<\infty$ such that it satisfies condition (ii) of Definition \ref{dfn ex_lin_Davie}. Then, 
  \begin{equation*}
      \bigg(X_\cdot,\int_0^1 \big(\Phi_\cdot(\lambda)X_\cdot + F'_\cdot(\lambda)\big)  \ud\lambda\bigg) \in \mathbf{D}_{\eta}^{\alpha,\beta}\mathbb{L}^k([0,T]; \mathbb{R}^{d}).
  \end{equation*}
\end{lemma}

\begin{prop}\label{Prop 3.5}
Let $F_0 =0$ and $X = (X_t)_{t\in[0,T]}$ be a continuous and adapted $\mathbb{R}^d$-valued process with $\sup_{t\in[0,T]}\|X_t\|_k <\infty$ and $\xi\in \mathbb{L}^k(\mathcal{F}_0;\mathbb{R}^d)$. Then, the following are equivalent:
\begin{enumerate}
    \item[(i)] $X$ is an $\mathbb{L}^k$-integrable solution of \eqref{eq ex_lin} in the sense of Davie.
    
    \item[(ii)] 
    $\Big(\int_0^1\Phi_\cdot(\lambda)X_\cdot \ud\lambda,  \int_0^1\big(\Phi'_\cdot(\lambda)X_\cdot +\Phi_\cdot(\lambda)\big[\int_0^1(\Phi_\cdot(\widetilde{\lambda})X_\cdot + F'_\cdot(\widetilde{\lambda})\big)\ud \widetilde{\lambda}\big]\big)\ud\lambda\Big) \in \mathbf{D}_{\eta}^{\beta,\beta'}\mathbb{L}^k([0,T]; \mathbb{R}^{d\times \ell})$, and for every $t\in[0,T]$ we have $\mathbb{P}$-a.s. that
    \begin{gather*}
        X_t = \xi + F_t + \int_0^t b_s X_s \ud s + \int_0^t \sigma_s X_s\ud W_s\\+\int_0^t \Big(\int_0^1\Phi_s(\lambda)X_s \ud\lambda,  \int_0^1\Big(\Phi'_s(\lambda)X_s +\Phi_s(\lambda)\Big[\int_0^1(\Phi_s(\widetilde{\lambda})X_s + F'_s(\widetilde{\lambda})\big)\ud \widetilde{\lambda}\Big]\Big)\ud\lambda\Big) \ud \boldsymbol{\eta}_s.
    \end{gather*}
\end{enumerate}
In particular we have $\|X^\natural_{s,t}\|_k \lesssim  |t-s|^{\alpha+\beta}\hspace{.2cm}\text{and}\hspace{.2cm}  \|\mathbb{E}_s[X^\natural_{s,t}]\|_k \lesssim |t-s|^{\alpha+\beta+\beta'}$, for every $(s,t)\in\Delta_{[0,T]}$.
\end{prop}

\subsubsection{Estimate of affine RSDEs using the \texorpdfstring{$p$-variation}{p-variation} seminorms}

\begin{lemma}\label{lem:uni-p-var}
Let $ 1\leq p < k$, and $X$ an $\mathbb{R}^d$-valued continuous process\footnote{$X$ could take values on any Banach space and the proof would remain the same.}. 
Then, for every $(s,t)\in \Delta([0,T])$ we have
\begin{equation}
 \|X\|_{\mathbb S^k;[s,t]}
 \leq \|X_s\|_{k}
      +C_{p,k}\|\delta X\|_{p\text{-var},k;[s,t]}.
\end{equation}
\end{lemma}

\begin{proof}
To ease the notation, let 
$
 \mathcal{V}:=\|\delta X\|_{p\text{-var},k;[s,t]},
 \ 
 {w}_X(u,v):=
 \|\delta X\|_{p\text{-var},k;[u,v]}^p.
$
Without loss of generality, we assume $0<\mathcal{V}<\infty$.
Then, by \cite{allan2026rough}[Lemma 2.9] and the fact that $X$ is pathwise continuous we get that $(X_v)_{v\in [s,t]}$ is continuous with respect to the $\mathbb{L}^k$ norm. Hence, the map $w_X$ is a deterministic (continuous) control. Next, let
\begin{equation*}
 a_X(v):=w_X(s,v).
\end{equation*}
Thus, by superadditivity we get
\begin{equation}\label{ineq 58}
 0\leq w_X(u,v)\leq a_X(v)-a_X(u),
 \quad s\leq u\leq v\leq t,
\end{equation}
and that $a_X$ is continuous non-decreasing with $a_X(s) = 0, a_X(t) = \mathcal{V}^p$. Now, let $\tau^0_0 =s$ and $\tau^0_1 =t$, and for every $n\in\mathbb{N}$ we define
\begin{equation*}
 \tau_i^n
 :=\inf\Big\{v\in[s,t]:a_X(v)\geq \frac{i\mathcal{V}^p}{2^n}\Big\},
 \qquad 0 < i <2^n,
\end{equation*}
with \(\tau_0^n=s\) and \(\tau_{2^n}^n=t\). We shall denote with $\mathcal{D}^n_{[s,t]}:=\{\tau^n_i:0\leq i \leq 2^n\}$ and $\mathcal{D}_{[s,t]}:=\bigcup_{n=1}^\infty \mathcal{D}^n_{[s,t]}$. Every new point
$\tau_{2i+1}^{n+1}$ has left parent
$\tau_i^n=\tau_{2i}^{n+1}$, and hence by \eqref{ineq 58} in conjunction with the previous definition we get
\begin{equation*}
 \big\|\delta X_{\tau_i^n,\tau_{2i+1}^{n+1}}\big\|_{k}
 \leq 2^{-\frac{n+1}{p}} \mathcal{V}.
\end{equation*}
Next, for $n\in \mathbb{N}\cup\{0\}$ we set $Z_n:=\max_{0\leq i<2^n-1}
 \big|\delta X_{\tau_i^n,\tau_{2i+1}^{n+1}}\big|,$ trivially $|Z_n|^k \leq \sum_{i=0}^{2^{n}-1}|\delta X_{\tau_i^n,\tau_{2i+1}^{n+1}}|^k$, so we have
\begin{equation}\label{est:dyadic-X-est}
 \|Z_n\|_{k}
 \leq\left(\sum_{i=0}^{2^n-1}
 \big\|\delta X_{\tau_i^n,\tau_{2i+1}^{n+1}}\big\|_{k}^k
 \right)^{1/k}
 \leq 2^{(\frac{n}{k}-\frac{n+1}{p})} \mathcal{V}.
\end{equation}
On the other hand, because for $0\leq u \leq v \leq t$ with $a_X(u) = a_X(v)$ from \eqref{ineq 58} we have $X_u(\omega) = X_v(\omega)$ $\mathbb{P}-a.s.$, by the pathwise continuity of $X$ we get
\begin{equation}\label{est:deltaX}
 \sup_{r\in[s,t]}|\delta X_{s,r}(\omega)| = \sup_{r\in \mathcal{D}_{[s,t]}}|\delta X_{s,r}(\omega)|
 \leq |\delta X_{s,t}(\omega)|+\sum_{n = 0}^\infty Z_n,\quad\mathbb{P}-a.s.
\end{equation}
Then, by \eqref{est:dyadic-X-est}, \eqref{est:deltaX}, the triangle inequality and the monotone convergence Theorem we have
\begin{align*}
 \big\|\sup_{r\in[s,t]}|\delta X_{s,r}|\big\|_{k}
 &\leq  \| \delta X_{s,t}\|_k + \Big\|\sum_{n = 1}^\infty Z_n\Big\|_k\\
 &\leq \mathcal{V} + \sum_{n = 0}^\infty \|Z_n\|_k\\
 &\leq\Big(1+2^{-\frac{1}{p}}\sum_{n = 0}^{\infty}
 2^{-n (\frac{1}{p} - \frac{1}{k})}\Big) \mathcal{V} = C_{p,k}\mathcal{V},
\end{align*}
with $C_{p,k} := 1+2^{-\frac{1}{p}}\sum_{n = 0}^{\infty}
 2^{-n (\frac{1}{p} - \frac{1}{k})}$,
where we used the fact that the geometric series converges exactly when $k > p$. So, our conclusion follows from the fact that
\begin{equation*}
 \|X\|_{\mathbb S^k;[s,t]}
 \leq\|X_s\|_{k}
 +\big\|\sup_{r\in[s,t]}|\delta X_{s,r}|\big\|_{k}.
\end{equation*}
\end{proof}
The following estimate result for affine RSDEs under the $p$-variation seminorms is a key ingredient in the spike argument.
 
\begin{thm}\label{thm 3.9}
Let $\frac{1}{3}< \frac{1}{p} \leq \beta''\leq \beta'\leq\alpha < \frac{1}{2}$, $p < k$, $\xi\in \mathbb{L}^k(\mathcal{F}_0;\mathbb{R}^d)$ and $X$ an $\mathbb{L}^k$-integrable solution of \eqref{eq ex_lin} in the sense of Davie\footnote{Where $F', \Phi, \Phi'$ were assumed constant with respect to $\lambda$ without loss of generality due to Proposition \ref{Prop 3.3}.} 
with $F_0 = 0$. 
Then, for every $\mathscr{M} > 0 $ satisfying
\begin{align}\label{eq:BV-coefficient-bound}
   \sup_{t\in [0,T]}\big\{
   \|\| b_t |\mathcal{F}_t \|_k \|_\infty + \|\| \sigma_t |\mathcal{F}_t \|_k \|_\infty +
   \|\Phi_t\|_\infty + \|\Phi'_t\|_\infty \big\} 
   + \|(\Phi,\Phi')\|_{\mathbf{D}_{\eta}^{(p,p)\text{-var}} \mathbb{L}^{k,\infty}_{lin}}
   + \tnorm{\boldsymbol{\eta}}_{{p}\text{-var}} \leq \mathscr{M}
\end{align}
there exists $ C =  C(\mathscr{M},p,k,T)\geq 0$ such that
\begin{equation}\label{eq:BV-primary-estimate}
 \|X\|_{\mathbb S^k}
 +\|(X,X')\|
 _{\mathbf D_\eta^{(p,p)\text{-var}}\mathbb L^k}
 \leq C \big(
 \|\xi\|_{k}
 + \|(F,F')\|_{
    \mathbf D_\eta^{(p,p)\text{-var}}\mathbb L^k}
 \big).
\end{equation}
\end{thm}

\begin{proof}
To ease the notation we clarify that during the proof the value of the constant $C$ is allowed to change from line to line up to \eqref{eq. 25}. Moreover, in the same spirit, for every $(s,t)\in \Delta([0,T])$, we also make the following definitions\footnote{Although we have previously defined the rough seminorms only over $[0,T]$, it is trivial to see that the same definitions can extend over any interval $I$.} 
\begin{align*}
 \mathcal N_{k;[s,t]}(X)
 &:=
 \|X\|_{\mathbb S^k;[s,t]}
 + \|(X,X')\|_{
    \mathbf D_\eta^{(p,p)\text{-var}}\mathbb L^k;[s,t]},\\
 \mathcal{F}_{k;[s,t]}(F) &:= \|(F,F')\|_{
    \mathbf D_\eta^{(p,p)\text{-var}}\mathbb L^k;[s,t]}\\
 &\hspace{.1cm}= 
 \|F'_s \|_k +
 \|\delta F\|_{p\text{-var},k;[s,t]} + \|\delta F'\|_{p\text{-var},k;[s,t]}
 +\|\mathbb E_\bullet R^F\|_{
       \frac{p}{2}\text{-var},k;[s,t]}.
\end{align*}
At this point we introduce the deterministic controls (see Remark \ref{rmk 2.16})
\begin{align*}
 w_{\Phi,k}(s,t)
 &:=
 \|\delta \Phi\|_{p\text{-var},k,\infty;[s,t]}^p
 + \|\delta \Phi'\|_{p\text{-var},k,\infty;[s,t]}^p
 +\big\|\mathbb E_\bullet R^{\Phi}\big\|_{
       \frac{p}{2}\text{-var},k,\infty;[s,t]}^{\frac{p}{2}},\\
 w_\eta(s,t)
 &:=
 \|\delta\eta\|_{p\text{-var};[s,t]}^p
 + \|\eta^{(2)}\|_{\frac{p}{2}\text{-var};[s,t]}^{\frac{p}{2}},\\
 w_{\mathcal L,k}(s,t)
 &:=
 \bigl(\sup_{u\in[0,T]} \|\| b_t |\mathcal{F}_t \|_k \|_\infty + \sup_{u\in[0,T]}\|\| \sigma_t |\mathcal{F}_t \|_k \|_\infty^2\bigr)(t-s)
 + w_{\Phi,k}(s,t)+ w_\eta(s,t).
\end{align*}
First we derive a local estimate. For $u\in [s,t]$, we define
\begin{equation*}
 A^X_u:= \int_s^u b_v X_v\,\ud v,\quad
 M^X_u:= \int_s^u \sigma _v X_v\,\ud W_v,\quad
 \mathcal{I}^X_u:= \int_s^u(H_v,H'_v)\,\ud\boldsymbol{\eta}_v,
\end{equation*}
where
\begin{align*}
H_v &:= \Phi_v X_v,\\ 
H'_v&:= \Phi'_v X_v + \Phi_v X'_v =\Phi'_v X_v + \Phi^2_v X_v + \Phi_v F'_v.
\end{align*}
For $s\leq u_1\leq u_2\leq t$ we have by Minkowski's integral inequality
\begin{align}
    \|\delta A^X_{u_1,u_2}\|_k &\leq  \int_{u_1}^{u_2} \|b_v X_v\|_k \ud v \leq \|X\|_{\mathbb S^k;[s,t]} \int_{u_1}^{u_2}\|\| b_v |\mathcal{F}_v \|_k \|_\infty \ud v,\label{eq. 012}\\
    \bigg\| \bigg(\int_{u_1}^{u_2} |\sigma_v X_v|^2 \ud v \bigg)^{\frac{1}{2}}\bigg\|_k &= \bigg\|\int_{u_1}^{u_2} |\sigma_v X_v|^2 \ud v \bigg\|_{\frac{k}{2}}^{\frac{1}{2}} \leq \bigg( \int_{u_1}^{u_2} \|\sigma_v X_v\|_k^2 \ud v\bigg)^{\frac{1}{2}}\nonumber\\ 
    &\leq \|X\|_{\mathbb S^k;[s,t]} \left(\int_{u_1}^{u_2} \|\| \sigma_v |\mathcal{F}_v \|_k \|_\infty^2 \ud v\right)^{\frac{1}{2}}.\label{eq. 013}
\end{align}
Furthermore, by \cite[Exercise 1.9.]{FV10} we have for $p \geq 2$ that $(u_1,u_2) \mapsto \left(\int_{u_1}^{u_2} \|\| \sigma_v |\mathcal{F}_v \|_k \|_\infty^2\ud v\right)^{\frac{p}{2}}$ is a control.
Next, by \cite[Lemma 4.8]{allan2026rough} we get\footnote{Note that we work in a continuous setting, hence we can take the norms in closed intervals.}
\begin{multline}\label{ineq 12}
    \|\delta \mathcal{I}^X\|_{p\text{-var},k;[s,t]} +  \|\mathbb E_\bullet(
       \delta \mathcal{I}^X-H_\bullet\delta\eta)\|_{
       \frac{p}{2}\text{-var},k;[s,t]} \leq C \Big(\|\delta H\|_{p\text{-var},k;[s,t]} + \|\delta H'\|_{p\text{-var},k;[s,t]}\\
       + \|\mathbb{E}_\bullet R^H\|_{\frac{p}{2}\text{-var},k;[s,t]} + \sup_{u\in [s,t]}\|H_u\|_k \\+ (1 + \|\delta \eta\|_{p\text{-var};[s,t]}) \sup_{u\in [s,t]}\|H'_u\|_k\Big) ( \|\delta\eta\|_{p\text{-var};[s,t]}
 + \|\eta^{(2)}\|_{\frac{p}{2}\text{-var};[s,t]}).
\end{multline}
For $s\leq u \leq v \leq t$ we have
\begin{align}
    &\|\delta H_{u,v}\|_k \leq \| \delta \Phi_{u,v}\|_{k,\infty} \|X_u\|_k + \| \Phi_{v}\|_{\infty} \|\delta X_{u,v}\|_k,\label{ineq 13}
    \\
    &\|\delta H'_{u,v}\|_k \leq \| \delta \Phi'_{u,v}\|_{k,\infty} \|X_u\|_k + \| \Phi'_{v}\|_{\infty} \|\delta X_{u,v}\|_k \nonumber\\
    &\qquad+ \| \delta \Phi_{u,v}\|_{k,\infty} \| \Phi_{u}\|_{\infty} \|X_u\|_k + \| \Phi_{v}\|_{\infty} \big(\| \delta \Phi_{u,v}\|_{k,\infty} \|X_u\|_k + \| \Phi_{v}\|_{\infty} \|\delta X_{u,v}\|_k \big)\nonumber\\
    &\qquad+ \| \delta \Phi_{u,v}\|_{k,\infty} \|F'_u\|_k + \| \Phi_{v}\|_{\infty} \|\delta F'_{u,v}\|_k \label{ineq 14}
\end{align}
and from standard algebraic manipulations in conjunction with conditional H\"older's inequality, see for example the proof of \cite[Proposition 3.3]{bugini2024malliavin}, we see that
\begin{align}
    &\|\mathbb{E}_u R^H_{u,v}\|_k = \|\mathbb{E}_u[\delta \Phi_{u,v} \delta X_{u,v} + \Phi_u R^{X}_{u,v} + R^\Phi_{u,v} X_u]\|_k \nonumber\\
    &\qquad\leq \|\mathbb{E}_u [\delta \Phi_{u,v} \delta X_{u,v}]\|_k + \|\mathbb{E}_u [\Phi_u R^{X}_{u,v}] \|_k + \|\mathbb{E}_u [R^\Phi_{u,v} X_u] \|_k \nonumber\\
    &\qquad\leq \|\delta \Phi_{u,v}\|_{k,\infty} \|\delta X_{u,v}\|_k + \|\Phi_u\|_\infty \|\mathbb{E}_u  R^{X}_{u,v} \|_k + \|\mathbb{E}_u  R^{\Phi}_{u,v} \|_\infty \|X_u\|_k. \label{ineq 15}
\end{align}
Thus, from Cauchy-Schwarz inequality and because $p > 2$ we have
\begin{align}\label{eq. 018}
  &\|\delta H\|_{p\text{-var},k;[s,t]} +  \|\delta H'\|_{p\text{-var},k;[s,t]} +  \|\mathbb{E}_\bullet R^H\|_{\frac{p}{2}\text{-var},k;[s,t]}+ \sup_{u\in [s,t]}\|H_u\|_k + \sup_{u\in [s,t]}\|H'_u\|_k\nonumber\\
  & \leq \mathscr{M}\Big( (1 + \mathscr{M}) \mathcal{N}_{k;[s,t]}(X) + \|F'_s\|_k +  \|\delta F'\|_{p\text{-var},k;[s,t]} \Big).
\end{align}

So, from the contraction property of conditional expectation, BDG inequality 
and the above we get
\begin{align}
 &\|\delta A^X\|_{p\text{-var},k;[s,t]}
 +\|\mathbb E_\bullet\delta A^X\|_{
       \frac{p}{2}\text{-var},k;[s,t]}
 \leq 2 \mathcal{N}_{k;[s,t]}(X) \int_{s}^{t}\|\|b_v |\mathcal{F}_v \|_k \|_\infty \ud v ,\label{eq. 19}\\
 &\|\delta M^X\|_{p\text{-var},k;[s,t]}
 \leq \mathcal{N}_{k;[s,t]}(X)
 \left(\int_s^t\|\| \sigma_v |\mathcal{F}_v \|_k \|_\infty^2\,\ud v\right)^{\frac{1}{2}},
 \qquad \mathbb E_u\delta M^X_{u,v}=0,
 \label{eq. 20}\\
 &\|\delta \mathcal{I}^X\|_{p\text{-var},k;[s,t]} 
 +\|\mathbb E_\bullet(
       \delta \mathcal{I}^X-H_\bullet\delta\eta)\|_{
       \frac{p}{2}\text{-var},k;[s,t]}\nonumber\\
 &\hspace{28mm}\leq C\mathcal{F}_{k;[s,t]}(F) +
 C \bigl(w_\eta(s,t)^{\frac{1}{p}}
          + w_\eta(s,t)^{\frac{2}{p}}\bigr)
 \mathcal N_{k;[s,t]}(X).
 \label{eq. 21}
\end{align}
Hence, because
\begin{align*}
 \delta X_{u,v} &= \delta F_{u,v} + \delta A^X_{u,v} + \delta M^X_{u,v} + \delta \mathcal{I}^X_{u,v},\\
 R^X_{u,v} &= (\delta F_{u,v} -F'_u \delta\eta_{u,v}) +\delta A^X_{u,v}+\delta M^X_{u,v}
  + (\delta \mathcal{I}^X_{u,v}-H_u\delta\eta_{u,v}),
\end{align*}
and since $p > 2$, the estimates \eqref{eq. 19}--\eqref{eq. 21} imply
\begin{align}
 \|\delta X\|_{p\text{-var},k;[s,t]}
 &\leq
 C\mathcal{F}_{k;[s,t]}(F)
 + C w_{\mathcal{L},k}(s,t)^{\frac{1}{p}}
      \mathcal{N}_{k;[s,t]}(X),\label{eq. 22}\\
 \|\mathbb{E}_\bullet R^X\|_{
       \frac{p}{2}\text{-var},k;[s,t]}
 &\leq
 \|\mathbb E_\bullet R^F\|_{
       \frac{p}{2}\text{-var},k;[s,t]}
+ C\mathcal{F}_{k;[s,t]}(F) + C w_{\mathcal{L},k}(s,t)^{\frac{1}{p}}
      \mathcal{N}_{k;[s,t]}(X),
 \label{eq. 23}
\end{align}
whenever $$w_{\mathcal{L},k}(s,t)\leq 1.$$ 
As we already mentioned, the martingale vanishes after conditional expectation. Moreover,
\begin{equation*}
 \delta X'_{u,v} = \delta F'_{u,v} + \delta H_{u,v}= \delta F'_{u,v} + \Phi_u\delta X_{u,v} + \delta \Phi_{u,v} X_v,
\end{equation*}
so that
\begin{multline}\label{eq. 24}
 \|X'_s\|_{k}
 + \|\delta X'\|_{p\text{-var},k;[s,t]} \leq \|F'_s\|_k + \|H_s\|_k + \|\delta F'\|_{p\text{-var},k;[s,t]} + \|\delta H\|_{p\text{-var},k;[s,t]}
 \\
 \leq \|F'_s\|_k +  \|\delta F'\|_{p\text{-var},k;[s,t]} + C\big( \|X_s\|_{k}
 + \|\delta X\|_{p\text{-var},k;[s,t]}
 + w_{\Phi,k}(s,t)^{\frac{1}{p}}
       \|X\|_{\mathbb S^k;[s,t]}\big).
\end{multline}
Additionally, because $k > p$, Lemma~\ref{lem:uni-p-var} gives
\begin{equation*}
 \|X\|_{\mathbb S^k;[s,t]}
 \leq\|X_s\|_{k}
 +C \|\delta X\|_{p\text{-var},k;[s,t]}.
\end{equation*}
Combining this with
\eqref{eq. 22}--\eqref{eq. 24} we get for every $(s,t)\in \Delta([0,T])$ the direct local estimate
\begin{equation}\label{eq. 25}
 \mathcal N_{k;[s,t]}(X)
 \leq C \left(
       \|X_s\|_{k}
       +\mathcal{F}_{k;[s,t]}(F)\right)
 + C w_{\mathcal{L},k}(s,t)^{\frac{1}{p}}
       \mathcal{N}_{k;[s,t]}(X).
\end{equation}

To finish the proof, we define an appropriate partition and derive our global estimate. From the previous calculations immediately follows that by now $C\geq1$. So, let
$\delta_{C}:=
\frac{1}{(2C)^{p}}$, set $\tau_0 :=0$, and define
\begin{equation*}
 \tau_{i+1}
 :=\inf\{t\in(\tau_i,T]:
       w_{\mathcal{L},k}(\tau_i,t)\geq \delta_{C}\}\wedge T.
\end{equation*}
The number $N_*:=\min\{i\in\mathbb{N}:\tau_i=T\}$ satisfies
\begin{equation*}
 N_*\leq 1 + \delta_{C}^{-1} w_{\mathcal{L},k}(0,T),
\end{equation*}
and is therefore bounded under
\eqref{eq:BV-coefficient-bound}.  Moving the last term in
\eqref{eq. 25}, evaluated on
$I_i=[\tau_i,\tau_{i+1}]$, to the left hand side and then dividing both sides by the resulting coefficient gives\footnote{Note that $x\mapsto \frac{1}{1-x}$ is bounded over $[0,\frac{1}{2}]$ by $2$.} for every $i\in \{0,\dots,N_*-1 \}$
\begin{equation}\label{eq. 26}
 \mathcal N_{k;I_i}(X)
 \leq 2 C\left(
 \|X_{\tau_i}\|_{k}
 +\mathcal{F}_{k;I_i}(F)\right).
\end{equation}
Next, endpoint iteration and H\"older's inequality over the finite index set $\{1,...N_*\}$
gives
\begin{align*}
 \max_{0\leq i\leq N_*-1}\|X_{\tau_i}\|_{k}
 &\leq (2C)^{N_*-2}\bigg(
 \|\xi\|_{k}
 +\sum_{i = 0}^{N_* - 2}\mathcal{F}_{k;I_i}(F)\bigg),\\
 \sum_{i = 0}^{N_* - 1}\mathcal{F}_{k;I_i}(F)
 &\leq
 N_*\big(\|\delta F\|_{p\text{-var},k;[0,T]} + \|\delta F'\|_{p\text{-var},k;[0,T]}
 +\|\mathbb{E}_\bullet R^F\|_{
       \frac{p}{2}\text{-var},k;[0,T]}\big)+ \sum_{i=0}^{N_*-1}\|F'_{\tau_i}\|_k,
\end{align*}
where obviously
\begin{equation*}
    \sum_{i=0}^{N_*-1}\|F'_{\tau_i}\|_k \leq N_*(\|F'_0\|_k + \|\delta F'\|_{p\text{-var},k;[0,T]}).
\end{equation*}
Hence, we get
\begin{align*}
    \sum_{i = 0}^{N_* - 1}\mathcal{F}_{k;I_i}(F) &\leq 2 N_{*}\|(F,F')\|_{
    \mathbf D_\eta^{(p,p)\text{-var}}\mathbb L^k},\\ \max_{0\leq i\leq N_*-1}\|X_{\tau_i}\|_{k}
 &\leq 2N_* (2C)^{N_*-2} \big(
 \|\xi\|_{k}
 +|(F,F')\|_{
    \mathbf D_\eta^{(p,p)\text{-var}}\mathbb L^k}\big),
\end{align*}
and thus from \eqref{eq. 26} we have
\begin{align}\label{eq. 27}
  \sum_{i=0}^{N_*-1} \mathcal N_{k;I_i}(X) 
\leq  4N_* (2C)^{N_*-1} \big(
 \|\xi\|_{k}
 + \|(F,F')\|_{
    \mathbf D_\eta^{(p,p)\text{-var}}\mathbb L^k}
 \big).
\end{align}
Now, for $Y \in \{X,X'\}$ because $\delta Y$ is additive and $k, p\geq 1$, for every $0\leq a\leq b\leq c\leq T$ we have
\begin{equation*}
\|\delta Y\|_{p\text{-var},k;[a,c]}
\leq 2\big(
\|\delta Y\|_{p\text{-var},k;[a,b]}
+
\|\delta Y\|_{p\text{-var},k;[b,c]}\big).
\end{equation*}
Consequently, we get
\begin{equation*}
\|\delta Y\|_{p\text{-var},k;[0,T]}
\leq 2^{N_*}
\sum_{i=0}^{N_*-1}
\|\delta Y\|_{p\text{-var},k;[\tau_i,\tau_{i+1}]} = 2^{N_*} \sum_{i=0}^{N_*-1}
\|\delta Y\|_{p\text{-var},k;I_i}.
\end{equation*}
Furthermore, because $\sup_{t\in[0,T]}|X_t| \leq \sum_{i=0}^{N_*-1} \sup_{t\in I_i}|X_t|$, we immediately see that
\begin{equation*}
  \|X\|_{\mathbb S^k;[0,T]}  \leq \sum_{i=0}^{N_*-1} \|X\|_{\mathbb S^k;I_i}.
\end{equation*}
Finally, we need to deal with $R^X$ which is not additive as we have
\begin{equation*}
 R^X_{u,z}
 =R^X_{u,v}+R^X_{v,z}
  +\delta X'_{u,v}\delta\eta_{v,z},
 \qquad 0\leq u\leq v\leq z \leq T.
\end{equation*}
To do this we remind that from the previous relation one can use the same method of proof as in the additive case above to show that for $p > 2$ and $0\leq a\leq b\leq c \leq T$ we have
\begin{align*}
\|\mathbb E_\bullet R^X\|_{\frac{p}{2}\text{-var},k;[a,c]}
&\leq 3\big(
\|\mathbb E_\bullet R^X\|_{\frac{p}{2}\text{-var},k;[a,b]}
+
\|\mathbb E_\bullet R^X\|_{\frac{p}{2}\text{-var},k;[b,c]}
+
\|\delta X'\|_{p\text{-var},k;[a,b]}
\|\delta\eta\|_{p\text{-var};[b,c]}\big)\nonumber\\
&\leq 3\big(\|\mathbb E_\bullet R^X\|_{\frac{p}{2}\text{-var},k;[a,b]}
+
\|\mathbb E_\bullet R^X\|_{\frac{p}{2}\text{-var},k;[b,c]} + \|\delta X'\|_{p\text{-var},k;[a,b]}
\|\delta\eta\|_{p\text{-var};[0,T]}\big).
\label{eq:remainder-concatenation}
\end{align*}
Hence, by refining $[0,T]$ with the finitely many partition endpoints we get
\begin{align*}
  \|\mathbb E_\bullet R^X\|_{\frac{p}{2}\text{-var},k;[0,T]} &\leq 3^{N_*}\bigg( \sum_{i=0}^{N_*-1} \|\mathbb E_\bullet R^X\|_{\frac{p}{2}\text{-var},k;I_i} + \|\delta\eta\|_{p\text{-var};[0,T]} \sum_{i=0}^{N_*-1} \|\delta X'\|_{p\text{-var},k;I_i}\bigg)\\
  &\leq 3^{N_*}(1\vee\|\delta\eta\|_{p\text{-var};[0,T]}) \bigg( \sum_{i=0}^{N_*-1} \|\mathbb E_\bullet R^X\|_{\frac{p}{2}\text{-var},k;I_i} +  \sum_{i=0}^{N_*-1} \|\delta X'\|_{p\text{-var},k;I_i}\bigg).
\end{align*}
Combining all the above and plugging them to the left hand side of \eqref{eq. 27} yields \eqref{eq:BV-primary-estimate}.
\end{proof}

\begin{rmk}
\begin{enumerate}
    \item[(i)] Although Theorem \ref{thm 3.9} assumes that $X$ is a solution to an RSDE in a Hölder setting, since this is the version needed in the remainder of the text, the proof also applies for a solution notion defined in the more general $p$-variation setting, as it does not rely on this assumption.

    \item[(ii)]Inspection of the proof immediately yields that we can weaken our bound for $\mathscr{M}$ to 
    \begin{equation*}
        \int_0^T\big(\|\| b_t |\mathcal{F}_t \|_k \|_\infty +\|\| \sigma_t|\mathcal{F}_t \|_k \|_\infty ^2\big)\ud t + \sup_{t\in [0,T]}\{ 
   \|\Phi_t\|_\infty + \|\Phi'_t\|_\infty \} 
   + \|(\Phi,\Phi')\|_{\mathbf{D}_{\eta}^{(p,p)\text{-var}} \mathbb{L}^{k,\infty}_{lin}}
   + \tnorm{\boldsymbol{\eta}}_{{p}\text{-var}}.
    \end{equation*}
\end{enumerate}
\end{rmk}

\subsubsection{Estimate of affine RSDEs using the H\"older seminorms}
By replacing $\|X\|_{\mathbb S^k}$ with $\sup_{t\in[0,T]}\|X_t\|_k$ we can avoid using Lemma \ref{lem:uni-p-var} in the proof of Theorem \ref{thm 3.9}, thus removing the $p<k$ condition. Hence, taking $p=\frac{1}{\beta'}$, we immediately obtain the following slightly weakened but sufficient for our purposes\footnote{Alternatively, one can use \cite[Theorem 3.10]{bugini2024malliavin}.} H\"older version of Theorem \ref{thm 3.9}, which will be our main tool for deriving the rough SMP via linear perturbation.

\begin{thm}\label{thm 3.11}
Let $\frac{1}{3}<  \beta'\leq \beta\leq\alpha < \frac{1}{2}$, $k\in[2,\infty)$, $\xi\in \mathbb{L}^k(\mathcal{F}_0;\mathbb{R}^d)$ and $X$ an $\mathbb{L}^k$-integrable solution of \eqref{eq ex_lin} in the sense of Davie\footnote{Where $F', \Phi, \Phi'$ were assumed constant with respect to $\lambda$ without loss of generallity due to Proposition \ref{Prop 3.3}.} 
with $F_0 = 0$. 
Then, for every $\mathscr{M} > 0 $ satisfying
\begin{align*}
   \sup_{t\in [0,T]}\Big\{\|\| b_t |\mathcal{F}_t \|_k \|_\infty + \|\| \sigma_t |\mathcal{F}_t \|_k \|_\infty + 
   \|\Phi_t\|_\infty + \|\Phi'_t\|_\infty \Big\} 
   + \|(\Phi,\Phi')\|_{\mathbf{D}_{\eta}^{(\beta',\beta')} \mathbb{L}^{k,\infty}_{lin}}
   + \tnorm{\boldsymbol{\eta}}_{\beta'} \leq \mathscr{M}
\end{align*}
there exists $ C =  C(\mathscr{M},\beta',k,T)\geq 0$ such that
\begin{equation*}
 \sup_{t\in[0,T]}\|X_t\|_k
 \leq C \big(
 \|\xi\|_{k}
 + \|(F,F')\|_{
    \mathbf D_\eta^{\beta',\beta'}\mathbb L^k}
 \big).
\end{equation*}
\end{thm}

\subsection{General affine rough BSDEs}\label{sec:a-rbsde}
In this Subsection and in Section \ref{sec:max_principles} we assume that $(\Omega,\mathcal F,(\mathcal F_t)_{t\in\mathbb{R}_+},\mathbb{P})$ is the usual augmentation of $W$ and a $\sigma$-algebra $\mathcal{F}_0^0$ independent of $W$.

To allow a variational calculus argument, in this subsection we provide a strong solution to the following class of rough BSDEs

\begin{equation}\label{eq:l-bsde}
    Y_t= \xi + \int_t^T f(s,Y_s, Z_s) \ud s - \int_t^T Z_s \ud W_s + \int_t^T (g_s Y_s + h_s,g_sY_s'+g'_sY_s+h'_s) \ud\boldsymbol{\eta}_s.
\end{equation}
In contrast to the linear case considered in \cite{LiangTang2025}, here we allow $(g,g')$ to be in $\mathbf{D}^{(p,p')\text{-var}}_\eta {\LL}^{m,\infty}_{ {lin}}\big([0,T];$ $L(\mathbb{R}^{d},\mathbb{R}^{d\times \ell})\big)$, and moreover, build the solution in $\mathbb{L}^m$ for $2\leq m\leq p<3$. The above two points are  crucial for the adjoint equation considered in the next section, note also Remark \ref{rmk 3.3}. 

Only for this subsection, assuming $\boldsymbol{\eta} \in \mathscr{C}^{p\text{-var}}([0,T];\R^\ell)$, to ease the calculations, we shall use the following norm on stochastic controlled  linear vector fields: \newline For every 
$ \ (g,g')\in
  \mathbf D_\eta^{(q,q')\text{-var}}\mathbb L^{m,\infty}_{lin}\big([0,T];L(\mathbb{R}^{d},\mathbb{R}^{d\times \ell})\big)$ we define
\begin{align*}
  & \|(g,g')\|_{\mathbf D_\eta^{(q,q')\text{-var}}\mathbb L^{m,\infty}_{lin} }:= \|(g,g')\|_{\mathbf D_\eta^{(q,q')\text{-var}}\mathbb L^{m,\infty}_{lin} } 
  +\sup_{t\in [0,T]} \|g_t\|_{\infty} +  \sup_{t\in [0,T]}\|g'_t\|_{\infty},
\end{align*}
where by abusing notation for brevity, on the right hand side is the standard definition of the seminorm.
Again for simplicity of calculations, we shall always assume $p' \geq p \geq 2$ and $q' \ge q$. 

Now, for the well-posedness of \eqref{eq:l-bsde} we need several lemmas. 

\begin{lemma}[Product of stochastic controlled rough paths]\label{lem:product}
Let $\boldsymbol{\eta} \in \mathscr{C}^{p\text{-var}}([0,T];\R^\ell)$. Suppose
$ (g,g')\in
  \mathbf D_\eta^{(q,q')\text{-var}}\mathbb L^{m,\infty}_{lin}\big([0,T];L(\mathbb{R}^{d},\mathbb{R}^{d\times \ell})\big)$ and $(y,y') \in \mathbf D_\eta^{(q,q')\text{-var}}\mathbb L^{m}([0,T];\mathbb{R}^d)$. Define
$
  z_t:=g_t y_t,\ 
  z'_t:=g'_t y_t+g_t y'_t .
$ Then,
$
  (z,z')\in
  \mathbf D_\eta^{(q,q')\text{-var}}\mathbb L^{m}([0,T];\mathbb{R}^{d\times \ell}).
$
Moreover,
\be\label{est:cr-product}
  \|(z,z')\|_{\mathbf D_\eta^{(q,q')\text{-var}}\mathbb L^{m }}
  \lesssim
  \|(g,g')\|_{\mathbf D_\eta^{(q,q')\text{-var}}\mathbb L^{m,\infty}_{lin}}
  \bigl(\|y_0\|_{m}+\|(y,y')\|_{\mathbf D_\eta^{(q,q')\text{-var}}\mathbb L^{m }}\bigr) .
\ee
\end{lemma}
\begin{proof}
For \(s<t\), set
$
  R^g_{s,t}:=\delta g_{s,t}-g'_s \delta \eta_{s,t},
\ 
  R^y_{s,t}:=\delta y_{s,t}-y'_s \delta\eta_{s,t}.
$
Then
\begin{equation}\label{eq:decom-z}
\begin{aligned}
  \delta z_{s,t}
  &= g_ty_t-g_sy_s                                           \\
  &= (\delta g_{s,t})y_s+g_s(\delta y_{s,t})
     +(\delta g_{s,t})(\delta y_{s,t})                       \\
  &= (g'_s y_s+g_s y'_s) \delta\eta_{s,t}+R^z_{s,t}
   = z'_s \delta\eta_{s,t}+R^z_{s,t},
\end{aligned}
\end{equation}
where
$
  R^z_{s,t}
  =
  R^g_{s,t}y_s
  +g_sR^y_{s,t}
  +(\delta g_{s,t})(\delta y_{s,t}).
$
We now estimate each term. Since \(g \in \mathbf D_\eta^{(q,q')\text{-var}}\mathbb L^{m,\infty}_{lin} \),
\begin{equation*}
  \|\delta z_{s,t}\|_{{m}} \le
  \|\delta g_{s,t}\|_{m,\infty}\|y_s\|_{m}
  +\|g_t\|_{\infty}\|\delta y_{s,t}\|_{m}.
\end{equation*}
Hence
\begin{equation}\label{est:dz}
  \|\delta z\|_{q\text{-var};m}
  \lesssim
  \|g\|_{\infty}\|y\|_{q\text{-var},m}
  +\sup_{r\in[0,T]}\|y_r\|_{m}\|g\|_{q\text{-var},m,\infty}.
\end{equation}
For the Gubinelli derivative $
  z'=g'y+gy',
$
we have
\[
  \delta z'_{s,t}
  = (\delta g'_{s,t})y_s+g'_t(\delta y_{s,t})
     +(\delta g_{s,t})y'_s+g_t(\delta y'_{s,t}),   
\]
which implies
\begin{equation}\label{est:dz'}
\begin{split}
  \|\delta z'\|_{q'\text{-var},m}
  \lesssim & \sup_{r\in[0,T]} \|y_r\|_{m}\|\delta g'\|_{q'\text{-var},m,\infty}
  + \|\delta y\|_{q'\text{-var},m} \|g'\|_{\infty}  \\         
  & +\sup_{r\in[0,T]} \|y_r'\|_{m}\|\delta g\|_{q'\text{-var},m,\infty}
  +\| \delta y'\|_{q'\text{-var},m }\|g\|_{\infty}.  
\end{split} 
\end{equation}
Finally, for the controlled remainder, note that
\[
  \mathbb E_s R^z_{s,t}=
  y_s \mathbb E_sR^g_{s,t}
  +g_s \mathbb E_sR^y_{s,t}
  +\mathbb E_s \left[(\delta g_{s,t})(\delta y_{s,t})\right].
\]
Thus
\be\label{est:ERZ}
\begin{aligned}
  \|\mathbb E_\bullet R^z\|_{\frac{qq'}{q+q'}\text{-var},m}
  &\lesssim
  \sup_{r\in[0,T]} \|y_r\|_{m}
  \|\mathbb E_\bullet R^g\|_{\frac{qq'}{q+q'}\text{-var},m,\infty}
  +\|g\|_{\infty}
  \|\mathbb E_\bullet  R^y\|_{\frac{qq'}{q+q'}\text{-var},m}        \\
  &\quad
  +\| \E_\bullet[\delta g\,\delta y]\|_{\frac{qq'}{q+q'}\text{-var},m} .
\end{aligned}
\ee
For the last term on the right hand side of the above inequality, note that
$$
 \E_s[(\delta g_{s,t})(\delta y_{s,t})] \le \| \delta g_{s,t} | \MF_s\|_{2} \| \delta y_{s,t} | \MF_s\|_{2} \le \|\delta g\|_{q\text{-var},m,\infty;[s,t]} \| \delta y_{s,t} | \MF_s\|_{m} .
$$
It follows that 
$$
\| \E_\bullet[\delta g\,\delta y]\|_{\frac{q}{2}\text{-var},m} \lesssim \|\delta g\|_{q\text{-var},m,\infty} \|\delta y\|_{q\text{-var},m }.
$$
Combine the above estimate with \eqref{est:ERZ}, we obtain
\be\label{est:ERZ-2}
 \|\mathbb E_\bullet R^z\|_{\frac{qq'}{q+q'}\text{-var},m} \lesssim  \|(g,g')\|_{\mathbf D_\eta^{(q,q')\text{-var}}\mathbb L^{m,\infty}_{lin}}
  \bigl(\|y_0\|_{m}+\|(y,y')\|_{\mathbf D_\eta^{(q,q')\text{-var}}\mathbb L^{m }}\bigr),
\ee
where we apply the assumption $q'\ge q.$ Note that $$
\|z'_0\|_{m}\le\|g'_0\|_\infty\|y_0\|_{m}+\|g_0\|_\infty\|y'_0\|_{m}.
$$
Therefore, in view of \eqref{est:dz}, \eqref{est:dz'} and \eqref{est:ERZ-2}, our claim follows.
\end{proof}

We need more lemmas as shown in the following, which are generalizations of those in \cite{LiangTang2025} from $\LL^2$ to $\LL^m$, so their proofs are omitted.

\begin{lemma}\label{lem:maring}
Suppose $M$ is a martingale with finite $m$-th moment. Then, for any $\boldsymbol{\eta} \in \mathscr{C}^{p\text{-var}}([0,T];\mathbb{R}^\ell)$, and any $q' \ge q \ge m  \ge 2,$ we have 
$$
\| (M,0)\|_{\mathbf D_\eta^{(q,q')\text{-var}} \LL^m} \lesssim \|\delta M_{0,T}\|_{m}.
$$
\end{lemma}

\begin{lemma}\label{lem:smoothint}
   For any $\phi \in \HH^m([0,T];\mathbb{R}^{d\times w})$ we have 
   $\int_0^. \phi_r \ud r \in C^{\frac{2m}{m+2}\text{-var}} \LL^m([0,T];\mathbb{R}^{d})$, and moreover, for any $(\boldsymbol{\eta},q,q') $ as in the above lemma,
\be
\bigg\| \bigg(\int_0^. \phi_r \ud r, 0\bigg) \bigg\|_{\mathbf D_\eta^{(q,q')\text{-var}} \LL^m} \lesssim T^{\frac{1}{2}} \|\phi\|_{\HH^m}.
\ee
\end{lemma}

Now we are ready to prove the well-posedness of \eqref{eq:l-bsde}. 

\begin{assumptions}\label{assu:bsde}
    \begin{itemize}
    \item[(i)] It holds $\xi \in \LL^m(\MF_T)$.
    
    \item[(ii)] Let $f: \Omega \times [0,T]\times \R^d \times \R^{d\times w} \rightarrow \R^d$ be progressively measurable 
    for each $(y,z)\in \mathbb{R}^d\times\mathbb{R}^{d\times w}$, and such that $f(\cdot,\cdot,0,0) \in \HH^{m}([0,T];\mathbb{R}^{d\times w})$ while 
$$
|f(t,y,z)-f(t,\bar y, \bar z)| \lesssim |y-\bar y|+|z- \bar z|.
$$

\item[(iii)] Let $(g,g') \in \mathbf D_\eta^{(p,p')\text{-var}}\mathbb L^{m,\infty }_{lin}\big([0,T];L(\mathbb{R}^{d},\mathbb{R}^{d\times \ell})\big),$ $(h,h')\in \mathbf D_\eta^{(p,p')\text{-var}}\mathbb L^m([0,T];\mathbb{R}^{d\times \ell})$, $h_0\in\mathbb L^m(\mathcal{F}_0)$, with $2\le m\le p<3$, $p\le p'<\infty$, and $2/p+1/p'>1$.
    \end{itemize}
\end{assumptions}

\begin{thm}\label{ex_uni_lin_bsde}
Suppose Assumptions \ref{assu:bsde} hold. Then, there exists a unique solution $((Y,Y'),Z)$ to \eqref{eq:l-bsde}, with $Y'=-(gY+h)$, such that $(Y,Y') \in \BD^{(p,p)-\text{var}}_{\eta}\LL^m([0,T];\mathbb{R}^{d})$ and $Z \in \HH^m([0,T];\mathbb{R}^{d\times w})$.
\end{thm}

\begin{proof}

Following the arguments in \cite{FrizZhang2018} and \cite{LiangTang2025}, we establish a fixed-point theorem on a small time interval, but with respect to a modified norm on the space of stochastic controlled rough paths. Consider the space of $\BD:= \BD^{(p,p)\text{-var}}_{\eta}\LL^m \times \HH^m$, equipped with the norm 
\[
\begin{split}
\|(y,y',z)\|_{(K)}:=& K(\|\delta y\|_{p-var, m} + \|y_T\|_{m})+ (\|\delta y'\|_{p-var, m} + \|y'_T\|_{m})+ K \|\E_\bullet R^y\|_{\frac{p}{2}-var, m} + \|z\|_{\HH^m}, \\
=& K \|y\|_{C^{p\text{-var}}\LL^m}+ \|y'\|_{C^{p\text{-var} } \LL^m} + K \|\E_\bullet R^y\|_{\frac{p}{2}-var, m} + \|z\|_{\HH^m}\\
=:& \|(y,y')\|_{\BD^{(p,p)\text{-var} }_{\eta}\LL^m}^{(K)} + \|z\|_{\HH^m}
\end{split}
\]
where $K>1$ is some constant to be determined, and $\|\cdot \|_{\BD^{(p,p)\text{-var}}_{\eta}\LL^m}^{(K)}$ is defined naturally by the above. It is obvious that the above metric is equivalent to the natural metric in the product space, and thus $\BD$ is complete under the above norm. Consider the following mapping $\psi$ on $\BD.$ For any $(y,y',z) \in \BD,$ let $Z$ be the unique $\HH^m$ element by the martingale decomposition theorem, such that 
\[
\begin{split}
\int_0^t Z_r \ud W_r &= \E_t \bigg[\xi + \int_0^T f(r,y_r, z_r) \ud r + \int_0^T (g_r y_r+h_r) \ud \BBeta_r \bigg]\\
&\ \ \ - \E\bigg[\xi + \int_0^T f(r,y_r, z_r) \ud r + \int_0^T (g_r y_r+h_r) \ud \BBeta_r \bigg]\\
&=: M_t-M_0. 
\end{split}
\]
It is easy to check $M$ is a martingale with finite $m$-the moment, and thus the above $Z$ is well-defined. Moreover, let 
\be\label{eq:defpsi}
\begin{split}
Y_t :=& \xi + \int_t^T f(r,y_r, z_r) \ud r + \int_t^T (g_r y_r+h_r) \ud \BBeta_r - \int_t^T Z_r \ud W_r,  \\
=& M_t - \int_0^t f(r,y_r, z_r) \ud r - \int_0^t (g_r y_r+h_r) \ud \BBeta_r,
\end{split}
\ee
and $Y':=-(gy+h).$ In view of lemmas \ref{lem:product}, \ref{lem:roughstochint}, we see that $(Y,Y') \in \BD^{(p,p)\text{-var}}_{\eta} \LL^m.$ Let 
$$
\psi(y,y',z) := (Y,Y',Z),
$$
and it defines a mapping on $\BD.$ 

Now we show that $\psi$ defines a contraction when $T$ small. For any $(y,z)$, $(\by, \bz) \in \BD,$ let $\Delta y:= y-\by$, and $\Delta z, \Delta f$ defined similarly. We take the difference of identities \eqref{eq:defpsi} with initial input $(y,y',z)$ and $(\by, \by', \bz)$, to obtain
\be\label{eq:Delta-y}
\Delta Y_t= \int_t^T \Delta f_r \ud r + \int_t^T g_r \Delta y_r \ud\BBeta_r - \int_t^T \Delta Z_r \ud W_r.
\ee
In the following we estimate every term on the right hand side of the above identity. By Lipschitzness of $f$, we have 
$$
\|\Delta f\|_{\HH^m} \lesssim \sup_{t \in [0,T]} \| \Delta y_t\|_{m} + \| \Delta z \|_{\HH^m},
$$
and thus by Lemma \ref{lem:smoothint}, we have 
\be\label{est:soothest}
\bigg\|\bigg(\int_0^. \Delta f_r \ud r, 0\bigg) \bigg\|_{\BD^{(p,p)\text{-var} }_\eta \LL^m}^{ (K)} \lesssim T^{\frac{1}{2}}K  \Big( \|\Delta y \|_{C^{p\text{-var}}\LL^m}+\|\Delta z\|_{\HH^m} \Big).
\ee
Moreover, by Lemma \ref{lem:product}, 
$$
\|(g\Delta y, g \Delta y'+ g' \Delta y)\|_{\BD^{(p,p')\text{-var}}_\eta \LL^m}^{(K)} \lesssim \|(g,g')\|_{\BD^{(p,p')\text{-var} }_\eta \LL^{m,\infty}_{lin}} \|(\Delta y,  \Delta y')\|_{\BD^{(p,p)\text{-var}}_\eta \LL^m}^{(K)},
$$
and thus, in view of Lemma \ref{lem:roughstochint},
\begin{align}\nonumber
&\bigg\|\bigg(\int_0^. g_r \Delta y_r \ud\BBeta_r, g \Delta y \bigg)\bigg\|_{\BD^{(p,p)\text{-var}}_\eta \LL^m}^{(K)} \\ \nonumber
&\ \ \  \lesssim K\|(g\Delta y, g \Delta y'+ g' \Delta y)\|_{\BD^{(p,p')\text{-var} }_\eta \LL^m}^{(K)}\tnorm{\BBeta}(1 \vee \tnorm{\BBeta}) + \|(g,g')\|_{\BD^{(p,p')\text{-var} }_\eta \LL^{m,\infty}_{lin}} \|\Delta y\|_{C^{p\text{-var} }\LL^m } \\ \label{est:roughint}
&\ \ \  \lesssim \|(g,g')\|_{\BD^{(p,p')\text{-var} }_\eta \LL^{m,\infty}_{lin}} (K \tnorm{\BBeta }(1 \vee \tnorm{\BBeta }) + \frac1K )\|(\Delta y,  \Delta y')\|_{\BD^{(p,p)\text{-var}}_\eta \LL^m}^{(K)}.
\end{align}
For the last term on the right hand side of \eqref{eq:Delta-y},
note that 
$$
\Delta M_t= \E_t\bigg[\int_0^T \Delta f_r \ud r + \int_0^T g_r \Delta y_r \ud\BBeta_r \bigg]=\Delta M_0 + \int_0^t \Delta Z_r \ud W_r.
$$
It follows by Doob's maximal inequality that 
\begin{align}\nonumber
 \|\Delta Z\|_{\HH^m}
&\lesssim  T^{\frac{1}{2}} \|\Delta f\|_{\HH^m} + \bigg\|\int_0^T g_r \Delta y_r \ud \BBeta_r \bigg\|_{m}\\  \nonumber
& \lesssim_{(g,g')}  T^{\frac{1}{2}} (\|\Delta y\|_{p\text{-var}, m} + \|\Delta z\|_{\HH^m}) + 
 K \tnorm{\BBeta }(1 \vee \tnorm{\BBeta }) \|(\Delta y,  \Delta y')\|_{\BD^{(p,p)\text{-var}}_\eta \LL^m}^{(K)}\\ \label{est:DeltaZ}
& \lesssim_{(g,g')} (T^{\frac{1}{2}} + K\tnorm{\BBeta }(1 \vee \tnorm{\BBeta })) \|\Delta y, \Delta y', \Delta z\|_{(K)}.
\end{align}
Then by Lemma \ref{lem:maring} and the above inequality, 
\be\label{est:deltaZ}
\bigg\|\bigg( \int_0^. \Delta Z_r dW_r, 0 \bigg) \bigg\|_{\BD^{(p,p)-\text{var}}_\eta \LL^m}^{(K)} \lesssim_{(g,g')}(KT^{\frac{1}{2}} + K^2\tnorm{\BBeta }(1 \vee \tnorm{\BBeta })) \|(\Delta y, \Delta y', \Delta z)\|_{(K)}.
\ee
In view of \eqref{eq:Delta-y}, \eqref{est:soothest}, \eqref{est:roughint} and \eqref{est:deltaZ},
\be
\|(\Delta Y, \Delta Y')\|^{(K)}_{\BD^{(p,p)-var}_\eta \LL^m} \lesssim_{(g,g')}(T^{\frac{1}{2}}K+K^2\tnorm{\BBeta}(1 \vee \tnorm{\BBeta})+ \frac1K)\| (\Delta y, \Delta y', \Delta z) \|_{(K)}.
\ee
Moreover, by \eqref{est:DeltaZ} and the above estimate, we can choose $K$ large and then $T$ small, such that 
$$
\|(\Delta Y, \Delta Y', \Delta Z) \|_{(K)} \le \frac12 \| (\Delta y, \Delta y', \Delta z) \|_{(K)},
$$
which justify a contraction of $\psi.$ Furthermore, note that the above choice of $K$ and $T$ is independent of the terminal $\xi.$ Thus, we can solve \eqref{eq:l-bsde} piece by piece from time $T$ backwardly. 
\end{proof}

\begin{rmk}[Scope of the multidimensional statement]
\label{rmk:multidimensional-rough-bsde}
Theorem~\ref{ex_uni_lin_bsde} is genuinely
multidimensional, while the well-posedness highly relies on the affine structure of the rough integrand.
For the general nonlinear and multidimensional case, \cite{LiangTang2025} imposes either a smallness condition or a structured condition, in view of the fact that transformed rough BSDEs are quadratic BSDEs. More recently, \cite{FrizSongZhangZhang2026}
develops an intrinsic theory of nonlinear rough BSDEs using rough
semimartingales. It turns out that the well-posedness of nonlinear multidimensional equations
requires an a priori uniform $\mathbb{L}^\infty$ bound on $Y$ (see \cite[Theorems~4.17 and~4.21]{FrizSongZhangZhang2026}).
\end{rmk}

\subsection{General rough stochastic product rule}

As before, we assume that $(\Omega,\mathcal F,(\mathcal F_t)_{t\in\mathbb{R}_+},\mathbb{P})$ is a complete filtered probability space which satisfies the usual conditions, and $W$ a $w-$dimensional Brownian motion with respect to $(\mathcal{F}_t)_{t\in[0,T]}$.

\begin{assumptions}\label{assu:rpf}
Let $\alpha, \beta', \beta'' \in (0,1), k_1, k_2, p \in (2,\infty)$ such that 
 $\frac{1}{3}< \frac{1}{p} \leq \beta''\leq \beta'\leq\alpha < \frac{1}{2}$.\footnote{Of course under this conditions we have $\alpha + \beta' + \beta'' >1$ and $\frac{2}{p} + \alpha > 1$.}
    \begin{itemize}
    \item[(i)] It holds 
    $$
\frac{1}{k_1} + \frac{1}{k_2} \leq \frac{1}{2}.
    $$
    \item[(ii)] It holds 
 $$\alpha+\beta'+ \beta'' \wedge \frac{1}{ k_2}>1.$$
    
    \item[(iii)] It holds
    $$
     \frac{2}{p} + \frac{1}{k_1} > 1.
    $$
    \end{itemize}
\end{assumptions}

\begin{rmk}\label{rmk 3.18}
The following Theorem is written in a mixed Hölder--$p$-variation setting as this is the version used in the rest of this work. However, one can easily extend it to purely $p$-variation setting, with the same proof under the obvious changes, with only altering (ii) of Assumptions \ref{assu:rpf} to $ \frac{2}{p} + \frac{1}{k_2} > 1$.
\end{rmk}

\begin{thm}\label{rough_product_formula}
 Let $e, d\in\mathbb{N}, \boldsymbol{\eta} \in \mathscr{C}^{0,\alpha}_g([0,T];\mathbb{R}^\ell)$ and Assumptions \ref{assu:rpf} hold true. Moreover, let $(X,X')\in \mathbf{D}^{\alpha,\beta'}_{\eta}\mathbb{L}^{k_1}([0,T];\mathbb{R}^{e\times d})$, $(Y,Y')\in \mathbf{D}^{(p,p)\text{-var}}_{\eta}\mathbb{L}^{k_2}([0,T];\mathbb{R}^{d})$,
 while
 $(X',X'')\in \mathbf{D}^{\beta',\beta''}_{\eta}\mathbb{L}^{k_1}\big([0,T];L(\mathbb{R}^\ell,\mathbb{R}^{e\times d})\big),$ $ (Y',Y'')\in \mathbf{D}^{(p,p)\text{-var}}_{\eta}\mathbb{L}^{k_2}\big([0,T];\mathbb{R}^{d\times \ell}\big)$, $X_0\in \mathbb{L}^{k_1}(\mathcal{F}_0;\mathbb{R}^{e\times d}), Y_0\in \mathbb{L}^{k_2}(\mathcal{F}_0;\mathbb{R}^{d})$, and $b^1, b^2, \sigma^1, \sigma^2$ progressively measurable processes with
 $$\mathbb{E}[|X'_0|^{k_1}]+\mathbb{E}[|Y'_0|^{k_2}] + \sum_{i\in\{1,2\}} \Bigg[\mathbb{E}\bigg[\bigg(\int_0^T |b^i_t|^2\ud t\bigg)^{\frac{k_i}{2}}\bigg] + \mathbb{E}\bigg[\bigg(\int_0^T |\sigma^i_t|^2\ud t\bigg)^{\frac{k_i}{2}}\bigg] \Bigg] < \infty,
 $$ 
 such that
 \begin{align*}
     X_t &= X_0 + \int_0^t b^1_s\ud s +\int_0^t \sigma^1_s\ud W_s + \int_0^t(X',X'')_s\ud \boldsymbol{\eta}_s,\\
     Y_t &= Y_0 + \int_0^t b^2_s\ud s + \int_0^t \sigma^2_s\ud W_s + \int_0^t(Y',Y'')_s\ud \boldsymbol{\eta}_s.
 \end{align*}
 Then, for every $t\in[0,T]$, $\mathbb{P}-$a.s., we have
 \begin{gather*}
      X_tY_t = X_0Y_0 + \int_0^t \Big(X_sb^2_s+b^1_sY_s+ \sum_{i=1}^q (\sigma^1_i)_s(\sigma^2_i)_s\Big)\ud s + \int_0^t \big( X_s \sigma^2_s + \sigma^1_s Y_s \big) \ud W_s + \int_0^t ({Z}',{Z}'')_s \ud \boldsymbol{\eta}_s,
 \end{gather*}
 where $({Z}',{Z}''):=\big(X Y' + X'Y,2X'Y'+  X Y'' + X'' Y\big) \in \mathbf{D}^{(p,p)\text{-var}}_{\eta}\mathbb{L}^{\frac{k_1k_2}{k_1+k_2}}\big([0,T];\mathbb{R}^{e\times \ell}\big)$.
\end{thm}

\begin{proof}
 First, by \cite[Proposition 4.13]{LiangTang2025} we have that $(Z',Z'') \in \mathbf{D}^{(p,p)\text{-var}}_{\eta}\mathbb{L}^{\frac{k_1k_2}{k_1+k_2}}([0,T];\mathbb{R}^{e\times \ell})$. Thus, the process $(\int_0^t ({Z}',{Z}'')_s \ud \boldsymbol{\eta}_s)_{t\in [0,T]}$ is well defined.

Assume for a moment that the Theorem is true when $\big(b^1=\sigma^1=0\big)$ and when $\big( b^2=\sigma^2 =0\big) $\footnote{Above we abused notation for brevity, since each process takes values in spaces with different zero elements.}.
We define
\begin{align*}
    \widetilde{X}_t &:= X_t - \int_0^t(X',X'')_s\ud \boldsymbol{\eta}_s,\\
    \widetilde{Y}_t &:= Y_t - \int_0^t(Y',Y'')_s\ud \boldsymbol{\eta}_s.
\end{align*}
Note that $\widetilde{X}, \widetilde{Y}$ are semimartingales.
We have
\begin{align}
    X_t Y_t &= \Big(\widetilde{X}_t +  \int_0^t(X',X'')_s\ud \boldsymbol{\eta}_s\Big) \Big(\widetilde{Y}_t +  \int_0^t(Y',Y'')_s\ud \boldsymbol{\eta}_s\Big) \nonumber\\
    &= \widetilde{X}_t \widetilde{Y}_t + \widetilde{X}_t \int_0^t(Y',Y'')_s\ud \boldsymbol{\eta}_s + \int_0^t(X',X'')_s\ud \boldsymbol{\eta}_s Y_t. \label{eq 025}
\end{align}
From the standard integration by parts formula for semimartingales we have
\begin{align}
    \widetilde{X}_t \widetilde{Y}_t = X_0Y_0 + \int_0^t \Big(\widetilde{X}_sb^2_s+b^1_s \widetilde{Y}_s+ \sum_{i=1}^q (\sigma^1_i)_s(\sigma^2_i)_s\Big) \ud s + \int_0^t \Big(\sigma^1_s \widetilde{Y}_s + \widetilde{X}_s \sigma^2_s\Big) \ud W_s. \label{eq 026}
\end{align}
Furthermore, from our assumption we have
\begin{align}
    \widetilde{X}_t \int_0^t(Y',Y'')_s\ud \boldsymbol{\eta}_s &= \int_0^t b^1_s  \int_0^s(Y',Y'')_r \ud \boldsymbol{\eta}_r  \ud s + \int_0^t  \sigma^1_s  \int_0^s(Y',Y'')_r \ud \boldsymbol{\eta}_r  \ud W_s\nonumber\\
    &\qquad+ \int_0^t (\widetilde{X}Y',\widetilde{X}Y'')_s\ud \boldsymbol{\eta}_s,\label{eq 027}\\
    \int_0^t(X',X'')_s\ud \boldsymbol{\eta}_s Y_t &= \int_0^t  \int_0^s(X',X'')_r \ud \boldsymbol{\eta}_r  b^2_s \ud s + \int_0^t \int_0^s(X',X'')_r \ud \boldsymbol{\eta}_r  \sigma^2_s \ud W_s \nonumber\\
    &\qquad+ \int_0^t ({Z}' -\widetilde{X}Y',{Z}'' -\widetilde{X}Y'')_s \ud \boldsymbol{\eta}_s. \label{eq 028}
\end{align}
Combining \eqref{eq 025}, \eqref{eq 026}, \eqref{eq 027} and \eqref{eq 028}, under our assumption, the general case is proved.

Now, all left to prove are two special cases  $ (b^1=\sigma^1=0 )$ and $(b^2=\sigma^2 =0)$. 
There is, however, a minor difference in the estimates required for the latter case.
In the proof below, which addresses the first case, we use \cite[Lemma 2.12]{FHL21} to bound $\big\|\sup_{u\in[s,t]}|\delta X_{s,u}|\big\|_{k_1}$, where in the second case we have to use \cite[ineq. (A.5), p. 56]{allan2026rough}\footnote{Along with a trivial extension of $\sup$ argument.} to bound $\big\|\sup_{u\in[s,t]}|\delta Y_{s,u}|\big\|_{k_2}$ appropriately. Actually, from \cite[Exercise 1.9.]{FV10} in conjunction with \cite[ineq. (A.5), p. 56]{allan2026rough}, \cite[Theorem 3.4]{allan2026rough} and the line with the definitions of (using their notation) $\hat{w}_{1,1}, \hat{w}_{2,1}, \hat{w}_{1,2}, \hat{w}_{2,2}$ in the proof of \cite[Lemma 4.6]{allan2026rough}, because 
\begin{align*}
    \Big( \hat{w}_{1,1}^{\frac{q}{p}} \hat{w}_{2,1}^{\frac{q}{p}} + \hat{w}_{1,2}^{\frac{q}{p}} \hat{w}_{2,2}^{\frac{2q}{p}} \Big)^{\frac{1}{q}} \leq (\hat{w}_{1,1}^{\frac{1}{2}} \hat{w}_{2,1}^{\frac{1}{2}})^{\frac{2}{p}} + \hat{w}_{1,2}^{\frac{1}{p}} \hat{w}_{2,2}^{\frac{2}{p}} 
\end{align*}
with $p < 3$, we can deduce that there exist continuous controls $$\tilde{w}_1 := w_1 + \hat{w}_{1,2}^{\frac{1}{p}} \hat{w}_{2,2}^{\frac{2}{p}},\hspace{.1cm} w_2, \hspace{.1cm}\tilde{w}_3 := \hat{w}_{1,1} \hat{w}_{2,1}$$
such that
\begin{equation*}
    \big\|\sup_{u\in[s,t]}|\delta Y_{s,u}|\big\|_{k_2} \lesssim \tilde{w}_1(s,t) + w_2(s,t)^{\frac{1}{2}} + \tilde{w}_3(s,t)^{\frac{2}{p}}.
\end{equation*} 
Finally, below we use as main tool \cite[Theorem 3.2]{allan2026rough} which is stated for products of pairs of controls. For the second case we need to extend this for products of arbitrary (finite) number of controls with the appropriate exponents. This can be immediately done by using again \cite[Exercise 1.9.]{FV10}, for example, let controls $\{w_i\}_{i\in \{1,\dots,N\}}$ and $0\leq a_1\leq \dots\leq a_N$ such that $\sum_{i=1}^N a_i >0$. Then, we have
\begin{align*}
    \prod_{i=1}^N w_i^{a_i} = \bigg(\prod_{i=2}^{N} w_i^{\frac{a_i}{\sum_{j=2}^{N} a_j}}\bigg)^{\sum_{j=2}^{N} a_j} w_1^{a_1},
\end{align*}
where inductively one can easily prove that $\big(\prod_{i=2}^{N} w_i^{\frac{a_i}{\sum_{j=2}^{N} a_j}}\big)^{\sum_{j=2}^{N} a_j}$ is a control.

So, below we give the proof of case $\big(b^1=\sigma^1=0\big)$, where we only need (i) and (ii) of Assumptions \ref{assu:rpf}\footnote{(i) and (iii) of Assumptions \ref{assu:rpf} are needed for the case $\big(b^2=\sigma^2=0\big)$.}. To ease the notation, we shall denote $b^2, \sigma^2$ by $b, \sigma$.
Hence, for every $t\in [0,T]$, we define
\begin{equation*}
    \mathcal{I}_t:= X_0Y_0 + \int_0^t X_s b_s \ud s + \int_0^t X_s \sigma_s \ud W_s + \int_0^t ({Z}',{Z}'')_s \ud \boldsymbol{\eta}_s.
\end{equation*}
Next, for $s,t\in [0,T]$ with $s\leq t$ we have
\begin{align*}
    \delta (XY)_{s,t}:= X_t Y_t - X_s Y_s = X_s \delta Y_{s,t} + \delta X_{s,t} Y_s + \delta X_{s,t} \delta Y_{s,t}.
\end{align*}
By definition $(\delta(XY)_{s,t})_{(s,t)\in\Delta([0,T])}$ is a two index increment process of a continuous path, hence for every $0\leq s \leq u \leq t \leq T$ we have
\begin{align*}
    \delta(XY)_{s,u,t} = 0.
\end{align*}
By the uniqueness part of \cite[Theorem 3.2]{allan2026rough} applied to $(\delta(XY)_{s,t})_{(s,t)\in\Delta([0,T])}$, it remains to show necessary estimates of $\delta \mathcal{I}- \delta (XY)$ concluded from the stochastic sewing lemma:
an unconditional $\mathbb{L}^{\frac{k_1k_2}{k_1+k_2}}$-estimate with total control exponent strictly
greater than $\frac12$, and a conditional $\mathbb{L}^{\frac{k_1k_2}{k_1+k_2}}$-estimate with total control exponent strictly greater than $1$.

So, by using the definitions of $\mathcal{I}, \delta(XY)$, expanding $Z', Z''$ and adding-subtracting the relevant terms we get
\begin{align}
    &\delta \mathcal{I}_{s,t} - \delta (XY)_{s,t} \nonumber\\
    &= \int_s^t \delta X_{s,u} b_u \ud u + \int_s^t \delta X_{s,u} \sigma_u \ud W_u \label{eq 46}\\
    & \qquad+ \int_s^t (Z',Z'')_u \ud \boldsymbol{\eta}_u - Z'_s \delta \eta_{s,t} - Z''_s \eta^{(2)}_{s,t} \label{eq 47}\\
    &\qquad- X_s\Big( \int_s^t(Y',Y'')_u \ud\boldsymbol{\eta}_u - Y'_s \delta\eta_{s,t} - Y''_s \eta^{(2)}_{s,t}\Big) \label{eq 48}\\
    &\qquad -\Big( \int_s^t(X',X'')_u \ud\boldsymbol{\eta}_u - X'_s \delta\eta_{s,t} - X''_s \eta^{(2)}_{s,t}\Big) Y_s \label{eq 49}\\
    &\qquad - \Big(\delta X_{s,t} \delta Y_{s,t} - 2X'_s Y'_s \text{Symm}\big(\eta^{(2)}_{s,t}\big) \Big). \label{eq 50}
\end{align}

For \eqref{eq 46} we have by H\"older's inequality and BDG inequalities
\begin{align*}
     &\Big\|\int_s^t \delta X_{s,u} b_u \ud u\Big\|_{\frac{k_1k_2}{k_1+k_2}} \leq \big\|\sup_{u\in[s,t]}|\delta X_{s,u}|\big\|_{k_1} \Big\|\int_s^t | b_u| \ud u \Big\|_{k_2} \leq \big\|\sup_{u\in[s,t]}|\delta X_{s,u}|\big\|_{k_1} \Big\|\int_s^t |b_u|^2 \ud u \Big\|_{\frac{k_2}{2}}^{\frac{1}{2}} |t-s|^{\frac{1}{2}},\\
      &\Big\|\int_s^t \delta X_{s,u} \sigma_u \ud W_u\Big\|_{\frac{k_1k_2}{k_1+k_2}} \lesssim \big\|\sup_{u\in[s,t]}|\delta X_{s,u}|\big\|_{k_1} \Big\|\int_s^t | \sigma_u|^2 \ud u \Big\|_{\frac{k_2}{2}}^{\frac{1}{2}},
\end{align*}
where from \cite[Lemma 2.12]{FHL21}\footnote{Note that $(X_t)_{t\in [0,T]}, (\eta_t)_{t\in [0,T]}$ are continuous processes.} and the proof of \cite[Theorem 3.4]{FHL21}\footnote{Last line of p.20, definition of $\epsilon_3$.}
\begin{align*}
   \big\|\sup_{u\in[s,t]}|\delta X_{s,u}|\big\|_{k_1} &\lesssim |t-s|^{\frac{1}{k_1} + \epsilon} + \big\|\sup_{u\in[s,t]}|X'_s\delta\eta_{s,u} + X''_s \eta^{(2)}_{s,u}|\big\|_{k_1} \\
   &\lesssim  |t-s|^{\frac{1}{k_1} + \epsilon} + \|X'_s \|_{k_1} |t-s|^{\alpha} + \|X''_s\|_{k_1} |t-s|^{2\alpha}\\
   &\lesssim \big(1+\sup_{u\in[0,T]}\|X'_u \|_{k_1} + \sup_{u\in[0,T]}\|X''_u \|_{k_1} \big) \big(|t-s|^{(\alpha + \beta)\wedge \frac{1}{2}} + |t-s|^{\alpha} +  |t-s|^{2\alpha}\big).
\end{align*}
Furthermore, we obviously have 
\begin{align*}
    \Big\|\mathbb{E}_s\int_s^t \delta X_{s,u} b_u \ud u\Big\|_{\frac{k_1k_2}{k_1+k_2}} \leq   \Big\|\int_s^t \delta X_{s,u} b_u \ud u\Big\|_{\frac{k_1k_2}{k_1+k_2}}\hspace{.2cm}\text{and}\hspace{.2cm}  \Big\|\mathbb{E}_s\int_s^t \delta X_{s,u} \sigma_u \ud W_u\Big\|_\frac{k_1k_2}{k_1+k_2} = 0.
\end{align*}
Finally, note that $w(s,t):=\Big\|\int_s^t | b_u|^2 \ud u \Big\|_{\frac{k_2}{2}}^{\frac{k_2}{2}}$ and $\bar{w}(s,t):=\Big\|\int_s^t | \sigma_u|^2 \ud u \Big\|_{\frac{k_2}{2}}^{\frac{k_2}{2}}$ are continuous controls, and by (ii) of Assumptions \ref{assu:rpf} we have
\begin{equation*}
    \alpha +\frac{1}{k_2} > \frac{1}{2}.
\end{equation*}

Now, for \eqref{eq 47}, \eqref{eq 48} and \eqref{eq 49}, we can again apply H\"older's inequality along with the fact that we are dealing with terms which are rough stochastic integrals minus their local expansions, to get from \cite[Lemma 4.3]{allan2026rough}
\begin{align*}
    \|\text{\eqref{eq 47}}\|_{\frac{k_1k_2}{k_1+k_2}} &\lesssim \sum_{j=1}^{M_{1}} w_{1,1,j}(s,t)^{\beta_{1,1,j}} w_{1,2,j}(s,t)^{\beta_{1,2,j}},\\ 
    \|\mathbb{E}_s\text{\eqref{eq 47}}\|_{\frac{k_1k_2}{k_1+k_2}} &\lesssim  \sum_{j=1}^{M_{2}} w_{1,3,j}(s,t)^{\alpha_{1,3,j}} w_{1,4,j}(s,t)^{\alpha_{1,4,j}},\\
    \|\text{\eqref{eq 48}}\|_{\frac{k_1k_2}{k_1+k_2}} &\lesssim \sum_{j=1}^{M_{3}} w_{2,1,j}(s,t)^{\beta_{2,1,j}} w_{2,2,j}(s,t)^{\beta_{2,2,j}},\\
    \|\mathbb{E}_s\text{\eqref{eq 48}}\|_{\frac{k_1k_2}{k_1+k_2}} &\lesssim \sum_{j=1}^{M_{4}} w_{2,3,j}(s,t)^{\alpha_{2,3,j}} w_{2,4,j}(s,t)^{\alpha_{2,4,j}},\\
    \|\text{\eqref{eq 49}}\|_{\frac{k_1k_2}{k_1+k_2}} \lesssim |t-s|^{\alpha + \beta'}\hspace{.2cm}&\text{and}\hspace{.2cm} \|\mathbb{E}_s\text{\eqref{eq 49}}\|_{\frac{k_1k_2}{k_1+k_2}} \lesssim |t-s|^{\alpha + \beta'+\beta''},
\end{align*}
where $\{w_{1,1,j},w_{1,2,j}\}_{j\in M_1}, \{w_{1,3,j},w_{1,4,j}\}_{j\in M_2}, \{w_{2,1,j},w_{2,2,j}\}_{j\in M_3}, \{w_{2,3,j},w_{2,4,j}\}_{j\in M_4}$ are continuous controls, with suitable $M_1, M_2, M_3, M_4$, and $\beta_{1,1,j}, \beta_{1,2,j}, \beta_{2,1,j}, \beta_{2,2,j}, \alpha_{1,3,j}, \alpha_{1,4,j}, \alpha_{2,3,j}, \alpha_{2,4,j}$ strictly positive real numbers such that $\beta_{1,1,j} + \beta_{1,2,j} > \frac{1}{2} ,  \beta_{2,1,j} + \beta_{2,2,j} > \frac{1}{2}$ and $\alpha_{1,3,j}+ \alpha_{1,4,j} > 1, \alpha_{2,3,j} + \alpha_{2,4,j} > 1$, for every $j\in M_1, M_2, M_3, M_4$ respectively.

To finish the proof, we expand \eqref{eq 50} a little further using that $\boldsymbol{\eta}$ is a geometric rough path and we have
\begin{align*}
    &\delta X_{s,t} \delta Y_{s,t} - 2X'_s Y'_s \text{Symm}\big(\eta^{(2)}_{s,t}\big)\\ &= \delta X_{s,t} \delta Y_{s,t} - (X'_s \delta\eta_{s,t}) (Y'_s \delta\eta_{s,t})\\
    &= (X'_s \delta\eta_{s,t}) \delta Y_{s,t} + \Big( X''_s \eta^{(2)}_{s,t} + \bar{R}^X_{s,t}\Big) \delta Y_{s,t} - (X'_s \delta\eta_{s,t}) (Y'_s \delta\eta_{s,t})\\
    &= (X'_s \delta\eta_{s,t}) \Big( Y''_s \eta^{(2)}_{s,t} + \bar{R}^Y_{s,t}\Big) + \Big( X''_s \eta^{(2)}_{s,t} + \bar{R}^X_{s,t}\Big) \delta Y_{s,t}\\
    &\qquad +  \int_s^t (X'_s \delta\eta_{s,t}) b_u \ud u +  \int_s^t  (X'_s \delta\eta_{s,t}) \sigma_u \ud W_u\\
    &= (X'_s \delta\eta_{s,t} + X''_s \eta^{(2)}_{s,t}) \Big( Y''_s \eta^{(2)}_{s,t} + \bar{R}^Y_{s,t}\Big) +  \big(X''_s \eta^{(2)}_{s,t}\big) (Y'_s \delta\eta_{s,t}) + 
    \bar{R}^X_{s,t} \bar{R}^Y_{s,t}\\
    &\qquad + \bar{R}^{X}_{s,t}
  \bigl(Y'_s\,\delta\eta_{s,t}\bigr)
+ \bar{R}^{X}_{s,t}
  \bigl(Y''_s\,\eta^{(2)}_{s,t}\bigr) +  \int_s^t \bar{R}^X_{s,t} b_u \ud u +  \bar{R}^X_{s,t} \int_s^t \sigma_u \ud W_u\\
    &\qquad +  \int_s^t (X'_s \delta\eta_{s,t}) b_u \ud u +  \int_s^t  (X'_s \delta\eta_{s,t}) \sigma_u \ud W_u\\
    &\qquad +  \int_s^t \big(X''_s \eta^{(2)}_{s,t}\big) b_u \ud u +  \int_s^t  \big(X''_s \eta^{(2)}_{s,t} \big) \sigma_u \ud W_u,
\end{align*}
where we used the shorthand notation
\begin{align*}
    \bar{R}^X_{s,t} &:=  \int_s^t(X',X'')_u \ud\boldsymbol{\eta}_u - X'_s \delta\eta_{s,t} - X''_s \eta^{(2)}_{s,t}, \\
    \bar{R}^Y_{s,t} &:=  \int_s^t(Y',Y'')_u \ud\boldsymbol{\eta}_u - Y'_s \delta\eta_{s,t} - Y''_s \eta^{(2)}_{s,t}.
\end{align*}
By H\"older's and BDG inequalities we have similar to before that
\begin{align*}
    &\Big\|\int_s^t (X'_s \delta\eta_{s,t})b_u \ud u \Big\|_{\frac{k_1k_2}{k_1+k_2}} \lesssim |t-s|^{\alpha}  \Big\|\int_s^t | b_u|^2 \ud u \Big\|_{\frac{k_2}{2}}^{\frac{1}{2}} |t-s|^{\frac{1}{2}} , \\
     &\Big\|\int_s^t (X'_s \delta\eta_{s,t})\sigma_u \ud W_u \Big\|_{\frac{k_1k_2}{k_1+k_2}} \lesssim |t-s|^{\alpha}  \Big\|\int_s^t | \sigma_u|^2 \ud u \Big\|_{\frac{k_2}{2}}^{\frac{1}{2}},\\
     &\Big\|\int_s^t (X''_s \eta^{(2)}_{s,t})b_u \ud u \Big\|_{\frac{k_1k_2}{k_1+k_2}} \lesssim |t-s|^{2\alpha}  \Big\|\int_s^t | b_u|^2 \ud u \Big\|_{\frac{k_2}{2}}^{\frac{1}{2}} |t-s|^{\frac{1}{2}} , \\
     &\Big\|\int_s^t (X''_s \eta^{(2)}_{s,t}) \sigma_u \ud W_u \Big\|_{\frac{k_1k_2}{k_1+k_2}} \lesssim |t-s|^{2\alpha}  \Big\|\int_s^t | \sigma_u|^2 \ud u \Big\|_{\frac{k_2}{2}}^{\frac{1}{2}},
\end{align*}
together with
\begin{align*}
    &\Big\|\mathbb{E}_s\int_s^t  (X'_s \delta\eta_{s,t})b_u \ud u\Big\|_{\frac{k_1k_2}{k_1+k_2}} \leq   \Big\|\int_s^t(X'_s \delta\eta_{s,t})b_u \ud u\Big\|_{\frac{k_1k_2}{k_1+k_2}}\hspace{.2cm}\text{and}\hspace{.2cm}  \Big\|\mathbb{E}_s\int_s^t(X'_s \delta\eta_{s,t}) \sigma_u \ud W_u\Big\|_{\frac{k_1k_2}{k_1+k_2}} = 0,\\
    &\Big\|\mathbb{E}_s\int_s^t (X''_s \eta^{(2)}_{s,t})b_u \ud u\Big\|_{\frac{k_1k_2}{k_1+k_2}} \leq   \Big\|\int_s^t (X''_s \eta^{(2)}_{s,t}) b_u \ud u\Big\|_{\frac{k_1k_2}{k_1+k_2}}\hspace{.2cm}\text{and}\hspace{.2cm}  \Big\|\mathbb{E}_s\int_s^t (X''_s \eta^{(2)}_{s,t}) \sigma_u \ud W_u\Big\|_{\frac{k_1k_2}{k_1+k_2}} = 0.
\end{align*}
Next, we have
\begin{align*}
    &\big\|\mathbb{E}_s (X'_s \delta\eta_{s,t}) ( Y''_s \eta^{(2)}_{s,t})\big\|_{\frac{k_1k_2}{k_1+k_2}} \leq  \big\|(X'_s \delta\eta_{s,t}) ( Y''_s \eta^{(2)}_{s,t})\big\|_{\frac{k_1k_2}{k_1+k_2}} \lesssim \sup_{u\in [0,T]} \big\{\big\|X'_u\big\|_{k_1} \big\|Y''_u\big\|_{k_2}\big\} |t-s|^{3\alpha},\\
    &\big\|\mathbb{E}_s (X''_s \eta^{(2)}_{s,t}) ( Y''_s \eta^{(2)}_{s,t})\big\|_{\frac{k_1k_2}{k_1+k_2}} \leq  \big\|(X''_s \eta^{(2)}_{s,t}) ( Y''_s \eta^{(2)}_{s,t})\big\|_{\frac{k_1k_2}{k_1+k_2}} \lesssim \sup_{u\in [0,T]} \big\{\big\|X''_u\big\|_{k_1} \big\|Y''_u\big\|_{k_2}\big\} |t-s|^{4\alpha},\\
    &\big\|\mathbb{E}_s (X''_s \eta^{(2)}_{s,t}) ( Y'_s \delta\eta_{s,t})\big\|_\frac{k_1k_2}{k_1+k_2} \leq  \big\|(X''_s \eta^{(2)}_{s,t}) ( Y'_s \delta\eta_{s,t})\big\|_\frac{k_1k_2}{k_1+k_2} \lesssim \sup_{u\in [0,T]} \big\{\big\|X''_u\big\|_{k_1} \big\|Y'_u\big\|_{k_2}\big\} |t-s|^{3\alpha},\\
  &  \big\| \E_s (\bar{R}^{X}_{s,t})
   (Y'_s \delta \eta_{s,t} + Y''_s\,\eta^{(2)}_{s,t})  \big\|_\frac{k_1k_2}{k_1+k_2} \le \big\| ( \bar{R}^{X}_{s,t})
   (Y'_s \delta \eta_{s,t} +Y''_s\,\eta^{(2)}_{s,t} )\big\|_\frac{k_1k_2}{k_1+k_2} \\
   &\ \ \ \  \lesssim \big\|\bar{R}^X \big\|_{\alpha+\beta',k_1} \sup_{u\in [0,T]} \big\{ \big\|Y'_u\big\|_{k_2} \vee \big\|Y''_u\big\|_{k_2}\big\} |t-s|^{2 \alpha+ \beta'}.                
\end{align*}
Additionally,
\begin{align*}
     &\Big\|\int_s^t \bar{R}^X_{s,t} b_u \ud u\Big\|_\frac{k_1k_2}{k_1+k_2} \leq   \big\|\bar{R}^X_{s,t}\big\|_{k_1} \Big\|\int_s^t | b_u| \ud u \Big\|_{k_2} \leq \big\| \bar{R}^X_{s,t}\big\|_{k_1} \Big\|\int_s^t | b_u|^2 \ud u \Big\|_{\frac{k_2}{2}}^{\frac{1}{2}} |t-s|^{\frac{1}{2}}\\
     &\qquad \lesssim |t-s|^{\alpha + \beta' + \frac{1}{2}} \Big\|\int_s^t | b_u|^2 \ud u \Big\|_{\frac{k_2}{2}}^{\frac{1}{2}},\\
      &\Big\|\bar{R}^X_{s,t} \int_s^t \sigma_u \ud W_u\Big\|_{\frac{k_1k_2}{k_1+k_2}} \leq \big\|\bar{R}^X_{s,t}\big\|_{k_1} \Big\|\int_s^t \sigma_u \ud W_u\Big\|_{k_2} \lesssim \big\|\bar{R}^X_{s,t}\big\|_{k_1} \Big\|\int_s^t |\sigma_u|^2 \ud u \Big\|_{\frac{k_2}{2}}^{\frac{1}{2}}\\
      &\lesssim |t-s|^{\alpha + \beta'} \Big\|\int_s^t |\sigma_u|^2 \ud u \Big\|_{\frac{k_2}{2}}^{\frac{1}{2}},
\end{align*}
where in the last line above we use (ii) of Assumptions \ref{assu:rpf} to conclude. Then, we also get
\begin{align*}
    \Big\|\mathbb{E}_s\int_s^t  \bar{R}^X_{s,t} b_u \ud u\Big\|_{\frac{k_1k_2}{k_1+k_2}} &\leq   \Big\|\int_s^t \bar{R}^X_{s,t} b_u \ud u\Big\|_{\frac{k_1k_2}{k_1+k_2}},\\   \Big\|\mathbb{E}_s\Big[\bar{R}^X_{s,t}\int_s^t \sigma_u \ud W_u\Big]\Big\|_{\frac{k_1k_2}{k_1+k_2}} &\leq   \Big\|\bar{R}^X_{s,t}\int_s^t  \sigma_u \ud W_u\Big\|_{\frac{k_1k_2}{k_1+k_2}}.
\end{align*}
Lastly, we have by \cite[Lemma 4.3]{allan2026rough} that
\begin{align*}
    \big\|\mathbb{E}_s  \bar{R}^X_{s,t} \bar{R}^Y_{s,t} \big\|_{\frac{k_1k_2}{k_1+k_2}} \leq \big\| \bar{R}^X_{s,t} \bar{R}^Y_{s,t} \big\|_{\frac{k_1k_2}{k_1+k_2}} \leq  \big\| \bar{R}^X_{s,t}\big\|_{k_1} \big\| \bar{R}^Y_{s,t} \big\|_{k_2}\lesssim |t-s|^{\alpha + \beta'}  \|\boldsymbol{\eta\|}_{\frac{1}{\alpha}\text{-var},[s,t]},
\end{align*}
where $\|\boldsymbol{\eta\|}_{\frac{1}{\alpha}\text{-var},[s,t]} := \|\eta\|_{\frac{1}{\alpha}\text{-var},[s,t]} + \|\eta^{(2)}\|_{\frac{1}{2\alpha}\text{-var},[s,t]}$. Note that $\tilde{w}_1(s,t):=\|\eta\|_{\frac{1}{\alpha}\text{-var},[s,t]}^{\frac{1}{\alpha}}$ and $\tilde{w}_2(s,t):= \|\eta^{(2)}\|_{\frac{1}{2\alpha}\text{-var},[s,t]}^{\frac{1}{2\alpha}}$ are continuous controls, which completes the proof.
\end{proof}

\begin{rmk}\label{rmk 3.21}
Inspection of the proof immediately yields the following observations:
\begin{itemize}
    \item[(i)] By using the natural homeomorphisms between finite-dimensional Banach spaces, the above rough product formula immediately extends to the setting in which $X$ takes values in $L(V,Q)$ and $Y$ in $V$, for arbitrary finite-dimensional Banach spaces $V$ and $Q$. Moreover, one readily sees that the theorem also extends to the case where $X$ is controlled in a $p$-variation rough seminorm and $Y$ in a Hölder rough seminorm.

   \item[(ii)] In the case that 
   $$ \sup_{u\in[0,T]}\|b^1_u\|_{k_1} + \sup_{u\in[0,T]}\|\sigma^1_u\|_{k_1} < \infty,$$ 
   because by Minkowski's integral inequality we have
   \begin{align*}
        \Big\|\int_s^t |b^1_u|^2 \ud u \Big\|_{\frac{k_1}{2}}^{\frac{1}{2}} +  \Big\|\int_s^t |\sigma^1_u|^2 \ud u \Big\|_{\frac{k_1}{2}}^{\frac{1}{2}} &\leq \bigg(\int_s^t \|b^1_u\|_{k_1}^2\ud u\bigg)^{\frac{1}{2}} + \bigg(\int_s^t \|\sigma^1_u\|_{k_1}^2 \ud u\bigg)^{\frac{1}{2}}\\
        &\leq (\sup_{u\in[0,T]}\|b^1_u\|_{k_1}^2 + \sup_{u\in[0,T]}\|\sigma^1_u\|_{k_1}^2) |t-s|^{\frac{1}{2}},
   \end{align*}
   assumption (iii) can be dropped from Assumptions \ref{assu:rpf}. 
   
   For the same reason, in the case that
   $$ \sup_{u\in[0,T]}\|b^2_u\|_{k_2} + \sup_{u\in[0,T]}\|\sigma^2_u\|_{k_2} < \infty,$$ 
   assumption (ii) can be dropped from Assumptions \ref{assu:rpf}.

   \item[(iii)] In the case that $ (Y',Y'')\in \mathbf{D}^{\beta',\beta''}_{\eta}\mathbb{L}^{k_2}\big([0,T];\mathbb{R}^{d\times \ell}\big)$, then $(Z',Z'') \in \mathbf{D}^{\beta',\beta''}_{\eta}\mathbb{L}^{\frac{k_1k_2}{k_1+k_2}}\big([0,T];\mathbb{R}^{e\times \ell}\big)$ and assumptions (iii) of Assumptions \ref{assu:rpf} becomes
   $$
\alpha+\beta'+ \beta'' \wedge \frac{1}{k_1}>1.
   $$ 
\end{itemize}
\end{rmk}

\section{The rough stochastic maximum principles}\label{sec:max_principles}
This section is the heart of this work, its goal is to derive necessary and sufficient conditions for the control problem \eqref{eq:rough_stochastic_dynamics}--\eqref{eq: cost_function_intro}. Without loss of generality, to ease the notation, in what follows we are going to assume that $t=0$ and instead of $\mathcal{J}^{\boldsymbol{\eta}}(u;0,x)$ we are going to use the simpler notation $\mathcal{J}(u)$. In addition, we are going to allow the initial condition of the dynamics to be an appropriate random variable $\xi$.

\subsection{Variation via a linear perturbation and the rough SMP}\label{subsec: 4.1}
\subsubsection{General setting}
Let $\boldsymbol{\eta} \in \mathscr{C}^{0,\alpha}_g([0,T];\mathbb{R}^\ell), \ell\in\mathbb{N}$, with 
$\frac{1}{3}<  \beta'\leq \beta\leq\alpha < \frac{1}{2}$.
In this subsection the following assumptions hold true:
\begin{enumerate}[label=(A\arabic*), ref=(A\arabic*)]
\item\label{A1} The functions $b: \Omega \times [0,T]\times \mathbb{R}^d\times\mathbb{R}^\kappa\rightarrow \mathbb{R}^d,\sigma: \Omega\times [0,T]\times \mathbb{R}^d\times\mathbb{R}^\kappa\rightarrow \mathbb{R}^{d\times w}$ are $\text{Prog}(\Omega\times[0,T]) \otimes \mathcal{B}(\mathbb{R}^d)\otimes \mathcal{B}(\mathbb{R}^\kappa)/ \mathcal{B}(\mathbb{R}^{d}) \big(\text{resp.}\hspace{.1cm} \mathcal{B}(\mathbb{R}^{d\times w})\big)$-measurable, and differentiable in $(x,u)\in \mathbb{R}^d\times \mathbb{R}^\kappa$. Also, for every $\phi \in \{b,\sigma,\partial_x b, \partial_x\sigma, \partial_u b, \partial_u\sigma\}$ we have  
$$
\|\sup_{t\in[0,T]} \sup_{x\in\mathbb{R}^d} \sup_{u\in\mathbb{R}^\kappa} |\phi(t,x,u)|\|_\infty <\infty.
$$
 Furthermore, let $\theta\in(0,1)$, for every $\phi\in\{\partial_x b,\partial_x \sigma, \partial_u b, \partial_u\sigma\}$ we have
 \begin{equation*}
\|\sup_{t\in[0,T]}[\phi(t,\cdot,\cdot)]_{\theta}\|_{\infty} < \infty.\footnotemark
\end{equation*}
\footnotetext{Note that from the separability properties of the involved spaces 
$[\phi(t,\cdot,\cdot)]_{\theta}$
are progressively measurable, and by the projection Theorem 
$\sup_{t\in[0,T]}[\phi(t,\cdot,\cdot)]_{\theta}$ are $\mathcal{F}$-measurable.}

\item\label{A2} Let $\theta \in (0,1)$ and $k \in [2,\infty)$. The functions $f:\Omega \times [0,T]\times \mathbb{R}^d\rightarrow \mathbb{R}^{d\times \ell}, f': \Omega \times [0,T]\times\mathbb{R}^d \rightarrow L(\mathbb{R}^\ell,\mathbb{R}^{d\times \ell})$ form a pair $(f,f')\in \mathbf{D}^{\alpha,\alpha}_{\eta}\mathbb{L}^{k,\infty}\mathcal{C}_b^{3 + \theta}$. 
Also, $(\partial_x f,\partial_x f')\in \mathbf{D}^{\alpha, \beta'}_{\eta}\mathbb{L}^{k,\infty}\mathcal{C}_b^{2 + \theta}$ and
$$
\|\sup_{t\in[0,T]}|f(t,\cdot)|_{3+\theta}\|_\infty + \|\sup_{t\in[0,T]}|f'(t,\cdot)|_{2+\theta}\|_\infty <\infty.
$$
Furthermore, $\partial_{xt}^{2}f :=\partial_t(\partial_xf), \partial_{xxt}^{3}f := \partial_t(\partial^2_{xx}f), \partial_{xt}^{2}f' := \partial_t(\partial_xf')$ exist, are  
continuous in $t\in [0,T]$, and for every $\phi\in \{\partial_{xt}^{2}f, \partial_{xxt}^{3}f, \partial_{xt}^{2}f'\}$ we have
\begin{equation*}
\|\sup_{t\in[0,T]}|\phi(t,\cdot)|_{\theta}\|_{\infty} < \infty.
\end{equation*}

\item\label{A3} Let $m\in[2,\infty)$. The function $h : \Omega \times [0,T] \times \mathbb{R}^d \times \mathbb{R}^\kappa \rightarrow \mathbb{R}$ is
 $\text{Prog}(\Omega\times[0,T]) \otimes \mathcal{B}(\mathbb{R}^d)\otimes \mathcal{B}(\mathbb{R}^\kappa)/ \mathcal{B}(\mathbb{R})$-measurable, differentiable in $(x,u) \in \mathbb{R}^d \times \mathbb{R}^\kappa$ and $\partial_{x}h,\partial_uh$ are Lipschitz in $(x,u) \in \mathbb{R}^d \times \mathbb{R}^\kappa$ uniformly with respect to $(\omega,t)\in \Omega\times[0,T]$, furthermore, 
$h(\cdot,\cdot,0,0), \partial_xh(\cdot,\cdot,0,0), \partial_uh(\cdot,\cdot,0,0)\in \mathbb{H}^m([0,T];\mathbb{R})$. 
Finally, the function $g:\Omega\times \mathbb{R}^d \rightarrow \mathbb{R}$ is $\mathcal{F}_T \otimes \mathcal{B}(\mathbb{R}^d)/\mathcal{B}(\mathbb{R})$-measurable, differentiable in $x\in \mathbb{R}^d$, and $\partial_x g$ is Lipschitz in $x \in \mathbb{R}^d$ uniformly with respect to $\omega\in \Omega$, furthermore $g(\cdot,0),\partial_xg(\cdot,0)\in \mathbb{L}^m(\mathcal{F}_T;\mathbb{R})$.
\end{enumerate}

The space of admissible controls for this subsection is the following
\begin{align*}
   \mathcal{A}:= \bigg\{ u:\Omega \times [0,T] \rightarrow \mathbb{R}^\kappa : (u_t)_{t\in [0,T]}\hspace{.1cm}\text{progressively measurable and}\hspace{.1cm} \sup_{t\in[0,T]}\|u_t\|_k < \infty\bigg\}. 
\end{align*}

Fix $u, v \in \mathcal{A} $ and for every $\epsilon \geq 0$ define $ {u}^\epsilon = (u^\epsilon_t)_{t \in [0, T]}$ by 
$$u^\epsilon_t := u_t + \epsilon v_t,$$
for all $t\in[0,T]$. To ease the notation, we define $X^\epsilon := X^{\boldsymbol{\eta},u^\epsilon}$, \emph{i.e.} the process which solves the equation \eqref{eq:rough_stochastic_dynamics} given $\xi\in\mathbb{L}^k(\mathcal{F}_0;\mathbb{R}^d)$ and the control $u^\epsilon$
\begin{equation*}
\begin{cases}
\ud X^{\boldsymbol{\eta},u^\epsilon}_s
=
b\bigl(s,X^{\boldsymbol{\eta},u^\epsilon}_s,u^\epsilon_s\bigr)\,\ud s
+
\sigma\bigl(s,X^{\boldsymbol{\eta},u^\epsilon}_s,u^\epsilon_s\bigr)\,\ud W_s
+
(f,f')\bigl(s,X^{\boldsymbol{\eta},u^\epsilon}_s\bigr)\,\ud\boldsymbol{\eta}_s,
\qquad s \in [0,T], \\
X^{\boldsymbol{\eta},u^\epsilon}_0
= \xi,
\end{cases}
\end{equation*}
and $X:= X^0:= X^{\boldsymbol{\eta},u^0}$. Our target is to minimize the cost function \eqref{eq: cost_function_intro}, and provide description of the optimal control for \eqref{eq:rough-vf}.

\begin{rmk}
We allow the elements of $\mathcal{A}$ to take values in the whole space $\mathbb{R}^\kappa$ purely for notational simplicity. The arguments below remain valid, with the appropriate modifications, if the values are instead restricted to, for example, a closed convex subset of $\mathbb{R}^\kappa$.
\end{rmk}

\subsubsection{Derivation of the Gateux derivative}
Heuristically, we assume $\boldsymbol{\eta}$ to be smooth, and expand to a first order the coefficients to get that
\begin{equation*}
    \begin{aligned}
        b\bigl(s,X^{\epsilon}_s,u^\epsilon_s\bigr) &\approx b\bigl(s,X_s,u_s\bigr) + \partial_x b(s,X_s,u_s\bigr) \big(X^{\epsilon}_s-X_s\big) + \partial_u b(s,X_s,u_s\bigr)(u^\epsilon_s-u_s ),\\
        \sigma\bigl(s,X^{\epsilon}_s,u^\epsilon_s\bigr) &\approx \sigma\bigl(s,X_s,u_s\bigr) + \partial_x \sigma(s,X_s,u_s\bigr) \big(X^{\epsilon}_s-X_s\big) + \partial_u \sigma(s,X_s,u_s\bigr)(u^\epsilon_s-u_s ),\\
        f\bigl(s,X^{\epsilon}_s\bigr) \dot{\eta}_s &\approx f\bigl(s,X_s\bigr) \dot{\eta}_s + \partial_x f(s,X_s\bigr) \big(X^{\epsilon}_s-X_s\big) \dot{\eta}_s.
    \end{aligned}
\end{equation*}
Hence for the difference we get that
\begin{align*}
         \frac{X^{\epsilon}_t - X_t}{\epsilon}\approx &  \int_0^t \Big(\partial_x b(s,X_s,u_s\bigr)  \frac{X^{\epsilon}_s - X_s}{\epsilon} + \partial_u b(s,X_s,u_s\bigr)v_s \Big)\ud s \\
         &+ \int_0^t \Big(\partial_x \sigma(s,X_s,u_s\bigr)  \frac{X^{\epsilon}_s - X_s}{\epsilon} + \partial_u \sigma(s,X_s,u_s\bigr)v_s \Big)\ud W_s\\
         &+\int_0^t \partial_x f(s,X_s\bigr)  \frac{X^{\epsilon}_s - X_s}{\epsilon} \dot{\eta}_s \ud s.
\end{align*}
Thus, in order to rigorously identify the limit $\lim_{\epsilon\rightarrow 0^+}  \frac{X^{\epsilon}_t - X_t}{\epsilon}$ we focus our attention to the first order \emph{variational equation}
\begin{equation}\label{eq: var_process}
\begin{cases}
 \ud V_s = \left( \partial_x b(s,X_s, u_s)V_s +  \partial_u b(s,X_s, u_s) v_s \right)\ud s + (\partial_x \sigma(s, X_s, u_s) V_s + \partial_u \sigma(s, X_s, u_s)v_s )\ud W_s \\\qquad\quad+ \big((\partial_x f(s,{X}_s), \partial_{xx}^2f(s,{X}_s) f(s,{X}_s) + \partial_x f'(s,{X}_s))V_s\big)\ud  \boldsymbol{\eta}_s,  \\ 
 V_0 = 0.    
\end{cases}
\end{equation}
The following Lemma is an extension of \cite[Lemma 3.5]{bugini2024malliavin}.

\begin{lemma}\label{lemma 4.1}
Let $k\in [2,\infty), u \in\mathcal{A}$ and $(b^i)_{i\in \{1,2\}}$, $(\sigma^i)_{i\in \{1,2\}}$ satisfying \ref{A1}, $(f,f'), (\partial_x f,\partial_x f')$ satisfying \ref{A2}. Furthermore, let $\big(X^i,f(\cdot,X^i)\big)\in\mathbf{D}_{\eta}^{\alpha,\beta} \mathbb{L}^{k,\infty}\big([0,T];\mathbb{R}^d\big)$, $i\in \{1,2\}$, be the unique solution of the following RSDE\footnote{Provided by \cite[Theorem 4.6]{FHL21}.} 
\begin{equation*}
    \ud X^i_t = b^i(t,X^i_t,u_t) \ud t + \sigma^i(t,X^i_t,u_t)\ud W_t + \big(f(t,X^i_t),\partial_xf(t,X^i_t)f(t,X^i_t)+ f'(t,X^i_t)\big) \ud \boldsymbol{\eta}_t.
\end{equation*}
Then, for every $\lambda\in[0,1]$ we define $\widetilde{X}^{\{1,2\},\lambda}_t:= X^1_t +\lambda(X^2_t-X^1_t) = (1-\lambda)X^1_t+\lambda X^2_t$, and we have that 
\begin{gather*}\Big(\partial_x f(\cdot,\widetilde{X}^{\{1,2\},\lambda}_\cdot),\partial^2_{xx}f(\cdot,\widetilde{X}^{\{1,2\},\lambda}_\cdot) \big(f(\cdot,{X}^{1}_\cdot) + \lambda \big[f(\cdot,{X}^2_\cdot) -f(\cdot,X^1_\cdot)\big]\big) + \partial_x f'(\cdot,\widetilde{X}^{\{1,2\},\lambda}_\cdot)\Big) \\\in
\mathbf{D}_{\eta}^{\beta,\beta'} \mathbb{L}^{k,\infty}_{\lambda-ex,lin}\big([0,T];L(\mathbb{R}^d,\mathbb{R}^{d\times \ell})\big).
\end{gather*}
\end{lemma}
\begin{proof}
We define $\Phi_t(\lambda):= \partial_xf(t,\widetilde{X}^{\{1,2\},\lambda}_t)$ and $\Phi'_t(\lambda):= \partial^2_{xx}f(t,\widetilde{X}^{\{1,2\},\lambda}_t) \big(f(t,{X}^{1}_t) + \lambda \big[f(t,{X}^2_t) -f(t,X^1_t)\big]\big)+\partial_x f'(t,\widetilde{X}^{\{1,2\},\lambda}_t)$. To prove that 
$$\sup_{\lambda\in[0,1]}\|\delta \Phi_\cdot(\lambda)\|_{\alpha;k,\infty} + \sup_{\lambda\in[0,1]}\| \delta \Phi_\cdot'(\lambda)\|_{\beta';k,\infty} < \infty$$
one can follow exactly the proof of \cite[Lemma 3.5]{bugini2024malliavin}.

For $\mathbb{E}_s[R_{s,t}^{\Phi_\cdot(\lambda)}]$ we initially calculate
\begin{align*}
    R_{s,t}^{\Phi_\cdot(\lambda)} &= \partial_xf(t,\widetilde{X}^{\{1,2\},\lambda}_t) - \partial_xf(s,\widetilde{X}^{\{1,2\},\lambda}_s) -\\
    &\qquad\Big[\partial^2_{xx}f(s,\widetilde{X}^{\{1,2\},\lambda}_s)\Big(f(s,{X}^{1}_s) + \lambda \big[f(s,{X}^2_s) -f(s,X^1_s)\big]\Big) + \partial_x f'(s,\widetilde{X}^{\{1,2\},\lambda}_s)\Big] \eta_{s,t}\\
    &= \partial_xf(s,\widetilde{X}^{\{1,2\},\lambda}_t) - \partial_xf(s,\widetilde{X}^{\{1,2\},\lambda}_s) - \partial^2_{xx}f(s,\widetilde{X}^{\{1,2\},\lambda}_s)\Big(f(s,{X}^{1}_s) + \lambda \big[f(s,{X}^2_s) -f(s,X^1_s)\big]\Big) \eta_{s,t}\\
    &\qquad+ \partial_xf(t,\widetilde{X}^{\{1,2\},\lambda}_s) - \partial_xf(s,\widetilde{X}^{\{1,2\},\lambda}_s) - \partial_x f'(s,\widetilde{X}^{\{1,2\},\lambda}_s) \eta_{s,t}\\
    &\qquad+ \Big[\partial_xf(t,\widetilde{X}^{\{1,2\},\lambda}_t)-\partial_xf(s,\widetilde{X}^{\{1,2\},\lambda}_t)\Big] - \Big[\partial_xf(t,\widetilde{X}^{\{1,2\},\lambda}_s) - \partial_xf(s,\widetilde{X}^{\{1,2\},\lambda}_s)\Big]\\
     &= \partial_xf(s,\widetilde{X}^{\{1,2\},\lambda}_t) - \partial_xf(s,\widetilde{X}^{\{1,2\},\lambda}_s) - \partial^2_{xx}f(s,\widetilde{X}^{\{1,2\},\lambda}_s)\delta \widetilde{X}^{\{1,2\},\lambda}_{s,t}\\
     &\qquad+ \partial^2_{xx}f(s,\widetilde{X}^{\{1,2\},\lambda}_s) \Big[\delta \widetilde{X}^{\{1,2\},\lambda}_{s,t} - \Big(f(s,{X}^{1}_s) + \lambda \big[f(s,{X}^2_s) -f(s,X^1_s)\big]\Big)\eta_{s,t}\Big]\\
     &\qquad+ \partial_xf(t,\widetilde{X}^{\{1,2\},\lambda}_s) - \partial_xf(s,\widetilde{X}^{\{1,2\},\lambda}_s) - \partial_x f'(s,\widetilde{X}^{\{1,2\},\lambda}_s) \eta_{s,t}\\
     &\qquad+ \delta (\partial_xf)_{s,t}(\widetilde{X}^{\{1,2\},\lambda}_t) - \delta (\partial_xf)_{s,t}(\widetilde{X}^{\{1,2\},\lambda}_s)\\
     &= \frac{1}{2}\int_0^1 (1-\vartheta) \partial^3_{xxx}f(s,\widetilde{X}^{\{1,2\},\lambda}_{s} + \vartheta \delta\widetilde{X}^{\{1,2\},\lambda}_{s,t}) \ud \vartheta \big(\delta \widetilde{X}^{\{1,2\},\lambda}_{s,t}\big)^{\otimes 2}\\
     &\qquad+ \partial^2_{xx}f(s,\widetilde{X}^{\{1,2\},\lambda}_s) \Big[\delta \widetilde{X}^{\{1,2\},\lambda}_{s,t} -\Big(f_s({X}^{1}_s) + \lambda \big[f(s,{X}^2_s) -f(s,X^1_s)\big]\Big) \eta_{s,t}\Big]\\
     &\qquad+ \partial_xf(t,\widetilde{X}^{\{1,2\},\lambda}_s) - \partial_xf(s,\widetilde{X}^{\{1,2\},\lambda}_s) - \partial_x f'(s,\widetilde{X}^{\{1,2\},\lambda}_s) \eta_{s,t}\\
     &\qquad+ \delta (\partial_xf)_{s,t}(\widetilde{X}^{\{1,2\},\lambda}_t) - \delta (\partial_xf)_{s,t}(\widetilde{X}^{\{1,2\},\lambda}_s).
\end{align*}
Now, for the first term we have
\begin{align*}
    &\Big|\mathbb{E}_s\Big[\frac{1}{2}\int_0^1 (1-\vartheta) \partial^3_{xxx}f(s,\widetilde{X}^{\{1,2\},\lambda}_{s} + \vartheta \delta\widetilde{X}^{\{1,2\},\lambda}_{s,t}) \ud \vartheta \big(\delta \widetilde{X}^{\{1,2\},\lambda}_{s,t}\big)^{\otimes 2}\Big]\Big|\leq \|\partial^3_{xxx}f\|_\infty \mathbb{E}_s\big[|\delta\widetilde{X}^{\{1,2\},\lambda}_{s,t}|^2\big]\\
    &\qquad\qquad\leq 2 \|\partial^3_{xxx}f\|_\infty \Big( (1-\lambda) \mathbb{E}_s\big[|\delta{X}^{1}_{s,t}|^2\big] + \lambda \mathbb{E}_s\big[|\delta{X}^{2}_{s,t}|^2\big]\Big) \\
    &\qquad\qquad\leq 2 \|\partial^3_{xxx}f\|_\infty \Big(  \|\delta{X}^{1}\|^2_{\alpha; k,\infty} + \|\delta{X}^{2}\|^2_{\alpha; k,\infty}\big]\Big) |t-s|^{2\alpha}.
\end{align*}
For the second term we have
\begin{align*}
    &\bigg|\mathbb{E}_s\bigg[\partial^2_{xx}f(s,\widetilde{X}^{\{1,2\},\lambda}_s) \Big[\delta \widetilde{X}^{\{1,2\},\lambda}_{s,t} -\Big(f_s({X}^{1}_s) + \lambda \big[f(s,{X}^2_s) -f(s,X^1_s)\big]\Big) \eta_{s,t}\Big]\bigg]\bigg|\\
    &\qquad\qquad\leq  \|\partial^2_{xx}f\|_\infty \Big( (1-\lambda) \big|\mathbb{E}_s\big[R^{X^1}_{s,t}\big]\big| + \lambda \big|\mathbb{E}_s\big[R^{X^2}_{s,t}\big]\big|\Big)\\
    &\qquad\qquad\leq  \|\partial^2_{xx}f\|_\infty \Big(  \big\|\mathbb{E}_\bullet R^{X^1}\big\|_{\alpha+\beta;\infty} + \big\|\mathbb{E}_\bullet R^{X^2}\big\|_{\alpha + \beta;\infty}\Big) |t-s|^{\alpha+\beta}.
\end{align*}
For the third term we have
\begin{align*}
   \Big|\mathbb{E}_s\Big[\partial_xf(t,\widetilde{X}^{\{1,2\},\lambda}_s) - \partial_xf(s,\widetilde{X}^{\{1,2\},\lambda}_s) - \partial_x f'(s,\widetilde{X}^{\{1,2\},\lambda}_s) \eta_{s,t}\Big]\Big| \leq \llbracket \mathbb{E}_\bullet R^{\partial_x f}\rrbracket_{\beta + \beta';\infty} |t-s|^{\beta+\beta'}.
\end{align*}
And for the fourth term we have
\begin{align*}
    &\Big|\mathbb{E}_s\Big[\delta (\partial_xf)_{s,t}(\widetilde{X}^{\{1,2\},\lambda}_t) - \delta (\partial_xf)_{s,t}(\widetilde{X}^{\{1,2\},\lambda}_s)\Big]\Big| \leq  \|\delta(\partial^2_{xx}f)_{s,t}\|_\infty \mathbb{E}_s\big[|\delta\widetilde{X}^{\{1,2\},\lambda}_{s,t}|\big]\\
    &\qquad\qquad\leq \|\delta(\partial^2_{xx}f)_{s,t}\|_\infty\Big((1-\lambda) \mathbb{E}_s\big[|\delta{X}^{1}_{s,t}|\big] + \lambda \mathbb{E}_s\big[|\delta{X}^{2}_{s,t}|\big] \Big)\\
     &\qquad\qquad\leq \|\delta(\partial^2_{xx}f)_{s,t}\|_\infty\Big( \|\delta{X}^{1}\|_{\alpha; k,\infty} + \|\delta{X}^{2}\|_{\alpha; k,\infty} \Big) |t-s|^{\alpha}\\
     &\qquad\qquad\leq C_{\partial^2_{xx}f}\Big( \|\delta{X}^{1}\|_{\alpha; k,\infty} + \|\delta{X}^{2}\|_{\alpha; k,\infty} \Big) |t-s|^{\alpha + \beta'}.
\end{align*}
\end{proof}

Using Lemma \ref{lemma 4.1} with $\lambda = 0$, under Assumptions \ref{A1}--
\ref{A3}, by \cite[Theorem 3.11]{bugini2024malliavin}, for $k \in (4,\infty)$, we get that \eqref{eq: var_process} has a unique solution $(V,\partial_x f(\cdot,X)V)\in \mathbf{D}_{\eta}^{\alpha,\beta}\mathbb{L}^k([0,T]; \mathbb{R}^{d})$.

Let us compute the directional derivative $\lim_{\epsilon\rightarrow0^+}\frac{J(u^\epsilon)-J(u)}{\epsilon}$. By the fundamental Theorem of calculus we have that
\begin{multline}\label{eq. 053}
    \frac{\mathcal{J}(u^\epsilon)-\mathcal{J}(u)}{\epsilon}\\ = \mathbb{E}\bigg[\int_0^T \Big(\partial_x h(s,X_s,u_s)\frac{X^{\epsilon}_s - X_s}{\epsilon} + \partial_u h(s,X_s,u_s) v_s\Big) \ud s + \partial_x g(X_T)\frac{X^{\epsilon}_T - X_T}{\epsilon}\bigg] + \mathcal{R}_1(\epsilon)\\
    = \mathbb{E}\bigg[\int_0^T \Big(\partial_x h(s,X_s,u_s)V_s + \partial_u h(s,X_s,u_s) v_s\Big) \ud s + \partial_x g(X_T) V_T\bigg] + \mathcal{R}_1(\epsilon) + \mathcal{R}_2(\epsilon),
\end{multline}
    where
\begin{align*}
    \mathcal{R}_1(\epsilon) &:= \mathbb{E}\bigg[\int_0^1 \Big(\partial_x g({X}_T+\lambda (X^\epsilon_T-{X}_T)) -\partial_x g({X}_T) \Big)\ud\lambda \frac{X^\epsilon_T-{X}_T}{\epsilon}\bigg]\\
   &\qquad+\mathbb{E}\bigg[\int_0^T \int_0^1  \Big(\partial_x h(s,{X}_s + \lambda (X^\epsilon_s - {X}_s),u^\epsilon_s) -\partial_x h(s,{X}_s,u_s) \Big)\ud\lambda \frac{X^\epsilon_s - {X}_s}{\epsilon} \ud s \bigg]\\
   &\qquad+ \mathbb{E}\bigg[\int_0^T \int_0^1  \Big(\partial_u h(s,{X}_s,u_s +\lambda \epsilon v_s) -\partial_u h(s,{X}_s,u_s) \Big)\ud\lambda v_s \ud s \bigg],\\
   \mathcal{R}_2(\epsilon) &:= \mathbb{E}\bigg[\int_0^T  \partial_x h (s,{X}_s,{u}_s) \Big(\frac{X^\epsilon_s - {X}_s}{\epsilon} - V_s\Big) \ud s +  \partial_x g(\bar{X}_T) \Big(\frac{X^\epsilon_T - {X}_T}{\epsilon} - V_T\Big)\bigg].
\end{align*}
So, by the Lipschitz properties of Assumption \ref{A3} we get that    
    \begin{align*}
    |\mathcal{R}_1(\epsilon)|&\lesssim \frac{\sup_{t\in[0,T]}\mathbb{E}\big[|X^\epsilon_t - X_t|^2\big]}{\epsilon} + \epsilon  \int_0^T\|v_s\|_2 \left\| \frac{X^{\epsilon}_t - X_t}{\epsilon}\right\|_2 \ud s +\frac{\mathbb{E}\big[|X^\epsilon_T-X_T|^2\big]}{\epsilon}\\
   & \lesssim \epsilon \sup_{t \in [0,T]} \left\| \frac{X^{\epsilon}_t - X_t}{\epsilon} \right\|_2^2 
   + \epsilon \sup_{t \in [0,T]} \left\| v_t \right\|_2  \sup_{t \in [0,T]} \left\| \frac{X^{\epsilon}_t - X_t}{\epsilon} \right\|_2,\\
    |\mathcal{R}_2(\epsilon)|&\lesssim \sup_{t \in [0,T]}\left\| \frac{X^{\epsilon}_t - X_t}{\epsilon} - V_t \right\|_2 + \left\| \frac{X^{\epsilon}_T - X_T}{\epsilon} - V_T \right\|_2 \lesssim \sup_{t \in [0,T]}\left\| \frac{X^{\epsilon}_t - X_t}{\epsilon} - V_t \right\|_2. 
\end{align*}
Hence, to show that $\lim_{\epsilon\rightarrow0^+}[\mathcal{R}_1(\epsilon) + \mathcal{R}_2(\epsilon)] = 0$ it suffices to prove the following Lemma.

\begin{lemma}\label{lem:firstoderleftover1}
Let $k \in (4,\infty)$ and $X,X^{\epsilon}, V$ defined as above. Then, under Assumptions \ref{A1}--
\ref{A3} it holds that
    \begin{equation}
        \lim_{\epsilon \rightarrow 0^+} \sup_{t\in[0,T]} \bigg\| \frac{X^{\epsilon}_t - X_t}{\epsilon} - V_t \bigg\|_2 = 0.   
    \end{equation}
\end{lemma}
In the proof of Lemma \ref{lem:firstoderleftover1} we shall need the next result.
\begin{lemma}\label{lem:diff1}
Let $k \in (4,\infty)$ and $X,X^\epsilon$ defined as above. Then, under Assumptions \ref{A1}--
\ref{A3} it holds that
\begin{equation}
    \big\|\sup_{t\in[0,T]}|X^\epsilon_t -X_t|\big\|_k +\|\delta X^{\epsilon}-\delta X\|_{\alpha;k} \lesssim \epsilon,
\end{equation}
where the implicit constant is independent of $\epsilon$.
\end{lemma}
\begin{proof}[Proof of Lemma \ref{lem:diff1}]
We have by the Lipschitz properties of $b$ and $\sigma$ that
\begin{equation*}
\begin{aligned}
    \sup_{t\in[0,T]}\big\|\sup_{x\in\mathbb{R}^d}|b(t,x,u^\epsilon_t)-b(t,x,u_t)|\big\|_k +  \sup_{t\in[0,T]}\big\|\sup_{x\in\mathbb{R}^d}|\sigma(t,x,u^\epsilon_t)-\sigma(t,x,u_t)|\big\|_k 
    \lesssim  \sup_{t\in[0,T]}\big\|v_t \big\|_k \epsilon.
    \end{aligned}
\end{equation*}
So, we define 
\begin{equation*}
\theta:= \sup_{t\in[0,T]}\big\|\sup_{x\in\mathbb{R}^d}|b(t,x,u^\epsilon_t)-b(t,x,u_t)|\big\|_k +  \sup_{t\in[0,T]}\big\|\sup_{x\in\mathbb{R}^d}|\sigma(t,x,u^\epsilon_t)-\sigma(t,x,u_t)|\big\|_k
\end{equation*}
and by \cite[Theorem 4.9]{FHL21} we conclude.
\end{proof}

\begin{proof}[Proof of Lemma \ref{lem:firstoderleftover1}]
To ease the notation, we define $V^\epsilon_t := \frac{X^{\epsilon}_t - X_t}{\epsilon}-V_t$, for every $t\in [0,T]$. Furthermore, we define
\begin{align*}
    \widetilde{\Phi}^\epsilon_t &:= \frac{f(t,X^\epsilon_t) - f(t,X_t)}{\epsilon}-\partial_x f(t,X_t) V_t,\\
    (\widetilde{\Phi}^\epsilon)'_t&:=  \frac{\Big[\partial_xf(t,X^\epsilon_t)f(t,X^\epsilon_t) +f'(t,X^\epsilon_t) \Big] - \Big[\partial_xf(t,X_t)f(t,X_t) +f'(t,X_t) \Big]}{\epsilon}\\
    &\qquad - \Big[\partial_{xx}^2f(t,X_t) f(t,X_t) + (\partial_xf)'(t,X_t)\Big]V_t - \big(\partial_xf(t,X_t)\big)^2 V_t
\end{align*}
and we have that $\big(\widetilde{\Phi}^\epsilon,(\widetilde{\Phi}^\epsilon)'\big)\in \mathbf{D}^{\beta,\beta'}_\eta\mathbb{L}^{k}\big([0,T];\mathbb{R}^{d\times \ell}\big)$. 

Next, $V^\epsilon$ satisfies the following equation
\begin{align*}
   V^\epsilon_t =& \int_0^t \bigg(\frac{b(s,X^\epsilon_s,u^\epsilon_s) - b(s,X_s,u_s)}{\epsilon}-\partial_x b(s,X_s,u_s) V_s - \partial_u b(s,X_s,u_s) v_s\bigg) \ud s \\& + \int_0^t \bigg(\frac{\sigma(s,X^\epsilon_s,u^\epsilon_s) - \sigma(s,X_s,u_s)}{\epsilon}-\partial_x \sigma(s,X_s,u_s) V_s - \partial_u \sigma(s,X_s,u_s) v_s\bigg) \ud W_s \\ &+ \int_0^t \bigg(\frac{f(s,X^\epsilon_s) - f(s,X_s)}{\epsilon}-\partial_x f(s,X_s) V_s \bigg) \ud \boldsymbol{\eta}_s \\
    = & \int_0^t \int_0^1 \bigg(\partial_x b(s,X_s + \lambda \epsilon (V^\epsilon_s+V_s),u_s + \lambda \epsilon v_s) (V^\epsilon_s + V_s) -\partial_x b(s,X_s,u_s) V_s \\ & + \partial_u b(s,X_s + \lambda \epsilon (V^\epsilon_s + V_s),u_s+\lambda\epsilon v_s) v_s -  \partial_u b(s,X_s,u_s) v_s\bigg) \ud \lambda \ud s \\ & + \int_0^t \int_0^1 \bigg(\partial_x \sigma(s,X_s + \lambda \epsilon (V^\epsilon_s+V_s),u_s + \lambda \epsilon v_s) (V^\epsilon_s + V_s) -\partial_x \sigma(s,X_s,u_s) V_s \\ &+ \partial_u \sigma(s,X_s + \lambda \epsilon (V^\epsilon_s + V_s),u_s+\lambda\epsilon v_s) v_s -  \partial_u \sigma(s,X_s,u_s) v_s\bigg) \ud \lambda \ud W_s\\
   & + \int_0^t \int_0^1 \bigg(\partial_x f(s,X_s + \lambda \epsilon (V^\epsilon_s + V_s)) (V^\epsilon_s + V_s) - \partial_x f(s,X_s) V_s \bigg) \ud \lambda \ud \boldsymbol{\eta}_s\\
    = & \int_0^t\Bigg( \bigg[\int_0^1 \partial_x b(s,X_s + \lambda \epsilon (V^\epsilon_s+V_s),u_s + \lambda \epsilon v_s) \ud\lambda\bigg] V^\epsilon_s \\ &+ \bigg[ \int_0^1\Big(\partial_x b(s,X_s + \lambda \epsilon (V^\epsilon_s+V_s),u_s + \lambda \epsilon v_s) - \partial_x b(s,X_s,u_s) \Big)\ud\lambda\bigg] V_s \\ & + \bigg[\int_0^1\Big(\partial_u b(s,X_s + \lambda \epsilon (V^\epsilon_s + V_s),u_s+\lambda\epsilon v_s)-  \partial_u b(s,X_s,u_s)\Big) \ud \lambda \bigg] v_s \Bigg)\ud s \\
  &+ \int_0^t\Bigg( \bigg[\int_0^1 \partial_x \sigma(s,X_s + \lambda \epsilon (V^\epsilon_s+V_s),u_s + \lambda \epsilon v_s) \ud\lambda\bigg] V^\epsilon_s \\ &+ \bigg[ \int_0^1\Big(\partial_x \sigma(s,X_s + \lambda \epsilon (V^\epsilon_s+V_s),u_s + \lambda \epsilon v_s) - \partial_x \sigma(s,X_s,u_s) \Big)\ud\lambda\bigg] V_s \\ &+ \bigg[\int_0^1\Big(\partial_u \sigma(s,X_s + \lambda \epsilon (V^\epsilon_s + V_s),u_s+\lambda\epsilon v_s)-  \partial_u \sigma(s,X_s,u_s)\Big) \ud \lambda \bigg] v_s \Bigg)\ud W_s\\
   & + \int_0^t\Bigg( \bigg[\int_0^1 \partial_x f(s,X_s + \lambda \epsilon (V^\epsilon_s + V_s)) V^\epsilon_s \ud\lambda\bigg]  \\ &+ \bigg[\int_0^1\Big(\partial_x f(s,X_s + \lambda \epsilon (V^\epsilon_s + V_s)) - \partial_x f(s,X_s) \Big) V_s \ud \lambda\bigg] \Bigg)\ud \boldsymbol{\eta}_s.
\end{align*}
Now, we define for every $\lambda\in[0,1]$
\begin{align*}
 \Phi^\epsilon_t(\lambda)&:=  \partial_x f(t,X_t + \lambda  (X^\epsilon_t - X_t)),\\ 
 (\Phi^\epsilon)'_t(\lambda)&:= \partial^2_{xx} f(t,X_t + \lambda  (X^\epsilon_t - X_t)) \Big(f_t({X}_t) + \lambda \big(f_t({X}^\epsilon_t) -f_t(X_t)\big)\Big) \\
 &\qquad+\partial_x f'(t,X_t + \lambda  (X^\epsilon_t - X_t)).
\end{align*}
From Lemma \ref{lemma 4.1} and its proof, along with the boundness properties of $f, \partial_x f, \partial^2_{xx} f, \partial_xf'$ and \cite[Proposition 4.5]{FHL21}, we see that $\big(\Phi^\epsilon(\lambda),(\Phi^\epsilon)'(\lambda)\big)\in \mathbf{D}_{\eta}^{\beta,\beta'} \mathbb{L}^{k,\infty}_{\lambda-ex,lin}\big([0,T];L(\mathbb{R}^d,\mathbb{R}^{d\times \ell})\big)$ and
\begin{align}\label{eq. 4.7}
\sup_{\epsilon > 0}\sup_{\lambda\in[0,1]}\Big\{\sup_{t\in[0,T]}\|\Phi^\epsilon_t(\lambda)\|_\infty +  \sup_{t\in[0,T]}\|(\Phi^\epsilon)'_t(\lambda)\|_\infty
+\big\|\big(\Phi^\epsilon_\cdot(\lambda),(\Phi^\epsilon)'_\cdot(\lambda)\big)\big\|_{\mathbf{D}_{\eta}^{\beta,\beta'} \mathbb{L}^{k,\infty}_{lin}} \Big\} < \infty.
\end{align}
Furthermore, we define
\begin{multline*}
    {F}^{\epsilon}_t: = \int_0^t\Bigg( \bigg[ \int_0^1\Big(\partial_x b(s,X_s + \lambda \epsilon (V^\epsilon_s+V_s),u_s + \lambda \epsilon v_s) - \partial_x b(s,X_s,u_s) \Big)\ud\lambda\bigg] V_s \\+ \bigg[\int_0^1\Big(\partial_u b(s,X_s + \lambda \epsilon (V^\epsilon_s + V_s),u_s+\lambda\epsilon v_s)-  \partial_u b(s,X_s,u_s)\Big) \ud \lambda \bigg] v_s \Bigg)\ud s\\
    + \int_0^t\Bigg(  \bigg[ \int_0^1\Big(\partial_x \sigma(s,X_s + \lambda \epsilon (V^\epsilon_s+V_s),u_s + \lambda \epsilon v_s) - \partial_x \sigma(s,X_s,u_s) \Big)\ud\lambda\bigg] V_s \\+ \bigg[\int_0^1\Big(\partial_u \sigma(s,X_s + \lambda \epsilon (V^\epsilon_s + V_s),u_s+\lambda\epsilon v_s)-  \partial_u \sigma(s,X_s,u_s)\Big) \ud \lambda \bigg] v_s \Bigg)\ud W_s\\
    + \int_0^t \bigg[\int_0^1\Big(\partial_x f(s,X_s + \lambda (X^\epsilon_s - X_s)) - \partial_x f(s,X_s) \Big) V_s  \ud \lambda\bigg] \ud \boldsymbol{\eta}_s.
\end{multline*}
and 
\begin{align*}
\hspace{.15cm}(F^\epsilon)'_t(\lambda)&:= \Big(\partial_x f(t,X_t + \lambda (X^\epsilon_t - X_t)) - \partial_x f(t,X_t) \Big) V_t,\\
    (F^\epsilon)''_t(\lambda)&:=  \bigg[\bigg(\partial^2_{xx} f(t,X_t + \lambda  (X^\epsilon_t - X_t))  \Big(f_t({X}_t) + \lambda \big[f_t({X}^\epsilon_t) -f_t(X_t)\big]\Big) \\&\qquad\qquad+\partial_x f'(t,X_t + \lambda  (X^\epsilon_t - X_t)) \bigg)
 -\partial_{xx}^2f(t,X_t) f(t,X_t) - \partial_xf'(t,X_t)\bigg] V_t\\&\qquad\qquad\qquad+ \bigg[\partial_x f(t,X_t + \lambda (X^\epsilon_t - X_t))\partial_x f(t,X_t) - \big(\partial_x f(t,X_t)\big)^2 \bigg] V_t.
\end{align*}
From Lemma \ref{lemma 4.1}, Proposition \ref{Prop 3.3} and Proposition \ref{Prop 3.2}, along with standard calculations, we can see that 
\begin{align*}
\bigg(\int_0^1 (F^\epsilon)'_\cdot(\lambda)\ud\lambda,\int_0^1 (F^\epsilon)''_\cdot(\lambda)\ud\lambda\bigg)&\in \mathbf{D}_{\eta}^{\beta,\beta'}\mathbb{L}^k([0,T]; \mathbb{R}^{d}),\\
    \bigg(F^\epsilon_\cdot,\int_0^1 (F^\epsilon)'_\cdot(\lambda)\ud\lambda\bigg)&\in \mathbf{D}_{\eta}^{\alpha,\beta}\mathbb{L}^k([0,T]; \mathbb{R}^{d}).
\end{align*}
Next, by the fundamental Theorem of calculus we have for every $t\in [0,T]$ that
\begin{multline*}
   (\widetilde{\Phi}^\epsilon)'_t = \int_0^1\bigg( \Big[\partial^2_{xx} f(t,X_t + \lambda  (X^\epsilon_t - X_t)) \Big(f_t({X}_t) + \lambda \big[f_t({X}^\epsilon_t) -f_t(X_t)\big]\Big) \\+ \partial_x f'(t,X_t + \lambda  (X^\epsilon_t - X_t)) \Big] (V^\epsilon_t+V_t) + f(t,X_t + \lambda  (X^\epsilon_t - X_t)) \big(f_t({X}^\epsilon_t) -f_t(X_t)\big) \\-\Big[\partial_{xx}^2f(t,X_t) f(t,X_t) + (\partial_xf)'(t,X_t)\Big]V_t - \big(\partial_xf(t,X_t)\big)^2 V_t\bigg) \ud \lambda\\
   = \int_0^1\Big((\Phi^\epsilon)'_t(\lambda)V^\epsilon_t + (\Phi^\epsilon)_t(\lambda)\Big[\int_0^1 \big((\Phi^\epsilon)_t(\widetilde{\lambda})V^\epsilon_t + (F^\epsilon)'_t(\widetilde{\lambda})\big) \ud\widetilde{\lambda}\Big] + (F^\epsilon)''_t(\lambda)\Big)\ud\lambda.
\end{multline*}
Hence, from Lemma \ref{lemma 3.4} and Proposition \ref{Prop 3.5} we see that $V^\epsilon$ satisfies the following equation
\begin{multline}
    V^\epsilon_t = F^\epsilon_t + \int_0^t \bigg[\int_0^1 \partial_x b(s,X_s + \lambda \epsilon (V^\epsilon_s+V_s),u_s + \lambda \epsilon v_s) \ud\lambda\bigg] V^\epsilon_s \ud s \\
  + \int_0^t \bigg[\int_0^1 \partial_x \sigma(s,X_s + \lambda \epsilon (V^\epsilon_s+V_s),u_s + \lambda \epsilon v_s) \ud\lambda\bigg] V^\epsilon_s \ud W_s\\
    + \int_0^t
    \bigg(\int_0^1 \Phi^\epsilon_s(\lambda)\ud\lambda,\int_0^1(\Phi^\epsilon)'_s(\lambda) \ud\lambda \bigg) V^\epsilon_s
    \ud \boldsymbol{\eta}_s.
\end{multline}
Thus, from Theorem \ref{thm 3.11} we have
\begin{align*}
    \sup_{t\in[0,T]} \left\| \frac{X^{\epsilon}_t - X_t}{\epsilon} - V_t \right\|_2 \lesssim \bigg\|\bigg(F^\epsilon_\cdot,\int_0^1 (F^\epsilon)'_\cdot(\lambda)\ud\lambda\bigg)\bigg\|_{\mathbf{D}_{\eta}^{\alpha,\beta} \mathbb{L}^{2}},
\end{align*}
where the implicit constant is independent of $\epsilon$, we remind the bound \eqref{eq. 4.7}.
So, to conclude we need to show that
\begin{equation}\label{eq 24}
    \lim_{\epsilon\rightarrow 0^+}\bigg\|\bigg(F^\epsilon_\cdot,\int_0^1 (F^\epsilon)'_\cdot(\lambda)\ud\lambda\bigg)\bigg\|_{\mathbf{D}_{\eta}^{\alpha,\beta} \mathbb{L}^{2}} = 0.
\end{equation}
We have from Jensen's inequality
\begin{align*}
    &\mathbb{E}\Bigg[\bigg|\int_s^t\Bigg( \bigg[ \int_0^1\Big(\partial_x b(r,X_r + \lambda \epsilon (V^\epsilon_r + V_r),u_r + \lambda \epsilon v_r) - \partial_x b(r,X_r,u_r) \Big)\ud\lambda\bigg] V_r \\
    &\qquad\qquad+ \bigg[\int_0^1\Big(\partial_u b(r,X_r + \lambda \epsilon (V^\epsilon_r + V_r),u_r + \lambda\epsilon v_r)-  \partial_u b(r,X_r,u_r)\Big) \ud \lambda \bigg] v_r \Bigg)\ud r\bigg|^2\Bigg]^\frac{1}{2}\\
    &\leq \sqrt{2} |t-s|\\& \mathbb{E}\Bigg[\int_s^t \int_0^1\sum_{i\in\{x,u\}}\Big|\partial_i b(r,X_r + \lambda (X^\epsilon_r - X_r),u_r + \lambda \epsilon v_r) - \partial_i b(r,X_r,u_r) \Big|^2 \Big(|V_r|^2 + |v_r|^2\Big)\ud\lambda \frac{\ud r}{|t-s|}\Bigg]^\frac{1}{2}\\
     &\leq \sqrt{2} |t-s|\Bigg(\int_s^t \int_0^1 \sum_{i\in\{x,u\}}\mathbb{E}\bigg[\Big|\partial_i b(r,X_r + \lambda (X^\epsilon_r - X_r),u_r + \lambda \epsilon v_r) - \partial_i b(r,X_r,u_r) \Big|^2 
     \\&\qquad\qquad
     \Big(|V_r|^2 + |v_r|^2\Big)\bigg]\ud\lambda \frac{\ud r}{|t-s|}\Bigg)^\frac{1}{2}\\
      &\leq \sqrt{2} |t-s|\Bigg(\int_s^t \int_0^1 \sum_{i\in\{x,u\}}\mathbb{E}\bigg[\Big|\partial_i b(r,X_r + \lambda (X^\epsilon_r - X_r),u_r + \lambda \epsilon v_r) - \partial_i b(r,X_r,u_r) \Big|^{4}\bigg]^{\frac{1}{2}} 
      \\&\qquad\qquad
      \mathbb{E}\bigg[\Big(|V_r|^2 + |v_r|^2\Big)^{2}\bigg]^\frac{1}{2}\ud\lambda \frac{\ud r}{|t-s|}\Bigg)^\frac{1}{2}\\
     &\leq \mathscr{C}_1(\epsilon) |t-s|,
\end{align*}
where
\begin{multline*}
    \mathscr{C}_1(\epsilon):= 2\sqrt{2} \hspace{.1cm} \Big(\sup_{r\in[0,T]}\mathbb{E}\big[|V_r|^{4}\big] + \sup_{r\in [0,T]}\mathbb{E}\big[|v_r|^{4}\big]\Big)^\frac{1}{4}\\
    \bigg(\sum_{i\in\{x,u\}}\sup_{\lambda\in[0,1]}\sup_{r\in[0,T]}\mathbb{E}\bigg[\Big|\partial_i b(r,X_r + \lambda (X^\epsilon_r - X_r),u_r + \lambda \epsilon v_r) - \partial_i b(r,X_r,u_r) \Big|^{4}\bigg]^{\frac{1}{4}}\bigg).
\end{multline*}
From the $\text{Lip}^\theta$ properties of \ref{A1} and Lemma \ref{lem:diff1} we have 
\begin{equation}\label{eq 12}
    \lim_{\epsilon\rightarrow 0^+}  \mathscr{C}_1(\epsilon) = 0.
\end{equation}

Next, from BDG inequality we have similarly to above
\begin{align*}
    &\mathbb{E}\Bigg[\bigg|\int_s^t\Bigg( \bigg[ \int_0^1\Big(\partial_x \sigma(r,X_r + \lambda \epsilon (V^\epsilon_r + V_r),u_r + \lambda \epsilon v_r) - \partial_x \sigma(r,X_r,u_r) \Big)\ud\lambda\bigg] V_r \\
    &\qquad\qquad+ \bigg[\int_0^1\Big(\partial_u \sigma(r,X_r + \lambda \epsilon (V^\epsilon_r + V_r),u_r + \lambda\epsilon v_r)-  \partial_u \sigma(r,X_r,u_r)\Big) \ud \lambda \bigg] v_r \Bigg)\ud W_r\bigg|^2\Bigg]^\frac{1}{2}\\
    &\leq \mathscr{C}_2(\epsilon) |t-s|^{\frac{1}{2}},
     \end{align*}
where
\begin{multline*}
    \mathscr{C}_2(\epsilon):= 2\hspace{.1cm}\Big(\sup_{r\in[0,T]}\mathbb{E}\big[|V_r|^{4}\big] + \sup_{r\in [0,T]}\mathbb{E}\big[|v_r|^{4}\big]\Big)^\frac{1}{4}\\
\bigg(\sum_{i\in\{x,u\}}\sup_{\lambda\in[0,1]}\sup_{r\in[0,T]}\mathbb{E}\bigg[\Big|\partial_i \sigma(r,X_r + \lambda (X^\epsilon_r - X_r),u_r + \lambda \epsilon v_r) - \partial_i \sigma(r,X_r,u_r) \Big|^{4}\bigg]^{\frac{1}{4}}\bigg).
\end{multline*}
Again, from the $\text{Lip}^\theta$ properties of \ref{A1} and Lemma \ref{lem:diff1} we have 
\begin{equation}\label{eq 13}
    \lim_{\epsilon\rightarrow 0^+}  \mathscr{C}_2(\epsilon) = 0.
\end{equation}
We now handle the rough component of $F^\epsilon$. First note that by the boundness properties of $f,\partial_xf,\partial^2_{xx}f$, their Lipschitz continuity properties, the integrability properties of $V$, H\"older inequality, the dominated convergence theorem, and the help of Lemma \ref{lem:diff1}, we get with the same arguments that concluded to \eqref{eq 12}, \eqref{eq 13} above that 
\begin{equation}\label{eq 27}
\lim_{\epsilon\rightarrow 0^+}\sup_{t\in[0,T]}\Big\|\int_0^1 (F^\epsilon)'_t(\lambda)\ud\lambda\Big\|_2 +\lim_{\epsilon\rightarrow 0^+}\sup_{t\in[0,T]}\Big\|\int_0^1 (F^\epsilon)''_t(\lambda)\ud\lambda\Big\|_2 = 0.
\end{equation}
Next, we need to estimate $\Big\|\delta\int_0^1 (F^\epsilon)'_\cdot(\lambda)\ud\lambda\Big\|_{\beta,2}, \Big\|\mathbb{E}_\bullet R^{\int_0^1 (F^\epsilon)'_\cdot(\lambda)\ud\lambda}\Big\|_{\beta + \beta',2}, \Big\|\mathbb{E}_\bullet \delta\int_0^1 (F^\epsilon)''_\cdot(\lambda)\ud\lambda\Big\|_{\beta',2}$.
So, by Tonelli's Theorem and Jensen's inequality we have
\begin{align*}
   \Big\| \int_0^1 (F^\epsilon)'_t(\lambda)\ud\lambda - \int_0^1 (F^\epsilon)'_s(\lambda)\ud\lambda \Big\|_2 \leq \bigg( \int_0^1 \mathbb{E}\Big[\Big|(F^\epsilon)'_t(\lambda) -(F^\epsilon)'_s(\lambda)\Big|^2\Big]\ud\lambda\bigg)^\frac{1}{2}.
\end{align*}
But for $s,t \in [0,T]$ with $s\leq t$ we have from H\"older's inequality and the fundamental Theorem of calculus
\begin{align*}
    &\mathbb{E}\Big[\Big|(F^\epsilon)'_t(\lambda) -(F^\epsilon)'_s(\lambda)\Big|^2\Big] \\&= \mathbb{E}\Big[ \Big|\Big(\partial_x f(t,X_t + \lambda (X^\epsilon_t - X_t)) - \partial_x f(t,X_t) \Big) V_t - \Big(\partial_x f(s,X_s + \lambda (X^\epsilon_s - X_s)) - \partial_x f(s,X_s) \Big) V_s \Big|^2\Big]\\
    &\leq 2 \hspace{.1cm}\mathbb{E} \Big[\Big|\Big(\partial_x f(t,X_t + \lambda (X^\epsilon_t - X_t)) - \partial_x f(t,X_t)\Big) \big(V_t - V_s\big)\Big|^2 \Big] \\
    &\quad + 2 \hspace{.1cm}\mathbb{E} \Big[\Big|\Big(\partial_x f(t,X_t + \lambda (X^\epsilon_t - X_t)) - \partial_x f(t,X_t) - \partial_x f(s,X_s + \lambda (X^\epsilon_s - X_s)) + \partial_x f(s,X_s)\Big) V_s\Big|^2 \Big]\\
     &\leq 2\hspace{.1cm} \mathbb{E}\bigg[\sup_{\lambda\in[0,1]}\sup_{r\in[0,T]}\Big|\partial_x f(r,X_r + \lambda (X^\epsilon_r - X_r)) - \partial_x f(r,X_r) \Big|^{4}\bigg]^{\frac{1}{2}} \mathbb{E}\big[|V_t-V_s|^{4}\big]^{\frac{1}{2}}\\
     & + 2 \hspace{.1cm}\mathbb{E} \Big[\Big|\partial_x f(t,X_t + \lambda (X^\epsilon_t - X_t)) - \partial_x f(t,X_t) - \partial_x f(s,X_s + \lambda (X^\epsilon_s - X_s)) + \partial_x f(s,X_s)\Big|^{4} \Big]^{\frac{1}{2}} \\
     &\hspace{1cm}\sup_{r\in [0,T]} \mathbb{E}\big[|V_r|^{4}\big]^{\frac{1}{2}}\\
      &\leq 2 \hspace{.1cm}\mathbb{E}\bigg[\sup_{\lambda\in[0,1]}\sup_{r\in[0,T]}\Big|\partial_x f(r,X_r + \lambda (X^\epsilon_r - X_r)) - \partial_x f(r,X_r) \Big|^{4}\bigg]^{\frac{1}{2}}  C \|\delta V\|^2_{\alpha;4} |t-s|^{2\alpha}\\
     & + 4 \hspace{.1cm}\mathbb{E} \Big[\Big|\partial_x f(s,X_t + \lambda (X^\epsilon_t - X_t)) - \partial_x f(s,X_t) - \partial_x f(s,X_s + \lambda (X^\epsilon_s - X_s)) + \partial_x f(s,X_s)\Big|^{4} \Big]^{\frac{1}{2}} \\
     &\hspace{1cm}\sup_{r\in [0,T]} \mathbb{E}\big[|V_r|^{4}\big]^{\frac{1}{2}} \\& + 4 \mathbb{E}\bigg[\sup_{\lambda\in[0,1]}\sup_{r_1\in[0,T]}\sup_{r_2\in [0,T]} \Big|\partial^2_{xt} f(r_1,X_{r_2} + \lambda (X^\epsilon_{r_2} - X_{r_2})) - \partial^2_{xt} f(r_1,X_{r_2})\Big|^{4}\bigg]^{\frac{1}{2}}  \mathbb{E}\big[|V_r|^{4}\big]^{\frac{1}{2}} |t-s|^2,
\end{align*}
where again from the fundamental Theorem of calculus we have
\begin{align*}
   &\mathbb{E} \bigg[\Big|\partial_x f(s,X_t + \lambda (X^\epsilon_t - X_t)) - \partial_x f(s,X_t) - \partial_x f(s,X_s + \lambda (X^\epsilon_s - X_s)) + \partial_x f(s,X_s)\Big|^{4} \bigg]^{\frac{1}{2}}\\
    &\leq \mathbb{E} \bigg[\Big|\partial_x f(s,X_t + \lambda (X^\epsilon_s - X_s)) - \partial_x f(s,X_t) - \partial_x f(s,X_s + \lambda (X^\epsilon_s - X_s)) + \partial_x f(s,X_s)\Big|^{4} \bigg]^{\frac{1}{2}}\\ 
    &\qquad+ C \|\delta X^{\epsilon}-\delta X\|^2_{\alpha;4} |t-s|^{2\alpha}\\
    &\leq \mathbb{E}\bigg[\sup_{\lambda\in[0,1]}\sup_{r_1\in[0,T]}\sup_{r_2\in [0,T]} |\partial^2_{xx} f(r_1,X_{r_2} + \lambda (X^\epsilon_{r_1} - X_{r_1})) - \partial^2_{xx} f(r_1,X_{r_2})|^{\frac{4+\varepsilon}{\varepsilon}}\bigg]^{\frac{2\varepsilon}{4+\varepsilon}}  C \|\delta X\|^2_{\alpha;4+\varepsilon} |t-s|^{2\alpha}\\
    &\qquad+ C \|\delta X^{\epsilon}-\delta X\|^2_{\alpha;4} |t-s|^{2\alpha},
\end{align*}
for every $\varepsilon >0$.
So, by gathering all the above we see that
\begin{equation}\label{eq 28}
    \lim_{\epsilon \rightarrow 0^+} \Big\|\delta \int_0^1 (F^\epsilon)'_\cdot(\lambda)\ud\lambda\Big\|_{\beta,2} = 0.
\end{equation}
We move on with the estimate of $\Big\|\mathbb{E}_\bullet R^{\int_0^1 (F^\epsilon)'_\cdot(\lambda)\ud\lambda}\Big\|_{\beta + \beta',2}$. Hence, by standard algebra, see for example \cite[eq. (23) p.19]{bugini2024malliavin}, we have
\begin{multline*}
    \mathbb{E}_s \Big[ R_{s,t}^{\int_0^1 (F^\epsilon)'_\cdot(\lambda)\ud\lambda}\Big] =  \mathbb{E}_s \bigg[  \int_0^1 \big((F^\epsilon)'_t(\lambda) - (F^\epsilon)'_s(\lambda) -  (F^\epsilon)''_s(\lambda) \delta\eta_{s,t}\big) \ud\lambda \bigg]\\
    = \mathbb{E}_s \bigg[\int_0^1\Big(\delta\big(\partial_x f(\cdot,X_\cdot + \lambda (X^\epsilon_\cdot - X_\cdot)) - \partial_x f(\cdot,X_\cdot)\big)_{s,t}\delta V_{s,t} \\+ \big(\partial_x f(s,X_s + \lambda (X^\epsilon_s - X_s)) - \partial_x f(s,X_s)\big) R^V_{s,t} + R^{\partial_x f(\cdot,X_\cdot + \lambda (X^\epsilon_\cdot - X_\cdot)) - \partial_x f(\cdot,X_\cdot)}_{s,t} V_s \Big)\ud \lambda \bigg].
\end{multline*}
For the first term we have by conditional Jensen's and H\"older inequalities, similar to the previous calculations, that
\begin{align*}
    &\mathbb{E}\Bigg[\bigg|\mathbb{E}_s \bigg[\int_0^1\delta\big(\partial_x f(\cdot,X_\cdot + \lambda (X^\epsilon_\cdot - X_\cdot)) - \partial_x f(\cdot,X_\cdot)\big)_{s,t}\delta V_{s,t} \ud\lambda\bigg]\bigg|^2\Bigg]^\frac{1}{2}\\
    &\leq \mathbb{E}\Bigg[\int_0^1 \Big|\partial_x f(t,X_t + \lambda (X^\epsilon_t - X_t)) - \partial_x f(t,X_t) - \partial_x f(s,X_s + \lambda (X^\epsilon_s - X_s)) + \partial_x f(s,X_s)\Big|^2  \ud\lambda\big|\delta V_{s,t}\big|^2 \Bigg]^\frac{1}{2}\\
    &\lesssim \mathbb{E}\bigg[\sup_{\lambda\in[0,1]}\sup_{r_1\in[0,T]}\sup_{r_2\in [0,T]} \Big|\partial^2_{xt} f(r_1,X_{r_2} + \lambda (X^\epsilon_{r_2} - X_{r_2})) - \partial^2_{xt} f(r_1,X_{r_2})\Big|^{4}\bigg]^{\frac{1}{4}} \|\delta V\|_{\alpha;4} |t-s|^{1+\alpha}\\
    &\quad+  \mathbb{E}\bigg[\sup_{\lambda\in[0,1]}\sup_{r_1\in[0,T]}\sup_{r_2\in [0,T]} |\partial^2_{xx} f(r_1,X_{r_2} + \lambda (X^\epsilon_{r_1} - X_{r_1})) - \partial^2_{xx} f(r_1,X_{r_2})\Big|^{\frac{4+\varepsilon}{\varepsilon}}\bigg]^{\frac{\varepsilon}{4+\varepsilon}}  \|\delta X\|_{\alpha;4+\varepsilon} \|\delta V\|_{\alpha;4} |t-s|^{2\alpha}\\
    &\quad+ \|\delta X^{\epsilon}-\delta X\|_{\alpha;4} \|\delta V\|_{\alpha;4} |t-s|^{2\alpha},
\end{align*}
for every $\varepsilon>0$.
For the second term
\begin{align*}
    &\mathbb{E}\Bigg[\bigg|\mathbb{E}_s \bigg[\int_0^1 \big(\partial_x f(s,X_s + \lambda (X^\epsilon_s - X_s)) - \partial_x f(s,X_s)\big) R^V_{s,t}\ud\lambda\bigg]\bigg|^2\Bigg]^\frac{1}{2}\\
    &\leq \mathbb{E}\Bigg[\bigg|\int_0^1 \big(\partial_x f(s,X_s + \lambda (X^\epsilon_s - X_s)) - \partial_x f(s,X_s)\big) \ud\lambda\bigg|^2 \Big|\mathbb{E}_s\big[R^V_{s,t}\big]\Big|^2\Bigg]^\frac{1}{2}\\
    &\leq \mathbb{E}\bigg[\sup_{\lambda\in[0,1]}\sup_{r\in [0,T]}\big|\partial_x f(r,X_r + \lambda (X^\epsilon_r - X_r)) - \partial_x f(r,X_r)\Big|^{4}\bigg]^{\frac{1}{4}} \big\| \mathbb{E}_\bullet R^V\big\|_{\alpha + \beta, 4} |t-s|^{\alpha + \beta}.
\end{align*}
For the third term we first expand and we have for every $\lambda\in [0,1]$ that
\begin{align*}
   &R^{\partial_x f(\cdot,X_\cdot + \lambda (X^\epsilon_\cdot - X_\cdot)) - \partial_x f(\cdot,X_\cdot)}_{s,t} \\&=\partial_x f(t,X_t + \lambda (X^\epsilon_t - X_t)) - \partial_x f(t,X_t)-\partial_x f(s,X_s + \lambda (X^\epsilon_s - X_s)) + \partial_x f(s,X_s)\\
    &\qquad-\Big(\partial^2_{xx} f(s,X_s + \lambda  (X^\epsilon_s - X_s)) \Big(f_s({X}_s) + \lambda \big(f_s({X}^\epsilon_s) -f_s(X_s)\big)\Big) +\partial_x f'(s,X_s + \lambda  (X^\epsilon_s - X_s))\\
    &\qquad\qquad -\partial_{xx}^2f(s,X_s) f(s,X_s) - \partial_xf'(s,X_s) \Big)\eta_{s,t}\\
    &= -\lambda\bigg[ \partial^2_{xx} f(s,X_s + \lambda  (X^\epsilon_s - X_s))  - \partial^2_{xx} f(t,X_t + \lambda  (X^\epsilon_s - X_s)) \bigg] \big(f(s,{X}^\epsilon_s) -f(s,X_s)\big) \eta_{s,t}\\  
    &+\bigg[\partial_x f(t,X_t + \lambda (X^\epsilon_t - X_t)) -  \partial_x f(t,X_t + \lambda (X^\epsilon_s - X_s)) - \lambda \partial^2_{xx} f(t,X_t + \lambda  (X^\epsilon_s - X_s)) \big(f_s({X}^\epsilon_s) -f_s(X_s)\big) \eta_{s,t} \bigg]\\
    &+\bigg[\partial_x f(t,X_t + \lambda (X^\epsilon_s - X_s)) - \partial_x f(t,X_t)-\partial_x f(s,X_s + \lambda (X^\epsilon_s - X_s)) + \partial_x f(s,X_s)\\
    &\qquad-\Big(\Big[\partial^2_{xx} f(s,X_s + \lambda  (X^\epsilon_s - X_s))-\partial_{xx}^2f(s,X_s)\Big] f(s,X_s)+\partial_x f'(s,X_s + \lambda  (X^\epsilon_s - X_s)) -\partial_x f'(s,X_s)\Big)\eta_{s,t}\bigg].
\end{align*}
Next, we have
\begin{align*}
    &\partial_x f(t,X_t + \lambda (X^\epsilon_t - X_t)) -  \partial_x f(t,X_t + \lambda (X^\epsilon_s - X_s)) - \lambda \partial^2_{xx} f(t,X_t + \lambda  (X^\epsilon_s - X_s)) \big(f(s,{X}^\epsilon_s) - f(s,X_s)\big) \eta_{s,t}\\
    &= \frac{\lambda^2}{2}\int_0^1 (1-\vartheta) \partial^3_{xxx}f\big(t,X_t  + \lambda (X^\epsilon_s - X_s) + \vartheta \lambda \delta(X^\epsilon_\cdot - X_\cdot)_{s,t}\big) \ud \vartheta \big(\delta(X^\epsilon_\cdot - X_\cdot)_{s,t}\big)^{\otimes 2} \\
    &\qquad+ \lambda \partial^2_{xx} f(t,X_t + \lambda  (X^\epsilon_s - X_s)) \Big[\delta(X^\epsilon_\cdot - X_\cdot)_{s,t} - \big(f(s,{X}^\epsilon_s) - f(s,X_s)\big) \eta_{s,t}\Big]\\
     &= \frac{\lambda^2}{2}\int_0^1 (1-\vartheta) \partial^3_{xxx}f\big(t,X_t + \lambda (X^\epsilon_s - X_s) + \vartheta \lambda \delta(X^\epsilon_\cdot - X_\cdot)_{s,t}\big) \ud \vartheta \big(\delta(X^\epsilon_\cdot - X_\cdot)_{s,t}\big)^{\otimes 2} \\
     &\qquad+ \lambda \partial^2_{xx} f(s,X_s + \lambda  (X^\epsilon_s - X_s)) \Big[\delta(X^\epsilon_\cdot - X_\cdot)_{s,t} - \big(f(s,{X}^\epsilon_s) - f(s,X_s)\big) \eta_{s,t}\Big]\\
     &\qquad+ \lambda \Big(\partial^2_{xx} f(t,X_t + \lambda  (X^\epsilon_s - X_s)) - \partial^2_{xx} f(s,X_s + \lambda  (X^\epsilon_s - X_s))\Big) \Big[\delta(X^\epsilon_\cdot - X_\cdot)_{s,t} - \big(f(s,{X}^\epsilon_s) - f(s,X_s)\big) \eta_{s,t}\Big],
\end{align*}
and
\begin{align*}
   &\partial_x f(t,X_t + \lambda (X^\epsilon_s - X_s)) - \partial_x f(t,X_t)-\partial_x f(s,X_s + \lambda (X^\epsilon_s - X_s)) + \partial_x f(s,X_s)\\
    &\qquad-\Big(\Big[\partial^2_{xx} f(s,X_s + \lambda  (X^\epsilon_s - X_s))-\partial_{xx}^2f(s,X_s)\Big] f(s,X_s) + \partial_x f'(s,X_s + \lambda  (X^\epsilon_s - X_s)) -\partial_x f'(s,X_s)\Big)\eta_{s,t}\\
    &=\frac{1}{2} \int_0^1 (1-\vartheta) \Big(\partial^3_{xxx}f(s,X_s +\vartheta \delta X_{s,t} + \lambda (X^\epsilon_s - X_s)) - \partial^3_{xxx}f(s,X_s +\vartheta \delta X_{s,t} )\Big)\ud \vartheta \big(\delta X_{s,t}\big)^{\otimes 2}\\
    &\qquad+ \Big[\partial^2_{xx} f(s,X_s + \lambda  (X^\epsilon_s - X_s))-\partial_{xx}^2f(s,X_s)\Big] \big(\delta X_{s,t} - f(s,X_s)\eta_{s,t}\big)\\
    &\qquad + \Big[\partial_{x} f(t,X_s + \lambda  (X^\epsilon_s - X_s))-\partial_{x} f(t,X_s)\Big] - \Big[\partial_{x} f(s,X_s + \lambda  (X^\epsilon_s - X_s))-\partial_{x}f(s,X_s)\Big]\\
    &\qquad\qquad - \Big(\partial_x f'(s,X_s + \lambda  (X^\epsilon_s - X_s)) -\partial_x f'(s,X_s) \Big) \eta_{s,t}\\
    &\qquad+ \Big[\delta(\partial_x f)_{s,t}(X_t + \lambda  (X^\epsilon_s - X_s)) - \delta(\partial_x f)_{s,t}(X_s + \lambda  (X^\epsilon_s - X_s))\Big] - \Big[\delta(\partial_x f)_{s,t}(X_t) - \delta(\partial_x f)_{s,t}(X_s)\Big]\\
    &=\frac{1}{2} \int_0^1 (1-\vartheta) \Big(\partial^3_{xxx}f(s,X_s +\vartheta \delta X_{s,t} + \lambda (X^\epsilon_s - X_s)) - \partial^3_{xxx}f(s,X_s +\vartheta \delta X_{s,t} )\Big)\ud \vartheta \big(\delta X_{s,t}\big)^{\otimes 2}\\
    &\qquad+ \Big[\partial^2_{xx} f(s,X_s + \lambda  (X^\epsilon_s - X_s))-\partial_{xx}^2f(s,X_s)\Big] \big(\delta X_{s,t} - f(s,X_s)\eta_{s,t}\big)\\
    &\qquad + \lambda \int_0^1 \Big[\partial^2_{xx} f(t,X_s + \vartheta \lambda  (X^\epsilon_s - X_s)) - \partial^2_{xx} f(s,X_s + \vartheta \lambda  (X^\epsilon_s - X_s)) \\
    &\qquad\qquad\qquad - \partial^2_{xx} f'(s,X_s + \vartheta \lambda  (X^\epsilon_s - X_s)) \eta_{s,t} \Big] \ud\vartheta (X^\epsilon_s - X_s)  \\
    &\qquad+ \int_0^1 \Big[ \delta(\partial^2_{xx} f)_{s,t}(X_s + \vartheta(X_t - X_s) + \lambda  (X^\epsilon_s - X_s)) - \delta(\partial^2_{xx} f)_{s,t}(X_s + \vartheta(X_t - X_s)) \Big] \ud\vartheta  (X_t - X_s).
\end{align*}
So, for example we have
\begin{align*}
    &\mathbb{E}\Bigg[\bigg|\mathbb{E}_s \bigg[  \int_0^1 \Big[ \delta(\partial^2_{xx} f)_{s,t}(X_s + \vartheta(X_t - X_s) + \lambda  (X^\epsilon_s - X_s)) - \delta(\partial^2_{xx} f)_{s,t}(X_s + \vartheta(X_t - X_s)) \Big] \ud\vartheta  (X_t - X_s) V_s\bigg]\bigg|^2\Bigg]^\frac{1}{2}\\
    &\leq \mathbb{E}\bigg[\sup_{\lambda\in[0,1]}\sup_{r_1\in[0,T]}\sup_{r_2\in [0,T]}\sup_{r_3\in[0,T]} \Big|\partial^3_{xxt} f(r_1,X_{r_2} + \vartheta (X_{r_3} - X_{r_2}) +\lambda (X^\epsilon_{r_2} - X_{r_2})) - \partial^3_{xxt} f(r_1,\vartheta (X_{r_3} - X_{r_2}))\Big|^{\frac{4+\varepsilon}{\varepsilon}}\bigg]^{\frac{\varepsilon}{4+\varepsilon}}\\ 
    &\qquad\|\delta X\|_{\alpha;4+\varepsilon}  \sup_{r\in[0,T]}\mathbb{E}\big[|V_r|^{4+\varepsilon}\big]^{\frac{1}{4+\varepsilon}} |t-s|^{1+\alpha},
\end{align*}
for every $\varepsilon >0$.

Thus, by similar calculations as above for every one of the previous terms, we get
\begin{equation}\label{eq 29}
    \lim_{\epsilon\rightarrow 0^+} \Big\|\mathbb{E}_\bullet R^{\int_0^1 (F^\epsilon)'_\cdot(\lambda)\ud\lambda}\Big\|_{\beta + \beta',2}= 0.
\end{equation}

Finally, we estimate the term $\Big\|\mathbb{E}_\bullet \delta\int_0^1 (F^\epsilon)''_\cdot(\lambda)\ud\lambda\Big\|_{\beta',2}$. We have
\begin{align*}
&\delta\bigg(\int_0^1 (F^\epsilon)''_\cdot(\lambda)\ud\lambda\bigg)_{s,t}=\int_0^1\bigg[\Big[\partial^2_{xx} f(t,X_t + \lambda (X^\epsilon_t - X_t))\big(f_t({X}_t) + \lambda [f_t({X}^\epsilon_t) -f_t(X_t)]\big) \\&\qquad\qquad+\partial_x f'(t,X_t + \lambda  (X^\epsilon_t - X_t)) 
 -\partial_{xx}^2f(t,X_t) f(t,X_t) - \partial_xf'(t,X_t)\Big] V_t\\&\qquad\qquad\qquad+ \Big[\partial_x f(t,X_t + \lambda (X^\epsilon_t - X_t))\partial_x f(t,X_t) - \big(\partial_x f(t,X_t)\big)^2 \Big] V_t\\
 &\qquad-\Big[\partial^2_{xx} f(s,X_s + \lambda  (X^\epsilon_s - X_s))  \Big(f(s,{X}_s) + \lambda \big(f(s,{X}^\epsilon_s) - f(s,X_s)\big)\Big) \\&\qquad\qquad+\partial_x f'(s,X_s + \lambda  (X^\epsilon_s - X_s)) 
 -\partial_{xx}^2f(s,X_s) f(s,X_s) - \partial_xf'(s,X_s)\Big] V_s\\&\qquad\qquad\qquad- \Big[\partial_x f(s,X_s + \lambda (X^\epsilon_s - X_s))\partial_x f(s,X_s) - \big(\partial_x f(s,X_s)\big)^2 \Big] V_s \bigg] \ud \lambda.
\end{align*}
Hence, it is enough to study the difference
\begin{align*}
&\int_0^1\bigg[\Big[\partial^2_{xx} f(t,X_t + \lambda  (X^\epsilon_t - X_t))  \big(f(t,{X}_t) + \lambda [f(t,{X}^\epsilon_t) - f(t,X_t)]\big) \\&\qquad\qquad + \partial_x f'(t,X_t + \lambda  (X^\epsilon_t - X_t)) 
 -\partial_{xx}^2f(t,X_t) f(t,X_t) - \partial_xf'(t,X_t)\Big] \\&\qquad\qquad\qquad+ \Big[\partial_x f(t,X_t + \lambda (X^\epsilon_t - X_t))\partial_x f(t,X_t) - \partial_x f(t,X_t)^2 \Big] \\
 &\qquad-\Big[\partial^2_{xx} f(s,X_s + \lambda  (X^\epsilon_s - X_s))  \big(f(s,{X}_s) + \lambda [f(s,{X}^\epsilon_s) - f(s,X_s)]\big) \\&\qquad\qquad+\partial_x f'(s,X_s + \lambda  (X^\epsilon_s - X_s)) 
 -\partial_{xx}^2f(s,X_s) f(s,X_s) - \partial_xf'(s,X_s)\Big] \\&\qquad\qquad\qquad- \Big[\partial_x f(s,X_s + \lambda (X^\epsilon_s - X_s))\partial_x f(s,X_s) - \partial_x f(s,X_s)^2 \Big] \bigg] V_s \ud \lambda.
\end{align*}
We have
\begin{align*}
  &\Big[ \partial^2_{xx} f(t,X_t + \lambda  (X^\epsilon_t - X_t))  \big(f(t,{X}_t) + \lambda [f(t,{X}^\epsilon_t) - f(t,X_t)]\big) - \partial_{xx}^2f(t,X_t) f(t,X_t) \Big]\\
  &\qquad - \Big[ \partial^2_{xx} f(s,X_s + \lambda  (X^\epsilon_s - X_s))  \big(f(s,{X}_s) + \lambda [f(s,{X}^\epsilon_s) - f(s,X_s)]\big) - \partial_{xx}^2f(s,X_s) f(s,X_s) \Big]\\
  &= \Big[ \partial^2_{xx} f(t,X_t + \lambda  (X^\epsilon_t - X_t))  - \partial_{xx}^2f(t,X_t) \Big] f(t,X_t) - \Big[ \partial^2_{xx} f(s,X_s + \lambda  (X^\epsilon_s - X_s))  - \partial_{xx}^2f(s,X_s) \Big]f(s,X_s)\\
  &\qquad + \lambda \Big[ \partial^2_{xx} f(t,X_t + \lambda  (X^\epsilon_t - X_t)) \big(f_t({X}^\epsilon_t) - f_t(X_t)\big) -\partial^2_{xx} f(s,X_s + \lambda  (X^\epsilon_s - X_s)) \big(f_s({X}^\epsilon_s) - f_s(X_s)\big) \Big]\\
   &= \Big[ \partial^2_{xx} f(s,X_s + \lambda  (X^\epsilon_s - X_s))  - \partial_{xx}^2f(s,X_s) \Big] \big(f(t,X_t) - f(s,X_s)\big)\\
   &\qquad+ \lambda \Big[ \int_0^1\Big(\partial^3_{xxx}f(t,X_t + \vartheta \lambda  (X^\epsilon_t - X_t)) - \partial^3_{xxx}f(s,X_s + \vartheta \lambda  (X^\epsilon_s - X_s)) \Big) \ud\vartheta \big( X^\epsilon_t - X_t\big)\Big] f(t,X_t)\\
   &\qquad+ \lambda \Big[ \int_0^1 \partial^3_{xxx} f(s,X_s + \vartheta \lambda  (X^\epsilon_s - X_s))\ud \vartheta (\delta X^\epsilon_{s,t} - \delta X_{s,t})  \Big] f(t,X_t)\\
   &\qquad+ \lambda \Big[\partial^2_{xx} f(t,X_t + \lambda  (X^\epsilon_t - X_t)) - \partial^2_{xx} f(s,X_s + \lambda  (X^\epsilon_s - X_s))\Big]  \big(f_t({X}^\epsilon_t) - f_t(X_t)\big)\\
  &\qquad+ \lambda \Big[\partial^2_{xx} f(s,X_s + \lambda  (X^\epsilon_s - X_s)) \int_0^1 \partial_x f(s,X_s+\vartheta(X_t-X_s))\ud \vartheta \big(\delta X^\epsilon_{s,t} - \delta X_{s,t}\big) \Big]\\
  &\qquad+ \lambda \Big[\partial^2_{xx} f(s,X_s + \lambda  (X^\epsilon_s - X_s))  \int_0^1 \Big[\partial_x f(s,X^\epsilon_s+\vartheta(X^\epsilon_t-X^\epsilon_s)) - \partial_x f_s(X_s+\vartheta(X_t-X_s))\Big]\ud \vartheta \hspace{.1cm}\delta X^\epsilon_{s,t} \Big]\\
  &\qquad+ \lambda \Big[\partial^2_{xx} f(s,X_s + \lambda  (X^\epsilon_s - X_s)) \int_s^t \Big( \partial_t f(\vartheta,X^\epsilon_t) - \partial_t f(\vartheta,X_t)\Big)\ud\vartheta \Big],
\end{align*}
and
\begin{align*}
    &\Big[\partial_x f'(t,X_t + \lambda  (X^\epsilon_t - X_t)) - \partial_x f'(t,X_t ) \Big] - \Big[\partial_x f'(s,X_s + \lambda  (X^\epsilon_s - X_s)) - \partial_x f'(s,X_s)\Big]\\
    &= \int_s^t \Big[\partial^2_{xt} f'(\vartheta,X_t + \lambda  (X^\epsilon_t - X_t)) - \partial^2_{xt} f'(\vartheta,X_t) \Big] \ud\vartheta\\
    &\qquad + \Big[\partial_x f'(s,X_t + \lambda  (X^\epsilon_t - X_t)) - \partial_x f'(s,X_s + \lambda  (X^\epsilon_s - X_s)) \Big] - \Big[\partial_x f'(s,X_t)- \partial_x f'(s,X_s)\Big]\\
    &= \int_s^t \Big[\partial^2_{xt} f'(\vartheta,X_t + \lambda  (X^\epsilon_t - X_t)) - \partial^2_{xt} f'(\vartheta,X_t) \Big] \ud\vartheta\\
    &\qquad + \int_0^1 \Big[ \partial^2_{xx} f'\big(s,X_s + \lambda  (X^\epsilon_s - X_s) + \vartheta\delta(X_\cdot + \lambda  (X^\epsilon_\cdot - X_\cdot))_{s,t}\big) - \partial^2_{xx}f'(s,X_s +\vartheta \delta X_{s,t})\Big]\ud\vartheta \hspace{.1cm} \delta X_{s,t}\\
&\qquad+ \lambda  \int_0^1 \partial^2_{xx} f'\big(s,X_s + \lambda  (X^\epsilon_s - X_s) + \vartheta\delta(X_\cdot + \lambda  (X^\epsilon_\cdot - X_\cdot))_{s,t}\big) \ud\vartheta \hspace{.1cm} \delta( (X^\epsilon_\cdot - X_\cdot))_{s,t}.
\end{align*}
Furthermore,
\begin{align*}
    &\Big[\partial_x f(t,X_t + \lambda (X^\epsilon_t - X_t))\partial_x f(t,X_t) - \big(\partial_x f(t,X_t)\big)^2 \Big] - \Big[\partial_x f(s,X_s + \lambda (X^\epsilon_s - X_s))\partial_x f(s,X_s) - \big(\partial_x f(s,X_s)\big)^2 \Big]\\
    &= \Big[\partial_x f(t,X_t + \lambda (X^\epsilon_t - X_t)) - \partial_x f(t,X_t) \Big] \big(\partial_x f(t,X_t) - \partial_x f(s,X_s) \big)\\
    &\qquad + \bigg[\Big[\partial_x f(t,X_t + \lambda (X^\epsilon_t - X_t)) - \partial_x f(t,X_t) \Big] - \Big[\partial_x f(s,X_s + \lambda (X^\epsilon_s - X_s)) - \partial_x f(s,X_s) \Big]\bigg] \partial_x f(s,X_s)\\
    &= \Big[\partial_x f(t,X_t + \lambda (X^\epsilon_t - X_t)) - \partial_x f(t,X_t) \Big] \big(\partial_x f(t,X_t) - \partial_x f(s,X_s) \big)\\
    &\qquad+ \int_s^t \Big[ \partial^2_{xt} f(\vartheta,X_t + \lambda (X^\epsilon_t - X_t)) - \partial^2_{xt} f(\vartheta,X_t)\Big]\ud\vartheta \hspace{.1cm}f(s,X_s) \\
    &\qquad + \lambda\bigg[\int_0^1 \Big[\partial^2_{xx} f(s,X_t + \vartheta \lambda (X^\epsilon_t - X_t)) - \partial^2_{xx} f(s,X_s + \vartheta \lambda (X^\epsilon_s - X_s)) \Big] \ud\vartheta \big(X^\epsilon_t - X_t\big) \bigg] \hspace{.1cm}\partial_x f(s,X_s)\\
    &\qquad + \lambda \int_0^1 \partial^2_{xx} f(s,X_s + \vartheta \lambda (X^\epsilon_s - X_s)) \ud\vartheta \big(\delta X^\epsilon_{s,t} - \delta X_{s,t}\big) \hspace{.1cm}\partial_x f(s,X_s).
\end{align*}
Using H\"older's inequality per term in the above three expansions, with the $\sup$ operator whenever needs be, and combine them we get
\begin{equation}\label{eq 30}
    \lim_{\epsilon\rightarrow 0^+} \Big\|\mathbb{E}_\bullet \delta\int_0^1 (F^\epsilon)''_\cdot(\lambda)\ud\lambda\Big\|_{\beta',2} = 0.
\end{equation}
From \eqref{eq 12}--
\eqref{eq 30} and the stochastic sewing estimates of Proposition \ref{prop:roughinteg} we have
\begin{equation}\label{eq 31}
    \lim_{\varepsilon\rightarrow0^+} \big\|\delta F^\epsilon\big\|_{\alpha;2} + \big\|\mathbb{E}_\bullet R^{F^{\epsilon}}\big\|_{\alpha+\beta;2}= 0.
\end{equation}
Thus, from \eqref{eq 28} and \eqref{eq 31} we get \eqref{eq 24}, and the proof is complete.
\end{proof}
From \eqref{eq. 053} and Lemma \ref{lem:firstoderleftover1} we conclude that for $k\in(4,\infty)$, under Assumptions \ref{A1}--
\ref{A3}, $\mathcal{J}$ is indeed Gateaux differentiable for every $v \in \mathcal{A}$ and
\begin{equation}\label{eq 32}
    \frac{\ud}{\ud\epsilon}\mathcal{J}( {u} + \epsilon {v})|_{\epsilon = 0}
    = \mathbb{E}\Bigg[\int_0^T \big(\partial_x h(s,X_s,u_s)V_s + \partial_u h(s,X_s,u_s) v_s\big) \ud s + \partial_x g(X_T) V_T\Bigg].
\end{equation}

\subsubsection{Extracting the SMP}\label{subsubsub 4.1.3}
First, we need to strengthen our assumptions with the following.

\begin{assumptions}\label{assu: lin PMP}
Let $\alpha, \beta', \beta'' \in (0,1), k, m, p \in (2,\infty)$ such that 
 $\frac{1}{3}< \frac{1}{p} \leq \beta''\leq \beta'\leq\alpha < \frac{1}{2}$.
 
    \begin{itemize}
    \item[(i)] It holds 
    $$
\frac{1}{k} + \frac{1}{m} \leq \frac{1}{2}.
    $$
    \item[(ii)] It holds 
 $$\alpha+\beta'+ \beta'' \wedge \frac{1}{m}>1.$$
    
    \item[(iii)] It holds
    $$
      m \leq p.
    $$
    \end{itemize}
\end{assumptions}
Note that by combining assumptions (i) and (iii) of Assumptions \ref{assu: lin PMP} we get that necessarily $k \in (6,\infty)$.

Next, we define the 
Hamiltonian $H:\Omega\times[0,T]\times \mathbb{R}^d \times \mathbb{R}^d \times \mathbb{R}^{d\times w} \times \mathbb{R}^\kappa \rightarrow \mathbb{R}$ as
\begin{equation}\label{eq Ham_lin}
H(\omega,t,x,y,z,u):= y^\top b(\omega,t,x,u) + \text{Tr}\big(z^\top\sigma(\omega,t,x,u)\big) + h(\omega,t,x,u).
\end{equation}
We consider the first order adjoint processes $((Y,Y'),Z)$ which satisfy the following equation
\begin{equation}
    Y_t= \partial_x g(X_T)^\top + \int_t^T \partial_x H(s,X_s, Y_s, Z_s,u_s)^\top \ud s+ \int_t^T \partial_x f(s,X_s)^\top Y_s \ud \boldsymbol{\eta}_s - \int_t^T Z_s \ud W_s.
\end{equation}
According to Assumption \ref{assu:bsde}, the above rough BSDE is well-posed by Theorem \ref{ex_uni_lin_bsde} with 
\begin{align*}
    \big((Y,Y'),Z\big) \in  \mathbf{D}^{(p,p)\text{-var}}_{\eta}\mathbb{L}^m([0,T];\mathbb{R}^{d})
 \times
 \mathbb H^m ([0,T]; \mathbb R^{d\times w}).
\end{align*}
Moreover, we also see that
\begin{align*}
    Y_t = Y_0  - \int_0^t \partial_x H(s,X_s, Y_s, Z_s,u_s)^\top \ud s - \int_0^t \partial_x f(s,X_s)^\top Y_s \ud \boldsymbol{\eta}_s + \int_0^t Z_s \ud W_s.
\end{align*}

At this point we could apply directly the rough product formula Theorem \ref{rough_product_formula}, and noting (ii) of Remark \ref{rmk 3.21}, to $V^\top Y$ or ($Y^\top V$) to derive equation \eqref{eq. 65}. But, just for demonstration purposes we also present the following method.

Let us introduce the solution of the linear RSDE
\begin{align}\label{eq:KkernelRDE}
 U^{\boldsymbol{\eta}}_{t\leftarrow s}&=\text{Id}_{d\times d} + \int_s^t \partial_x f(r, X_r) U^{\boldsymbol{\eta}}_{r\leftarrow s}\ud\boldsymbol{\eta}_\mathcal{r}.
\end{align}
The matrix $U^{\boldsymbol{\eta}}_{t\leftarrow s}$ is invertible for every $t\in [0,T]$, this follows by applying Theorem \ref{rough_product_formula}, while noting (ii) of Remark \ref{rmk 3.21}. Its inverse $(U^{\boldsymbol{\eta}}_{t\leftarrow s})^{-1} =:U^{\boldsymbol{\eta}}_{s\leftarrow t}$ 
is the solution of
\begin{equation*}
     U^{\boldsymbol{\eta}}_{s\leftarrow t}=\text{Id}_{d\times d} - \int_s^t U^{\boldsymbol{\eta}}_{s\leftarrow r} \partial_x f(r, X_r) \ud\boldsymbol{\eta}_r.
\end{equation*} 
\begin{rmk}\label{rmk 4.4}
Note that, for every $s,t\in [0,T]$ with $s\leq t$ we have 
\begin{align*}
 (U^{\boldsymbol{\eta}}_{t\leftarrow s})^\top&=\text{Id}_{d\times d} + \int_s^t (U^{\boldsymbol{\eta}}_{r\leftarrow s})^\top \partial_x f(r, X_r)^{\top}\ud\boldsymbol{\eta}_r,\\
    (U^{\boldsymbol{\eta}}_{s\leftarrow t})^\top&=\text{Id}_{d\times d} - \int_s^t \partial_x f(r, X_r)^{\top} (U^{\boldsymbol{\eta}}_{s\leftarrow r})^\top\ud\boldsymbol{\eta}_r.
\end{align*}
\end{rmk}
Applying the rough product formula Theorem \ref{rough_product_formula}, and noting (ii) of Remark \ref{rmk 3.21}, we have 
\begin{multline*}
    U^{\boldsymbol{\eta}}_{0\leftarrow t}  V_t
   = \int_0^t U^{\boldsymbol{\eta}}_{0\leftarrow s} \big(\partial_x b(s,X_s,u_s) V_s + \partial_u b(s,X_s,u_s) v_s\big)  \ud s \\+  \int_0^t U^{\boldsymbol{\eta}}_{0\leftarrow s} \big(\partial_x \sigma(s,X_s,u_s) V_s + \partial_u \sigma(s,X_s,u_s) v_s\big) \ud W_s
\end{multline*}
and
\begin{equation*}
   (U^{\boldsymbol{\eta}}_{t\leftarrow 0})^\top Y_t = Y_0 - \int_0^t (U^{\boldsymbol{\eta}}_{s\leftarrow 0})^\top \partial_x H(s,X_s, Y_s, Z_s,u_s)^\top \ud s + \int_0^t (U^{\boldsymbol{\eta}}_{s\leftarrow 0})^\top Z_s \ud W_s.
\end{equation*}
The integration-by-parts formula for semimartingales then gives us
\begin{align}
    Y_T^\top V_T &= \big((U^{\boldsymbol{\eta}}_{T\leftarrow 0})^\top Y_T\big)^\top  \big(U^{\boldsymbol{\eta}}_{0\leftarrow T}  V_T\big) \nonumber\\
    &= \int_0^T \big((U^{\boldsymbol{\eta}}_{s\leftarrow 0})^\top Y_s\big)^\top \ud\big(U^{\boldsymbol{\eta}}_{0\leftarrow s}  V_s\big) + \int_0^T \big(U^{\boldsymbol{\eta}}_{0\leftarrow s}  V_s\big)^\top \ud\big((U^{\boldsymbol{\eta}}_{s\leftarrow 0})^\top Y_s\big) + \Big[(U^{\boldsymbol{\eta}}_{\cdot\leftarrow 0})^\top Y_\cdot,U^{\boldsymbol{\eta}}_{0\leftarrow\cdot}  V_\cdot\Big]_T \nonumber\\
    &=\int_0^T Y_s^\top \big(\partial_x b(s,X_s,u_s) V_s + \partial_u b(s,X_s,u_s) v_s\big)\ud s \nonumber\\
    &\qquad + \int_0^T  Y_s^\top \big(\partial_x \sigma(s,X_s,u_s) V_s + \partial_u \sigma(s,X_s,u_s) v_s\big)\ud W_s \nonumber\\
    &\qquad- \int_0^T V_s^\top \partial_x H(s,X_s, Y_s, Z_s,u_s)^\top \ud s + \int_0^T V_s^\top Z_s\ud W_s \nonumber\\
    &\qquad + \int_0^T \text{Tr}\Big(\big(\partial_x \sigma(s,X_s,u_s) V_s + \partial_u \sigma(s,X_s,u_s) v_s\big)^\top Z_s\Big) \ud s.\label{eq. 65}
\end{align}

Note that $V\in \mathbb{S}^k$ by Kolmogorov's continuity Theorem, see for example \cite[Theorem 3.1]{friz2020course}, $Z \in \mathbb{H}^m$ and
$$
 \sup_{t\in [0,T]} \|Y_t\|_m + \sup_{t\in [0,T]}\|v_t\|_k < \infty.
 $$
So, by H\"older's inequality and Tonelli's Theorem we have\footnote{By (i) of Assumptions \ref{assu: lin PMP} we have $2\leq \frac{km}{k+m}$.}
\begin{align*}
    \mathbb{E}\bigg[\bigg(\int_0^T |Y_t|^2 |v_t|^2 \ud t\bigg)^{\frac{1}{2}}\bigg] &\leq \mathbb{E}\bigg[\int_0^T |Y_t|^2 |v_t|^2 \ud t\bigg]^{\frac{1}{2}} \leq \bigg(\int_0^T \| |Y_t| |v_t|\|_2^2\ud t \bigg)^{\frac{1}{2}}\\
    &\leq \bigg(\int_0^T \| |Y_t| |v_t|\|_\frac{k m}{k+m}^2\ud t \bigg)^{\frac{1}{2}} \leq \bigg(\int_0^T \|Y_t\|_m^2 \|v_t\|_k^2\ud t \bigg)^{\frac{1}{2}}\\
    &\leq T^{\frac{1}{2}} \sup_{t\in [0,T]} \|Y_t\|_m \sup_{t\in [0,T]}\|v_t\|_k < \infty.
\end{align*}
Thus, the above local martingale terms are actually martingales and by taking expectations we get
\begin{equation}\label{eq 35}
    \mathbb{E}\big[Y_T^\top V_T\big] = \mathbb{E}\bigg[\int_0^T \Big( Y_s^\top \partial_u b(s,X_s,u_s) v_s + \text{Tr}\big(Z_s^\top \partial_u \sigma(s,X_s,u_s) v_s \big) - \partial_x h(s,X_s,u_s) V_s\Big)\ud s\bigg].
\end{equation}

\begin{prop}\label{Der_formula_lin}
Let Assumptions \ref{A1}--
\ref{A3}, \ref{assu: lin PMP}  hold true, $H$ given by \eqref{eq Ham_lin}, and $(X,X'), (Y,Y'), Z$ defined as above. Then, for every $(v_t)_{t\in[0,T]} \in \mathcal{A}$ we have
    \begin{equation*}
         \frac{\ud}{\ud\epsilon}\mathcal{J}( {u} + \epsilon {v})\Big|_{\epsilon = 0} = \E\bigg[\int_0^T \partial_u H(t, X_t, Y_t, Z_t, u_t) \cdot v_t \ud t \bigg].
    \end{equation*}
\end{prop}

\begin{proof}
Immediate by combining \eqref{eq 32} and \eqref{eq 35} and noting that $Y_T^\top = \partial_x g(X_T)$.    
\end{proof}
Finally, we are now ready to give the maximum principle.

\begin{thm}[SMP for rough stochastic control]\label{thm:roughpmp_lin}
Let Assumptions \ref{A1}--
\ref{A3}, \ref{assu: lin PMP}  hold true, $H$ given by \eqref{eq Ham_lin}, and $(X,X'), (Y,Y'), Z$ defined as above. If the admissible control $(u_t)_{t\in[0,T]}\in \mathcal{A}$ is optimal for $\mathcal{J}$, then we have
\begin{enumerate}
    \item[(i)] 
  $
        \partial_u H(t, X_t, Y_t, Z_t, u_t) = 0,\hspace{.2cm} \mathbb{P}\otimes\text{Leb}(\ud t)-a.s.
$

    \item[(ii)] If additionally $H$ is convex in $u \in \mathbb{R}^\kappa$, then for every $u\in \mathbb{R}^\kappa$ we have
    \begin{equation}\label{ineq. 68}
    H(t,X_t,Y_t,Z_t,u_t) \leq H(t,X_t,Y_t,Z_t,u),\hspace{.2cm} \mathbb{P}\otimes\text{Leb}(\ud t)-a.s.
\end{equation}
\end{enumerate}
\end{thm}
\begin{proof}
Because $ \partial_u H(t, X_t, Y_t, Z_t, u_t)$ is progressively measurable $(i)$ is immediate from Proposition \ref{Der_formula_lin}, see for example \cite[Proposition 6.15]{carmona2018probabilistic}.

For $(ii)$ one can follow again the standard calculations using Proposition \ref{Der_formula_lin} and the aforementioned convexity property of $H$, see for example \cite[Theorem 6.14]{carmona2018probabilistic}.   
\end{proof}

\subsubsection{Sufficiency conditions}

\begin{thm}\label{thm 4.8}
Let Assumptions \ref{A1}--
\ref{A3}, \ref{assu: lin PMP} hold true, $H$ given by \eqref{eq Ham_lin}, and $(X,X'), (Y,Y'), Z$ defined as above. If 
the function $g$ is convex in $x\in\mathbb{R}^d$, 
the function ${H}(t,\cdot,{Y}_t, {Z}_t,\cdot)$ is convex in $(x,u)\in \mathbb{R}^d\times \mathbb{R}^k$, and for the admissible control $(u_t)_{t\in[0,T]}\in \mathcal{A}$, for every 
$u\in \mathbb{R}^\kappa$ holds
\begin{equation*}
      H(t,X_t,Y_t,Z_t,u_t) \leq H(t,X_t,Y_t,Z_t,u),\hspace{.2cm}\mathbb{P}\otimes\text{Leb}(\ud t)-a.s.,
\end{equation*}
then for every admissible pair $((\widetilde{X},\widetilde{X}'),(v_t)_{t\in[0,T]})$ we have 
\begin{align*}
    \mathcal{J}(v) - \mathbb{E}\bigg[\int_0^T  \big({Y}_s^\top f(s,\widetilde{X}_s) - {Y}_s^\top f(s,{X}_s) - {Y}_s^\top \partial_xf(s,{X}_s) [\widetilde{X}_s - X_s]\big)
     \ud\boldsymbol{\eta}_s \bigg] \geq \mathcal{J}({u}).
\end{align*}
\end{thm}

\begin{proof}
Let $(v_t)_{t\in[0,T]} \in \mathcal{A}$, and $\widetilde{X}$ its corresponding control state. By the definitions of $\mathcal{J}$ and $H$ we have
\begin{align}\label{eq. 069}
    \mathcal{J}(u) - \mathcal{J}(v) &= \mathbb{E}\bigg[g(X_T) - g(\widetilde{X}_T) + \int_0^T \big(H(s,X_s,Y_s,Z_s,u_s) - H(s,\widetilde{X}_s,Y_s,Z_s,v_s) \big) \ud s \bigg] \nonumber\\
    &\hspace{.2cm}- \mathbb{E}\bigg[\int_0^T \Big(Y_s^\top [b(s,X_s,u_s) - b(s,\widetilde{X}_s,v_s)] + \text{Tr}\big(Z^\top_s[\sigma(s,X_s,u_s) - \sigma(s,\widetilde{X}_s,v_s)]\big)\Big) \ud s \bigg]. 
\end{align}
From the convexity of $g$ with respect to $x\in \mathbb{R}^d$ we get
\begin{equation}\label{eq. 69}
    \mathbb{E}[g(X_T) - g(\widetilde{X}_T)] \leq \mathbb{E}[\partial_xg(X_T) (X_T-\widetilde{X}_T)] = \mathbb{E}[Y^\top_T(X_T-\widetilde{X}_T)].
\end{equation}
Additionally, from \eqref{ineq. 68} we immediately have
\begin{equation}\label{eq. 70}
  \partial_uH(t,X_t,Y_t,Z_t,u_t) = 0,\hspace{.2cm} \text{Leb}(\ud t)\otimes\mathbb{P}-a.s.  
\end{equation}
So, from the convexity of ${H}(\omega,t,\cdot,{Y}_t,{Z}_t,\cdot)$ with respect to $(x,u)\in \mathbb{R}^d\times \mathbb{R}^k$ and \eqref{eq. 70} we get
\begin{multline}\label{eq. 71}
    \mathbb{E}\bigg[\int_0^T \big(H(s,X_s,Y_s,Z_s,u_s) - H(s,\widetilde{X}_s,Y_s,Z_s,v_s) \big) \ud s \bigg] 
    \leq \mathbb{E}\bigg[\int_0^T \partial_xH(s,X_s,Y_s,Z_s,u_s) (X_s-\widetilde{X}_s) \ud s \bigg].
\end{multline}
Thus, from \eqref{eq. 069}, \eqref{eq. 69} and \eqref{eq. 71} we have
\begin{align}\label{eq. 73}
    \mathcal{J}(u) - \mathcal{J}(v) &\leq \mathbb{E}\bigg[Y^\top_T(X_T-\widetilde{X}_T) + \int_0^T \partial_xH(s,X_s,Y_s,Z_s,u_s) (X_s-\widetilde{X}_s) \ud s \bigg] \nonumber\\
    &\hspace{.2cm}- \mathbb{E}\bigg[\int_0^T \Big(Y_s^\top [b(s,X_s,u_s) - b(s,\widetilde{X}_s,v_s)] + \text{Tr}\big(Z^\top_s[\sigma(s,X_s,u_s) - \sigma(s,\widetilde{X}_s,v_s)]\big)\Big) \ud s \bigg]. 
\end{align}
Moreover, the difference $X-\widetilde{X}$ satisfies 
\begin{align*}
    X_t - \widetilde{X}_t &= \int_0^t [b(s,X_s,u_s) - b(s,\widetilde{X}_s,v_s)]\ud s + \int_0^t [\sigma(s,X_s,u_s) - \sigma(s,\widetilde{X}_s,v_s)] \ud W_s\\
    &\qquad+\int_0^t \big[\partial_xf(s,{X}_s) ({X}_s - \widetilde{X}_s) + \widetilde{\gamma}_s\big] \ud \boldsymbol{\eta}_s,
\end{align*}
where
\begin{align*}
      \widetilde{\gamma}_t &:= - \partial_xf(t,{X}_t) (X_t - \widetilde{X}_t) + f(t,{X}_t) - f(t,\widetilde{X}_t),\\
     (\widetilde{\gamma})'_t &:= - \partial_{xx}^2f(t,{X}_t) [f(t,{X}_t)\otimes (X_t - \widetilde{X}_t)] - \partial_{x}f'(t,{X}_t)(X_t - \widetilde{X}_t) - \partial_{x}f(t,{X}_t) \big(f(t,{X}_t) - f(t,\widetilde{X}_t)\big)\\
     &\qquad+ \partial_{x}f(t,{X}_t)f(t,{X}_t) +f'(t,{X}_t) - \partial_{x}f(t,\bar{X}_t)f(t,\widetilde{X}_t) - f'(t,\widetilde{X}_t).
\end{align*}
Applying the rough product formula Theorem \ref{rough_product_formula}, while noting (i) and (ii) of Remark \ref{rmk 3.21}, and then taking expectations we see as in \ref{subsubsub 4.1.3} that
\begin{align*}
    &\mathbb{E}[Y^\top_T(X_T-\widetilde{X}_T)] = \mathbb{E}\bigg[ \int_0^T \Big(- \partial_xH(s,X_s,Y_s,Z_s,u_s) (X_s-\widetilde{X}_s) \\ 
    &+ Y_s^\top [b(s,X_s,u_s) - b(s,\widetilde{X}_s,v_s)] + \text{Tr}\big(Z^\top_s[\sigma(s,X_s,u_s) - \sigma(s,\widetilde{X}_s,v_s)]\big) \Big) \ud s\bigg] + \mathbb{E}\bigg[\int_0^T\bar{Y}_s^\top \widetilde{\gamma}_s
     \ud\boldsymbol{\eta}_s \bigg].
\end{align*}
Hence, from \eqref{eq. 73} we get
\begin{equation*}
     \mathcal{J}(u) - \mathcal{J}(v) \leq \mathbb{E}\bigg[\int_0^T\bar{Y}_s^\top \widetilde{\gamma}_s
     \ud\boldsymbol{\eta}_s \bigg].
\end{equation*}
\end{proof}

\begin{coroll}\label{coroll 4.9}
Let the Assumptions of Theorem \ref{thm 4.8} hold true, and $(X,X'),(Y,Y')$ defined as above. If for every 
admissible pair $((\widetilde{X},\widetilde{X}'),(v_t)_{t\in[0,T]})$ holds
\begin{equation}\label{eq. 074}
    \mathbb{E}\bigg[\int_0^T \Big({Y}_s^\top f\big(s,\widetilde{X}(s)\big) - {Y}_s^\top f\big(s,{X}(s)\big) -{Y}_s^\top\partial_xf\big(s,{X}(s)\big) \big[\widetilde{X}(s) - X(s)\big]\Big)
     \ud\boldsymbol{\eta}_s \bigg] \geq 0,
\end{equation}
then we have that $(({X},X'),(u_t)_{t\in[0,T]})$ is an optimal pair.
\end{coroll}

\begin{proof}
Immediate from Theorem \ref{thm 4.8}.
\end{proof}

\begin{rmk}
If $f$ is independent of $x\in\mathbb{R}^d$ then \eqref{eq. 074} is trivially satisfied. More generally, our results clarify why, in the affine case, the standard sufficiency conditions should, in principle, guarantee optimality, as again \eqref{eq. 074} is automatically satisfied, this was indeed shown in \cite{horst2025pontryagin}.
\end{rmk}

\subsection{Variation via the spike perturbation and the rough SMP}\label{subsec: 4.2}
\subsubsection{General setting}
Let $\boldsymbol{\eta} \in \mathscr{C}^{0,\alpha}_g([0,T];\mathbb{R}^\ell), \ell\in\mathbb{N}$, with $\alpha\in(\frac{1}{3},\frac{1}{2})$. In this subsection the following assumptions hold true:

\begin{enumerate}[label=(B\arabic*),ref=(B\arabic*)]

\item\label{P-control}
The action space $(U,\rho)$ is a separable metric space.

\item\label{D-dynamics}
The functions $b: \Omega \times [0,T]\times \mathbb{R}^d\times U\rightarrow \mathbb{R}^d,\sigma: \Omega \times [0,T]\times \mathbb{R}^d\times U\rightarrow \mathbb{R}^{d\times w}$ are 
$ \text{Prog}(\Omega\times[0,T]) \otimes \mathcal{B}(\mathbb{R}^d)\otimes \mathcal{B}(U)/ \mathcal{B}(\mathbb{R}^{d}) \big(\text{resp.}\hspace{.1cm} \mathcal{B}(\mathbb{R}^{d\times w})\big)$-measurable, and two times differentiable in $x\in\mathbb{R}^d$. Also, for every $\phi \in \{b,\sigma, \partial_x b, \partial_x\sigma, \partial^2_{xx} b, \partial^2_{xx}\sigma\}$ we have
$$
\|\sup_{t\in[0,T]} \sup_{x\in\mathbb{R}^d} \sup_{u\in U} |\phi(t,x,u)|\|_\infty <\infty.
$$
Furthermore, there exist a constant $\mathscr{L} > 0$ and a modulus of continuity $\overline{\omega} :[0,\infty) \rightarrow [0,\infty)$ such that for $\phi \in \{b,\sigma\}$ we have for every $(\omega,t)\in \Omega\times [0,T], x_1,x_2 \in \mathbb{R}^d, u_1, u_2 \in U$ that
\begin{align*}
    |\phi(t,x_2,u_2) - \phi(t,x_1,u_1)| +  |\partial_x \phi(t,x_2,u_2) - \partial_x \phi(t,x_1,u_1)| &\leq \mathscr{L} |x_2 - x_1| + \overline{\omega}\big(\rho(u_2,u_1)\big),\\
    |\partial^2_{xx} \phi(t,x_2,u_2) - \partial^2_{xx} \phi(t,x_1,u_1)| &\leq \overline{\omega}\big(|x_2 - x_1| + \rho(u_2,u_1)\big).
\end{align*}

\item\label{P-rough}
Let $k \in [2,\infty)$. The functions $f:\Omega \times [0,T]\times \mathbb{R}^d\rightarrow \mathbb{R}^{d\times \ell}, f': \Omega \times [0,T]\times\mathbb{R}^d \rightarrow L(\mathbb{R}^p,\mathbb{R}^{d\times \ell})$ form a pair $(f,f')\in \mathbf{D}^{\alpha,\alpha}_{\eta}\mathbb{L}^{k,\infty}\mathcal{C}_b^{5}$. Also, $(\partial_x f,\partial_x f')\in \mathbf{D}^{\alpha, \alpha}_{\eta}\mathbb{L}^{k,\infty}\mathcal{C}_b^{4}$, $(\partial_{xx}^2 f,\partial_{xx}^2 f')\in \mathbf{D}^{\alpha, \alpha}_{\eta}\mathbb{L}^{k,\infty}\mathcal{C}_b^{3}$ while $(\partial^3_{xxx} f,\partial^3_{xxx} f')\in \mathbf{D}^{\alpha, \alpha}_{\eta}\mathbb{L}^{k,\infty}\mathcal{C}_b^{2}$ and
$$
\|\sup_{t\in[0,T]}|f(t,\cdot)|_{5}\|_\infty + \|\sup_{t\in[0,T]}|f'(t,\cdot)|_{4}\|_\infty <\infty.
$$

\item\label{D-cost}
Let $m\in[2,\infty)$. The function $h : \Omega\times [0,T] \times \mathbb{R}^d \times U \rightarrow \mathbb{R}$ is $\text{Prog}(\Omega\times[0,T]) \otimes \mathcal{B}(\mathbb{R}^d)\otimes \mathcal{B}(U)/ \mathcal{B}(\mathbb{R})$-measurable, and the function  $g: \Omega\times \mathbb{R}^d \rightarrow \mathbb{R}$ is $\mathcal{F}_T \otimes \mathcal{B}(\mathbb{R}^d)/\mathcal{B}(\mathbb{R})$-measurable. Also, $h,g$ are two times differentiable in $x\in\mathbb{R}^d$ and for every $\phi_1\in\{h, \partial_x h, \partial^2_{xx} h\}$, $\phi_2 \in \{g,\partial_xg,\partial^2_{xx} g\}$ we have
$$
\|\sup_{t\in[0,T]} \sup_{u\in U} |\phi_1(t,0,u)|\|_m + \|\phi_2(0)\|_m <\infty.
$$
Furthermore, there exist a constant $\mathscr{L} > 0$ and a modulus of continuity $\overline{\omega} :[0,\infty) \rightarrow [0,\infty)$ such that for $h,g$ we have for every $(\omega,t) \in \Omega\times [0,T], x_1,x_2 \in \mathbb{R}^d, u_1, u_2 \in U$ that
\begin{align*}
    |h(t,x_2,u_2) - h(t,x_1,u_1)| +  |\partial_x h(t,x_2,u_2) - \partial_x h(t,x_1,u_1)| &\leq \mathscr{L} |x_2 - x_1| + \overline{\omega}\big(\rho(u_2,u_1)\big),\\
    |\partial^2_{xx} h(t,x_2,u_2) - \partial^2_{xx} h(t,x_1,u_1)| &\leq \overline{\omega}\big(|x_2 - x_1| + \rho(u_2,u_1)\big)\\
    \text{and}\hspace{8.8cm}\\
    |g(x_2) - g(x_1)| +  |\partial_x g(x_2) - \partial_x g(x_1)| &\leq \mathscr{L} |x_2 - x_1| ,\\
    |\partial^2_{xx} g(x_2) - \partial^2_{xx} g(x_1)| &\leq \overline{\omega}\big(|x_2 - x_1|\big).
\end{align*}
\end{enumerate}

The space of admissible controls for this subsection is $\mathcal{A}$ such that
\begin{align*}
    &\mathcal{A} := \bigg\{ u:\Omega \times [0,T] \rightarrow U : (u_t)_{t\in [0,T]}\hspace{.1cm}\text{progressively measurable}\bigg\},
\end{align*}
which obviously is closed with respect to the pasting operation introduced below.

Here we are going to use a different method to perturb our controls since linear perturbation is not in general available. Let $((\bar{X},\bar{X}'),(\bar u_t)_{t\in[0,T]})$ be an optimal pair for our problem and $I\subseteq[0,T]$ an interval with length less or equal from $\epsilon$, for every $(u_t)_{t\in[0,T]} \in \mathcal{A}$ we define
\begin{equation}
    u^\epsilon_t:= \left\{
    \begin{array}{ll}
     u_t & t \in I, \\
     \bar{u}_t & {t\in [0,T]\setminus I}. 
    \end{array}
    \right.
\end{equation}
Moreover, for $\xi\in\mathbb{L}^k(\mathcal{F}_0;\mathbb{R}^d)$ we shall denote with $(X^{\boldsymbol{\eta},u^\epsilon},(X^{\boldsymbol{\eta},u^\epsilon})')$ the solution of
\begin{equation*}
\begin{cases}
\ud X^{\boldsymbol{\eta},u^\epsilon}_s
=
b\bigl(s,X^{\boldsymbol{\eta},u^\epsilon}_s,u^\epsilon_s\bigr) \ud s
+
\sigma\bigl(s,X^{\boldsymbol{\eta},u^\epsilon}_s,u^\epsilon_s\bigr)\ud W_s
+
(f,f')\bigl(s,X^{\boldsymbol{\eta},u^\epsilon}_s\bigr) \ud\boldsymbol{\eta}_s,
\qquad s \in [0,T], \\
X^{\boldsymbol{\eta},u^\epsilon}_0
= \xi.
\end{cases}
\end{equation*}

\subsubsection{Main approximation result}
In order to rigorously identify asymptotically the difference $(X^{\boldsymbol{\eta},u^\epsilon} - \bar{X})$ as $\epsilon\rightarrow 0^+$ we focus our attention to the following \emph{variational equations} along the optimal pair $((\bar{X},\bar{X}'),(\bar u_t)_{t\in[0,T]})$
\begin{equation}\label{eq:diffusion-first-variation}
\begin{cases}
 \ud V^\epsilon_s = (\partial_x b(s,\bar{X}_s, \bar{u}_s)V^\epsilon_s) \ud s + (\partial_x \sigma(s,\bar{X}_s, \bar{u}_s)V^\epsilon_s)\ud W_s \\\qquad\qquad + \big((\partial_x f(s, \bar{X}_s), \partial_{xx}^2f(s,\bar{X}_s) f(s,\bar{X}_s) + \partial_x f'(s,\bar{X}_s))V^\epsilon_s\big) \ud \boldsymbol{\eta}_s
  + (\mathds{1}_I(s) \delta\sigma_s) \ud W_s ,  \\ 
 V^\epsilon_0 = 0,    
\end{cases}
\end{equation}
and
\begin{equation}\label{eq:diffusion-second-variation}
\begin{cases}
 \ud \zeta^\epsilon_s = (\partial_x b(s,\bar{X}_s, \bar{u}_s) \zeta^\epsilon_s) \ud s + (\partial_x \sigma(s,\bar{X}_s, \bar{u}_s)\zeta^\epsilon_s)\ud W_s \\
 \qquad\quad+ \big((\partial_x f(s, \bar{X}_s), \partial_{xx}^2f(s,\bar{X}_s) f(s,\bar{X}_s) +\partial_x f'(s,\bar{X}_s))\zeta^\epsilon_s\big) \ud \boldsymbol{\eta}_s\\
 \qquad\quad+ (\frac{1}{2}\partial_{xx}^2 b(s,\bar{X}_s, \bar{u}_s) (V^\epsilon_s)^{\otimes 2} + \mathds{1}_I(s) \delta b_s) \ud s + (\frac{1}{2}\partial_{xx}^2\sigma(s,\bar{X}_s, \bar{u}_s) (V^\epsilon_s)^{\otimes 2} + \mathds{1}_I(s) \delta \sigma_{x,s} V^\epsilon_s) \ud W_s\\
 \qquad\quad+ (\frac{1}{2}\partial_{xx}^2f(s,\bar{X}_s) (V^\epsilon_s)^{\otimes 2},\frac{1}{2}[\partial_{xxx}^3f(s,\bar{X}_s) f(s,\bar{X}_s) + \partial_{xx}^2f'(s,\bar{X}_s)](V^\epsilon_s)^{\otimes 2} \\
 \hspace{8.5cm}+ \partial_{xx}^2f(s,\bar{X}_s)[V^\epsilon_s \otimes (\partial_xf(s,\bar{X}_s) V^\epsilon_s)]) \ud \boldsymbol{\eta}_s\\ 
 \zeta^\epsilon_0 = 0,    
\end{cases}
\end{equation}
where for every $t\in [0,T]$ 
\begin{align*}
    \delta b_t &:= b(t,\bar{X}_t,u_t) - b(t,\bar{X}_t,\bar{u}_t),\\
    \delta \sigma_t &:= \sigma(t,\bar{X}_t,u_t) - \sigma(t,\bar{X}_t,\bar{u}_t),\\
    \delta \sigma_{x,t} &:= \partial_x \sigma(t,\bar{X}_t,u_t) - \partial_x \sigma(t,\bar{X}_t,\bar{u}_t).
\end{align*}

\begin{rmk}
Let $(Y,Y') \in \mathbf{D}_{\eta}^{\alpha,\alpha} \mathbb{L}^{{k}}\big([0,T];\mathbb{R}^d\big)$, then one can easily see that $(Y^{\otimes 2},(Y^{\otimes 2})') \in \mathbf{D}_{\eta}^{\alpha,\alpha} \mathbb{L}^{\frac{k}{2}}\big([0,T];\mathbb{R}^{d\times d}\big)$, where $Y^{\otimes 2} := Y \otimes Y$ and $(Y^{\otimes 2})' := Y' \otimes Y + Y \otimes Y'$.
\end{rmk}

\begin{prop}\label{prop 4.8}
Let $k \in [4,\infty)$ and $(\bar{X},f(\cdot,\bar{X}))\in \mathbf{D}_{\eta}^{\alpha,\alpha} \mathbb{L}^{k,\infty}\big([0,T];\mathbb{R}^d\big)$ the optimal state process, \emph{i.e.} the unique solution under \ref{P-control}--
\ref{P-rough} of the RSDE\footnote{Provided by \cite[Theorem 4.6]{FHL21}, with of course $\bar{X}'=f(\cdot,\bar{X})$.} 
\begin{equation*}
    \ud \bar{X}_s = b(s,\bar{X}_s,\bar{u}_s) \ud s + \sigma(s,\bar{X}_s,\bar{u}_s)\ud W_s + \big(f(s,\bar{X}_s),\partial_xf(s,\bar{X}_s)f(s,\bar{X}_s)+ f'(s,\bar{X}_s)\big) \ud \boldsymbol{\eta}_s.
\end{equation*}
Then, we have that:
\begin{enumerate}
\item[(i)] 
$ \big(\partial^2_{xx}f(\cdot,\bar{X}),\partial^3_{xxx}f(\cdot,\bar{X}) f(\cdot,\bar{X}) + \partial^2_{xx}f'(\cdot,\bar{X})\big)
    \in \mathbf{D}_{\eta}^{\alpha,\alpha} \mathbb{L}^{k,\infty}_{lin}\big([0,T];L(\mathbb{R}^{d\times d},\mathbb{R}^{d\times \ell})\big). $

\item[(ii)] Furthermore, let $V^\epsilon$ be the unique solution of \eqref{eq:diffusion-first-variation}. 
Then, we have that
\begin{gather*}
    \big(\partial_{xx}^2f(\cdot,\bar{X}) (V^\epsilon)^{\otimes 2},[\partial_{xxx}^3f(\cdot,\bar{X}) f(\cdot,\bar{X}) + \partial_{xx}^2f'(\cdot,\bar{X})](V^\epsilon)^{\otimes 2} + 2\partial_{xx}^2f(\cdot,\bar{X})[V^\epsilon\otimes (\partial_xf(\cdot,\bar{X}) V^\epsilon)] \big)\\
    \in \mathbf{D}_{\eta}^{\alpha,\alpha} \mathbb{L}^{\frac{k}{2}}\big([0,T];\mathbb{R}^d\big).
\end{gather*}
\end{enumerate}
\end{prop}

\begin{proof}
For $(i)$ the proof is similar to that of Lemma \ref{lemma 4.1} and thus omitted.

Next for $(ii)$,
we remind that $\big((V^\epsilon_t)^{\otimes 2}\big)'= V^\epsilon_t\otimes(\partial_xf(t,\bar{X}_t) V^\epsilon_t) + (\partial_xf(t,\bar{X}_t) V^\epsilon_t)\otimes V^\epsilon_t$.  
We can immediately see that
\begin{align*}
     \sup_{t\in [0,T]}\big\| \partial_{xx}^2f(t,\bar{X}_t) (V^\epsilon_t)^2\big\|_{\frac{k}{2}} \leq \sup_{t\in[0,T]} \|\partial_{xx}^2f(t,\cdot)\|_{\infty} \sup_{t\in [0,T]} \|V^\epsilon_t\|^2_k 
 \end{align*}
 and 
 \begin{align*}
    \sup_{t\in [0,T]}\big\| \big(\partial_{xx}^2f(t,\bar{X}_t)\big)'(V^\epsilon_t)^{\otimes 2} + \partial_{xx}^2f(t,\bar{X}_t)\big((V^\epsilon_t)^{\otimes 2}\big)'\big\|_{\frac{k}{2}}   &\leq \sup_{t\in[0,T]} \|\partial_{xxx}^3f(t,\cdot) f(t,\cdot) + \partial_{xx}^2 f'(t,\cdot)\|_{\infty} \sup_{t\in [0,T]} \|V^\epsilon_t\|^2_k  \\
    &\hspace{.2cm}+ \sup_{t\in[0,T]} \|\partial_{xx}^2 f(t,\cdot)\|_{\infty} \sup_{t\in [0,T]}
     \|\partial_{x}f(t,\cdot)\|_{\infty} \sup_{t\in [0,T]} \|V^\epsilon_t\|^2_k. 
 \end{align*}   
 Next, for $ 0\leq s \leq t \leq T$ we have
\begin{align*}
    &\big\| \delta\big(\partial_{xx}^2f(\cdot,\bar{X}) (V^\epsilon)^{\otimes 2} \big)_{s,t} \big\|_{\frac{k}{2}} \leq \big\|  \delta\big(\partial_{xx}^2f(\cdot,\bar{X})\big)_{s,t} (V_s^\epsilon)^{\otimes 2} \big\|_{\frac{k}{2}} + \big\| \partial_{xx}^2f(t,\bar{X}_t) \delta (V^\epsilon)^{\otimes 2}_{s,t}\big\|_{\frac{k}{2}}\\
   &\leq \Big(\big\|\partial_{xx}^2f(\cdot,\bar{X})\big\|_{\alpha ; k,\infty} \sup_{r\in [0,T]}\|V^\epsilon_r\|^2_k  + 2\sup_{r\in[0,T]}\|\partial_{xx}^2f(r,\cdot)\|_{\infty} \sup_{r\in [0,T]}\|V^\epsilon_r\|_k \|\delta V^\epsilon \|_{\alpha;k} \Big) |t-s|^\alpha.
\end{align*}
Additionally, with similar calculations as above, we get
\begin{align*}
    &\Big\| \delta\Big(\big(\partial_{xx}^2f(\cdot,\bar{X})\big)'(V^\epsilon)^{\otimes 2} + \partial_{xx}^2f(\cdot,\bar{X})\big((V^\epsilon)^{\otimes 2}\big)' \Big)_{s,t}\Big\|_{\frac{k}{2}} \leq \Big\| \delta\Big(\big(\partial_{xx}^2f(\cdot,\bar{X})\big)'(V^\epsilon)^{\otimes 2}  \Big)_{s,t}\Big\|_{\frac{k}{2}} \\
   &\qquad+ \Big\| \delta\Big( \partial_{xx}^2f(\cdot,\bar{X})\big((V^\epsilon)^{\otimes 2}\big)' \Big)_{s,t}\Big\|_{\frac{k}{2}} \lesssim |t-s|^{\alpha}.
\end{align*}
Finally, because $1 <\frac{k}{k-2}\leq k$,
by the conditional H\"older inequality we get
\begin{align*}
    &\bigg\|\mathbb{E}_s\bigg[ \delta\big(\partial_{xx}^2 f(\cdot,\bar{X}) (V^\epsilon)^{\otimes 2} \big)_{s,t} - \Big(\big(\partial_{xx}^2f(s,\bar{X}_s)\big)'(V_s^\epsilon)^{\otimes 2} + \partial_{xx}^2f(s,\bar{X}_s)\big((V_s^\epsilon)^{\otimes 2}\big)' \Big) \delta \eta_{s,t} \bigg] \bigg\|_{\frac{k}{2}}
    \\
     &= \bigg\|\mathbb{E}_s\bigg[  \Big(\delta\big(\partial_{xx}^2f(\cdot,\bar{X})\big)_{s,t} (V^\epsilon_s)^{\otimes 2} + \partial_{xx}^2f(s,\bar{X}_s) \delta (V^\epsilon)^{\otimes 2}_{s,t} + \delta\big(\partial_{xx}^2f(\cdot,\bar{X})\big)_{s,t}  \delta (V^\epsilon)^{\otimes 2}_{s,t} \\
     &\qquad\qquad- \Big(\big(\partial_{xx}^2f(s,\bar{X}_s)\big)'(V_s^\epsilon)^{\otimes 2} + \partial_{xx}^2f(s,\bar{X}_s)\big((V_s^\epsilon)^{\otimes 2}\big)' \Big) \delta \eta_{s,t}\Big) \bigg] \bigg\|_\frac{k}{2}\\
     &\leq  \Big\| \mathbb{E}_s\Big[\delta\big(\partial_{xx}^2f(\cdot,\bar{X})\big)_{s,t} (V^\epsilon_s)^{\otimes 2} - \big(\partial_{xx}^2f(s,\bar{X}_s)\big)'(V_s^\epsilon)^{\otimes 2} \delta \eta_{s,t}\Big]\Big\|_\frac{k}{2} \\
     &\qquad+ \Big\| \mathbb{E}_s\Big[\partial_{xx}^2f(s,\bar{X}_s) \delta (V^\epsilon)^{\otimes 2}_{s,t} - \partial_{xx}^2f(s,\bar{X}_s)\big((V_s^\epsilon)^{\otimes 2}\big)' \delta \eta_{s,t}\Big]\Big\|_\frac{k}{2} \\
     &\qquad + \Big\|\mathbb{E}_s\Big[\delta\big(\partial_{xx}^2f(\cdot,\bar{X})\big)_{s,t}  \delta (V^\epsilon)^{\otimes 2}_{s,t}\Big]\Big\|_{\frac{k}{2}}\\
    &\leq \Big(\big\|\mathbb{E}_\bullet R^{\partial_{xx}^2 f(\cdot,\bar{X})}\big\|_{2\alpha;\infty} \sup_{r\in [0,T]}\|V^\epsilon_r\|^2_k + \sup_{r\in[0,T]} \|\partial_{xx}^2 f(r,\cdot)\|_{\infty} \sup_{r\in [0,T]}\|V^\epsilon_r\|_k \|\mathbb{E}_\bullet R^{V^\epsilon} \|_{2\alpha;k}\\
    &\qquad+ \sup_{r\in[0,T]} \|\partial_{xx}^2f(r,\cdot)\|_{\infty} \|\delta V^\epsilon \|^2_{\alpha;k}\Big) |t-s|^{2\alpha}\\
     &\qquad+ \bigg(\Big\|\big\|\delta\big(\partial_{xx}^2f(\cdot,\bar{X})\big)_{s,t} \big|\mathcal{F}_s \big\|_{\frac{k}{k-2}} \big\|\delta (V^\epsilon)^{\otimes 2}_{s,t}\big|\mathcal{F}_s\big\|_{\frac{k}{2}}\Big\|_{\frac{k}{2}}^{\frac{k}{2}}  \bigg)^{\frac{2}{k}}\\
     &\leq \Big(\big\|\mathbb{E}_\bullet R^{\partial_{xx}^2 f(\cdot,\bar{X})}\big\|_{2\alpha;\infty} \sup_{r\in [0,T]}\|V^\epsilon_r\|^2_k + \sup_{r\in[0,T]} \|\partial_{xx}^2 f(r,\cdot)\|_{\infty} \sup_{r\in [0,T]}\|V^\epsilon_r\|_k \|\mathbb{E}_\bullet R^{V^\epsilon} \|_{2\alpha;k}\\
   &\qquad+ \sup_{r\in[0,T]} \|\partial_{xx}^2f(r,\cdot)\|_{\infty} \|\delta V^\epsilon \|^2_{\alpha;k} + 2 \|\delta \big(\partial_{xx}^2f(\cdot,\bar{X})\big) \|_{\alpha;k,\infty} \sup_{r\in [0,T]}\|V^\epsilon_r\|_k \|\delta V^\epsilon \|_{\alpha;k}\Big) |t-s|^{2\alpha}.
\end{align*}
\end{proof}
So, by \cite[Theorem 3.11]{bugini2024malliavin}  
\eqref{eq:diffusion-second-variation} is well posed for $k \in [6,\infty)$ with a unique solution $(\zeta^\epsilon,\partial_xf(\cdot,\bar{X})\zeta^\epsilon + \frac{1}{2}\partial_{xx}^2f(\cdot,\bar{X}) (V^\epsilon)^{\otimes 2}) \in \mathbf{D}_{\eta}^{\alpha,\alpha}\mathbb{L}^{\frac{k}{2}}([0,T]; \mathbb{R}^{d})$.

Now we are ready to give our main approximation result for this section.

\begin{lemma}\label{prop:diffusion-spike-estimates}
Let $2 < \frac{1}{\alpha} \leq p <3$ and $k \in (3p,\infty)$. Then, under Assumptions \ref{P-control}--
\ref{D-cost} it holds that
\begin{align}
 &\|X^{u^\epsilon} - \bar{X}\|_{\mathbb S^k}
 +\|(X^{u^\epsilon} - \bar{X},(X^{u^\epsilon} - \bar{X})')\|_
 {\mathbf D_\eta^{(p,p)\text{-var}}\mathbb L^k}
 = O(\epsilon^{1/2}),
 \label{eq:state-diff}\\
 &\|V^\epsilon\|_{\mathbb S^k}
 + \|(V^\epsilon,(V^\epsilon)')\|_
 {\mathbf D_\eta^{(p,p)\text{-var}}\mathbb L^k}
 = O(\epsilon^{1/2}),
 \label{eq:diffusion-first-order-size}\\
 &\|\zeta^\epsilon\|_{\mathbb S^{\frac{k}{2}}}
 +\|(\zeta^\epsilon,(\zeta^\epsilon)')\|_
 {\mathbf D_\eta^{(p,p)\text{-var}}\mathbb L^{\frac{k}{2}}}
 = O(\epsilon),
 \label{eq:diffusion-second-order-size}\\
 &\|X^{\boldsymbol{\eta},u^\epsilon}-\bar{X} - V^\epsilon\|_{\mathbb S^{\frac{k}{2}}}
 +\|(X^{\boldsymbol{\eta},u^\epsilon}-\bar{X} - V^\epsilon,(X^{\boldsymbol{\eta},u^\epsilon}-\bar{X} - V^\epsilon)')\|_
 {\mathbf D_\eta^{(p,p)\text{-var}}\mathbb L^{\frac{k}{2}}}=O(\epsilon),
 \label{eq:diffusion-state-remainder_0}\\
 &\|X^{\boldsymbol{\eta},u^\epsilon}-\bar{X} - V^\epsilon - \zeta^\epsilon\|_{\mathbb S^{\frac{k}{3}}} +\|(X^{\boldsymbol{\eta},u^\epsilon}-\bar{X} - V^\epsilon - \zeta^\epsilon,(X^{\boldsymbol{\eta},u^\epsilon}-\bar{X} - V^\epsilon - \zeta^\epsilon)')\|_
 {\mathbf D_\eta^{(p,p)\text{-var}}\mathbb L^{\frac{k}{3}}}
 =o(\epsilon).
 \label{eq:diffusion-state-remainder}
\end{align}
Finally, for the cost functional we have the expansion
\begin{multline}\label{eq 76}
 \mathcal{J}(u^\epsilon) - \mathcal{J}(\bar u)
 = \mathbb{E}\bigg[
 \partial_x g(\bar X_T)(V^\epsilon_T + \zeta^\epsilon_T)
 +\frac{1}{2}\partial_{xx}^2 g(\bar X_T)
       (V^\epsilon_T)^{\otimes 2}\\
 +\int_0^T\Big(
 \partial_x h(s,\bar{X}_s,\bar{u}_s)(V^\epsilon_s + \zeta^\epsilon_s)
 +\frac{1}{2}\partial_{xx}^2 h(s,\bar{X}_s,\bar{u}_s)
      (V^\epsilon_s)^{\otimes 2}\\
 \qquad+ \mathds{1}_{I}(s)
 \bigl[h(s,\bar X_s, u_s)-h(s,\bar X_s,\bar u_s)\bigr]\Big)\ud s
 \bigg] + o(\epsilon).
\end{multline}
\end{lemma}

\begin{proof}
To ease the notation, we define $X^\epsilon_t:= X^{\boldsymbol{\eta},u^\epsilon}_t$. Furthermore, we define
\begin{align*}
    \widetilde{\Phi}^\epsilon_t &:= f(t,X^\epsilon_t) - f(t,\bar{X}_t),\\
    (\widetilde{\Phi}^\epsilon)'_t&:=  \Big[\partial_xf(t,X^\epsilon_t)f(t,X^\epsilon_t) +f'(t,X^\epsilon_t) \Big] - \Big[\partial_xf(t,\bar{X}_t)f(t,\bar{X}_t) +f'(t,\bar{X}_t) \Big]
\end{align*}
and we have that $\big(\widetilde{\Phi}^\epsilon,(\widetilde{\Phi}^\epsilon)'\big)\in \mathbf{D}^{\alpha,\alpha}_\eta\mathbb{L}^{k}\big([0,T];\mathbb{R}^{d\times \ell}\big)$.
So, for every $t\in [0,T]$ we have
\begin{align*}
  &X^\epsilon_t - \bar{X}_t = \int_0^t \mathds{1}_I(s) \big(b(s,\bar{X}_s,u_s) - b(s,\bar{X}_s,\bar{u}_s) \big) \ud s + \int_0^t \big(b(s,X^\epsilon_s,{u}^\epsilon_s) - b(s,\bar{X}_s,{u}^\epsilon_s) \big) \ud s \\
  &\qquad + \int_0^t \mathds{1}_I(s) \big(\sigma(s,\bar{X}_s,u_s) - \sigma(s,\bar{X}_s,\bar{u}_s) \big) \ud W_s + \int_0^t \big(\sigma(s,X^\epsilon_s,{u}^\epsilon_s) - \sigma(s,\bar{X}_s,{u}^\epsilon_s) \big) \ud W_s\\
  &\qquad+ \int_0^t \big( f(s,X^{\epsilon}_s) - f(s,\bar{X}_s) \big) \ud \boldsymbol{\eta}_s\\
  &= \int_0^t \mathds{1}_I(s) \big(b(s,\bar{X}_s,u_s) - b(s,\bar{X}_s,\bar{u}_s) \big) \ud s + \int_0^t \bigg[\int_0^1 \partial_x b(s,\bar{X}_s + \lambda (X^\epsilon_s - \bar{X}_s),{u}^\epsilon_s) \ud \lambda \bigg] (X^\epsilon_s - \bar{X}_s) \ud s\\
  &+  \int_0^t \mathds{1}_I(s) \big(\sigma(s,\bar{X}_s,u_s) - \sigma(s,\bar{X}_s,\bar{u}_s) \big) \ud W_s + \int_0^t \bigg[\int_0^1 \partial_x \sigma(s,\bar{X}_s + \lambda (X^\epsilon_s - \bar{X}_s),{u}^\epsilon_s) \ud \lambda \bigg] (X^\epsilon_s - \bar{X}_s) \ud W_s\\
  &\qquad+ \int_0^t \bigg[\int_0^1 \partial_x f\big(s,\bar{X}_s + \lambda (X^\epsilon_s - \bar{X}_s)\big) \ud \lambda \bigg] (X^\epsilon_s - \bar{X}_s) \ud \boldsymbol{\eta}_s.
\end{align*}
Now, we define for every $\lambda\in[0,1]$
\begin{align*}
 \Phi^\epsilon_t(\lambda)&:=  \partial_x f(t,\bar{X}_t + \lambda  (X^\epsilon_t - \bar{X}_t)),\\ 
 (\Phi^\epsilon)'_t(\lambda)&:= \partial^2_{xx} f(t,\bar{X}_t + \lambda  (X^\epsilon_t - \bar{X}_t)) \big(f_t(\bar{X}_t) + \lambda [f_t({X}^\epsilon_t) -f_t(\bar{X}_t)]\big) \\
 &\qquad+\partial_x f'(t,\bar{X}_t + \lambda  (X^\epsilon_t - \bar{X}_t)).
\end{align*}
From Lemma \ref{lemma 4.1} and its proof, along with the boundness properties of $f, \partial_x f, \partial^2_{xx} f, \partial_xf'$ and \cite[Proposition 4.5]{FHL21}, we see that $\big(\Phi^\epsilon(\lambda),(\Phi^\epsilon)'(\lambda)\big)\in \mathbf{D}_{\eta}^{\alpha,\alpha} \mathbb{L}^{k,\infty}_{\lambda-ex,lin}\big([0,T];L(\mathbb{R}^d,\mathbb{R}^{d\times p})\big)$ and
\begin{align}\label{eq. 59}
\sup_{\epsilon > 0}\sup_{\lambda\in[0,1]}\Big\{\sup_{t\in[0,T]}\|\Phi^\epsilon_t(\lambda)\|_\infty +  \sup_{t\in[0,T]}\|(\Phi^\epsilon)'_t(\lambda)\|_\infty
+\big\|\big(\Phi^\epsilon_\cdot(\lambda),(\Phi^\epsilon)'_\cdot(\lambda)\big)\big\|_{\mathbf{D}_{\eta}^{\alpha,\alpha} \mathbb{L}^{k,\infty}_{lin}} \Big\} < \infty.
\end{align}

Next, by the fundamental Theorem of calculus we have for every $t\in [0,T]$ that
\begin{multline}\label{eq 78}
   (\widetilde{\Phi}^\epsilon)'_t = \int_0^1\bigg( \Big[\partial^2_{xx} f(t,\bar{X}_t + \lambda  (X^\epsilon_t - \bar{X}_t)) \Big(f_t(\bar{X}_t) + \lambda \big(f_t({X}^\epsilon_t) -f_t(\bar{X}_t)\big)\Big) \\+ \partial_x f'(t,\bar{X}_t + \lambda  (X^\epsilon_t - \bar{X}_t)) \Big] (X^\epsilon_t - \bar{X}_t) + \partial_x f(t,\bar{X}_t + \lambda  (X^\epsilon_t - \bar{X}_t)) \big(f_t({X}^\epsilon_t) -f_t(\bar{X}_t)\big) \bigg) \ud \lambda\\
   = \int_0^1\Big((\Phi^\epsilon)'_t(\lambda) (X^\epsilon_t - \bar{X}_t) + \Phi^\epsilon_t(\lambda)  (X^\epsilon_t - \bar{X}_t)'
   \Big)\ud\lambda.
\end{multline}
Furthermore, we define
\begin{align*}
    F^\epsilon_t := \int_0^t \mathds{1}_I(s) \delta b_s \ud s + \int_0^t \mathds{1}_I(s) \delta\sigma_s \ud W_s.
\end{align*}
Hence, from Proposition \ref{Prop 3.3}, Lemma \ref{lemma 3.4} and Proposition \ref{Prop 3.5} we see that $X^\epsilon - \bar{X}$ satisfies the following equation
\begin{multline}\label{eq 77}
    X^\epsilon_t - \bar{X}_t = F^\epsilon_t + \int_0^t \bigg[\int_0^1 \partial_x b(s,\bar{X}_s + \lambda (X^\epsilon_s - \bar{X}_s),{u}^\epsilon_s) \ud \lambda \bigg] (X^\epsilon_s - \bar{X}_s) \ud s\\
  +  \int_0^t \bigg[\int_0^1 \partial_x \sigma(s,\bar{X}_s + \lambda (X^\epsilon_s - \bar{X}_s),{u}^\epsilon_s) \ud \lambda \bigg] (X^\epsilon_s - \bar{X}_s) \ud W_s\\
    + \int_0^t  
    \bigg(\int_0^1 \Phi^\epsilon_s(\lambda)\ud\lambda, \int_0^1(\Phi^\epsilon)'_s(\lambda) \ud\lambda \bigg) (X^\epsilon_s - \bar{X}_s)
    \ud \boldsymbol{\eta}_s.
\end{multline}
We can easily see with the methods used in the proof of Theorem \ref{thm 3.9} (see also below) that 
\begin{align*}
\big\|\big(F^\epsilon,0 \big)\big\|_
 {\mathbf D_\eta^{(p,p)\text{-var}}\mathbb L^k}&\leq \bigg\|\bigg(\int_0^\cdot \mathds{1}_I(s) \delta b_s \ud s,0 \bigg)\bigg\|_
 {\mathbf D_\eta^{(p,p)\text{-var}}\mathbb L^k} + \bigg\|\bigg(\int_0^\cdot \mathds{1}_I(s) \delta\sigma_s \ud W_s,0 \bigg)\bigg\|_
 {\mathbf D_\eta^{(p,p)\text{-var}}\mathbb L^k}\\ 
 &\lesssim \epsilon + \epsilon^{1/2}.
\end{align*}
Thus, via an application of \eqref{eq:BV-primary-estimate} from Theorem \ref{thm 3.9} to \eqref{eq 77} we obtain \eqref{eq:state-diff}.

Similarly, the martingale term in \eqref{eq:diffusion-first-variation} can be analyzed as above. By \eqref{eq:BV-primary-estimate} of Theorem \ref{thm 3.9} again, we have
\eqref{eq:diffusion-first-order-size}.

Next, note that every inhomogeneous term in
\eqref{eq:diffusion-second-variation} is then of order $\epsilon$. Indeed, for the drift spike we have from \ref{D-dynamics}, \eqref{eq. 012} (where $\frac{k}{2}$ is in place of $\{k,\infty\}$ and $\infty$ in place of $k$), the H\"older inequality \eqref{ineq 84} and \eqref{eq:diffusion-first-order-size} that
\begin{align*}
    &\bigg\|\bigg(\int_0^\cdot\Big[\frac{1}{2}\partial_{xx}^2 b(s,\bar{X}_s, \bar{u}_s) (V^\epsilon_s)^{\otimes 2} + \mathds{1}_I(s) \delta b_s\big] \ud s,0\bigg) \bigg\|_{\mathbf D_\eta^{(p,p)\text{-var}}\mathbb L^\frac{k}{2}} \\
    &\leq \bigg\|\bigg(\int_0^\cdot\frac{1}{2}\partial_{xx}^2 b(s,\bar{X}_s, \bar{u}_s) (V^\epsilon_s)^{\otimes 2} \ud s,0\bigg) \bigg\|_{\mathbf D_\eta^{(p,p)\text{-var}}\mathbb L^\frac{k}{2}} + \bigg\|\bigg(\int_0^\cdot \mathds{1}_I(s) \delta b_s \ud s,0\bigg) \bigg\|_{\mathbf D_\eta^{(p,p)\text{-var}}\mathbb L^\frac{k}{2}}\\
    &\lesssim \int_0^T \|V^\epsilon_s\|_k^2\ud s + \int_0^T\mathds{1}_I(s)\ud s \lesssim \epsilon.
\end{align*}
Similarly, for the diffusion spike we have from \ref{D-dynamics}, \eqref{eq. 013} (where $\frac{k}{2}$ is in place of $\{k,\infty\}$ and $\infty$ in place of $k$), the H\"older inequality \eqref{ineq 84} and \eqref{eq:diffusion-first-order-size} that
\begin{align*}
    &\bigg\|\bigg(\int_0^\cdot \Big[\frac{1}{2}\partial_{xx}^2\sigma(s,\bar{X}_s, \bar{u}_s) (V^\epsilon_s)^{\otimes 2} + \mathds{1}_I(s) \delta\sigma_{x,s} V^\epsilon_s \Big] \ud W_s,0 \bigg)\bigg\|_
 {\mathbf D_\eta^{(p,p)\text{-var}}\mathbb L^\frac{k}{2}}\\
 &\leq \bigg\|\bigg(\int_0^\cdot \frac{1}{2}\partial_{xx}^2\sigma(s,\bar{X}_s, \bar{u}_s) (V^\epsilon_s)^{\otimes 2}  \ud W_s,0 \bigg)\bigg\|_
 {\mathbf D_\eta^{(p,p)\text{-var}}\mathbb L^\frac{k}{2}} + \bigg\|\bigg(\int_0^\cdot  \mathds{1}_I(s) \delta\sigma_{x,s} V^\epsilon_s  \ud W_s,0 \bigg)\bigg\|_
 {\mathbf D_\eta^{(p,p)\text{-var}}\mathbb L^\frac{k}{2}}\\
 &\lesssim \bigg( \int_0^T\|V^\epsilon_s\|_k^4 \ud s\bigg)^{\frac{1}{2}} + \bigg(\int_0^T \mathds{1}_I(s) \|V^\epsilon_s\|_k^2 \ud s \bigg)^{\frac{1}{2}} \lesssim \epsilon.
\end{align*}
Lastly, for the rough integral term we have from \ref{P-rough}, \eqref{ineq 12}, \eqref{eq. 018} (where $(V^\epsilon)^{\otimes 2}$ is in place of $X$ and $(0,0)$ in place of $(F,F')$), the H\"older inequality \eqref{ineq 84} and \eqref{eq:diffusion-first-order-size}
that
\begin{align*}
    &\bigg\| \bigg(\int_0^\cdot \frac{1}{2}\partial_{xx}^2f(s,\bar{X}_s) (V^\epsilon_s)^{\otimes 2}\ud \boldsymbol{\eta}_s,\frac{1}{2}\partial_{xx}^2f(\cdot,\bar{X}) (V^\epsilon)^{\otimes 2}\bigg)\bigg\|_ {\mathbf D_\eta^{(p,p)\text{-var}}\mathbb L^\frac{k}{2}}\\
    &\lesssim \big\| \big((V^\epsilon)^{\otimes 2}, (V^\epsilon)' \otimes V^\epsilon+ V^\epsilon\otimes (V^\epsilon)'\big)\big\|_ {\mathbf D_\eta^{(p,p)\text{-var}}\mathbb L^\frac{k}{2}}\\
    &\lesssim \Big( \|V^\epsilon\|_{\mathbb{S}^k} + \|(V^\epsilon,(V^\epsilon)')\|_
 {\mathbf D_\eta^{(p,p)\text{-var}}\mathbb L^k}\Big)^2 \lesssim \epsilon.
\end{align*}
This proves \eqref{eq:diffusion-second-order-size}.

We move on to \eqref{eq:diffusion-state-remainder_0}. We subtract the equations and we get
\begin{align*}
    X^\epsilon_t - \bar{X}_t - V^\epsilon_t =  \widetilde{F}^\epsilon_t + \int_0^t \bigg[\int_0^1 \partial_x b(s,\bar{X}_s + \lambda (X^\epsilon_s - \bar{X}_s),{u}^\epsilon_s) \ud \lambda \bigg] (X^\epsilon_s - \bar{X}_s - V^\epsilon_s) \ud s\\
  +  \int_0^t \bigg[\int_0^1 \partial_x \sigma(s,\bar{X}_s + \lambda (X^\epsilon_s - \bar{X}_s),{u}^\epsilon_s) \ud \lambda \bigg] (X^\epsilon_s - \bar{X}_s - V^\epsilon_s) \ud W_s\\
   + \int_0^t  
    \bigg(\int_0^1\Phi^\epsilon_s(\lambda)\ud\lambda,\int_0^1(\Phi^\epsilon)'_s(\lambda)\ud\lambda\bigg) (X^\epsilon_s - \bar{X}_s - V^\epsilon_s) 
    \ud \boldsymbol{\eta}_s
\end{align*}
with
\begin{align*}
  \widetilde{F}^\epsilon_t :=  \int_0^t \bigg( \mathds{1}_I(s) \delta b_s +\bigg[\int_0^1 \partial_x b(s,\bar{X}_s + \lambda (X^\epsilon_s - \bar{X}_s),{u}^\epsilon_s) \ud \lambda - \partial_x b(s,\bar{X}_s, \bar{u}_s)\bigg] V^\epsilon_s  \bigg) \ud s\\
  + \int_0^t \bigg[\int_0^1 \partial_x \sigma(s,\bar{X}_s + \lambda (X^\epsilon_s - \bar{X}_s),{u}^\epsilon_s) \ud \lambda - \partial_x \sigma(s,\bar{X}_s, \bar{u}_s)\bigg] V^\epsilon_s  \ud W_s\\
  + \int_0^t \big( \widehat{H}^\epsilon_s,(\widehat{H}^\epsilon)'_s\big) V^\epsilon_s \ud \boldsymbol{\eta}_s,
\end{align*}
and
\begin{align*}
     \widehat{H}^\epsilon_t &:= \int_0^1 
     {\Phi}^\epsilon_t(\lambda)
     \ud \lambda - \partial_{x}f(t,\bar{X}_t),\\
    (\widehat{H}^\epsilon)'_t &: = \int_0^1 
    ({\Phi}^\epsilon)'_t(\lambda)
    \ud \lambda - \partial^2_{xx}f(t,\bar{X}_t) f(t,\bar{X}_t) - \partial_{x}f'(t,\bar{X}_t).
\end{align*}
So, by Proposition \ref{Prop 3.5} and applying \eqref{eq:BV-primary-estimate} of Theorem \ref{thm 3.9} with forcing $\widetilde{F}^\epsilon$ we get \eqref{eq:diffusion-state-remainder_0}. Indeed, it only remains to estimate the following
\begin{multline*}
\|(\widetilde{F}^\epsilon,(\widetilde{F}^\epsilon)')\|_{
    \mathbf D_\eta^{(p,p)\text{-var}}\mathbb L^{\frac{k}{2}}}\leq\\
    \bigg\|\bigg(\int_0^\cdot \bigg( \mathds{1}_I(s) \delta b_s +\bigg[\int_0^1 \partial_x b(s,\bar{X}_s + \lambda (X^\epsilon_s - \bar{X}_s),{u}^\epsilon_s) \ud \lambda - \partial_x b(s,\bar{X}_s, \bar{u}_s)\bigg] V^\epsilon_s  \bigg) \ud s,0 \bigg)\bigg\|_
 {\mathbf D_\eta^{(p,p)\text{-var}}\mathbb L^{\frac{k}{2}}}\\
 + \bigg\|\bigg(\int_0^\cdot \bigg[\int_0^1 \partial_x \sigma(s,\bar{X}_s + \lambda (X^\epsilon_s - \bar{X}_s),{u}^\epsilon_s) \ud \lambda - \partial_x \sigma(s,\bar{X}_s, \bar{u}_s)\bigg] V^\epsilon_s  \ud W_s,0 \bigg)\bigg\|_
 {\mathbf D_\eta^{(p,p)\text{-var}}\mathbb L^{\frac{k}{2}}}\\
 + \bigg\|\bigg(\int_0^\cdot \widehat{H}^\epsilon_s V^\epsilon_s \ud\boldsymbol{\eta}_s,\widehat{H}^\epsilon V^\epsilon \bigg)\bigg\|_
 {\mathbf D_\eta^{(p,p)\text{-var}}\mathbb L^{\frac{k}{2}}}.
\end{multline*} 
Now, for the second term, let $0\leq u_1 \leq u_2 \leq T$, by BDG and Minkowski's integral inequalities we get
\begin{align*}
    \bigg\|\int_{u_1}^{u_2} \bigg[\int_0^1 \partial_x \sigma(s,\bar{X}_s + \lambda (X^\epsilon_s - \bar{X}_s),{u}^\epsilon_s) \ud \lambda - \partial_x \sigma(s,\bar{X}_s, \bar{u}_s)\bigg] V^\epsilon_s  \ud W_s\bigg\|_{\frac{k}{2}}\\
    \lesssim  \bigg\|\bigg(\int_{u_1}^{u_2} \bigg|\int_0^1 \partial_x \sigma(s,\bar{X}_s + \lambda (X^\epsilon_s - \bar{X}_s),{u}^\epsilon_s) \ud \lambda - \partial_x \sigma(s,\bar{X}_s, \bar{u}_s)\bigg|^2 |V^\epsilon_s|^2  \ud s \bigg)^{\frac{1}{2}}\bigg\|_{\frac{k}{2}}\\
    \leq \bigg(\int_{u_1}^{u_2} \bigg\|\bigg|\int_0^1 \partial_x \sigma(s,\bar{X}_s + \lambda (X^\epsilon_s - \bar{X}_s),{u}^\epsilon_s) \ud \lambda - \partial_x \sigma(s,\bar{X}_s, \bar{u}_s)\bigg| |V^\epsilon_s|\bigg\|^2_{\frac{k}{2}}  \ud s \bigg)^{\frac{1}{2}}\\
    \leq \|V^\epsilon\|_{\mathbb S^k} \bigg(\int_{u_1}^{u_2} \bigg\|\int_0^1 \partial_x \sigma(s,\bar{X}_s + \lambda (X^\epsilon_s - \bar{X}_s),{u}^\epsilon_s) \ud \lambda - \partial_x \sigma(s,\bar{X}_s, \bar{u}_s) \bigg\|^2_{k}  \ud s \bigg)^{\frac{1}{2}},
\end{align*}
where we used the H\"older inequality 
\begin{equation}\label{ineq 84}
 \|\cdot \cdot\hspace{.1cm}\|_{\frac{k}{2}} \leq \|\cdot\|_{k} \|\cdot\|_{k}.
\end{equation}
Hence, 
    \begin{align*}
        &\bigg\|\bigg(\int_0^\cdot \bigg[\int_0^1 \partial_x \sigma(s,\bar{X}_s + \lambda (X^\epsilon_s - \bar{X}_s),{u}^\epsilon_s) \ud \lambda - \partial_x \sigma(s,\bar{X}_s, \bar{u}_s)\bigg] V^\epsilon_s  \ud W_s,0\bigg)\bigg\|_{
    \mathbf D_\eta^{(p,p)\text{-var}}\mathbb L^{\frac{k}{2}}}\\
    &\lesssim \|V^\epsilon\|_{\mathbb S^k} \bigg(\int_{0}^{T} \bigg\|\int_0^1 \partial_x \sigma(s,\bar{X}_s + \lambda (X^\epsilon_s - \bar{X}_s),{u}^\epsilon_s) \ud \lambda - \partial_x \sigma(s,\bar{X}_s, \bar{u}_s) \bigg\|^2_{k}  \ud s \bigg)^{\frac{1}{2}}\\
    &\leq  \|V^\epsilon\|_{\mathbb S^k} \bigg(\int_{0}^{T} \mathds{1}_{I}(s) \bigg\|\int_0^1 \partial_x \sigma(s,\bar{X}_s + \lambda (X^\epsilon_s - \bar{X}_s),{u}^\epsilon_s) \ud \lambda - \partial_x \sigma(s,\bar{X}_s, \bar{u}_s) \bigg\|^2_{k}  \ud s\\
    &\qquad+ \int_{0}^{T} \mathds{1}_{I^c}(s) \bigg\|\int_0^1 \partial_x \sigma(s,\bar{X}_s + \lambda (X^\epsilon_s - \bar{X}_s),{u}^\epsilon_s) \ud \lambda - \partial_x \sigma(s,\bar{X}_s, \bar{u}_s) \bigg\|^2_{k}  \ud s\bigg)^{\frac{1}{2}}\\
    &\lesssim \|V^\epsilon\|_{\mathbb S^k} \bigg[\epsilon^{\frac{1}{2}} +  \bigg(\int_{0}^{T} \mathds{1}_{I^c}(s)\bigg\|\int_0^1 \partial_x \sigma(s,\bar{X}_s + \lambda (X^\epsilon_s - \bar{X}_s),\bar{u}_s) \ud \lambda - \partial_x \sigma(s,\bar{X}_s, \bar{u}_s) \bigg\|^2_{k}  \ud s \bigg)^{\frac{1}{2}}\bigg]\\
    &\lesssim \|V^\epsilon\|_{\mathbb S^k} \big[\epsilon^{\frac{1}{2}} + T^{\frac{1}{2}} \|X^\epsilon - \bar{X}\|_{\mathbb S^k}\big] \lesssim \epsilon.
    \end{align*}
 The rough variation seminorm of the drift term is estimated in essentially the same way as above while also using the contraction property of conditional expectation. 
 For the rough integral term, its rough variation seminorm is estimated by noticing the uniform bound over $\epsilon$ from \eqref{eq. 59} and applying \eqref{ineq 12}, together with the structure of \eqref{ineq 13}–\eqref{ineq 15} (where $V^\epsilon$ is in place of $X$ and (0,0) in place of $(F,F')$) the H\"older inequality \eqref{ineq 84} and \eqref{eq:state-diff}, \eqref{eq:diffusion-first-order-size} such that
\begin{align*}
 &\bigg\|\bigg(\int_0^\cdot \widehat{H}^\epsilon_s V^\epsilon_s \ud\boldsymbol{\eta}_s,\widehat{H}^\epsilon V^\epsilon \bigg)\bigg\|_
 {\mathbf D_\eta^{(p,p)\text{-var}}\mathbb L^{\frac{k}{2}}}\\
 &\lesssim  \Big[\big\|\big(\widehat{H}^\epsilon,(\widehat{H}^\epsilon)'\big)\big\|_{\mathbf{D}_{\eta}^{(p,p)\text{-var}} \mathbb{L}^{k}_{lin}} + \sup_{t\in[0,T]}\|\widehat{H}^\epsilon_t\|_k + \sup_{t\in[0,T]}\|(\widehat{H}^\epsilon)'_t \|_k\Big] \\
 &\qquad\cdot \Big[\|V^\epsilon\|_{\mathbb{S}^k} + \|(V^\epsilon,(V^\epsilon)')\|_
 {\mathbf D_\eta^{(p,p)\text{-var}}\mathbb L^{k}}\Big]\\
 &\lesssim \Big[\|X^{u^\epsilon} - \bar{X}\|_{\mathbb S^k}
 +\|(X^{u^\epsilon} - \bar{X},(X^{u^\epsilon} - \bar{X})')\|_
 {\mathbf D_\eta^{(p,p)\text{-var}}\mathbb L^k}\Big] \Big[\|V^\epsilon\|_{\mathbb{S}^k} + \|(V^\epsilon,(V^\epsilon)')\|_
 {\mathbf D_\eta^{(p,p)\text{-var}}\mathbb L^{k}}\Big]\\
 &\lesssim \epsilon.
\end{align*}
Above we assumed the inequality
\begin{multline}\label{eq. 091}
     \big\|\big(\widehat{H}^\epsilon,(\widehat{H}^\epsilon)'\big)\big\|_{\mathbf{D}_{\eta}^{(p,p)\text{-var}} \mathbb{L}^{k}_{lin}} + \sup_{t\in[0,T]}\|\widehat{H}^\epsilon_t\|_k + \sup_{t\in[0,T]}\|(\widehat{H}^\epsilon)'_t \|_k\ \\
     \lesssim \|X^{u^\epsilon} - \bar{X}\|_{\mathbb S^k}
 +\|(X^{u^\epsilon} - \bar{X},(X^{u^\epsilon} - \bar{X})')\|_
 {\mathbf D_\eta^{(p,p)\text{-var}}\mathbb L^k}.
\end{multline}
From the fundamental Theorem of calculus we have
\begin{align*}
    \widehat{H}^\epsilon_t &= \bigg[\int_0^1 \int_0^1 \lambda\partial_{xx}^2 f(t,\bar{X}_t + \vartheta \lambda  (X^\epsilon_t - \bar{X}_t))\ud \vartheta \ud\lambda \bigg] (X^\epsilon_t - \bar{X}_t) = \int_0^1 \widehat{{H}}^\epsilon_t(\lambda)\ud\lambda (X^\epsilon_t - \bar{X}_t),\\
    (\widehat{H}^\epsilon)'_t&= \int_0^1\int_0^1\lambda \partial_{xxx}^3 f(t,\bar{X}_t + \vartheta\lambda   (X^\epsilon_t - \bar{X}_t)) \big(f(t,\bar{X}_t) + \vartheta\lambda \big[f(t,X^\epsilon_t ) -f(t,\bar{X}_t)\big] \big)\ud \vartheta \ud \lambda (X^\epsilon_t - \bar{X}_t) \\
    &\qquad+\int_0^1\int_0^1 \lambda\partial_{xx}^2 f(t,\bar{X}_t + \vartheta \lambda  (X^\epsilon_t - \bar{X}_t)) \big[f(t,X^\epsilon_t ) -f(t,\bar{X}_t)\big] \ud \vartheta \ud\lambda\\
    &\qquad+ \int_0^1 \int_0^1 \lambda\partial_{xx}^2 f'(t,\bar{X}_t + \vartheta\lambda  (X^\epsilon_t - \bar{X}_t))\ud \vartheta \ud \lambda (X^\epsilon_t - \bar{X}_t)\\
    &= \int_0^1 (\widehat{{H}}^\epsilon)'_t(\lambda)\ud\lambda (X^\epsilon_t - \bar{X}_t) + \int_0^1 \widehat{{H}}^\epsilon_t(\lambda) \ud\lambda (X^\epsilon_t - \bar{X}_t)',
\end{align*}
where 
\begin{align*}
\widehat{{H}}^\epsilon_t(\vartheta,\lambda)&:= \lambda\partial_{xx}^2 f(t,\bar{X}_t + \vartheta \lambda  (X^\epsilon_t - \bar{X}_t)),\\
(\widehat{{H}}^\epsilon)'_t(\vartheta,\lambda)&:= \lambda \partial_{xxx}^3 f(t,\bar{X}_t + \vartheta\lambda (X^\epsilon_t - \bar{X}_t)) \big(f(t,\bar{X}_t) + \vartheta\lambda \big[f(t,X^\epsilon_t ) -f(t,\bar{X}_t)\big] \big) \\
&\qquad+  \lambda\partial_{xx}^2 f'(t,\bar{X}_t + \vartheta \lambda  (X^\epsilon_t - \bar{X}_t)) ,\\
\widehat{{H}}^\epsilon_t(\lambda) &:= \int_0^1 \widehat{{H}}^\epsilon_t(\vartheta,\lambda)\ud \vartheta,\\
(\widehat{{H}}^\epsilon)'_t(\lambda)&:= \int_0^1 (\widehat{{H}}^\epsilon)'_t(\vartheta,\lambda)\ud \vartheta.
\end{align*}
So, for every $\lambda\in[0,1]$, from Lemma \ref{lemma 4.1} and its proof, along with the boundness properties of $f, \partial_{xx}^2 f, \partial^3_{xxx} f, \partial_{xx}^2f'$ and \cite[Proposition 4.5]{FHL21}, we see that
$\big(\widehat{{H}}^\epsilon(\vartheta,\lambda),(\widehat{{H}}^\epsilon)'(\vartheta,\lambda)\big)\in \mathbf{D}_{\eta}^{\alpha,\alpha} \mathbb{L}^{k,\infty}_{\vartheta-ex,lin}$ $\big([0,T];L(\mathbb{R}^d,L(\mathbb{R}^d,\mathbb{R}^{d\times \ell}))\big)$ and
\begin{multline*}
\sup_{\epsilon > 0}\sup_{\lambda\in[0,1]}\sup_{\vartheta\in[0,1]}\Big\{\sup_{t\in[0,T]}\|\widehat{{H}}_t^\epsilon(\vartheta,\lambda)\|_\infty +  \sup_{t\in[0,T]}\|(\widehat{{H}}^\epsilon)_t'(\vartheta,\lambda)\|_\infty
+\big\|\big(\widehat{{H}}_\cdot^\epsilon(\vartheta,\lambda),(\widehat{{H}}^\epsilon)_\cdot'(\vartheta,\lambda)\big)\big\|_{\mathbf{D}_{\eta}^{\alpha,\alpha} \mathbb{L}^{k,\infty}_{lin}} \Big\} < \infty.
\end{multline*}
Consequently, we get that $\big(\widehat{{H}}^\epsilon(\lambda),(\widehat{{H}}^\epsilon)'(\lambda)\big)\in \mathbf{D}_{\eta}^{\alpha,\alpha} \mathbb{L}^{k,\infty}_{\lambda-ex,lin}\big([0,T];L(\mathbb{R}^d,L(\mathbb{R}^d,\mathbb{R}^{d\times \ell}))\big)$ and
\begin{align*}
\sup_{\epsilon > 0}\sup_{\lambda\in[0,1]}\Big\{\sup_{t\in[0,T]}\|\widehat{{H}}_t^\epsilon(\lambda)\|_\infty +  \sup_{t\in[0,T]}\|(\widehat{{H}}^\epsilon)_t'(\lambda)\|_\infty
+\big\|\big(\widehat{{H}}_\cdot^\epsilon(\lambda),(\widehat{{H}}^\epsilon)_\cdot'(\lambda)\big)\big\|_{\mathbf{D}_{\eta}^{\alpha,\alpha} \mathbb{L}^{k,\infty}_{lin}} \Big\} < \infty.
\end{align*}
Thus, \eqref{eq. 091} follows from \eqref{eq. 018} (where $X^\epsilon - \bar{X}$ is in place of $X$ and $(0,0)$ in place of $(F,F')$).

Then, for \eqref{eq:diffusion-state-remainder}, by the Taylor expansion
of the state equation up to the second order (see e.g. \cite[Lemma 4.3]{yong1999stochastic} for the classical case), we have
\begin{align*}
    &X^\epsilon_t - \bar X_t = \int_0^t \mathds{1}_I(s) \delta b_s \ud s  +\int_0^t \mathds{1}_I(s) \delta \sigma_s \ud W_s\\ &+ \int_0^t \bigg( \partial_{x} b(s,\bar{X}_s,u^\epsilon_s) (X^\epsilon_s - \bar{X}_s) + \bigg[\int_0^1 \lambda \partial^2_{xx} b(s,{X}^\epsilon_s - \lambda (X^\epsilon_s - \bar{X}_s),{u}^\epsilon_s) \ud \lambda \bigg] (X^\epsilon_s - \bar{X}_s)^{\otimes 2} \bigg)\ud s\\ 
    &+ \int_0^t \bigg(\partial_{x} \sigma(s,\bar{X}_s,u^\epsilon_s) (X^\epsilon_s - \bar{X}_s) + \bigg[\int_0^1 \lambda \partial^2_{xx} \sigma(s,{X}^\epsilon_s - \lambda (X^\epsilon_s - \bar{X}_s),{u}^\epsilon_s) \ud \lambda \bigg] (X^\epsilon_s - \bar{X}_s)^{\otimes 2} \bigg)\ud W_s\\
    &+ \int_0^t\partial_{x} f(s,\bar{X}_s) (X^\epsilon_s - \bar{X}_s) \ud\boldsymbol{\eta}_s 
    +\int_0^t \bigg[\int_0^1 \lambda \partial^2_{xx} f(s,{X}^\epsilon_s - \lambda (X^\epsilon_s - \bar{X}_s)) \ud \lambda \bigg] (X^\epsilon_s - \bar{X}_s)^{\otimes 2}\ud \boldsymbol{\eta}_s,
\end{align*}
where for the difference of the rough integral, the corresponding Gubinelli derivative is 
\begin{align}
    &\Big[\partial_xf(t,X^\epsilon_t)f(t,X^\epsilon_t) +f'(t,X^\epsilon_t) \Big] - \Big[\partial_xf(t,\bar{X}_t)f(t,\bar{X}_t) +f'(t,\bar{X}_t) \Big] = \nonumber\\
    &\Big[\partial^2_{xx}f(t,\bar{X}_t) f(t,\bar{X}_t) + \partial_x f'(t,\bar{X}_t)\Big] (X^\epsilon_t - \bar{X}_t) + \partial_{x}f(t,\bar{X}_t) \Big[f(t,{X}^\epsilon_t) - f(t,\bar{X}_t)\Big] \nonumber\\
    &\qquad+ \int_0^1 \Big( \Big[\lambda\partial^3_{xxx} f(t,{X}^\epsilon_t - \lambda  (X^\epsilon_t - \bar{X}_t)) \big(f_t({X}^\epsilon_t) - \lambda [f_t({X}^\epsilon_t) -f_t(\bar{X}_t)]\big) \nonumber\\
    &\qquad\qquad+ \lambda\partial^2_{xx} f'(t,{X}^\epsilon_t - \lambda  (X^\epsilon_t - \bar{X}_t))\Big](X^\epsilon_t - \bar{X}_t)^{\otimes 2} \nonumber\\
    &\qquad+ 2 \lambda \partial^2_{xx} f(t,{X}^\epsilon_t - \lambda (X^\epsilon_t - \bar{X}_t))(X^\epsilon_t - \bar{X}_t) \otimes \Big[f(t,{X}^\epsilon_t) - f(t,\bar{X}_t)\Big] \Big)\ud\lambda.\label{eq 80}
\end{align}
Indeed, by \eqref{eq 78} and the fundamental Theorem of calculus, the left hand side of the above identity can be written as  
\begin{align*}
    &\int_0^1\Big((\Phi^\epsilon)'_t(\lambda) (X^\epsilon_t - \bar{X}_t) + \Phi^\epsilon_t(\lambda)(X^\epsilon_t -\bar{X}_t)' \Big)\ud\lambda\\
    &= \Big[\partial^2_{xx}f(t,\bar{X}_t) f(t,\bar{X}_t) + \partial_x f'(t,\bar{X}_t)\Big] (X^\epsilon_t - \bar{X}_t) + \partial_{x}f(t,\bar{X}_t) \Big[f(t,{X}^\epsilon_t) - f(t,\bar{X}_t)\Big]\\
    &\qquad+ \int_0^1\int_0^\lambda \partial_\vartheta(\Phi^\epsilon)'_t(\vartheta) (X^\epsilon_t - \bar{X}_t) \ud \vartheta \ud\lambda + \int_0^1\int_0^\lambda \partial_\vartheta\Phi^\epsilon_t(\vartheta) (X^\epsilon_t - \bar{X}_t)' \ud \vartheta \ud\lambda.
\end{align*}
For the third term on the right hand side of the above identity, by Fubini's Theorem, we get
\begin{align*}
     &\int_0^1\int_0^\lambda \partial_\vartheta(\Phi^\epsilon)'_t(\vartheta)  (X^\epsilon_t - \bar{X}_t) \ud \vartheta \ud \lambda = \int_0^1\int_0^1 \mathds{1}_{[0,\lambda]}(\vartheta) \Big( \Big[\partial^3_{xxx} f(t,\bar{X}_t + \vartheta(X^\epsilon_t - \bar{X}_t)) \big(f_t(\bar{X}_t) \\ 
     &\qquad+ \vartheta [f_t({X}^\epsilon_t) -f_t(\bar{X}_t)]\big) + \partial^2_{xx} f'(t,{X}^\epsilon_t + \vartheta(X^\epsilon_t - \bar{X}_t)) \Big](X^\epsilon_t - \bar{X}_t)^{\otimes 2} \\
     &\qquad+ \partial^2_{xx} f(t,\bar{X}_t + \vartheta(X^\epsilon_t - \bar{X}_t)) \big(f_t({X}^\epsilon_t) - f_t(\bar{X}_t)\big) \otimes (X^\epsilon_t - \bar{X}_t) \bigg) \ud \vartheta \ud \lambda\\
      &=\int_0^1 (1-\vartheta) \Big( \Big[\partial^3_{xxx} f(t,\bar{X}_t + \vartheta(X^\epsilon_t - \bar{X}_t)) \big(f_t(\bar{X}_t) \\ 
     &\qquad+ \vartheta [f_t({X}^\epsilon_t) -f_t(\bar{X}_t)]\big) + \partial^2_{xx} f'(t,{X}^\epsilon_t + \vartheta (X^\epsilon_t - \bar{X}_t)) \Big](X^\epsilon_t - \bar{X}_t)^{\otimes 2} \\
     &\qquad+ \partial^2_{xx} f(t,\bar{X}_t + \vartheta(X^\epsilon_t - \bar{X}_t)) \big(f_t({X}^\epsilon_t) - f_t(\bar{X}_t)\big) \otimes (X^\epsilon_t - \bar{X}_t) \Big) \ud \vartheta, 
\end{align*}
and for the fourth term
\begin{align*}
     &\int_0^1\int_0^\lambda \partial_\vartheta\Phi^\epsilon_t(\vartheta)  (X^\epsilon_t - \bar{X}_t)' \ud \vartheta \ud \lambda \\
     &= \int_0^1\int_0^1 \mathds{1}_{[0,\lambda]}(\vartheta) \partial^2_{xx} f(t,\bar{X}_t + \vartheta(X^\epsilon_t - \bar{X}_t)) (X^\epsilon_t - \bar{X}_t) \otimes (X^\epsilon_t - \bar{X}_t)' \ud \vartheta \ud \lambda\\
     &= \int_0^1 (1-\vartheta)\partial^2_{xx} f(t,\bar{X}_t + \vartheta(X^\epsilon_t - \bar{X}_t)) (X^\epsilon_t - \bar{X}_t)\otimes (X^\epsilon_t - \bar{X}_t)'  \ud \vartheta.
\end{align*}
Substituting $\vartheta$ by $1-\lambda$, we prove \eqref{eq 80}. Thus, we define
\begin{align*}
    \overline{\Phi}^\epsilon_t(\lambda) &:= \lambda \partial^2_{xx} f(t,{X}^\epsilon_t - \lambda (X^\epsilon_t - \bar{X}_t)),\\
    (\overline{\Phi}^\epsilon)'_t(\lambda) &:=  \lambda\partial^3_{xxx} f(t,{X}^\epsilon_t - \lambda  (X^\epsilon_t - \bar{X}_t)) \big(f_t({X}^\epsilon_t) - \lambda [f_t({X}^\epsilon_t) -f_t(\bar{X}_t)]\big) \nonumber\\
    &\qquad+ \lambda\partial^2_{xx} f'(t,{X}^\epsilon_t - \lambda  (X^\epsilon_t - \bar{X}_t)).
\end{align*}
From Lemma \ref{lemma 4.1} and its proof, along with the boundness properties of $f, \partial_{xx}^2 f, \partial^3_{xxx} f, \partial_{xx}^2f'$ and \cite[Proposition 4.5]{FHL21}, we see that $\big(\overline{\Phi}^\epsilon(\lambda),(\overline{\Phi}^\epsilon)'(\lambda)\big)\in \mathbf{D}_{\eta}^{\alpha,\alpha} \mathbb{L}^{{k},\infty}_{\lambda-ex,lin}\big([0,T];L(\mathbb{R}^{d\times d},\mathbb{R}^{d\times \ell})\big)$ and
\begin{align}\label{eq. 86}
\sup_{\epsilon > 0}\sup_{\lambda\in[0,1]}\Big\{\sup_{t\in[0,T]}\|\overline{\Phi}^\epsilon_t(\lambda)\|_\infty +  \sup_{t\in[0,T]}\|(\overline{\Phi}^\epsilon)'_t(\lambda)\|_\infty
+\big\|\big(\overline{\Phi}^\epsilon_\cdot(\lambda),(\overline{\Phi}^\epsilon)'_\cdot(\lambda)\big)\big\|_{\mathbf{D}_{\eta}^{\alpha,\alpha} \mathbb{L}^{{k},\infty}_{lin}} \Big\} < \infty.
\end{align}
Hence, from Proposition \ref{Prop 3.3}, Proposition \ref{Prop 3.5} and Proposition \ref{prop 4.8}, by subtracting $V^\epsilon + \zeta^\epsilon$ we see that
\begin{multline}
     X^\epsilon_t - \bar{X}_t - V^\epsilon_t - \zeta^\epsilon_t = \overline{F}^\epsilon_t + 
     \int_0^t \partial_{x} b(s,\bar{X}_s,\bar{u}_s) (X^\epsilon_s - \bar{X}_s - V^\epsilon_s - \zeta^\epsilon_s) \ud s\\
     +  \int_0^t \partial_{x} \sigma(s,\bar{X}_s,\bar{u}_s) (X^\epsilon_s - \bar{X}_s - V^\epsilon_s - \zeta^\epsilon_s) \ud W_s\\
     + \int_0^t \big(\partial_x f(s, \bar{X}_s), \partial_{xx}f^2(s,\bar{X}_s) f(s,\bar{X}_s) + \partial_x f'(s,\bar{X}_s)\big) (X^\epsilon_s - \bar{X}_s - V^\epsilon_s - \zeta^\epsilon_s) \ud\boldsymbol{\eta}_s,
\end{multline}
where
\begin{align*}
    &\overline{F}^\epsilon_t := \int_0^t B^\epsilon_s \ud s + \int_0^t \Sigma^\epsilon_s \ud W_s \\
    &\qquad+ \int_0^t \Big(\frac{1}{2}\partial^2_{xx}f(s,\bar{X}_s),\frac{1}{2}\partial^3_{xxx}f(s,\bar{X}_s) f(s,\bar{X}_s) + \frac{1}{2}\partial^2_{xx}f'(s,\bar{X}_s)\Big) \big[(X^\epsilon_s - \bar{X}_s)^{\otimes 2} - (V^\epsilon_s)^{\otimes 2}\big] \ud\boldsymbol{\eta}_s\\
    &\qquad+ \int_0^t \big(H^\epsilon_s,(H^\epsilon)'_s\big) (X^\epsilon_s - \bar{X}_s)^{\otimes 2} \ud \boldsymbol{\eta}_s,
\end{align*}
with
\begin{align*}
     H^\epsilon_t &:= \int_0^1 
     \overline{\Phi}^\epsilon_t(\lambda)
     \ud \lambda - \frac{1}{2} \partial^2_{xx}f(t,\bar{X}_t),\\
    (H^\epsilon)'_t &: = \int_0^1 
    (\overline{\Phi}^\epsilon)'_t(\lambda)
    \ud \lambda - \frac{1}{2}\partial^3_{xxx}f(t,\bar{X}_t) f(t,\bar{X}_t) - \frac{1}{2}\partial^2_{xx}f'(t,\bar{X}_t),
\end{align*}
and
\begin{align*}
B^\epsilon_t &:= 
 \big[\partial_x b(t,\bar{X}_t,u^\epsilon_t) - \partial_x b(t,\bar{X}_t,\bar{u}_t) \big] (X^\epsilon_t - \bar{X}_t) \\
 &\qquad+  \bigg[\int_0^1 \lambda \partial^2_{xx} b(t,{X}^\epsilon_t - \lambda (X^\epsilon_t - \bar{X}_t),{u}^\epsilon_t) \ud \lambda - \frac{1}{2} \partial^2_{xx} b(t,\bar{X}_t,{u}^\epsilon_t)\bigg] (X^\epsilon_t - \bar{X}_t)^{\otimes 2}\\
 &\qquad + \frac{1}{2} \big[\partial^2_{xx} b(t,\bar{X}_t,{u}^\epsilon_t) - \partial^2_{xx} b(t,\bar{X}_t,\bar{u}_t) \big] (X^\epsilon_t - \bar{X}_t)^{\otimes 2}\\
 &\qquad+\frac{1}{2} \partial_{xx}^2 b(t,\bar{X}_t,\bar{u}_t) \big[(X^\epsilon_t - \bar{X}_t)^{\otimes 2} - (V^\epsilon_t)^{\otimes 2}\big],\\
 \vspace{.5cm}\\
 \Sigma^\epsilon_t &:= 
 \big[\partial_x \sigma(t,\bar{X}_t,u^\epsilon_t) - \partial_x \sigma(t,\bar{X}_t,\bar{u}_t) \big] (X^\epsilon_t - \bar{X}_t -V^\epsilon_t) \\
 &\qquad+  \bigg[\int_0^1 \lambda \partial^2_{xx} \sigma(t,{X}^\epsilon_t - \lambda (X^\epsilon_t - \bar{X}_t),{u}^\epsilon_t) \ud \lambda - \frac{1}{2} \partial^2_{xx} \sigma(t,\bar{X}_t,{u}^\epsilon_t)\bigg] (X^\epsilon_t - \bar{X}_t)^{\otimes 2}\\
 &\qquad + \frac{1}{2} \big[\partial^2_{xx} \sigma(t,\bar{X}_t,{u}^\epsilon_t) - \partial^2_{xx} \sigma(t,\bar{X}_t,\bar{u}_t) \big] (X^\epsilon_t - \bar{X}_t)^{\otimes 2}\\
 &\qquad+\frac{1}{2} \partial_{xx}^2\sigma(t,\bar{X}_t,\bar{u}_t) \big[(X^\epsilon_t - \bar{X}_t)^{\otimes 2} - (V^\epsilon_t)^{\otimes 2}\big].
\end{align*}
So, by Proposition \ref{Prop 3.5} and applying \eqref{eq:BV-primary-estimate} of Theorem \ref{thm 3.9} with forcing $\overline{F}^\epsilon$ we get \eqref{eq:diffusion-state-remainder}. Indeed, it only remains to estimate the following
\begin{multline*}
\|(\overline{F}^\epsilon,(\overline{F}^\epsilon)')\|_{
    \mathbf D_\eta^{(p,p)\text{-var}}\mathbb L^{\frac{k}{3}}} \leq
     \bigg\|\bigg(\int_0^\cdot B^\epsilon_s \ud s, 0\bigg) \bigg\|_{
    \mathbf D_\eta^{(p,p)\text{-var}}\mathbb L^{\frac{k}{3}}} +  \bigg\|\bigg(\int_0^\cdot \Sigma^\epsilon_s \ud W_s, 0\bigg) \bigg\|_{
    \mathbf D_\eta^{(p,p)\text{-var}}\mathbb L^{\frac{k}{3}}} \\
    + \bigg\|\bigg(\int_0^\cdot\frac{1}{2}\partial^2_{xx}f(s,\bar{X}_s) \big[(X^\epsilon_s - \bar{X}_s)^{\otimes 2} - (V^\epsilon_s)^{\otimes 2}\big]\ud \boldsymbol{\eta}_s, \frac{1}{2}\partial^2_{xx}f(\cdot,\bar{X})\big[(X^\epsilon - \bar{X})^{\otimes 2} - (V^\epsilon)^{\otimes 2}\big]\bigg) \bigg\|_{
    \mathbf D_\eta^{(p,p)\text{-var}}\mathbb L^{\frac{k}{3}}}\\
    + \bigg\|\bigg(\int_0^\cdot H^\epsilon_s (X^\epsilon_s - \bar{X}_s)^{\otimes 2} \ud \boldsymbol{\eta}_s, H^\epsilon (X^\epsilon - \bar{X})^{\otimes 2}\bigg) \bigg\|_{
    \mathbf D_\eta^{(p,p)\text{-var}}\mathbb L^{\frac{k}{3}}}.
\end{multline*}
Now, for the second term, let $0\leq u_1 \leq u_2 \leq T$, by BDG and Minkowski's integral inequalities we get
\begin{align*}
    \bigg\|\int_{u_1}^{u_2} \Sigma^\epsilon_s \ud W_s \bigg\|_{\frac{k}{3}} &\lesssim  \bigg\|\bigg(\int_{u_1}^{u_2} |\Sigma^\epsilon_s|^2 \ud s\bigg)^{\frac{1}{2}} \bigg\|_{\frac{k}{3}} \leq \bigg(\int_{u_1}^{u_2} \|\Sigma^\epsilon_s\|_{\frac{k}{3}}^2 \ud s \bigg)^{\frac{1}{2}}.
\end{align*}
Hence,
\begin{align*}
    \bigg\|\bigg( \int_0^\cdot \Sigma^\epsilon_s \ud W_s, 0\bigg) \bigg\|_{
    \mathbf D_\eta^{(p,p)\text{-var}}\mathbb L^{\frac{k}{3}}} &\lesssim \bigg(\int_{0}^{T} \|\Sigma^\epsilon_s\|_{\frac{k}{3}}^2 \ud s \bigg)^{\frac{1}{2}}
\end{align*}
and we conclude as in \cite[eq. 4.52, p. 133]{yong1999stochastic}, note only that by H\"older's inequality it holds 
\begin{equation}\label{ineq 86}
\|\cdot \cdot\hspace{.1cm}\|_{\frac{k}{3}} \leq \|\cdot\|_{\frac{k}{2}} \|\cdot\|_{k}.
\end{equation}
The rough variation seminorm of the drift term is estimated in essentially the same way as above while also using the contraction property of conditional expectation. For the first rough integral term, its rough variation seminorm is estimated from applying \eqref{ineq 12}, together with \eqref{eq. 018} (where $(X^\epsilon - \bar{X})^{\otimes 2} - (V^\epsilon)^{\otimes 2}$ is in place of $X$ and (0,0) in place of $(F,F')$), the H\"older inequality \eqref{ineq 86} and \eqref{eq:state-diff}, \eqref{eq:diffusion-first-order-size}, \eqref{eq:diffusion-state-remainder_0} such that
\begin{align*}
    &\bigg\|\bigg(\int_0^\cdot\frac{1}{2}\partial^2_{xx}f(s,\bar{X}_s) \big[(X^\epsilon_s - \bar{X}_s)^{\otimes 2} - (V^\epsilon_s)^{\otimes 2}\big]\ud \boldsymbol{\eta}_s, \frac{1}{2}\partial^2_{xx}f(\cdot,\bar{X})\big[(X^\epsilon - \bar{X})^{\otimes 2} - (V^\epsilon)^{\otimes 2}\big]\bigg) \bigg\|_{
    \mathbf D_\eta^{(p,p)\text{-var}}\mathbb L^{\frac{k}{3}}}\\
    &\lesssim \big\|\big((X^\epsilon - \bar{X})^{\otimes 2} - (V^\epsilon)^{\otimes 2},\big[(X^\epsilon - \bar{X})^{\otimes 2} - (V^\epsilon)^{\otimes 2}\big]'\big) \big\|_{
    \mathbf D_\eta^{(p,p)\text{-var}}\mathbb L^{\frac{k}{3}}}\\
    &\lesssim \big\|\big((X^\epsilon - \bar{X} - V^\epsilon)\otimes (X^\epsilon - \bar{X}),\big[(X^\epsilon - \bar{X} - V^\epsilon)\otimes (X^\epsilon - \bar{X})\big]'\big) \big\|_{
    \mathbf D_\eta^{(p,p)\text{-var}}\mathbb L^{\frac{k}{3}}}\\
    &\qquad+ \big\|\big( V^\epsilon\otimes (X^\epsilon - \bar{X} - V^\epsilon), [V^\epsilon \otimes (X^\epsilon - \bar{X} - V^\epsilon)\big]'\big) \big\|_{
    \mathbf D_\eta^{(p,p)\text{-var}}\mathbb L^{\frac{k}{3}}}\\
    &\lesssim \Big[\|X^{\boldsymbol{\eta},u^\epsilon}-\bar{X} - V^\epsilon\|_{\mathbb S^{\frac{k}{2}}}
 +\|(X^{\boldsymbol{\eta},u^\epsilon}-\bar{X} - V^\epsilon,(X^{\boldsymbol{\eta},u^\epsilon}-\bar{X} - V^\epsilon)')\|_
 {\mathbf D_\eta^{(p,p)\text{-var}}\mathbb L^{\frac{k}{2}}} \Big]\\
 &\qquad \cdot \Big[\|X^{u^\epsilon} - \bar{X}\|_{\mathbb S^k}
 +\|(X^{u^\epsilon} - \bar{X},(X^{u^\epsilon} - \bar{X})')\|_
 {\mathbf D_\eta^{(p,p)\text{-var}}\mathbb L^k} \Big]\\
 &\qquad+ \Big[\|V^\epsilon\|_{\mathbb S^k}
 + \|(V^\epsilon,(V^\epsilon)')\|_
 {\mathbf D_\eta^{(p,p)\text{-var}}\mathbb L^k}\Big]\\
 &\qquad\cdot \Big[\|X^{\boldsymbol{\eta},u^\epsilon}-\bar{X} - V^\epsilon\|_{\mathbb S^{\frac{k}{2}}}
 +\|(X^{\boldsymbol{\eta},u^\epsilon}-\bar{X} - V^\epsilon,(X^{\boldsymbol{\eta},u^\epsilon}-\bar{X} - V^\epsilon)')\|_
 {\mathbf D_\eta^{(p,p)\text{-var}}\mathbb L^{\frac{k}{2}}} \Big]\\
&\lesssim \epsilon^{\frac{3}{2}}.
\end{align*}
For the second rough integral term, its rough variation seminorm is estimated by noticing the uniform bound over $\epsilon$ from \eqref{eq. 86} and applying \eqref{ineq 12}, together with the structure of \eqref{ineq 13}–\eqref{ineq 15} (where $(X^\epsilon - \bar{X})^{\otimes 2}$ is in place of $X$ and (0,0) in place of $(F,F')$), the H\"older inequality \eqref{ineq 86} and \eqref{eq:state-diff} such that
\begin{align*}
    &\bigg\|\bigg(\int_0^\cdot H^\epsilon_s (X^\epsilon_s - \bar{X}_s)^{\otimes 2} \ud \boldsymbol{\eta}_s, H^\epsilon (X^\epsilon - \bar{X})^{\otimes 2}\bigg) \bigg\|_{
    \mathbf D_\eta^{(p,p)\text{-var}}\mathbb L^{\frac{k}{3}}}\\
    &\lesssim \Big[\big\|\big({H}^\epsilon,({H}^\epsilon)'\big)\big\|_{\mathbf{D}_{\eta}^{(p,p)\text{-var}} \mathbb{L}^{k}_{lin}} + \sup_{t\in[0,T]}\|{H}^\epsilon_t\|_k + \sup_{t\in[0,T]}\|({H}^\epsilon)'_t \|_k\Big]\\
    &\cdot \Big[\|(X^\epsilon - \bar{X})^{\otimes 2}\|_{\mathbb{S}^{\frac{k}{2}}}+\big\|\big((X^\epsilon - \bar{X})^{\otimes 2},(X^\epsilon - \bar{X})'\otimes(X^\epsilon - \bar{X})+(X^\epsilon - \bar{X})\otimes(X^\epsilon - \bar{X})' \big) \big\|_{
    \mathbf D_\eta^{(p,p)\text{-var}}\mathbb L^{\frac{k}{2}}}\Big]\\
    &\lesssim  \Big[\|X^{u^\epsilon} - \bar{X}\|_{\mathbb S^k}
 +\|(X^{u^\epsilon} - \bar{X},(X^{u^\epsilon} - \bar{X})')\|_
 {\mathbf D_\eta^{(p,p)\text{-var}}\mathbb L^k}\Big] ^3 \lesssim \epsilon^{\frac{3}{2}}.
\end{align*}
The assumed inequality
\begin{multline*}
     \big\|\big({H}^\epsilon,({H}^\epsilon)'\big)\big\|_{\mathbf{D}_{\eta}^{(p,p)\text{-var}} \mathbb{L}^{k}_{lin}} + \sup_{t\in[0,T]}\|{H}^\epsilon_t\|_k + \sup_{t\in[0,T]}\|({H}^\epsilon)'_t \|_k\ \\
     \lesssim \|X^{u^\epsilon} - \bar{X}\|_{\mathbb S^k}
 +\|(X^{u^\epsilon} - \bar{X},(X^{u^\epsilon} - \bar{X})')\|_
 {\mathbf D_\eta^{(p,p)\text{-var}}\mathbb L^k}
\end{multline*}
is proved with the same method as \eqref{eq. 091} above.

Lastly, for \eqref{eq 76}, with the help of \eqref{eq:state-diff}
--\eqref{eq:diffusion-state-remainder}, it follows with exactly the same calculations as in the proof of \cite[eq. 4.31, p. 128]{yong1999stochastic}.
\end{proof}

\subsubsection{Extracting the SMP}
First, we need to strengthen our assumptions with the following.

\begin{assumptions}\label{assu: per PMP}
Let $\alpha \in (0,1), k, m, p \in (2,\infty)$ such that 
 $\frac{1}{3}< \frac{1}{p} \leq \alpha < \frac{1}{2}$.
 
    \begin{itemize}
    \item[(i)] It holds 
    $$
\frac{2}{k} + \frac{1}{m} \leq \frac{1}{2}.
    $$
    \item[(ii)] It holds 
 $$2\alpha + \alpha\wedge \frac{1}{m}>1.$$
    
    \item[(iii)] It holds
    $$
      m \leq p.
    $$
    \end{itemize}
\end{assumptions}
Note that by combining assumptions (i) and (iii) of Assumptions \ref{assu: per PMP} we get that necessarily $k \in (12,\infty)$.

Next, we define the Hamiltonian $\bar{H}:\Omega \times [0,T]\times \mathbb{R}^d \times \mathbb{R}^d \times \mathbb{R}^{d\times w} \times U\rightarrow \mathbb{R}$ as
\begin{equation}\label{eq bar_Ham_lin}
\bar{H}(\omega,t,x,y,z,u):= y^\top b(\omega,t,x,u) + \text{Tr}\big(z^\top\sigma(\omega,t,x,u)\big) + h(\omega,t,x,u).
\end{equation}
We consider the first order adjoint processes $((\bar{Y},\bar{Y}'),\bar{Z})$ which satisfy the following equation
\begin{equation}
    \bar{Y}_t= \partial_x g(\bar{X}_T)^\top + \int_t^T \partial_x \bar{H}(s,\bar{X}_s,\bar{Y}_s,\bar{Z}_s,\bar{u}_s)^\top \ud s+ \int_t^T \partial_x f(s,\bar{X}_s)^\top \bar{Y}_s \ud \boldsymbol{\eta}_s - \int_t^T \bar{Z}_s \ud W_s.
\end{equation} 
Under the assumption of Lemma \ref{prop:diffusion-spike-estimates}, the above rough BSDE is well-posed by Theorem \ref{ex_uni_lin_bsde}, with 
\begin{align*}
    \big((\bar{Y},\bar{Y}'),\bar{Z}\big) \in  \mathbf{D}^{(p,p)\text{-var}}_{\eta}\mathbb{L}^m([0,T];\mathbb{R}^{d})
 \times
 \mathbb H^m ([0,T];\mathbb R^{d\times w}).
\end{align*}
Moreover, we also see that      
\begin{align*}
    \bar{Y}_t = \bar{Y}_0  - \int_0^t \partial_x \bar{H}(s,\bar{X}_s, \bar{Y}_s, \bar{Z}_s,\bar{u}_s)^\top \ud s - \int_0^t \partial_x f(s,\bar{X}_s)^\top \bar{Y}_s \ud \boldsymbol{\eta}_s + \int_0^t \bar{Z}_s \ud W_s.
\end{align*}
Applying the rough product formula Theorem \ref{rough_product_formula} directly, while noting (i) and (ii) of Remark \ref{rmk 3.21},
we immediately see as in 
\ref{subsubsub 4.1.3} that 
\begin{equation}\label{eq 84}
    \mathbb{E}\big[\bar{Y}_T^\top V^\epsilon_T\big] = \mathbb{E}\Bigg[\int_0^T \Big( \mathds{1}_I(s) \text{Tr}\big(\bar{Z}_s^\top  \delta\sigma_s \big) - \partial_x h(s,\bar{X}_s,\bar{u}_s) V^\epsilon_s\Big)\ud s\Bigg].
\end{equation} 
Applying again the rough product formula, we get 
\begin{multline}\label{eq 85}
    \mathbb{E}\big[\bar{Y}_T^\top \zeta^\epsilon_T\big] = \mathbb{E}\Bigg[\int_0^T \Big( \frac{1}{2}\bar{Y}_T^\top\partial_{xx}^2 b(s,\bar{X}_s, \bar{u}_s) (V^\epsilon_s)^{\otimes 2} + \mathds{1}_I(s) \bar{Y}_T^\top \delta b_s  + \frac{1}{2} \text{Tr}\big(\bar{Z}_s^\top  \partial_{xx}^2\sigma(s,\bar{X}_s, \bar{u}_s) (V^\epsilon_s)^{\otimes 2}\big)\\
     + \mathds{1}_I(s) \text{Tr}(\bar{Z}_s^\top \delta \sigma_{x,s} V^\epsilon_s ) - \partial_x h(s,\bar{X}_s,\bar{u}_s) \zeta^\epsilon_s\Big)\ud s\Bigg] 
     + \mathbb{E}\Bigg[\int_0^T \frac{1}{2} \bar{Y}_s^\top \big(\partial_{xx}^2f(s,\bar{X}_s) (V^\epsilon_s)^{\otimes 2} 
     \ud\boldsymbol{\eta}_s \Bigg].
\end{multline}
Then, by adding together \eqref{eq 84} and \eqref{eq 85}, with the help of \eqref{eq:diffusion-first-order-size} we get
\begin{multline*}
    \mathbb{E}\Big[
 \partial_x g(\bar X_T)(V^\epsilon_T + \zeta^\epsilon_T) \Big] \\
 = \mathbb{E}\Bigg[\int_0^T \Big( \frac{1}{2}\bar{Y}_T^\top\partial_{xx}^2 b(s,\bar{X}_s, \bar{u}_s) (V^\epsilon_s)^{\otimes 2} + \mathds{1}_I(s) \bar{Y}_T^\top \delta b_s  + \frac{1}{2} \text{Tr}\big(\bar{Z}_s^\top  \partial_{xx}^2\sigma(s,\bar{X}_s, \bar{u}_s) (V^\epsilon_s)^{\otimes 2}\big) \\+ \mathds{1}_I(s) \text{Tr}\big(\bar{Z}_s^\top  \delta\sigma_s \big) - \partial_x h(s,\bar{X}_s,\bar{u}_s) (V^\epsilon_s+\zeta^\epsilon_s)\Big)\ud s\Bigg] 
 + \mathbb{E}\Bigg[\int_0^T \frac{1}{2} \bar{Y}_s^\top \big(\partial_{xx}^2f(s,\bar{X}_s) (V^\epsilon_s)^{\otimes 2} 
     \ud\boldsymbol{\eta}_s \Bigg] + o(\epsilon).
\end{multline*}
So, by \eqref{eq 76} and the optimality of $(\bar u_t)_{t\in[0,T]}$ we have
\begin{multline}\label{eq 87}
    0 \leq \mathcal{J}(u^\epsilon) - \mathcal{J}(\bar u) = \mathbb{E}\Bigg[ \frac{1}{2}\partial_{xx}^2 g(\bar X_T)
       (V^\epsilon_T)^{\otimes 2} \\+ \frac{1}{2} \int_0^T \Big(
 \partial_{xx}^2 h(s,\bar{X}_s,\bar{u}_s) (V^\epsilon_s)^{\otimes 2} +\bar{Y}_s^\top\partial_{xx}^2 b(s,\bar{X}_s, \bar{u}_s) (V^\epsilon_s)^{\otimes 2} +\text{Tr}\big(\bar{Z}_s^\top  \partial_{xx}^2\sigma(s,\bar{X}_s, \bar{u}_s) (V^\epsilon_s)^{\otimes 2}\big) \Big)\ud s\\
 + \int_0^T \mathds{1}_{I}(s)
 \Big(\bar{Y}_T^\top \delta b_s + \text{Tr}\big(\bar{Z}_s^\top  \delta\sigma_s \big) + \bigl[h(s,\bar X_s, u^\epsilon_s)-h(s,\bar X_s,\bar u_s)\bigr] \Big) \ud s \Bigg]\\
+  \mathbb{E}\Bigg[\int_0^T \frac{1}{2} \bar{Y}_s^\top \big(\partial_{xx}^2f(s,\bar{X}_s) (V^\epsilon_s)^{\otimes 2} \big)
     \ud\boldsymbol{\eta}_s \Bigg] + o(\epsilon)\\
     = \mathbb{E}\Big[\frac{1}{2}\partial_{xx}^2 g(\bar X_T)
       (V^\epsilon_T)^{\otimes 2}\Big]+ \mathbb{E}\Bigg[\int_0^T\Big( \frac12 \partial_{xx}^2 \bar{H}(s,\bar{X}_s, \bar{Y}_s, \bar{Z}_s,\bar{u}_s) (V^\epsilon_s)^{\otimes 2} \\
    + \mathds{1}_I(s) \big[\bar{H}(s,\bar{X}_s, \bar{Y}_s, \bar{Z}_s,{u}^\epsilon_s) - \bar{H}(s,\bar{X}_s, \bar{Y}_s, \bar{Z}_s,\bar{u}_s) \big] \Big) \ud s \Bigg] \\
    + \mathbb{E}\Bigg[\int_0^T \frac{1}{2} \bar{Y}_s^\top \big(\partial_{xx}^2f(s,\bar{X}_s) (V^\epsilon_s)^{\otimes 2} \big) 
     \ud\boldsymbol{\eta}_s \Bigg] + o(\epsilon).
\end{multline}
At this point we introduce the second order adjoint processes $((\bar{P},\bar{P}'),\bar{Q})$ which satisfy the following equation
\begin{multline}\label{eq 88}
    \bar{P}_t = \partial_{xx}^2 g(\bar X_T) + \int_t^T \Big(\partial_{xx}^2 \bar{H}(s,\bar{X}_s, \bar{Y}_s, \bar{Z}_s,\bar{u}_s) + \partial_xb(s,\bar{X}_s,\bar{u}_s)^\top \bar{P}_s + \bar{P}_s \partial_xb(s,\bar{X}_s,\bar{u}_s) \\
    + \sum_{j=1}^q \big[\partial_x\sigma^j(s,\bar{X}_s,\bar{u}_s)^\top \bar{P}_s \partial_x\sigma^j(s,\bar{X}_s,\bar{u}_s)\big] + \sum_{j=1}^q \big[\partial_x\sigma^j(s,\bar{X}_s,\bar{u}_s)^\top \bar{Q}^j_s + \bar{Q}^j_s \partial_x\sigma^j(s,\bar{X}_s,\bar{u}_s)\big]\Big) \ud s\\
    + \int_t^T\Big(\partial_xf(s,\bar{X}_s)^\top \bar{P}_s + \bar{P}_s \partial_xf(s,\bar{X}_s) +  \bar{Y}_s^\top \partial_{xx}^2f(s,\bar{X}_s)\Big) \ud \boldsymbol{\eta}_s - \int_t^T \bar{Q}_s \ud W_s,
\end{multline}
where $\sigma^j$ is the $j$-th column of $\sigma$, $\bar{Q}:[0,T]\rightarrow L(\mathbb{R}^w,\mathbb{R}^{d\times d})$ and $\bar{Q}^j$ the $j$-th column\footnote{Which actually is an element of $\mathbb{R}^{d\times d}$.} of $\bar{Q}$. Also, note that $\partial_{xx}^2f(\cdot,\bar{X})$ can be viewed as an element of $L(\mathbb{R}^{d\times d},\mathbb{R}^{d\times \ell})\cong B(\mathbb{R}^{d\times d}\times\mathbb{R}^\ell,\mathbb{R}^d)$, and consequently $Y^\top \partial_{xx}^2f(\cdot,\bar{X})$ as an element of $B(\mathbb{R}^{d\times d}\times\mathbb{R}^\ell,\mathbb{R})
\cong L(\mathbb{R}^\ell,L(\mathbb{R}^{d\times d},\mathbb{R}))
\cong L(\mathbb{R}^\ell,\mathbb{R}^{d\times d})$. Hence, because $m\leq k$, by standard calculations, as for example in \cite[Proposition 3.3]{bugini2024malliavin}, we have that $(\bar{Y}^\top \partial_{xx}^2f(\cdot,\bar{X}),(\bar{Y}')^\top \partial_{xx}^2f(\cdot,\bar{X}) + \bar{Y}^\top [\partial_{xxx}^3f(\cdot,\bar{X}) f(\cdot,\bar{X}) + \partial_{xx}^2f'(\cdot,\bar{X})]) \in \mathbf{D}^{(p,p)\text{-var}}_{\eta}\mathbb{L}^m([0,T];L(\mathbb{R}^\ell,\mathbb{R}^{d\times d}))$. From Theorem \ref{ex_uni_lin_bsde} we get that \eqref{eq 88} is well posed with
\begin{equation*}
 \big((\bar{P},\bar{P}'),\bar{Q}\big)\in \mathbf{D}^{(p,p)\text{-var}}_{\eta}\mathbb{L}^m\big([0,T];\mathbb{R}^{d\times d}\big)
 \times
 \mathbb H^m \big([0,T]; L(\mathbb{R}^w,\mathbb R^{d\times d})\big).
 \end{equation*}
Then, by applying the rough product formula Theorem \ref{rough_product_formula} two consecutive times, while noting (i) and (ii) of Remark \ref{rmk 3.21}, we calculate $\text{Tr}\big([V^\epsilon_t (V^\epsilon_t)^\top] \bar{P}_t\big)$\footnote{We remind that from the cyclic property of trace we have $\text{Tr}\big([V^\epsilon_t (V^\epsilon_t)^\top] \bar{P}_t\big) = (V^\epsilon_t)^\top [\bar{P}_t V^\epsilon_t] = \bar{P}_t(V^\epsilon_t)^{\otimes2}$.}
to get from \eqref{eq 87} and Lemma \ref{prop:diffusion-spike-estimates} that
\begin{multline}\label{eq 89}
    0 \leq \mathcal{J}(u^\epsilon) - \mathcal{J}(\bar u)\\
    = \mathbb{E}\Bigg[\int_0^T\mathds{1}_I(s) \Big( \big(\bar{H}(s,\bar{X}_s, \bar{Y}_s, \bar{Z}_s,{u}^\epsilon_s) - \bar{H}(s,\bar{X}_s, \bar{Y}_s, \bar{Z}_s,\bar{u}_s) + \frac{1}{2} \text{Tr} \big((\delta\sigma_s)^\top \bar{P}_s\delta\sigma_s\big)\Big) \ud s \Bigg] + o(\epsilon).
\end{multline}
Finally, we have our maximum principle.

\begin{thm}[SMP for rough stochastic control]\label{thm:roughpmp}
Let Assumptions \ref{P-control}--
\ref{D-cost}, \ref{assu: per PMP} hold true, $\bar{H}$ given by \eqref{eq bar_Ham_lin}, $((\bar{X},\bar{X}'),(\bar u_t)_{t\in[0,T]})$ an optimal pair for $\mathcal{J}$, and $(\bar{Y},\bar{Y}'), \bar{Z}, (\bar{P},\bar{P}'), \bar{Q}$ defined as above. Then, for every $u \in U$ we have
\begin{multline*}
    \bar{H}(t,\bar{X}_t, \bar{Y}_t, \bar{Z}_t,\bar{u}_t) - \bar{H}(t,\bar{X}_t, \bar{Y}_t, \bar{Z}_t,u)\\
    - \frac{1}{2}
\text{Tr} \big([\sigma(t,\bar{X}_t,u)- \sigma(t,\bar{X}_t,\bar{u}_t)]^\top \bar{P}_t[\sigma(t,\bar{X}_t,u) - \sigma(t,\bar{X}_t,\bar{u}_t)]\big) \leq 0, \hspace{.2cm}\mathbb{P}\otimes\text{Leb}(\ud t)-a.s.
\end{multline*}		 
\end{thm}

\begin{proof}
The desired result follows from \eqref{eq 89} by a standard calculation.
\end{proof}

\subsubsection{Sufficiency conditions}
In this part we shall additionally need the following assumption.
\begin{enumerate}[label=(B5), ref=B5]
    \item\label{(B5)} Let $U$ be a convex subset of $\mathbb{R}^\kappa$ with nonempty interior, and the maps $b, \sigma, h$ locally Lipschitz in $u\in U$.  
\end{enumerate}

Let now $((\bar{X},\bar{X}'),(\bar u_t)_{t\in[0,T]})$ be just an admissible pair, and not a priori optimal. Nevertheless, the corresponding $(\bar{Y},\bar{Y}'), \bar{Z}, (\bar{P},\bar{P}'), \bar{Q}$ can be defined exactly as before. Hence, we can define $\bar{\mathcal{H}} : \Omega \times [0,T] \times \mathbb{R}^d \times U$ as
\begin{align}\label{eq. 104}
    \bar{\mathcal{H}}(\omega,t,x,u) &:= \bar{H}(\omega,t,\bar{X}_t, \bar{Y}_t, \bar{Z}_t,u)  - \frac{1}{2}
\text{Tr} \big(\sigma(\omega,t,\bar{X}_t,\bar{u}_t)^\top \bar{P}_t  \sigma(\omega,t,\bar{X}_t,\bar{u}_t)\big) \nonumber\\
    &\qquad+ \frac{1}{2}
\text{Tr} \big([\sigma(\omega,t,x,u)- \sigma(\omega,t,\bar{X}_t,\bar{u}_t)]^\top \bar{P}_t [\sigma(\omega,t,x,u) - \sigma(\omega,t,\bar{X}_t,\bar{u}_t)]\big).
\end{align}
Observe that the condition of Theorem \ref{thm:roughpmp} can be equivalently written as 
\begin{equation}\label{eq. 92}
     \bar{\mathcal{H}}(t,\bar{X}_t,\bar{u}_t) = \min_{u\in U}  \bar{\mathcal{H}}(t,\bar{X}_t,u),\hspace{.2cm} \mathbb{P}\otimes\text{Leb}(\ud t)-a.s.
\end{equation}

\begin{thm}\label{thm 4.14}
Let Assumptions \ref{P-control}--
\ref{(B5)}, \ref{assu: per PMP} hold true, $\bar{\mathcal{H}}$ given by \eqref{eq. 104}, $((\bar{X},\bar{X}'),(\bar u_t)_{t\in[0,T]})$ an admissible pair, and $(\bar{Y},\bar{Y}'), \bar{Z}, (\bar{P},\bar{P}'), \bar{Q}$ defined as above. If the function $g$ is convex in $x\in\mathbb{R}^d$, 
the function $\bar{H}(t,\cdot,\bar{Y}_t,\bar{Z}_t,\cdot)$ is convex in $(x,u)\in \mathbb{R}^d\times U$, and 
\begin{equation*}
     \bar{\mathcal{H}}(t,\bar{X}_t,\bar{u}_t) = \min_{u\in U}  \bar{\mathcal{H}}(t,\bar{X}_t,u),\hspace{.2cm} \mathbb{P}\otimes\text{Leb}(\ud t)-a.s.,
\end{equation*}
then for every admissible pair $(({X},X'),(v_t)_{t\in[0,T]})$ we have 
\begin{align*}
    \mathcal{J}(v) - \mathbb{E}\bigg[\int_0^T \big(\bar{Y}_s^\top f(s,{X}_s) - \bar{Y}_s^\top f(s,\bar{X}_s) -\bar{Y}_s^\top \partial_xf(s,\bar{X}_s) [X_s - \bar{X}_s] \big)
     \ud\boldsymbol{\eta}_s \bigg]\geq \mathcal{J}(\bar{u}).
\end{align*}
\end{thm}

\begin{proof}
Duplicating the first part of the proof of \cite[Theorem 5.2]{yong1999stochastic} we get from the convexity of $\bar{H}(\omega,t,\cdot,\bar{Y}_t,\bar{Z}_t,\cdot)$ with respect to $(x,u)\in \mathbb{R}^d\times U$ and \eqref{eq. 92} that
\begin{align}\label{eq 99}
    \int_0^T \big( \bar{H}(t,{X}_t, \bar{Y}_t, \bar{Z}_t,v_t) -  \bar{H}(t,\bar{X}_t, \bar{Y}_t, \bar{Z}_t,\bar{u}_t)\big) \ud t \geq \int_0^T  \partial_x\bar{H}(t,\bar{X}_t, \bar{Y}_t, \bar{Z}_t,\bar{u}_t) (X_t - \bar{X}_t) \ud t, \hspace{.2cm}\mathbb{P}-a.s.,
\end{align}
for every admissible pair $\big({X},(v_t)_{t\in[0,T]}\big)$.
Next, the difference $X-\bar{X}$ satisfies
\begin{align*}
    X_t - \bar{X}_t &= \int_0^t \big[\partial_xb(s,\bar{X}_s,\bar{u}_s) (X_s - \bar{X}_s) + \alpha_s \big] \ud s + \int_0^t \big[\partial_x \sigma(s,\bar{X}_s,\bar{u}_s) (X_s - \bar{X}_s) + \beta_s \big] \ud W_s\\
    &\qquad+\int_0^t \big[\partial_xf(s,\bar{X}_s) (X_s - \bar{X}_s) + \gamma_s\big] \ud \boldsymbol{\eta}_s,
\end{align*}
where
\begin{align*}
    \alpha_t &:= - \partial_xb(t,\bar{X}_t,\bar{u}_t) (X_t - \bar{X}_t) + b(t,{X}_t,v_t) - b(t,\bar{X}_t,\bar{u}_t),\\
     \beta_t &:= - \partial_x\sigma(t,\bar{X}_t,\bar{u}_t) (X_t - \bar{X}_t) +\sigma(t,{X}_t,v_t) - \sigma(t,\bar{X}_t,\bar{u}_t),\\
     \gamma_t &:= - \partial_xf(t,\bar{X}_t) (X_t - \bar{X}_t) + f(t,{X}_t) - f(t,\bar{X}_t),\\
     \gamma'_t &:= - \partial_{xx}^2f(t,\bar{X}_t) [f(t,\bar{X}_t)\otimes (X_t - \bar{X}_t)] - \partial_{x}f'(t,\bar{X}_t)(X_t - \bar{X}_t) - \partial_{x}f(t,\bar{X}_t) \big(f(t,{X}_t) - f(t,\bar{X}_t)\big)\\
     &\qquad+ \partial_{x}f(t,{X}_t)f(t,{X}_t) +f'(t,{X}_t) - \partial_{x}f(t,\bar{X}_t)f(t,\bar{X}_t) - f'(t,\bar{X}_t).
\end{align*}
Applying the rough product formula Theorem \ref{rough_product_formula}, while noting (i) and (ii) of Remark \ref{rmk 3.21}, and then taking expectations we see as in \ref{subsubsub 4.1.3} that 
\begin{align*}
    &\mathbb{E}\big[\partial_xg(\bar{X}_T) (X_T - \bar{X}_T) \big] =  \mathbb{E}\big[\bar{Y}^\top_T (X_T - \bar{X}_T) \big] - \mathbb{E}\big[\bar{Y}^\top_0 (X_0 - \bar{X}_0) \big] \\
    &=  \mathbb{E}\bigg[\int_0^T \Big( Y_s^\top \alpha_s + \text{Tr}(Z_s^\top \beta_s) - \partial_x h(s,\bar{X}_s,\bar{u}_s) (X_s - \bar{X}_s) \Big)\ud s\bigg] + \mathbb{E}\bigg[\int_0^T\bar{Y}_s^\top \gamma_s
     \ud\boldsymbol{\eta}_s \bigg]\\
     &= \mathbb{E}\bigg[\int_0^T \Big( - \partial_x\bar{H}(s,\bar{X}_s, \bar{Y}_s, \bar{Z}_s,\bar{u}_s) (X_s - \bar{X}_s) + Y^\top_s [b(s,{X}_s,v_s) - b(s,\bar{X}_s,\bar{u}_s)] \\
     &\qquad+  \text{Tr}\big(Z_s^\top [\sigma(s,{X}_s,v_s) - \sigma(s,\bar{X}_s,\bar{u}_s)]\big) \Big)\ud s\bigg] + \mathbb{E}\bigg[\int_0^T\bar{Y}_s^\top \gamma_s
     \ud\boldsymbol{\eta}_s \bigg].
\end{align*}
So, by \eqref{eq 99} and \eqref{eq. 069} (with $\bar{H}$ in place of $H$) we have
\begin{align*}
     &\mathbb{E}\big[\partial_xg(\bar{X}_T) (X_T - \bar{X}_T) \big] \geq \mathbb{E}\bigg[\int_0^T \Big( - \big[\bar{H}(s,{X}_s, \bar{Y}_s, \bar{Z}_s,v_s) - \bar{H}(s,\bar{X}_s, \bar{Y}_s, \bar{Z}_s,\bar{u}_s) \big] \\
     &\qquad+ Y^\top_s [b(s,{X}_s,v_s) - b(s,\bar{X}_s,\bar{u}_s)] +  \text{Tr}\big(Z_s^\top [\sigma(s,{X}_s,v_s) - \sigma(s,\bar{X}_s,\bar{u}_s)]\big) \Big)\ud s\bigg] \\
     &\qquad+ \mathbb{E}\bigg[\int_0^T\bar{Y}_s^\top \gamma_s
     \ud\boldsymbol{\eta}_s \bigg]\\
     &=  \mathcal{J}(\bar{u}) - \mathcal{J}(u) + \mathbb{E}[g({X}_T) - g(\bar{X}_T)] + \mathbb{E}\bigg[\int_0^T\bar{Y}_s^\top \gamma_s
     \ud\boldsymbol{\eta}_s \bigg].
\end{align*}
Finally, by the convexity of $g$ with respect to $x\in \mathbb{R}^d$ we get that
\begin{align*}
    0 \geq \mathcal{J}(\bar{u}) -  \mathcal{J}(v) + \mathbb{E}\bigg[\int_0^T\bar{Y}_s^\top \gamma_s
     \ud\boldsymbol{\eta}_s \bigg].
\end{align*}
\end{proof}

\begin{coroll}\label{coroll 4.16}
Let the Assumptions of Theorem \ref{thm 4.14} hold true, and $(\bar{X},\bar{X}'),(\bar{Y},\bar{Y}')$ defined as above. If for every 
admissible pair $(({X},X'),(v_t)_{t\in[0,T]})$ holds
\begin{equation*}
    \mathbb{E}\bigg[\int_0^T  \big(\bar{Y}_s^\top f\big(s,{X}(s)\big) - \bar{Y}_s^\top f\big(s,\bar{X}(s)\big) - \bar{Y}_s^\top\partial_xf\big(s,\bar{X}(s)\big) \big[{X}(s) - \bar X(s)\big]\big)
     \ud\boldsymbol{\eta}_s \bigg] \geq 0,
\end{equation*}
then we have that $((\bar{X},\bar{X}'),(\bar{u}_t)_{t\in[0,T]})$ is an optimal pair.
\end{coroll}

\begin{proof}
Immediate from Theorem \ref{thm 4.14}.
\end{proof}

\section{Pathwise stochastic control with Brownian rough path noise}\label{sec:application}
In this section, let $(\Omega', \mathcal F', \mathbb P')$ and $(\Omega'', \mathcal F'', \mathbb P'')$ be probability spaces supporting Brownian motion $W$ and $W^\perp$, respectively. Let $\Omega:=\Omega'\times\Omega'', \mathbb{P}=\overline{\mathbb{P}'\otimes\mathbb{P}''}$ with  $W(\omega)=W(\omega'), {W}^\perp(\omega)= {W}^\perp (\omega'')$. We denote by $(\mathcal{F}_t)_{t\in\mathbb{R_+}}, (\mathcal{F}^{\perp}_t)_{t\in\mathbb{R_+}}$ the augmented filtrations generated by $W$ and $W^\perp$, respectively on $(\Omega', \mathcal F', \mathbb P')$ and $(\Omega'', \mathcal F'', \mathbb P'')$. It is well known that these filtrations satisfy the usual conditions, see for example \cite[Section 2.7]{karatzas2014brownian}. By abusing the notation we are going to denote the filtrations $(\mathcal{F}_t\times\Omega'')_{t\in\mathbb{R_+}}, (\Omega'\times \mathcal{F}^{\perp}_t)_{t\in\mathbb{R_+}}$ on $(\Omega, \mathbb P)$ by $(\mathcal{F}_t)_{t\in\mathbb{R_+}}, (\mathcal{F}^{\perp}_t)_{t\in\mathbb{R_+}}$. Similarly, we also use the same notations for random variables or processes on $(\Omega', \mathcal F', \mathbb P')$ (or $(\Omega'', \mathcal F'', \mathbb P'')$) and their natural extensions on $(\Omega, \mathbb P)$. Again, note that the filtrations $(\mathcal{F}_t\vee\mathcal{F}^{\perp}_t)_{t\in\mathbb{R}_+}, (\mathcal{F}_t\vee\mathcal{F}^{\perp}_T)_{t\in\mathbb{R}_+}$ also satisfy the usual conditions. We write $\mathbb{E}, \mathbb{E}', \mathbb{E}''$ as expectations with respect to $\mathbb{P}, \mathbb{P}', \mathbb{P}''$. Moreover, we denote by $\mathcal{O}'$ the optional $\sigma$-algebra on $\Omega'\times[0,T]$ with respect to filtration  $(\mathcal{F}_t)_{t\in\mathbb{R_+}}$.

Now we are going to lift $W^\perp$ to a geometric rough path.  To 
do that we define for each $(s,t)\in \Delta$, the Stratonovich integral 
$$
    (\mathbb{W}^{\perp}_{s,t})^{\mathrm{Strato}}:=\int_s^t \delta W^\perp_{s,r}\circ \ud W^{\perp}_r. 
$$ 
It is well known (see \cite[Chapter 3]{friz2020course}) that for all $\omega''$ outside a $\mathbb{P}''$-null set $\mathcal{N}''$ of $\Omega''$, the {\sl Brownian rough path}
$$
    \textbf{W}^{\perp} (\omega''):=\big(W^\perp(\omega''),(\mathbb{W}^\perp)^{\mathrm{Strato}}(\omega'')\big) 
$$    
belongs to the space $\mathscr{C}^{0,\alpha}_g$ for all $\alpha \in (1/3,1/2)$ and that the following {\it lifting map} is $(\mathcal{F}^\perp_t)_{t\in\mathbb{R}_+}-$prog. measurable:
\begin{equation*}\label{ito-lift}
\begin{array}{llll}
\textbf{W}^{\perp}: & [0,T]\times \Omega'' & \rightarrow & (\mathscr{C}^{0,\alpha}_g, \mathcal{B}(\mathscr{C}^{0,\alpha}_g))	\\
 & (t,\omega'') &\mapsto & \textbf{W}^{\perp}(\omega'')_{.\wedge t}.
\end{array}
\end{equation*}
For $\omega''\in \mathcal{N}''$ we set $\textbf{W}^{\perp}(\omega''):=(0,0)$, and the same convention will hold for every new process defined below whenever we randomize with respect to $\textbf{W}^{\perp}$. We consider the following spaces of admissible controls.
 

\begin{defn}\label{dfn 5.2}  Let $d\in\mathbb{N}, \alpha\in(\frac{1}{3},\frac{1}{2})$ and $0<\beta'\leq \beta\leq\alpha$ such that $\alpha+\beta+\beta'>1$.
    \begin{align*}
    \mathcal{A}&:= \big\{u:\Omega'\times[0,T]\rightarrow\mathbb{R}^\kappa:\text{prog. measurable with respect to}\hspace{.1cm}\nonumber\\&\hspace{2.3cm}(\mathcal{F}_s)_{s\in[0,T]},\hspace{.1cm}\text{and}\hspace{.1cm} \sup_{s\in[0,T]}\mathbb{E}'[|u_s|^k]<\infty \big\},\\
       \mathcal{A}^{1}&:= \big\{u:\Omega'\times\mathscr{C}^{0,\alpha}_g\times[0,T]\rightarrow\mathbb{R}^\kappa:\text{prog. measurable with respect to}\hspace{.1cm}\nonumber\\&\hspace{2.3cm}\big(\mathcal{F}_s\otimes\{\emptyset, \mathscr{C}^{0,\alpha}_g\}\big)_{s\in[0,T]},\footnotemark\hspace{.1cm}\text{and}\hspace{.1cm} u(\boldsymbol{\eta})\in \mathcal{A}\hspace{.1cm}\text{for every}\hspace{.1cm}\boldsymbol{\eta}\in \mathscr{C}^{0,\alpha}_g\big\},\\
        \mathcal{A}^{2}&:= \big\{u:\Omega'\times\mathscr{C}^{0,\alpha}_g\times[0,T]\rightarrow\mathbb{R}^\kappa:\text{prog. measurable with respect to}\hspace{.1cm}\nonumber\\&\hspace{2.3cm}\big(\mathcal{F}_s\otimes\mathcal{B}\big(\mathscr{C}^{0,\alpha}_g\big)\big)_{s\in[0,T]},\hspace{.1cm}\text{and}\hspace{.1cm} u(\boldsymbol{\eta})\in \mathcal{A}\hspace{.1cm}\text{for every}\hspace{.1cm}\boldsymbol{\eta}\in \mathscr{C}^{0,\alpha}_g\big\},\\ 
        \mathcal{A}^{\text{Causal}}&:= \big\{u:\Omega'\times\mathscr{C}^{0,\alpha}_g\times[0,T]\rightarrow\mathbb{R}^\kappa:\text{prog. measurable with respect to}\hspace{.1cm}\nonumber\\&\hspace{0.7cm}\big(\mathcal{F}_s\otimes \mathcal{B}(\mathscr{C}^{0,\alpha}_g)\big)_{s\in[0,T]}, u(\boldsymbol{\eta})\in \mathcal{A}\hspace{.1cm}\text{for every}\hspace{.1cm}\boldsymbol{\eta}\in \mathscr{C}^{0,\alpha}_g,\hspace{.1cm}\text{and}\hspace{.1cm}u_s(\omega',\boldsymbol{\eta})=u_s(\omega',\boldsymbol{\eta}_{\cdot\wedge s})\big\}.
    \end{align*}
    \end{defn}
    \footnotetext{$\mathcal{F}_t\otimes\{\emptyset,\mathscr{C}^{0,\alpha}_g\}=\{S\times \mathscr{C}^{0,\alpha}_g:S\in\mathcal{F}_t\}$.}

We also need the randomization of any $u \in \mathcal A^2$ to connect the rough stochastic control and pathwise stochastic control problems,
$$
\overline u_s(\omega',\omega'')
=
u_s\bigl(\omega',\mathbf W^\perp(\omega'')\bigr).
$$
Note that $\overline{u}$ is $(\mathcal{F}_s\vee\mathcal{F}^{\perp}_T)_{s\in[t,T]}$-progressively measurable.
\begin{rmk}\label{rmk 5.3}
Furthermore, note that $\mathcal{A}^1\subseteq\mathcal{A}^{\text{Causal}}\subseteq \mathcal{A}^2$, and for any $u\in\mathcal{A}^{\text{Causal}}$, $\overline{u}$ defined above are $(\mathcal{F}_t\vee\mathcal{F}^{\perp}_t)_{t\in\mathbb{R}_+}$ prog. measurable processes.
\end{rmk}

In general, the Gubinelli derivative depends on $\boldsymbol{\eta} \in \mathscr{C}^{0,\alpha}_g$, therefore, among other technicalities, for simplicity, we update our assumptions as seen below. Note, however, that in \ref{(A3-R)} the boundedness condition can be removed by employing \cite[Lemma 4.8]{horst2025pontryagin} and stronger (continuity) conditions on the control spaces, for more details on this route we refer to \cite[Section 4.3]{horst2025pontryagin}.

\begin{enumerate}[label=(A\arabic*-R), ref=(A\arabic*-R)]
    \item\label{(A1-R)}
    The functions $b, \sigma$ satisfy \ref{A1} while being independent from $\Omega$, \emph{i.e.} deterministic.

    \item\label{(A2-R)} The function $f$ is independent of $\Omega$, \emph{i.e.} deterministic, and the pair $(f,0)$ satisfies \ref{A2}.
    
    \item\label{(A3-R)}
    The functions $h, g$ satisfy \ref{A3} while being independent from $\Omega$, \emph{i.e.} deterministic. Moreover, we assume that they are uniformly bounded.
\end{enumerate}

Let $x\in\mathbb{R}^d$, from \cite[Theorem 4.5]{10.1214/26-AOP1852} we know that for $u\in \mathcal{A}^2$ under \ref{(A1-R)}, \ref{(A2-R)}, there exists an $\mathcal{O}'\otimes\mathcal{B}(\mathscr{C}^{0,\alpha}_g)/\mathcal{B}(\mathbb{R}^d)$-measurable version of the solution to 
\begin{equation*}
\begin{cases}
\ud X^{\boldsymbol{\eta},u}_s
=
b(s,X^{\boldsymbol{\eta},u}_s,u(\boldsymbol{\eta})_s)\,\ud s
+
\sigma(s,X^{\boldsymbol{\eta},u}_s,u(\boldsymbol{\eta})_s)\,\ud W_s
+
(f,0)(s,X^{\boldsymbol{\eta},u}_s)\ud\boldsymbol{\eta}_s,
\qquad s \in [t,T], \\
X^{\boldsymbol{\eta},u}_t
= x
\end{cases}
\end{equation*}
which is also denoted by $X^{\boldsymbol{\eta},u}.$ Thus, we have 
\begin{equation*}
    \overline{X}^{\overline{u}}(\omega,s):= X^{u(\boldsymbol{\eta})}_s(\omega',\boldsymbol{\eta})\big|_{\boldsymbol{\eta}=\textbf{W}^{\perp} (\omega'')} = X^{\boldsymbol{\eta},u(\eta)}_s(\omega')\big|_{\boldsymbol{\eta}=\textbf{W}^{\perp} (\omega'')}
\end{equation*}
defines an $(\mathcal{F}_s\vee\mathcal{F}^{\perp}_T)_{s\in[t,T]}$-progressively measurable process.
In the same spirit, we define
\begin{equation*}
\begin{aligned}
   \overline{\mathcal{A}}:= \big\{\overline{u}:\Omega'\times\Omega''\times[0,T]\rightarrow\mathbb{R}^\kappa: \text{there exists}\hspace{.1cm}u\in \mathcal{A}^{2}  \text{ such that}\hspace{.1cm} \overline{u}_s(\omega',\omega'')=u_s(\omega',\textbf{W}^{\perp} (\omega''))\big\}.
    \end{aligned}
\end{equation*}
Then, for every anticipating control $\overline u\in \overline{\mathcal{A}}$, 
$\overline{X}^{\overline{u}}$ formally defines a solution to the anticipating SDE  
\begin{equation}\label{eq. 99}
\begin{cases}
\ud \overline{X}^{\overline{u}}_s
=
b(s,\overline{X}^{\overline{u}}_s,\overline{u}_s)\,\ud s
+
\sigma(s,\overline{X}^{\overline{u}}_s,\overline{u}_s)\,\ud W_s(\omega')
+
f(s,\overline{X}^{\overline{u}}_s)\ud W^{\perp}_s (\omega''),
\quad s \in [t,T], \\
\overline{X}^{\overline{u}}_t
= x.
\end{cases}
\end{equation}

The pathwise stochastic control problem considers the cost functional
\begin{equation}\label{eq:def-bj}
    \overline{\mathcal{J}}(\overline u;t,x;\omega):= \mathbb E \bigg[  \int_t^T h(s, \overline X^{\overline u}_s, \overline u_s) \ud s + g(\overline X^{\overline u}_T) \bigg| \mathcal{F}^{\perp}_T \bigg]
\end{equation}
and the corresponding value function 
\begin{equation}\label{eq:stoch-cost}
\overline{\mathbb V}(t,x;\omega):=  \text{essinf}_{\overline u \in \overline{\mathcal A}} \overline{\mathcal{J}}(\overline u;t,x;\omega).
\end{equation}
Hence, the pathwise stochastic control problem is constituted by \eqref{eq. 99}--\eqref{eq:def-bj}, \emph{i.e.} we are looking for a control $\overline{u}\in \overline{\mathcal{A}}$ such that 
\begin{align*}
    \mathbb E \bigg[  \int_t^T h(s, \overline X^{\overline u}_s, \overline u_s) \ud s + g(\overline X^{\overline u}_T) \bigg| \mathcal{F}^{\perp}_T \bigg] = \overline{\mathbb V}(t,x;\omega).
\end{align*}

Note that different admissible controls correspond to different value functions 
\begin{align*}\label{eq:ValueFunctionsRandomisation}
\mathbb{V}^{1}(t,x;\omega)&:= \mathrm{essinf}_{u\in \mathcal{A}^1}\E\!\bigg[\int_t^T h\big(s,\overline{X}^{\overline{u}}_s,\overline{u}_s\big)\ud s + g\big(\overline{X}^{\overline{u}}_T\big)\bigg|\mathcal{F}^{\perp}_T\bigg],\\
    \mathbb{V}^{2}(t,x;\omega)&:= \mathrm{essinf}_{u\in \mathcal{A}^2}\E\!\bigg[\int_t^T h\big(s,\overline{X}^{\overline{u}}_s,\overline{u}_s\big)\ud s + g\big(\overline{X}^{\overline{u}}_T\big)\bigg|\mathcal{F}^{\perp}_T\bigg],\\
    \mathbb{V}^{\text{Causal}}(t,x;\omega)&:= \mathrm{essinf}_{u\in \mathcal{A}^{\text{Causal}}}\E\!\bigg[\int_t^T h\big(s,\overline{X}^{\overline{u}}_s,\overline{u}_s\big)\ud s + g\big(\overline{X}^{\overline{u}}_T\big)\bigg|\mathcal{F}^{\perp}_T\bigg].
\end{align*}
Then, from \cite[Remark 5.1, Theorem 7.4]{10.1214/26-AOP1852} and Remark \ref{rmk 5.3} we get under \ref{(A1-R)}--\ref{(A3-R)} that for every $(t,x) \in [0,T]\times\mathbb{R}^d$ we have
\begin{equation}\label{eq. 100}
    \overline{\mathbb{V}}(t,x;\omega) =  \mathbb{V}^{2}(t,x;\omega) =  \mathbb{V}^{1}(t,x;\omega) = \mathbb{V}^{\text{Causal}}(t,x;\omega),\hspace{.2cm} \mathbb{P}-a.s.
\end{equation}
Moreover, according to \cite[Proposition 7.2]{10.1214/26-AOP1852}, we have the following representation of $\overline{ \mathbb V}$
$$
 \overline{\mathbb{V}}(t,x;\omega)=  \textrm{inf}_{u\in\mathcal{A}}\E' \bigg[\int_t^T h\big(s,{X}^{\boldsymbol{\eta},u}_s,u_s\big)\ud s + g\big({X}^{\boldsymbol{\eta},u}_T\big) \bigg] \bigg|_{\boldsymbol{\eta}=\textbf{W}^{\perp} (\omega'')}= \mathbb{V}^{\boldsymbol{\eta}}(t,x)|_{\boldsymbol{\eta}=\textbf{W}^{\perp} (\omega'')}, \  \mathbb P-a.s.,
$$
where $\mathbb{V}^{\boldsymbol{\eta}}(t,x)$ is the value function defined in \eqref{eq:rough-vf} of the rough stochastic control problem.

\begin{thm}\label{thm-PMP-for-psc}
Let Assumptions \ref{(A1-R)}--\ref{(A3-R)}, \ref{assu: lin PMP} hold true, $({u}_s)_{s\in[0,T]}\in \mathcal{A}^2$, $H$ given by \eqref{eq Ham_lin}, and for every $\omega'' \in \Omega''$ let $({X}^{\boldsymbol{\eta},u},({X}^{\boldsymbol{\eta}, u})'),$ $(Y,Y'), Z$ be defined as in Section \ref{subsubsub 4.1.3}, with $\boldsymbol{\eta} = \textbf{W}^{\perp} (\omega'')$ and $u_s = u(\textbf{W}^{\perp} (\omega''))_s$. If the admissible control $(\overline{u}_s)_{s\in[0,T]}\in \overline{\mathcal{A}}$ is optimal for $\overline{\mathcal{J}}$, then we have $\mathbb{P}''-\text{a.s.}$ that
\begin{enumerate}
    \item[(i)] 
 $
        \partial_u H(s,\overline{X}_s^{\overline{u}}, Y_s, Z_s, \overline{u}_s) = 0,\hspace{.2cm} \mathbb{P}'\otimes\text{Leb}(\ud s)-a.s.
 $

    \item[(ii)] If additionally $H$ is convex in $u \in \mathbb{R}^\kappa$, then for every $u\in \mathbb{R}^\kappa$ we have
        \begin{equation*}
        H(s,\overline{X}_s^{\overline{u}},Y_s,Z_s,\overline{u}_s) \leq H(s,\overline{X}_s^{\overline{u}},Y_s,Z_s,u),\hspace{.2cm} \mathbb{P}'\otimes\text{Leb}(\ud s)-a.s.
        \end{equation*}
\end{enumerate}
\end{thm}

\begin{proof}
By \cite[Remark 7.3]{10.1214/26-AOP1852} we have $\mathbb{P}-\text{a.s.}$ that
\begin{align*}
    \E\!\bigg[\int_t^T h\big(s,\overline{X}^{\overline{u}}_s,\overline{u}_s\big)\ud s + g\big(\overline{X}^{\overline{u}}_T\big)\bigg|\mathcal{F}^{\perp}_T\bigg] = \E'\!\bigg[\int_t^T h\big(s,{X}^{\boldsymbol{\eta},{u}}_s,u(\boldsymbol{\eta})_s\big)\ud s + g\big({X}^{\boldsymbol{\eta},{u}}_T\big)\bigg]\bigg|_{\boldsymbol{\eta}=\textbf{W}^{\perp} (\omega'')}.
\end{align*}
Since also $\overline{\mathbb{V}}(t,x;\omega) = \overline{\mathbb{V}}(t,x;\omega'')$ we get
\begin{align*}
    S:= \bigg\{\E'\!\bigg[\int_t^T h\big(s,{X}^{\boldsymbol{\eta},{u}}_s,u(\boldsymbol{\eta})_s\big)\ud s + g\big({X}^{\boldsymbol{\eta},{u}}_T\big)\bigg]\bigg|_{\boldsymbol{\eta}=\textbf{W}^{\perp} (\omega'')} \neq  \overline{\mathbb{V}}(t,x;\omega'')\bigg\} \subseteq \Omega''.
\end{align*}
Then, by the optimality of $(u_s)_{s\in[0,T]}$  we have
$$
\mathbb{P}[S]= \mathbb{P}''[S] = 0.
$$
Thus, for every $\omega''\in S^c$ the result is immediate from Theorem \ref{thm:roughpmp_lin}.  
\end{proof}

\begin{coroll}\label{crl 5.4}
Let the Assumptions of Theorem \ref{thm-PMP-for-psc} hold true. If for every $(\omega',s) \in \Omega'\times [t,T]$ the function ${H}(s,\cdot,{Y}_s,{Z}_s,\cdot)$ is convex in $(x,u)\in \mathbb{R}^d\times \mathbb{R}^k$, for every $x\in\mathbb{R}^d$ 
\begin{align*}
g(x)-g(\overline{X}^{\overline{u}}_T)-\partial_xg(\overline{X}^{\overline{u}}_T)\cdot (x-\overline{X}^{\overline{u}}_T)\geq 0,\hspace{.2cm} \mathbb{P}-a.s.,
\end{align*}
for every $(\widetilde{X},\widetilde{X}') \in \mathbf{D}_{{W}^{\perp} (\omega'')}^{\alpha,\beta} \mathbb{L}^{k,\infty}\big(\Omega';[0,T];\mathbb{R}^d\big)$ we have $\mathbb{P}''-\text{a.s.}$ that
\begin{equation*}
    \mathbb{E}'\bigg[\int_t^T \Big({Y}_s^\top f\big(s,\widetilde{X}(s)\big) - {Y}_s^\top f\big(s,\overline{X}^{\overline{u}}(s)\big) -{Y}_s^\top\partial_xf\big(s,\overline{X}^{\overline{u}}(s)\big) \big[\widetilde{X}(s) - \overline{X}^{\overline{u}}(s)\big]\Big)
     \ud\textbf{W}^{\perp} (\omega'')_s \bigg] \geq 0,
\end{equation*}
and for the admissible control $(\overline{u}_s)_{s\in[0,T]}\in \overline{\mathcal{A}}$, for every $u\in \mathbb{R}^\kappa$ we have $\mathbb{P}''-\text{a.s.}$ that
\begin{equation*}
    H(s,\overline{X}_s^{\overline{u}},Y_s,Z_s,\overline{u}_s) \leq H(s,\overline{X}_s^{\overline{u}},Y_s,Z_s,u),\hspace{.2cm} \mathbb{P}'\otimes\text{Leb}(\ud s)-a.s.,
\end{equation*}
then $(\overline X^{\overline u}, \overline u)$ is an optimal pair for $\overline{\mathcal{J}}$.
\end{coroll}

\begin{proof}
Immediate from the equality   
\begin{align*}
    \E\!\bigg[\int_t^T h\big(s,\overline{X}^{\overline{u}}_s,\overline{u}_s\big)\ud s + g\big(\overline{X}^{\overline{u}}_T\big)\bigg|\mathcal{F}^{\perp}_T\bigg] = \E'\!\bigg[\int_t^T h\big(s,{X}^{\boldsymbol{\eta},{u}}_s,u(\boldsymbol{\eta})_s\big)\ud s + g\big({X}^{\boldsymbol{\eta},{u}}_T\big)\bigg]\bigg|_{\boldsymbol{\eta}=\textbf{W}^{\perp} (\omega'')}
\end{align*}
and Corollary \ref{coroll 4.9}.
\end{proof}

\begin{rmk}\label{rmk 5.5}
Similarly with above, one can write down the analog results corresponding to Theorem \ref{thm:roughpmp} and Corollary \ref{coroll 4.16}, given the same deterministic alterations to \ref{D-dynamics}--\ref{D-cost}.
\end{rmk}


\section*{Acknowledgments}
\smallskip
\noindent \textbf{Funding}  
PKF, UH, ST and HZ acknowledge support from 
Deutsche Forschungsgemeinschaft (DFG, German Research Foundation) – CRC/TRR 388 
"Rough Analysis, Stochastic Dynamics and Related Fields" – Project ID 516748464 (B04, B05). HZ is supported by the Fundamental Research Funds for the Central Universities and NSF of Shandong (ZR2023MA026). HZ gratefully acknowledges the support of Humboldt Universit\"at zu Berlin, where this project was commenced.

\printbibliography[heading=bibintoc]

\end{document}